\documentclass{report}
\usepackage[letterpaper, hmargin=1in, vmargin=1in]{geometry}
\usepackage{setspace}
\usepackage[T1]{fontenc}
\usepackage{amsthm}           %
\usepackage{amsmath}          %
\usepackage[toc,page]{appendix}
\usepackage{amssymb}
\usepackage{comment}
\usepackage{hyperref}
\usepackage{amsrefs}
\usepackage[table]{xcolor}
\usepackage{enumerate}
\usepackage{tikz-cd}
\usepackage{commath}

\usepackage{pdfpages}

\newtheorem{theorem}{Theorem}
\newtheorem*{theorem*}{Theorem}
\newtheorem{lemma}[theorem]{Lemma}
\newtheorem{proposition}[theorem]{Proposition}
\newtheorem{corollary}[theorem]{Corollary}

\theoremstyle{definition}
\newtheorem{definition}[theorem]{Definition}
\newtheorem{example}[theorem]{Example}

\theoremstyle{remark}

\newtheorem*{remark}{Remark}

\numberwithin{theorem}{section}
\numberwithin{equation}{section}

\newcommand{\F}{\mathcal{F}}
\newcommand{\pr}{\mathtt{pr}}
\renewcommand{\O}{\mathcal{O}}
\newcommand{\R}{\mathbb{R}}
\newcommand{\G}{\mathcal{G}}
\renewcommand{\S}{\mathbb{S}}
\newcommand{\Id}{\mathtt{Id}}
\renewcommand{\H}{\mathcal{H}}

\newcommand{\Mod}{\mathbf{Mod}}
\newcommand{\K}{\mathcal{K}}
\newcommand{\GStacks}{\mathtt{GStacks}}
\newcommand{\s}{\mathbf{s}}
\newcommand{\from}{\leftarrow}
\renewcommand{\t}{\mathbf{t}}
\DeclareMathOperator{\Aut}{Aut}
\DeclareMathOperator{\OutAut}{OutAut}
\DeclareMathOperator{\InnAut}{InnAut}
\DeclareMathOperator{\Pic}{Pic}
\DeclareMathOperator{\Diff}{Diff}
\DeclareMathOperator{\Hom}{Hom}
\DeclareMathOperator{\End}{End}
\newcommand{\Z}{\mathbb{Z}}
\newcommand{\Isom}{\mathbf{Isom}}
\newcommand{\Gr}{\mathfrak{Gr}}

\newcommand{\isom}{\cong}
\newcommand{\inv}{^{-1}}
\newcommand{\m}{\mathbf{m}}
\newcommand{\zero}{^{(0)}}

\newcommand{\two}{^{(2)}}
\newcommand{\enn}{^{(n)}}
\newcommand{\C}{\mathcal{C}}
\newcommand{\Hol}{\mathbf{Hol}}
\newcommand{\hol}{\mathbf{hol}}
\renewcommand{\u}{\mathbf{u}}
\newcommand{\Modd}{\mathrm{Mod}}
\newcommand{\Proj}{\Pr}
\newcommand{\Cy}{\mathcal{C}}
\newcommand{\aff}{\mathfrak{aff}}
\renewcommand{\i}{\mathbf{i}}
\newcommand{\g}{\mathfrak{g}}
\renewcommand{\til}{\widetilde}
\newcommand{\morpoid}{\mathtt{GMor}}
\newcommand{\dR}{\mathrm{dR}}
\newcommand{\diff}{\dif}
\newcommand{\Caff}{\mathcal{C}\mbox{aff}}
\newcommand{\caff}{\mathfrak{caff}}
\newcommand{\dd}[1]{\frac{\partial}{\partial {#1}}}
\newcommand{\Aff}{\mathrm{Aff}}
\newcommand{\into}{\hookrightarrow}
\newcommand{\Sub}{\mathbf{Sub}}

\newcommand{\darrow}{\arrow[d, shift left]\arrow[d, shift right]}
\newcommand{\laction}{\arrow[loop left]}
\newcommand{\raction}{\arrow[loop right, leftarrow]}

\newcommand{\X}{\mathcal{X}}
\newcommand{\Y}{\mathcal{Y}}

\renewcommand{\L}{\mathcal{L}}
\renewcommand{\P}{\mathcal{P}}
\newcommand{\bP}{\mathbf{P}}

\newcommand{\T}{\mathbb{T}}

\newcommand{\Zcal}{\mathcal{Z}}

\newcommand{\B}{\mathbf{B}}
\newcommand{\Man}{\mathtt{Man}}

\newcommand{\grpd}{\rightrightarrows}
\newcommand{\W}{\mathcal{W}}
\newcommand{\D}{\mathcal{D}}
\newcommand{\DMan}{\mathtt{DMan}}
\newcommand{\GRPD}{\mathtt{Grpd}}

\newcommand{\Set}{\mathtt{Set}}
\newcommand{\Top}{\mathtt{Top}}
\newcommand{\Scal}{\mathcal{S}}

\title{Stacks in Poisson Geometry}
\author{Joel Villatoro}
\date{}
\begin{document}
\maketitle
\bibliographystyle{spmpsci}      
\pagenumbering{roman}

\begin{abstract}
  This thesis is divided into four chapters. The first chapter discusses the relationship between stacks on a site and groupoids internal to the site. It includes a rigorous proof of the folklore result that there is an equivalence between the bicategory of internal groupoids and the bicategory of geometric stacks.
  The second chapter discusses standard concepts in the theory of geometric stacks, including Morita equivalence, stack symmetries, and some Morita invariants.
  The third chapter introduces a new site of Dirac structures and provides a rigorous answer to the question: What is the stack associated to a symplectic groupoid?
  The last chapter discusses a remarkable class of Poisson manifolds, called b-symplectic manifolds, giving a classification of them up to Morita equivalence and computing their Picard group.
\end{abstract}
\newenvironment{acknowledgements} {\renewcommand\abstractname{Acknowledgements}\begin{abstract}}{\end{abstract}}

\begin{acknowledgements}
  The author would like to express his deepest thanks to his Doctoral Advisor Rui Fernandes for his indispensable help and and guidance. The author would also like to extend his thanks to Eugene Lerman, Matias del Hoyo, Henrique Bursztyn, and Eva Miranda for their guidance at various points in the course of investigating the research presented here. Finally, the author would like to extend his thanks to the thesis committee: Rui Fernandes, Eugene Lerman, Pierre Albin and James Pascaleff.
\end{acknowledgements}

\spacing{1.5}
\tableofcontents
\spacing{1}

\newpage
\pagenumbering{arabic}
%
\chapter*{Introduction}
\addcontentsline{toc}{chapter}{Introduction}
\section*{Background}
One of the main goals of this thesis is to provide a rigorous foundation for the study of stacks in the context of Poisson geometry.
The motivation for this originates with the introduction of symplectic groupoids independently by Weinstein~\cite{SymplecticGroupoidW}, Karas\"{e}v~\cite{SymplecticGroupoidK}, and Zakrzewski~\cite{SymlpecticGroupoidZ} in the late 1980s.

The original intent of symplectic groupoids was as a method for attacking the quantization problem in Poisson geometry.
The hope was that it would be possible to reformulate the quantization problem of a Poisson manifold to an (easier) quantization problem on the associated symplectic groupoid.
While this has not been accomplished, many other applications of symplectic groupoids have been found towards understanding the geometry of Poisson manifolds.

A few years before symplectic groupoids were introduced, Hilsum and Skandalis~\cite{HilsumSkandalis} had noted the connections between bibundles of `foliation groupoids' and the transverse geometry of the foliation.
Such objects had a straightforward generalization to arbitrary Lie groupoids.
In the topological setting, Moerdijk~\cite{MoerdijkClass}~\cite{MoerdijkClass2}~\cite{MoerdijkClass3} related groupoid bibundles to morphisms of their associated sheaves.
This work and the subsequent work of other authors built upon this and led to the general belief that studying Lie groupoids and their bibundles is equivalent to studying geometric stacks (groupoid valued sheaves) on the site of manifolds.
From this point of view, an equivalence (called a \emph{Morita equivalence}) of Lie groupoids is a principal bibundle.

To connect these developments to Poisson geometry, consider the notion of a \emph{dual pair}~\cite{WeinsteinLocal}.
This consists of a pair of Poisson manifolds $M$ and $N$, together with a symplectic manifold $S$ equipped with Poisson maps:
\[ M \from S \to N  \]
such that the fibers of $M$ are symplectically orthogonal to the fibers of $N$.
Algebraically, a dual pair relates the Lie algebra of functions on $M$ to the Lie algebra of functions on $N$,
\[ C^\infty (M) \to C^\infty (S) \from C^\infty (N) \,  .\]
Marsden and Weinstein observed that the Poisson manifolds $M$ and $N$ seemed to be closely related.
In particular, under some connectedness and completeness conditions, they have isomorphic leaf spaces as well as isomorphic transverse geometry.
Furthermore, examples of dual pairs made frequent natural appearances in the subject.
For instance, an example of a dual pair arises from the moment map $S \to \g^*$ of a free and proper Hamiltonian actions of a Lie group on a symplectic manifold:
\[ \g^* \from S \to S/G \, . \]

The relationship between symplectic groupoids and dual pairs arises from the observation due to Xu~\cite{Morping} that a sufficiently well behaved dual pair inherits a principal bibundle structure:
\[
M \from S \to N
\quad \Rightarrow \quad
\begin{tikzcd}
  \G \darrow & S \arrow[dl] \arrow[dr] \raction \laction & \H \darrow \\
  M & & N
\end{tikzcd}
\]
In particular, $\G$ and $\H$ are Morita equivalent as Lie groupoids.
This formed the basis for Xu's definition of symplectic Morita equivalence of symplectic groupoids (and consequently Poisson manifolds).

Although Morita equivalence of Lie groupoids was motivated by a correspondence of sheaves.
It was not clear at the time how to understand symplectic groupoids sheaf theoretically and it was certainly not understood what sort of equivalence of sheaves one should expect to obtain from a symplectic Morita equivalence.

The study of symplectic Morita equivalences has been well developed since the pioneer work we have described above.
For instance, dual pairs and symplectic groupoids have appeared in classical mechanics~\cite{fluids} and deformation quantization~\cite{Karabegov}.
New Morita invariants, such as the Picard group, have been defined~\cite{BPic} and studied~\cite{Radko}.

Our goals here are twofold.
One is to develop a rigorous treatment of symplectic groupoids as presentations of stacks.
The other goal is to obtain classification results for interesting families symplectic groupoids, up to Morita equivalence. The first three chapters are concerned with the first goal, while the last chapter accomplishes the second.

\section*{Contents}
Although our primary interest is in Poisson geometry, it turns out that Poisson and symplectic geometry are not sufficiently well behaved categorically to provide a rigorous stacky treatment of symplectic Morita equivalences.
In order to resolve this, we will first make a detour into the general theory of sites, stacks, and groupoids in the first two chapters.
Thereby, we will obtain a nice criteria for exactly what sorts of categories are amenable to a theory of geometric stacks and groupoids.
The result is a general theory for understanding stacks over sites that are `similar to manifolds.'

In the third chapter we introduce and study a site called $\DMan$ which makes the stack of a symplectic groupoid rigorous.
Lastly, the fourth chapter is dedicated to solving the classification problem, up to Morita equivalence, for a special class of Poisson manifolds called b-symplectic. We will now proceed with a more detailed description of the contents of each chapter.

\subsection*{Chapter 1}
Chapter~\ref{chap:stacks} is dedicated to making the correspondence between groupoids and geometric stacks rigorous.
The setting that we will use is that of a \emph{good site} (see Definition~\ref{defn:goodsite}), which we have introduced in this thesis.

The main theorem (Theorem~\ref{thm:chapter1}) that we prove in Chapter~\ref{chap:stacks} is:
\begin{theorem*}[Fundamental theorem of geometric stacks]
  Suppose $\C$ is a good site. There is an equivalence of bicategories between the bicategory of groupoids internal to $\C$ (with bibundles as morphisms) and the bicategory of geometric stacks on $\C$.
\end{theorem*}
We should make a few comments about the context of this result.
The above theorem is false for an arbitrary (i.e. not necessary good) site.
We define a site $\C$ to be \emph{good enough} if the above theorem holds.
The the content of our work is therefore to prove that a good site is good enough, where good is a set of criteria that is fairly reasonable for geometric categories. For example, the sites of schemes, manifolds, and topological spaces are all good.

In the case that $\C$ is the site of smooth manifolds, the theorem roughly says that studying Lie groupoids and bibundles is equivalent to studying geometric stacks on the site of manifolds. This is fairly well known and the earliest sketch of the result is in Blohmann~\cite{BlohmannSLGs}. Later, Carchedi~\cite{carchedi} included a complete proof in his thesis. We will see in Chapter~\ref{chap:dman} that there is a good reason for us to generalize this result to sites beyond manifolds.

\subsection*{Chapter 2}
Chapter~\ref{chap:morita} continues the work of Chapter~\ref{chap:stacks} with a more in depth look at equivalences of stacks (Morita equivalence).
In particular, it develops the relationship between morphisms of groupoids internal to $\C$ and groupoid bibundles.
The key concept is that of a \emph{weak equivalence} of groupoids in $\C$.
The main theorem of the chapter is Theorem~\ref{thm:weakequivalences} which says:
\begin{theorem*}
  The following are equivalent:
  \begin{itemize}
    \item $\G$ and $\H$ are Morita equivalent as $\C$-groupoids.
    \item There exists a $\C$-groupoid $\K$ and a pair of weak equivalences $\H \from \K \to \G$.
  \end{itemize}
\end{theorem*}
When $\C$ is the site of smooth manifolds, this fact is well known.
Our main contribution is the observation that the result holds for any (good) site.
This theorem is a weaker version of a stronger fact that (in some sense) bibundles are the universal object which inverts weak equivalences and the construction of a (bicategory) of fractions. For more details on these ideas see Pronk~\cite{PronkBicat} or Pradines~\cite{PradinesFractions}. More recent work on the subject, and in a fairly general setting was done by Meyer and Zhu~\cite{MeyerZhu} as well as Roberts~\cite{RobertsAnafunctors}.
This result has significant utility in the study of Morita invariants since it is often-times much easier to prove invariance of a property under weak equivalence than bibundles.

\section*{Chapter 3}
Chapter~\ref{chap:dman} finally relates all of our work to Poisson geometry by introducing the good site of Dirac manifolds $\DMan$.
Dirac manifolds were originally described by Theodore Courant~\cite{Courant}.
They unify Poisson structures, foliations, and symplectic forms as Dirac structures.
Intuitively speaking, a Dirac structure should be thought of as a (possibly singular) foliation of the manifold by pre-symplectic manifolds.
Although we will state all of the most important details about Dirac structures, we refer the interested reader to \cite{BDiracintro} or \cite{MMbook} for more detail.

For us, Dirac structures are a convenient relaxation of the notion of a Poisson structure to improve categorical behavior.
The category Poisson manifolds is rather poorly behaved and a symplectic groupoid is not a groupoid internal to Poisson manifolds.
However, in $\DMan$ we have D-Lie groupoids and the punchline of the chapter is the Theorem~\ref{thm:main1} which says:
\begin{theorem*}
Let $\G$ and $\H$ be symplectic groupoids. Then $\G$ and $\H$ are also D-Lie groupoids and the following are equivalent:
\begin{enumerate}[(1)]
\item $\G$ and $\H$ are Morita equivalent as symplectic groupoids.
\item $\B\G$ is isomorphic to $\B\H$.
\item There exists a principal $(\G,\H)$-bibundle of D-Lie groupoids.
\item There exists a pre-symplectic groupoid $\G'$ and a pair of weak equivalences of D-Lie groupoids $\G \from \G' \to \H$.
\end{enumerate}
\end{theorem*}
At first sight it may seem that this theorem follows as an immediate corollary of the main theorems of the first two chapters.
However, there is some work to be done in order to show that $\DMan$ actually does what we want.
The main reason this all works is the, slightly surprising, Proposition~\ref{prop:symform}, which relates Morita equivalenc of D-Lie groupoids to the existing notion of symplectic Morita equivalence.
Throughout the chapter, we also take the time to give geometric characterizations of groupoids and bibundles in $\DMan$, as the categorical definition is not particularly amenable to doing geometry.

\subsection*{Chapter 4}
Chapter~\ref{chap:bsymp} shifts gears from theory building to computation.
The main goal is to classify, up to Morita equivalence, a special class of Poisson manifolds called b-symplectic (or log-symplectic).
A b-symplectic manifold is a type of Poisson manifold which is a mild degeneration of the notion of a symplectic manifold, where the symplectic form has a log type singularity along a prescribed hypersurface. The relative tameness of b-symplectic manifolds means that they form a tractable class for the explicit calculation of many invariants of Poisson manifolds (see \cite{GMP1}, \cite{BDirac}, \cite{Radko} and \cite{Radko2}).

In particular, the Picard group of a b-symplectic compact surface $M$ was computed by Radko and Shlyakhtenko in \cite{Radko2}. The goal of the chapter is to prove a generalization of their result to arbitrary (even) dimension: we calculate the Picard group of any \emph{stable} b-symplectic manifold. Every b-symplectic structure on an orientable compact manifold can be perturbed to a stable one, so in this sense stable structures are fairly generic.

For any stable b-symplectic manifold $M$ we will construct a collection of discrete data $\Gr$ called a \emph{discrete presentation} of $M$. The discrete presentation is a combinatorial object which takes the form of a heavily decorated graph that encodes the topological configuration of the symplectic leaves of $M$.
This graph resembles the data that was used previously in the aforementioned calculation of the Picard group of a compact surface and by Gualtieri and Li to classify integrations of b-symplectic manifolds~\cite{Gualt}.
The edges of the graph $\Gr$ represent the connected components of the singular locus of $M$ while the vertices represent the orbits.
The decorations take the form of the fundamental groups of open orbits $\pi_1(U)$, the fundamental groups of symplectic leaves $\pi_1(L)$ of the singular locus and homomorphisms between them.

The main results of the chapter are the following two theorems. The first says:
\begin{theorem*}
Suppose $M$ is a stable b-symplectic manifold and $\mathfrak{Gr}$ is a discrete presentation of $M$. Then:
\[ \Pic(M) \isom \left( \OutAut(\mathfrak{Gr}) \ltimes \R^N \right)/ H,  \]
where $H\subset \OutAut(\mathfrak{Gr}) \ltimes \R^N$ is a discrete normal subgroup.
\end{theorem*}
We give an explicit description of $H$ and the action of $\R^N$ on $\OutAut(\Gr)$. Of course, to make sense of this result we will define isomorphisms and inner automorphisms of discrete presentations. It turns out that isomorphisms of discrete presentations are a powerful tool for the classification of stable b-symplectic structures. This is the content of our second main theorem:
\begin{theorem*}
Suppose $M_1$ and $M_2$ are stable b-symplectic manifolds and $\Gr_1$ and $\Gr_2$ are discrete presentations of each, respectively. Then $M_1$ and $M_2$ are Morita equivalent if and only if there exists an isomorphism $\Gr_1 \to \Gr_2$.
\end{theorem*}
The relatively simple statements of Theorem \ref{maintheorem2} and Theorem \ref{maintheorem1} are somewhat deceptive as the definition of an isomorphism of discrete presentations $\Gr_1 \to \Gr_2$ is not so straightforward. On the other hand, the data can often be simplified when computing specific examples (see Section \ref{section:examples4}).

\chapter{Sites and stacks}\label{chap:stacks}
In this chapter, we will give a general treatment of stacks on a site.
Our aim is to develop a general theory of geometric stacks which resembles that of stacks over differentiable manifolds.
The focus of our treatment is to relate geometric stacks on a well behaved site $\C$ to groupoids internal to $\C$.

A significant background in category theory is not necessary to understand our work here.
However, it is not possible (or advisable) to completely avoid category theory and we will apply simple categorical reasoning when appropriate.
After all, the objects of study (stacks) are inherently categorical.
The author also hopes that the following will provide a reasonable bridge between category theory and geometry and hence we use a language that is intelligible to a person in either subject.
This chapter will proceed as follows:
\begin{itemize}
  \item In Section~\ref{section:categories_and_topology} we will introduce our notation and recall the basic definitions for categories as well as establish the notion of a category with topology.
  \item In Section~\ref{section:stacks} we will proceed to define categories fibered in groupoids and stacks and provide a precise definition for what we mean by \emph{geometric} stack.
  \item In Section~\ref{section:groupoids} we will define what we mean by a groupoid internal to a site. We will then proceed to define the functor $\B$ which will relate groupoids to geometric stacks. This section will conclude with a proof of Theorem~\ref{thm:chapter1} which states that $\B$ is an equivalence.
  \item In Section~\ref{section:examples1} we will take a look at a few simple example of sites to which we can apply our results.
\end{itemize}
The general strategy we will use is piecemeal borrowed from a variety of expository texts on the subject. In particular, our work is guided by Xu and Behrend~\cite{PXstacks} as well as Metlzer~\cite{Mstacks} which have both written excellent treatments of the subject.
Lerman~\cite{LOrbifolds} and del Hoyo~\cite{MR3089760} also provided the author with significant inspiration and intuition. Lastly, another short treatment of the main theorem of this chapter, in the setting of smooth manifolds, can be found in Blohmann~\cite{BlohmannSLGs}.
\section{Categories and topology}\label{section:categories_and_topology}
\subsection{Notation}
Categories will play a central role in both the study of groupoids and sites. It will be convenient to recall our basic definitions and notation.

\begin{definition}\label{defn:category}
A \emph{category} $\C$ is composed of a collection $\C_0$ called the objects, and a collection $\C_1$ called the morphisms, together with the following:
\begin{itemize}
\item to every morphism $f \in \C_1$ we associate two objects $M$ and $N$ called the \emph{source} and \emph{target} of $f$ respectively. This situation is summarized by the notation: $f\colon M \to N$;
\item to every object $M \in \C_0$ we associate a morphism $\Id_M$ called the \emph{unit} of $M$;
\item to every pair of morphisms $f\colon M \to N$ and $g \colon N \to O$ we associate a morphism $g \circ f\colon M \to O$.
\end{itemize}
This data must satisfy the following axioms:
\begin{enumerate}[(C1)]
\item the source and target of $\Id_M$ is $M$,
\item the source of $f \circ g$ is the source of $f$,
\item the target of $f \circ g$ is the target of $f$,
\item given any $f\colon M \to N$ and $\Id_M\colon M \to M$ then $f \circ 1_M =f$,
\item given any $f\colon M \to M$ and $\Id_M\colon y \to y$ then $\Id_M \circ f = f$,
\item given any three morphisms $f,g$ and $g$ then $(h \circ g) \circ f = h \circ (g \circ f)$ whenever defined.
\end{enumerate}
\end{definition}
\begin{example}[Sets]
The category whose objects are sets and whose morphisms are set theoretic mappings will be denoted by $\Set$.
\end{example}
\begin{example}[Topological spaces]
The category whose objects are topological spaces and whose morphisms are continuous mappings will be denoted by $\Top$.
\end{example}
\begin{example}[Smooth manifolds]
The category whose objects are smooth manifolds and whose morphisms are smooth maps will be denoted by $\Man$.
\end{example}
\begin{example}[Open sets]
  Given a topological space $X$, then we will denote the category whose objects are open subsets of $X$ and morphisms are given by subset inclusions.
\end{example}
\begin{example}[Opposite category]
  Suppose $\C$ is any category. Then we can construct a new category $\C^{op}$ called the \emph{opposite category} for which the source and target are swapped, and the order of composition is reversed.
\end{example}
Given any two objects, $M$ and $N \in \C_0$, we denote by $\Hom(M,N)$ the subcollection of $\C_1$ which consists of morphisms whose source is $M$ and target is $N$. Sometimes we may write $\End(M)$ to denote $\Hom(M,M)$ the \emph{endomorphisms} of $M$.

Given an arbitrary object $M \in \C$, then a \emph{generalized element} $x$ is a morphism $x$ whose codomain is $M$.
Generalized elements and morphisms will be distinguished by our notation.
For example, generalized elements of $M$ will usually be denoted by $x$ or $y$ whereas morphisms in $\C$ are denoted by $f$ or $g$.
One can `push forward' generalized elements along a morphisms  $f \colon M \to N$ by taking $f(x) := f \circ x$ which is a generalized element of $N$.
\begin{example}[Slice category]
  Given any category $\C$ and an object $M$ in $\C_0$, one can define a new  category $\C/M$ called the \emph{slice category of} $M$.
  An object in $\C/M$ is a morphism $f\colon N \to M$.
  A morphism in $\C/M$ is a commutative triangle:
  \begin{equation}\label{eqn:slice}
  \begin{tikzcd}
    N_2 \arrow[r, "f_1"] \arrow[d, "g"] & M \arrow[d, "1_M" swap] \\
    N_2 \arrow[r, "f_2"] & M \\
  \end{tikzcd}
\end{equation}
where the source is $f_1$, the target is $f_2$ and composition is defined by composing in $\C$ vertically.
\end{example}
\begin{definition}
A \emph{commutative} square consists of four morphisms illustrated below:
\[
\begin{tikzcd}
  M \arrow[r, "f"] \arrow[d, "h"] & N \arrow[d, "g"] \\
  O \arrow[r, "k"]                & P
\end{tikzcd}
\]
such that $g \circ f = k \circ h$. A commutative square is a \emph{pullback square} if for any $f' \colon M' \to N$ and $h' \colon M' \to O$ there is a unique morphism $e \colon M' \to M$ such that $f \circ e = f'$ and $g \circ e = g'$.

The pullback property implies that $g$ and $k$ define $M$ up to a unique isomorphism. For this reason we may write $M = O \times_{g,k} N$ or $M = O \times_P N$ and call $M$ the \emph{fiber product} of $k$ and $g$. Following this notational convention, we may use $f' \times h'$ or $(f', h')$ to denote the unique morphism defined by the universal property (called $e$ above). $\pr_1$ and $\pr_2$ will denote the projection maps $O \times_P N \to O$ and $O \times_P N \to N$, respectively.

For most categories, when fiber products exists there is a canonical such object. However, the universal property only guarantees uniqueness up to a unique isomorphism. To justify our use of the notation $O \times_{g,k} N$ we will assume that for any given category, we have made a choice of fiber product for every pair of morphisms $g$ and $k$ such that the fiber product exists.
For example, when $C = \Set$ then the most natural choice is the set $\{ (x,y) : g(x) = k(y) \}$.
\end{definition}
\begin{definition}
  Given a morphism $f \colon M \to N$ in a category $\C$, we say that $f$ is an \emph{isomorphism} if there exists a morphism such that $ f \circ g = \Id_N$ and $g \circ f = \Id_M$. In such a case, we may write $g = f\inv$.

  A morphism is $f \colon M \to N$ is called \emph{cartesian} if for all $g \colon M' \to N$ the fiber product $M' \times_M M$ exists. How strong the cartesian condition is depends on how many fiber products the category admits. For instance, in the category of manifolds the cartesian maps are the (local) submersions. On the other hand, in the category of sets every morphism is cartesian.
\end{definition}

\begin{definition}\label{defn:functor}
  A \emph{functor} $F\colon \C \to \D$ is a pair of mappings of collections $F_i \colon \C_i \to \D_i$, $i=0,1$ such that $F$ respects the source, target, and composition relationships. That is:
  \[  \left(f\colon M \to N \right)\, \mapsto \,  \left(F_1(f)\colon F_0(M) \to F_0(N)\right) \quad \mbox{ and } \quad F_1(f \circ g) = F_1(f) \circ F_1(g). \]
  For general categories, we will typically omit the subscripts in $F_0$ and $F_1$ for brevity.

  A functor $F\colon \C \to \D$ is called
  \begin{itemize}
    \item \emph{full} if the mapping $\Hom(M,N) \to \Hom(F(M),F(N))$ is surjective,
    \item \emph{faithful} if that same mapping is injective,
    \item \emph{fully faithful} if both of the above hold,
    \item \emph{surjective on objects} if $F_0$ is surjective,
    \item \emph{essentially surjective on objects} (or just \emph{essentially surjective}) if every object in $\D$ is isomorphic to one in the image of $F_0$,
    \item an \emph{isomorphism} if there exists another functor $G\colon \D \to \C$ such that $F \circ G$ and $G \circ F$ are identity functors, \item an \emph{equivalence of categories} if it is both essentially surjective and fully faithful.
  \end{itemize}
\end{definition}
\begin{remark}
  We do not assume that our categories are \emph{concrete} (meaning that they have a fully faithful functor into the category of sets).
  For this reason, all definitions of morphisms or proofs of equalities of morphisms should technically be done without reference to `points' or `elements'. However, the notation for composition of morphisms can be cumbersome and most computations done in terms of elements can be turned into computations involving compositions of morphisms in some category.

  Our use of the generalized element notation helps to alleviate this problem.
  For instance, if we want to define a morphism $f \colon M_1 \times_X M_2 \to N_1 \times_Y N_2$.
  Then we may write:
  \[ f(m_1,m_2) := (f_1(m_1),f_(m_2)) \, . \]
  Formally, this means $f := (f_1 \circ \pr_1) \times (f_2 \circ \pr_2)$. The expression can be made precise if we think of $m_1$ as the generalized element $\pr_1 \colon M_1 \times_X M_2 \to M_1$ and $m_2$ as the generalized element $\pr_2 \colon M_1 \times_X M_2 \to M_2$.
\end{remark}
\subsection{Grothendieck topology}
The categories that we are interested in come with topological structures.
The concept of a Grothendieck topology allows us to explore the consequences of such a structure in the abstract.
We will begin with the concept of a sieve, which is just a collection of morphisms closed under right composition. Throughout this subsection, $\C$ is a fixed category.

\begin{definition}\label{sieve}
  A \emph{sieve $S$ on an object $M \in \C_0$} is a (possibly empty) collection of morphisms whose target is $M$, such that $S$ is closed under precomposition. That is, for all $f \in S$ and $g \in \C_1$ then $f \circ g \in S$ whenever it is defined.
\end{definition}
Let us make a few remarks about sieves:
  \begin{itemize}
    \item An arbitrary subcollection $S' \subset \Hom(-,M)$ defines always defines a sieve $S := \{ f \circ g : f \in S', \, g \in \C_1 \}$. In such a case we say that $S$ is \emph{generated by} $S'$.
    \item Given a sieve $S$ on $M$ and a morphism $g \colon N \to M$, the collection $g^* S := \{ f \colon g \circ f \in S \}$ is called the \emph{pullback sieve} along $g$.
    \item Given two sieves $S$ and $T$ on $M$, the intersection $S \cap T$ is also a sieve.
    \item Give a sieves $S = \{ s_i \colon U_i \to M \}_{i \in I}$ and $T = \{ t_j \colon V_j \to M \}_{J \in J}$ let $S \times_M T$ denote the collection of fiber products $\{s_i \times_M t_j \}_{(i,j) \in I \times J}$.
    $S_1 \times_M S_2$ is a sieve. However, if either $S_1$ or $S_2$ is generated by cartesian morphisms, then $S_1 \times_M S_2$ generates $S_1 \cap S_2$.

    When $S = T$ we denote the elements of $S \times_M S$ by $\{ s_{i j} \colon U_{i j} \to M \}$. We extend this notation to higher products as well, where $s_{ijk} \colon U_{ijk} \to M$ denotes the map $U_i \times_M U_j \times_M U_k \to M$.
  \end{itemize}
  We can now define a Grothendieck topology. Roughly speaking, a Grothendieck topology is a collection of distinguished sieves which are designated covering sieves.
  Covering sieves should roughly behave like a sieve generated by open embeddings whose images cover the target.
  \begin{definition}\label{defn:grotop}
    A Grothendieck topology on a category $\C$ is an assignment to each object $M$, a collection of subsets of $\Hom(-,M)$ called \emph{covering sieves}. This assignment must satisfy the following properties:
  \begin{enumerate}[(T1)]
      \item The whole of $\Hom(-,M)$ is a covering sieve of $M$.
      \item If $S$ is a covering sieve of $M$ and $g\colon Y \to M$ is any morphism, then the pullback $g^*S$ is a covering sieve of $Y$.
      \item Suppose $S$ is a covering sieve on $M$ and $T$ is an arbitrary sieve on $M$ such that $g^*T$ is a covering sieve for all $g \in S$. Then $T$ is a covering sieve.
  \end{enumerate}
  \end{definition}
  The condition (T2) should be interpreted as being analogous to the fact that the inverse image of an open cover is an open cover. (T3) is corresponds to the observation that a collection of subsets is an open cover if and only if it covers every element of a covering. In practice, defining a Grothendieck topology is often easier to do in terms of generators called covering families.
\begin{definition}\label{defn:pretop}
Let $\C$ be a category.
A Grothendieck pre-topology is an assignment to each object $M$ in $\C$ of a collection of subcollections of the set $\Hom(-,M)$ called \emph{covering families}.
This assignment must satisfy some properties.
\begin{enumerate}[(PT1)]
\item If $f\colon N \to M$ is an isomorphism then $\{ f \}$ is a covering family.
\item If $\{ u_i\colon N_i \to M \}$ is a covering family then each $u_i$ is cartesian and $\{ \pr_2\colon N_i \times_M N \to N \}$ is a covering family of $N$.
\item If ${\{ u_i\colon N_i \to M \}}_{i \in I}$ is a covering family of $M$ and ${\{ N_{ij} \to N_i \}}_{j \in J_i}$ is covering family of $N_i$ for each $i$, then the compositions $\{ N_{ij} \to N_i \to M \}$ constitute a covering family of $M$.
\end{enumerate}
To a pre-topology, we associated a Grothendiek topology by defining the covering sieves to be those sieves which are generated by covering families.
\end{definition}
\begin{example}
  We can equip the category of open subsets of a topological sets $X$ with a pre-topology by defining $\{ u_i: U_i \into U \}$ be a covering family if $\bigcup U_i = U$. That is, a covering family is a covering in the conventional sense.
\end{example}
\begin{example}\label{example:pretopman1}
If $\C = \Man$, then we can give $\Man$ the following pre-topology: a covering $\{ u_i\colon U_i \to M \}$ of manifold $M$ is a collection of \'etale smooth maps $u_i$ whose images cover $M$.
This is the same pre-topology used in~\cite{PXstacks}.
\end{example}
\begin{example}\label{example:pretopman2}
If $\C = \Man$, then we can give $\Man$ the following pre-topology: a covering $\{ u_i\colon U_i \to M \}$ of manifold $M$ is a collection of open embeddings $u_i$ whose images cover $M$.
\end{example}

While Example~\ref{example:pretopman1} and Example~\ref{example:pretopman2} are different pre-topologies, it is easy to check that they generate the same Grothendieck topology.

\begin{definition}\label{defn:site}
  A \emph{site} is a category $\C$ together with a Grothendieck topology
\end{definition}
\section{Stacks}\label{section:stacks}

\subsection{Fibered categories}
There are a few different ways to present the subject, but for our purposes a fibered category will be a functor into $\C$ satisfying certain properties.
We will typically denote the domain of this functor with $\X$, $\Y$, or $\Zcal$. Objects in these categories will be typically denoted with upper case letters $X$, $Y$ and $Z$, while morphisms will be denoted $a$, $b$, and $c$.
\begin{definition}\label{defn:cfg}
A \emph{category fibered in groupoids} (abbreviated CFG) over a category $\C$ is a category $\X$ together with a functor $\pi\colon \X \to \C$ such that the following properties hold.
\begin{enumerate}[(CFG1)]
\item Given any morphism $f\colon M \to N$ in $\C$ and object $Y$ in $\X$ such that $\pi(Y) = N$, then there exists a morphism $a\colon X \to Y$ in $\X$ such that $\pi(a)=f$.
\[
\begin{tikzcd}
   \exists X \arrow[dashed, r, "\exists a"] \arrow[dash, d] & Y \arrow[dash, d] \\
   M \arrow[r, "f"] & N
\end{tikzcd}
\]
\item Given morphisms $f\colon M_1 \to M_2$ and $g\colon M_2 \to M_3$ in $\C$ together with $a\colon X \to Z$ and $b\colon Y \to Z$ such that $\pi(a) = g \circ f$ and $\pi(b) = g$, then there exists a unique morphism $c$ such that $\pi(c) = f$ and $b \circ c = a$.
\[
\begin{tikzcd}
  X \arrow[dash, d] \arrow[rr, bend left, "a"] \arrow[r, dashed, "\exists! c"] & Y \arrow[dash, d] \arrow[r, "b"] & Z\arrow[dash, d] \\
  M_1 \arrow[r, "f"]& M_2 \arrow[r, "g"] & M_3
\end{tikzcd}
\]
\end{enumerate}
\end{definition}
\begin{example}
  Given any object $M$ in $\C$, let $\bar{M}$ denote the slice category of $M$. Let $\pi\colon \bar{M} \to \C$ be the functor:
  \[
  \left(\begin{tikzcd}
    N_1 \arrow[d, "g"]\arrow[r, "f_1"] & M \arrow[d, equal] \\
    N_2 \arrow[r, "f_2"] & M
  \end{tikzcd}\right)
  \mapsto g \colon N_1 \to N_2
  \]
  It is straightforward to verify that this functor satisfies the axioms of a category fibered in groupoids.
  \end{example}
  \begin{example}
    Let $F \colon \C^{\mathbf{op}} \to \mathbf{Set}$ be a presheaf. That is, $F$ is a contravariant functor from $\C$ to the category of sets. Then we can define a CFG over $\C$ as follows. The object collection of $\X$ is defined to be the disjoint union of the sets $F(M)$ for each object $M$ in $\C$. There is a morphism $a \colon X \to Y$ in $\X$ if and only if there exists an $f$ such that $F(f)(Y) = X$.
  \end{example}
The name `fibered in groupoids' arises from the following observation.
Fix an object $M$ of $\C$ and consider the subcategory $X_M$ which is defined as below:
\[ (\X_M)_0 := \{X \in \X_0 : \pi(X) = M \} \qquad (\X_M)_1 := \{ a \in \X_1 : \pi(a) = 1_M \} \, . \]
\begin{lemma}\label{lemma:CFGgroupoid}
  $\X_M$ is a groupoid. That is, every morphism in $\X_M$ has an inverse.
\end{lemma}
\begin{proof}
  Let $a\colon X \to Y$ be a morphism in $\X_M$. Then consider the following diagram:
  \[
  \begin{tikzcd}
    Y \arrow[dash, d] \arrow[rr, bend left, "1_Y"]  & X \arrow[dash, d] \arrow[r, "a"] & Y\arrow[dash, d] \\
    M \arrow[r, "1_M"]& M \arrow[r, "1_M"] & M
  \end{tikzcd}
  \]
  (CFG2) says that there exists a unique $b: Y \to X$ such that $a \circ b = 1_Y$.
  This shows that $a$ admits a left inverse. On the other hand, consider the related diagram:
  \[
  \begin{tikzcd}
    X \arrow[dash, d] \arrow[rr, bend left, "a"]  & X \arrow[dash, d] \arrow[r, "a"] & Y\arrow[dash, d] \\
    M \arrow[r, "1_M"]& M \arrow[r, "1_M"] & M
  \end{tikzcd}
  \]
  We can see immediately that both $b \circ a$ and $1_X$ satisfy the existence invoked in (CFG2) and, since such a morphism must be unique, they are equal.
\end{proof}
The objects guaranteed by (CFG1) are called a \emph{pullback} of $Y$ along $f$ and we will denote them by $f^*Y$ or $Y|_M$.
Due to the fact that the pullback is only unique up to a unique isomorphism, this is a slight abuse of notation.
Generally, in the following text one should interpret statements about $f^* Y$ and $Y|_M$ as holding for an arbitrary choice of pullback representative.
Hence, we are claiming implicitly that our arguments and definitions are invariant under this choice.

Linguistically, we think of a CFG $\X$ as sitting `above' $\C$.
Hence a morphism in $a \in \X$ is said to \emph{over} a morphism $f$ in $\C$ if $\pi(a) = f$.
Similarly, an object $X \in \X$ is said to be \emph{over} $\pi(x) \in \C$.
\begin{lemma}
Given $\pi(Y) = N$ and $f: M \to N$, then $f^* Y$ is unique up to a unique isomorphism. That is, suppose we are given $a_1 \colon X_1 \to Y$ and $a_1 \colon X_1 \to Y$ such that $\pi(a_i) = f$ for $i = 1,2$.
Then there exists a unique isomorphism $b: X_1 \to X_2$ such that $a_2 \circ b = a_1$.
\end{lemma}
\begin{proof}
  Suppose there exists $a_1 \colon X_1 \to Y$ and $a_2\colon X_2 \to Y$ such that $\pi(a_1) = \pi(a_2) = f$. Then we have a diagram:
  \[
  \begin{tikzcd}
    X_1 \arrow[dash, d] \arrow[rr, bend left, "a_1"]  & X_2 \arrow[dash, d] \arrow[r, "a_2"] & Y\arrow[dash, d] \\
    M \arrow[r, "1_M"]& M \arrow[r, "f"] & N
  \end{tikzcd}
  \]
  Then (CFG2) tells us that there exists a unique morphism $b$ such that $a_2 \circ b = a_1$.
  Since $b \in \X_M$, Lemma~\ref{lemma:CFGgroupoid} tells us that $b$ is an isomorphism.
\end{proof}
\begin{definition}\label{defn:morphismofcfgs}
A morphism of CFGs $\F\colon \X \to \Y$ is a functor which commutes with the projections to $\C$.
A morphism of CFGs is called an \emph{isomorphism} if it is an equivalence of categories.
A 2-morphism $\eta\colon \F_1 \to \F_2$ is a (necessarily invertible) natural transformation of functors.
\end{definition}
\begin{definition}
  A CFG is called \emph{representable} if it is isomorphic to $\bar M$ for some object $M \in \C$.
\end{definition}
Formally, this notion of morphisms and 2-morphisms makes CFGs over $\C$ into a strict bicategory (see Definition~\ref{defn:strict2category}). We will keep the more cumbersome technical results regarding bicategories in the appendix for the sake of brevity.
The above bicategory data is clearly coherent since it is strict.
Bicategories have their own version of the fiber product which we call the \emph{homotopy fiber product} and is denoted with $\til\times$ (see Definition~\ref{defn:2pullback}).

\subsection{Stacks}\label{subsection:stacks}

A stack is just a CFG over a site which satisfies some `gluing' axioms. These axioms can roughly be translated as saying that both morphisms and objects are constructed from local data. Before we state the definition, we should clarify a minor bit of notation:

Let $\X$ be a CFG over $\C$. Let $u \colon U \to M$ be a morphism in $\C$
 and $\phi\colon P \to Q$ be a morphism in $\X$ covering $\Id_M$. Given pullbacks $P|_U$ and $Q|_U$, then we denote by $\phi|_U: P|_U \to Q|_U$ the unique morphism such that:
  \[
  \begin{tikzcd}
    P \arrow[r, "\phi"] & Q \\
    P|_U \arrow[u] \arrow[r, "\phi|_U"] & Q|_U \arrow[u]
  \end{tikzcd}
  \]
  commutes.
The existence and uniqueness of this morphism is an immediate consequence of (CFG2).

\begin{definition}\label{defn:stack}
A CFG $\X$ over a site $\C$ is called a \emph{stack} if it satisfies:
\begin{enumerate}[(S1)]
  \item Given the following data:
    \begin{itemize}
      \item a covering sieve ${\{ s_i \colon U_i \to M \}}_{i \in I}$,
      \item objects $P$ and $Q$ in $\X_M$,
      \item morphisms $\phi_i \colon P|_{U_i} \to Q|_{U_i}$ covering the identity,
    \end{itemize}
    such that for all $s_i$ and $s_j$:
  \[ \phi_i|_{U_{ij}} = \phi_j|_{{U_{ij}}} \, . \]
  There exists a unique morphism $\phi: P \to Q$ such that $\phi|_{P|_{U_i}} = \phi_i$.
\item Given the following data:
      \begin{itemize}
        \item a covering family ${\{ s_i \colon U_i \to M \}}_{a \in A}$ of an object $M$ in $\C$,
        \item a collection of objects ${\{ P_i \in \X_{U_i} \}}_{i \in A}$,
        \item morphisms $\phi_{ij}: P_j|_{U_{ij}} \to P_i|_{U_{ij}}$
      \end{itemize}
        such that, for all $i,j,k \in I$, $\phi_{ij}|_{U_{ijk}} \circ \phi_{jk}|_{U_{ijk}} = \phi_{ik}|_{U_{ijk}}$.

        There exists an object $P$ over $M$ together with morphisms ${\{ \phi_i: P|_{U_i} \to P_i \}}_{i \in A}$ such that
        \[ \phi_{ij} \circ \phi_{j}|_{U_{ij}} = \phi_i|_{U_{ij}}. \]
\end{enumerate}
A CFG which satisfies at least (S1) is called a \emph{pre-stack}.
\end{definition}
\begin{example}
  Let $\C$ be any site. Then $\Id_\C \colon \C \to \C$ is a stack. It is representable if and only if $\C$ has a terminal object.
\end{example}
\begin{example}
  Let $\C$ be any site, then any representable CFG $\X \isom \C/M$ is a stack if and only if $\Id \colon \C \to \C$ is a (pre-)stack.
\end{example}
\begin{example}
  Let $\C$ be the site of smooth manifolds and let $G$ be a Lie group. Then $\B G$, the category of principal $G$-bundles, is a stack.
\end{example}
\begin{example}
  Let $F\colon \C^{op} \to \mathbf{Set}$ be a presheaf. Then the associated CFG $\X$ is a pre-stack if an only if it satisfies the locality axiom, i.e.:
  \[ s|_{U_i} = t|_{U_i} \quad \forall U_i \quad \Rightarrow \quad s = t \, . \]
  Furthermore $F$ is a sheaf if and only if $\X$ is a stack.
\end{example}
\subsection{Morphisms of stacks}

Morphisms of stacks are defined to be morphisms of the underlying CFGs.
In this subsection we will review a few important classes of such morphisms.
\begin{definition}
  Let $F: \X \to \Y$ be a morphism of CFGs over a site.
  \begin{itemize}
  \item We say $F$ is \emph{locally essentially surjective} or an \emph{epimorphism} if for any object $Y \in \Y_M$, there is a covering sieve ${\{ U_a \}}_{a \in A}$ of $M$ such that each $Y|_{U_a}$ is isomorphic to an object in the image of $F$.
  \item We say $F$ is \emph{representable} if for any morphism $\bar{M} \to \Y$ with $M \in \C$, the fiber product $\X \til\times_\Y \bar{M}$ is representable.
  \item We say $F$ is a \emph{submersion} if it is a representable epimorphism.
  \item We say $F$ is a \emph{monomorphism} if it is fully faithful.

\end{itemize}
\end{definition}

There is a one-to-one correspondence between morphisms of representable stacks $F: \bar{M} \to \bar{N}$ and morphisms $F'\colon M \to N$ in $\C$. In other words, the functor $M \mapsto \bar{M}$ is fully faithful (the Yoneda lemma).
We should clarify that that for our examples, the term submersion will actually correspond to \emph{surjective} submersions.
\begin{example}\label{ex:morphismsofreps}
  Let $f\colon M \to N$ be a morphism in $\C$ and let $\bar{f}\colon \bar M \to \bar N$ be the associated morphism of representable CFGs. Then $\bar f$ is an epimorphism if and only if $f$ admits local sections. That is there exists a covering $S = \{ s_i \colon U_i \to M \}$ of $N$ together with morphisms $\{ \sigma_i \colon U_i \to M \}$ such that $f \circ \sigma_i = s_i$ for all $s \in S$.
  \[ \begin{tikzcd}
    & U_i \arrow[dr, "{s_i}"] \arrow[dl, dashed, "{\exists \sigma_i}" swap] & \\
    M \arrow[rr, "f"] & & N
  \end{tikzcd} \]

  The functor $\bar f$ is always a monomorphism. Lastly, $\bar{f}$ is representable if and only if it is cartesian. Generally, we will say that a morphism $f$ in a site $\C$ is a submersion if and only if $\bar{f}$ is a submersion.
\end{example}

\begin{definition}
  A stack $\X$ is \emph{geometric} or \emph{presentable} if there exists a submersion (i.e. a representable epimorphism) $\bar{M} \to \X$.
\end{definition}
The significance of this definition will be made more clear in the next section, where we begin our discussion of groupoids.

\section{Groupoids}\label{section:groupoids}

Throughout this section $\C$ will be some choice of site as defined above. Objects of $\C$ will be denoted with letters such as $M$ and $N$. Our goal is to develop the correspondence between groupoids internal to a site and geometric stacks. We will develop this correspondence in a relatively high level of generality to unify the proofs for specific cases. To begin, we will need to make sure our working conditions are reasonable, and so we will need to make a few additional assumptions about $\C$.

\begin{definition}\label{defn:sub}
To any site $\C$, there is a distinguished class of maps that we call \emph{submersions}. A morphism $f \colon M \to N$ in $\C$ is called a \emph{submersion} if it is cartesian (always admits fiber products) and $f$ admits local sections.
By local sections, we mean that there exists a covering $S = \{ s_i \colon U_i \to M \}$ of $N$ together with morphisms $\{ \sigma_i \colon U_i \to M \}$ such that $f \circ \sigma_i = s_i$ for all $s \in S$.
\[ \begin{tikzcd}
  & U_i \arrow[dr, "{s_i}"] \arrow[dl, dashed, "{\exists \sigma_i}" swap] & \\
  M \arrow[rr, "f"] & & N
\end{tikzcd} \]
The submersions in $\C$ are precisely those morphisms which give rise to submersions of the associated representable stacks (see Example~\ref{ex:morphismsofreps}.)

There is a CFG of submersions which we call $\mathbf{Sub}$.
The objects $\mathbf{Sub}$ are defined to be the submersions in $\C$. The morphisms of $\mathbf{Sub}$ are defined to be pullback squares:
\[
\begin{tikzcd}
  f^* P \arrow[r] \arrow[d] & P \arrow[d, "p"] \\
  N \arrow[r, "f"] & M
\end{tikzcd}
\]
Composition is defined by horizontal composition of pullback squares. The CFG projection $\mathbf{Sub} \to \C$ is defined by:
\[
\begin{tikzcd}
  f^* P \arrow[r] \arrow[d] & P \arrow[d, "p"] \\
  N \arrow[r, "f"] & M
\end{tikzcd}
\mapsto
\begin{tikzcd}
  N \arrow[r, "f"] & M
\end{tikzcd}
\]
It is not difficult to check that $\mathbf{Sub}$ is a CFG. Axiom (CFG1) follows from the fact that submersions are stable under base change, (corollary of Lemma~\ref{lemma:basechangelemma}). Axiom (CFG2) follows from the universal property of pullback squares.
\end{definition}
\begin{definition}\label{defn:goodsite}
  A site $\C$ is \emph{good} if the following hold:
  \begin{enumerate}[(GS1)]
    \item $\C$ has an initial object. That is, there exists an object $\emptyset \in \C$ such that for anther other object $M$ in $\C$, there is a unique morphism $\emptyset \to M$.
    \item Every covering sieve is generated by some set of cartesian morphisms.
    \item Morphisms in $\C$ are local. That is, given a pair of objects $M$ and $N$ together with a covering sieve $S = \{ s_i \colon U_i \to M \}$ on $M$ and maps $f_i \colon U_i \to N$ such that the following diagram commutes:
    \[
    \begin{tikzcd}
     U_{ij} \arrow[d] \arrow[r] &  U_i \arrow[d, "f_i"] \\
     U_j \arrow[r, "f_j"] & N
    \end{tikzcd}
    \]
    there exists a unique morphism $f \colon M \to N$ such that $f \circ s_i = f_i$ for all $s_i$.
    \item The CFG of submersions $\Sub$ (Definition~\ref{defn:sub}) is a stack.
    \item If $f \circ g$ is a submersion, then $f$ is a submersion.
  \end{enumerate}
\end{definition}

The requirement of the existence of an initial object is for fairly minor technical reasons.
For $\Man$ the initial object is the \emph{empty manifold}. We need such an object so that fiber products will exist in cases where the images of the smooth maps are disjoint. We do not require the existence of a terminal object as this fails for some of the examples of sites which we are interested (i.e $\DMan$). The requirement of cartesian generators is necessarily standard either, but it prevents some aberrant behavior and the author is not aware of an example of an interesting site which does not satisfy this property.

Condition (GS3) is necessary since we need a way of constructing new morphisms in $\C$.
It is also reasonable to expect that morphisms in $\C$ are local with respect to the ambient topology.
Note that (GS3) implies that $\Sub$ is a pre-stack and so (GS4) is about gluing of objects.
(GS4) also implies that being a submersion is local in the codomain.
The last condition, (GS5) is needed to prove Lemma~\ref{lemma:invertiblegmor}.
There are weaker versions that suffice to prove this lemma, but we have chosen an easy to state version that resembles the smooth case.

From here on out, and throughout the appendix, we assume that $\C$ is a good site.

\begin{remark}
  For some settings, it may be prudent to be more restrictive with the definition of a submersion. For instance, it may be the case that the set of cartesian epimorphisms do not satisfy (GS5). In this situation, the arguments in this text will still work so long as our designated class of maps satisfy the following properties:
\begin{itemize}
  \item they must be submersions in our weaker sense (cartesian epimorphisms);
  \item they must include isomorphisms;
  \item they must make $\C$ into a good site (i.e. they satisfy (GS4) and (GS5)).
\end{itemize}
In such a setting, one should adjust the definition of a presentation accordingly.
\end{remark}
\subsection{Internal groupoids}

Roughly speaking, an internal groupoid is a groupoid whose arrows and units are objects of a category (or a site). To do this, we will transfer the groupoid data into a collection of morphisms such that a collection of diagrams commutes. Throughout this subsection, $\C$ denotes an

\begin{definition}\label{defn:internalgroupoid}
  A groupoid internal to a site $\C$, also called a $\C$-groupoid, consists of the following data:
  \begin{itemize}
    \item objects $\G$ and $M$ in $\C$ called the arrows and units,
    \item submersions $\s, \t \colon \G \to M$, called the \emph{source} and \emph{target} (see Definition~\ref{defn:sub} for the definition of a submersion in a general site),
    \item morphisms $\u\colon M \to \G$ and $\i \colon \G \to \G$ called the \emph{unit} and \emph{inverse}
    \item a morphism $\m \colon \G \times_{\s,\t} \G \to \G$.
  \end{itemize}
    such that (A)-(F) of Definition~\ref{defn:category} hold (when interpreted as commutative diagrams in $\C$). We also require that the following diagrams commute:
    \begin{enumerate}[(G1)]
      \item ($\i$ is an involution)
      \[
      \begin{tikzcd}
         & \G \arrow[dr, "\i"] & \\
         \G \arrow[ur, "\i"] \arrow[rr, "\Id"] & & \G
      \end{tikzcd}
      \]
      \item (compatibility of $\i$ with $\s$ and $\t$)
      \[\begin{tikzcd}
       & M & \\
       \G \arrow[ur, "\s"] \arrow[rr, "\i"] & & \G \arrow[ul, "\t" swap]
      \end{tikzcd}\]
      \item (compatibility of $\i$ with $\m$)
      \[
      \begin{tikzcd}
        \G \arrow[d, "\t"] \arrow[r, "\Id \times \i"] & \G \times_{\s,\t} \G \arrow[d, "\m"] \\
        M \arrow[r, "\u"] & \G
      \end{tikzcd}
      \]
    \end{enumerate}
\end{definition}
\begin{example}
  A $\Man$-groupoid is called a \emph{Lie groupoid}.
\end{example}
\begin{example}
  Given any $\C$-groupoid $\G \grpd M$, we can construct the \emph{opposite} groupoid $\G^{op}$ for which the source and target maps are reversed, and the multiplication morphism comes from the composition:
   \[  \G \times_{\t, \s} \G \to \G \times_{\s,\t} \G \to \G . \]
\end{example}
\begin{example}
  Any object $M \in \C$ has an associated \emph{trivial groupoid} $1_M$ for which the arrows and objects are both $M$ and the source and target maps are identities.
\end{example}
We will use the notation $G \rightrightarrows M$ to denote a $\C$-groupoid whose arrows are $\G$ and units are $M$. The notation $\G\enn$ will denote the $n$-fold fiber product $\G \times_{\s,\t} \cdots \times_{\s, \t} \G$. By convention we define $\G\zero := M$.

Since $\G$ comes with two morphisms to $M$, there is some ambiguity when writing things such as $\G \times_M N$.
To alleviate this, given $f \colon N \to M$, when we take a fiber product on the right side, as in $\G \times_M N$ we always mean $\G \times_{\s, f} N$.
Similarly, $N \times_M \G = N \times_{f, \t} \G$. From this convention, it should be clear that we think of arrows in the groupoid as going from right to left.

\begin{definition}\label{defn:internalmorphism}
  Let $\G \grpd M$ and $\H \grpd N$ be $\C$-groupoids. A \emph{morphism} of $\C$-groupoids, denoted $F \colon \G \to \H$ consists of two morphisms $F_1 \colon \G \to \H$ and $F_0 \colon M \to N$ such that the following diagrams commute:
  \begin{enumerate}[(GM1)]
  \item (Compatibility of $F$ with $\s$ and $\t$)
  \[
    \begin{tikzcd}
      \G \arrow[d,"\s"] \arrow[r, "F_1"] & \H \arrow[d,"\s"] \\
      M \arrow[r, "F_0"] & N
    \end{tikzcd}
    \quad \mbox{ and } \quad
    \begin{tikzcd}
      \G \arrow[d,"\t"] \arrow[r, "F_1"] & \H \arrow[d,"\t"] \\
      M \arrow[r, "F_0"] & N
    \end{tikzcd}
    \]
    \item
    \[
    \begin{tikzcd}[column sep= large]
      \G\two \arrow[rr, "{(F \circ \pr_1) \times (F \circ \pr_2)}"] \arrow[d, "\m"] & & \H\two  \arrow[d, "\m"] \\
      \G \arrow[rr, "F_1"] & & \H
    \end{tikzcd}
    \]
\end{enumerate}
\end{definition}
\begin{remark}
  We may sometimes denote the multiplication morphism with $\cdot$. For instance, $f \cdot g$ is a synonym for $\m \circ (f \times g)$. Of course, the associativity of $\m$ implies that $(f \cdot g) \cdot h = f \cdot (g \cdot h)$ whenever such morphisms are well defined.
\end{remark}
\subsection{Internal actions}

\begin{definition}\label{defn:action}
Let $\G \grpd M$ be a $\C$-groupoid. Then a \emph{left} $\G$ action on an object $P$ in $\C$ consists of the following data:
\begin{itemize}
  \item a morphism $\t^P \colon P \to M$,
  \item a morphism $\m_L \colon \G \times_M P \to P$
\end{itemize}
such that the following diagrams commute:
\begin{enumerate}[GA1]
  \item (Compatibility of $\m_L$ with units)
  \[\begin{tikzcd}
     & \G \times_M P \arrow[dr, "\m_L"] & \\
     P \arrow[ur, "{(\u \circ J) \times \Id}"] \arrow[rr, "\Id"] & & P
  \end{tikzcd}\]
  \item (Associativity)
  \[
  \begin{tikzcd}[column sep = huge]
    \G \times_M \G \times_M P
    \arrow[d, "{(\m(\pr_1,\pr_2),\pr_3)}" swap]
    \arrow[rr, "{(\pr_1, (\m_L(\pr_2, \pr_3)))}"] & & \G \times_M P \arrow[d, "\m_L"] \\
    \G \times_M P \arrow[rr, "\m_L"]  & & P
  \end{tikzcd}
  \]
\end{enumerate}
A right action is defined in a symmetrical fashion. That is, we swap all instances of $\s$ and $\t$ and write $\m_R$ instead of $\m_L$.
\end{definition}
Like with groupoids, we will sometimes write $f \cdot g$ to mean $\m_L(f \times g)$.

In the case where $\C$ is manifolds or topological spaces. This corresponds to the standard definition of a groupoid action.

\begin{definition}
  Let $\G \grpd M$ be a $\C$-groupoid and $P$, $\m_L^{P}$, $\t^P$ be a left action of $\G$ on $P$ as above. Suppose $Q$, $m_L^Q$ and $\t^{Q}$ is another left $\G$ action. Then we say a morphism $\phi: P \to Q$ is $\G$-equivariant if:
  \begin{enumerate}[(GE1)]
    \item (Preserves the anchor)
    \[\begin{tikzcd}
       & M &
      \\
      P \arrow[ur, "\t^P"] \arrow[rr, "\phi"] & & Q \arrow[ul, "{\t^{Q}}" swap]
    \end{tikzcd}\]
    \item (Preserves the product)
    \[
    \begin{tikzcd}
      \G \times_M P \arrow[d, "\m_L^P"] \arrow[rr, "{(\pr_1,\phi \circ \pr_2)}"] & & \G \times_M Q \arrow[d, "\m_L^Q"] \\
      P \arrow[rr, "\phi"] & & Q
    \end{tikzcd}
    \]
  \end{enumerate}
\end{definition}

\begin{definition} Let $\G \grpd M$ be a $\C$-groupoid. A left $\G$-bundle over an object $N$ in $\C$, consists of a left $\G$ action on $P$, together with a submersion $\s^P \colon P \to M$ such that:
  \begin{enumerate}[GB1]
    \item (compatibility of $\m_L$ with $\s^P$)
    \[
    \begin{tikzcd}
      \G \times_M P \arrow[r, "\m_L"] \arrow[d, "\pr_2"] & P \arrow[d, "\s^P"] \\
      P \arrow[r, "\s^P"] & N
    \end{tikzcd}
    \]
  \end{enumerate}
  Furthermore, we say that the $\G$ is
  \begin{enumerate}[(GB2)]
    \item \emph{almost principal} if $\s^P \colon P \to N$ is a submersion and $\m_L \times \pr_2 \colon \G \times_M P \to P \times_{N} P$ is a submersion
    \item \emph{principal} if it is almost principal and $\m_L \times \pr_2 \colon \G \times_M P \to P \times_{N} P$ is an isomorphism (i.e. diagram (GB1) is a pullback square).
  \end{enumerate}
  Right $\G$-bundles are defined in the obvious symmetric manner.

  A left action on $P$ is called \emph{(almost) principal} if it can be made into a (almost) principal $\G$-bundle over some manifold $N$. Since such a submersion $P \to N$ is unique up to a unique isomorphism (see Lemma~\ref{lemma:equivariantmaps}) it makes sense to write $P /\G$ to indicate some canonical choice of $N$.

  If $P \to N$ and $Q \to N'$ are equipped with the structure of a $\G$-bundle, then a morphism of $\G$-bundles is a pair of maps $\phi \colon P \to Q$ and $f \colon N \to N'$ such that $\phi$ is $\G$-equivariant and the following commutes:
  \begin{enumerate}[(GBM)]
    \item (Compatible with projection)
    \[
    \begin{tikzcd}
      P \arrow[r, "\phi"] \arrow[d, "\s^P"] & Q \arrow[d, "\s^Q"] \\
      N \arrow[r, "f"] & N'
    \end{tikzcd}
    \]
    \end{enumerate}
\end{definition}
A left $\G$-bundle is illustrated with a diagram:
\[
\begin{tikzcd}
\G \darrow & P \laction \arrow[dl, "\t"] \arrow[dr, "\s" swap] & \\
M & & N
\end{tikzcd}
\]
The reader should check that in the case of $\C = \Man$ a principal $\G$-bundle is the standard object.
The morphism $A \colon \G \times_M P \to P \times_N P$ defined by
\[ A(g, p) = (g \cdot p, p) \] is called the \emph{total action}.
Note that the total action is an isomorphism implies the existence of a division map. That is, a morphism $\til \m_L \colon P \times_N P \to \G$ such that:
\[
\til \m_L(p,q) \cdot q = p
\]
The condition that the action is principal implies that the division map is uniquely defined by this property.

Given a $\G$-action on $P$. We say that the action is \emph{principal} if there exists an object $N$ and a submersion $P \to N$ which makes $P$ a principal $\G$-bundle over $N$.
In such a case, we denote $N$ by $P/\G$.
This notation is sensible since, by Lemma~\ref{lemma:equivariantmaps}, $N$ is unique up to a unique isomorphism.

To keep notation from getting out of hand, we will distinguish between $\s^P$ and $\s$ when it is not necessary. We think of ``arrows'' in a $\G$-bundle as going from right to left. For this reason, if $P$ is a \emph{right} $\G$-bundle, the morphism $\t^P$ is the projection to the base instead of $\s^P$.
\begin{example}
  Any $\C$-groupoid $\G \grpd M$ can be made into a $\G$-bundle over $M$ by left multiplication. The inverse of the total action can be constructed directly using the division map on the groupoid $\m \circ (\pr_1 \times (\i \circ pr_2))$.
\end{example}
\begin{example}
  Suppose $P$ is a left $\G$-bundle. Then we can form a right $\G$-bundle called $P^{op}$ whose total space is the same as $P$ and whose source is the target of $P$. The right action is define in the obvious way:
  \[ p \cdot g := \i(g) \cdot p \]
\end{example}
\begin{example}
  Let $\G \grpd M$ be a $\C$-groupoid and suppose $f \colon N \to M$ is a morphism in $\C$. Then $\G \times_M N$ naturally inherits the structure of a left principal $\G$-bundle. Such a bundle is called the \emph{trivial bundle} associated to $f$.
\end{example}
\begin{example}\label{ex:pullbackbundle}
  Suppose $P \to N$ is a principal (left) $\G$-bundle and $f \colon N' \to N$ is a morphism in $\C$. Then there exists a unique $\G$-bundle structure on $f^* P := P \times_N N'$ such that $f^* P \to P$ is a morphism of $\G$-bundles. Such a bundle is called the \emph{pullback} bundle. The multiplication map on $f^* P$ is defined below:
  \[ \m_L^{f^* P}(g,(g',p)) := (g \cdot g',p ) \]

  From this point of view, the trivial bundle associated to a morphism $f \colon N \to M$ is the pullback of $\G$ as a $\G$-bundle. For this reason we will use $f^* \G$ to denote the trivial $\G$ bundle associated to $f$.
\end{example}
A \emph{trivialization} of a $\G$-bundle $P$ is a $\G$-bundle isomorphism $f^* \G \to P$.
Our next lemma gives us some insight into the local geometry of principal $\G$-bundles.
\begin{lemma}\label{lemma:bundletrivialization}
  Suppose $P \to N$ is a principal left $\G$-bundle.
  There is a one-to-one correspondence between trivializations $\phi: f^* \G \to P$ and sections $\sigma \colon N \to P$ such that $\t^P \circ \sigma = f$.
\end{lemma}
\begin{proof}
  Suppose we have an $f \colon N \to M$ and recall $f^* \G := \G \times_M N$. Let $\til \sigma := (\u \circ f, \Id)$ be the canonical section of $f^*\G$. Then $\sigma := \phi \circ \til \sigma$ is a section of $P$. Furthermore, it clearly satisfies $\t \circ \sigma = f$.

  For the other direction of the correspondence, suppose $\sigma \colon N \to P$ is a section. Then we can define
  \[ \phi(g,x) = g \cdot \sigma(x) \, . \]
  That $\phi$ is equivariant is fairly routine. To see that $\phi$ is an isomorphism we will the division map to construct its inverse:
  \[ \phi\inv(p) := (\til \m_L(p,\sigma \circ \s(p)), \s(p)) \]
\end{proof}
\begin{definition}
  Let $\G \grpd M$ and $\H \grpd N$ be $\C$-groupoids. A $(\G, \H)$-bibundle consists of an object $P$ in $\C$ as well as a left $\G$ bundle structure on $P$ and right $\H$ bundle structure on $P$ such that the source and target maps are shared between them. Lastly, we require that:
  \begin{enumerate}[(BB)]
    \item (the actions commute)
    \[
    \begin{tikzcd}[column sep = huge]
      \G \times_M P \times_N \H
      \arrow[rr, "{(\m_L(\pr_1,\pr_2),\pr_3)}"]
      \arrow[d, "{(\pr_1,(\m_R (\pr_2 ,\pr_3))}" swap] & & P \times_N \H \arrow[d, "\m_R"] \\
      \G \times_M P \arrow[rr, "\m_L"] & & P
    \end{tikzcd}
    \]
  \end{enumerate}
  A bibundle is \emph{left(right) principal} if the left(right) action forms a principal left(right) $\G$ bundle over $N$($M$).
  We say that the bibundle is \emph{principal} or \emph{biprincipal} if it is both left and right principal.

  If $P$ and $P'$ are both equipped with $(\G, \H)$-bibundle structures, then a morphism $\phi \colon P \to P'$ is a morphism of $(\G, \H)$-bibundles if it is equivariant with respect to both actions.
\end{definition}
A $(\G,\H)$-bibundle can be illustrated with a diagram of the form:
\[
\begin{tikzcd}
\G \darrow & P \laction \raction \arrow[dl, "\t"] \arrow[dr, "\s" swap] & \H \darrow \\
M & & N
\end{tikzcd}
\]
\begin{example}
  Any $\C$-groupoid $\G \grpd M$ is naturally a biprincipal $(\G, \G)$-bibundle. It is called the \emph{trivial} $(\G,\G)$-bibundle.
\end{example}
\begin{example}
  Suppose $F \colon \G \to \H$ is a $\C$-groupoid homomorphism covering $f \colon M \to N$. Then we can construct a bibundle $P := f^* \H$. The left action on $P$ is the trivial $\H$ action while the left action of $\G$ is defined by the rule:
  \[ (h,x) \cdot g = (h F(g), x) \, . \]
\end{example}
We will take a closer look at this last example in Chapter~\ref{chap:morita}.
\subsection{Generalized morphisms}

Bibundles allow us to generalize weaken the notion of a groupoid morphism. We will see later that this weaker notion of morphism is precisely what is needed to study the stack associated to a $\C$-groupoid.

\begin{definition}
  Let $\G \grpd M$ and $\H \grpd N$ be $\C$-groupoids. A \emph{Hilsum-Skandalis map} or a \emph{generalized morphism} is a left principal $(\G, \H)$-bibundle.

  An \emph{equivalence} of generalized morphisms is a morphism of $(\G,\H)$-bibundles.
\end{definition}
We use the term equivalence due to the fact that all bibundle morphisms are automatically isomorphisms (see Lemma~\ref{lemma:bibundlemorphisms}).
Generalized morphisms can be composed via the following construction.
Let $P$ be a $(\G, \H)$-bibundle and $Q$ be a $(\H, \K)$-bibundle illustrated below:
  \[
  \begin{tikzcd}
  \G \darrow & P \laction \raction \arrow[dl, "\t"] \arrow[dr, "\s" swap] & \H \darrow & Q \arrow[dl, "{\t}"] \arrow[dr, "{\s}" swap] \laction \raction & \K \darrow \\
  M & & N & & O
\end{tikzcd}
  \]
  Consider the object $P \times_{N} Q$ and the canonical morphism $\hat \t \colon P \times_N Q \to N$. Note that this is the target map of a (left) action of $\H$ whose product is given by:
  \[ \hat \m_L(h,(p,q)) := (p \cdot \i(h), h \cdot q)  \]
  By Lemma~\ref{lemma:diagonalaction}, this action is principal. We denote $(P \times_N Q) / \H$ with $P \otimes_\H Q$.

  Notably, $P \otimes Q$ inherits a $(\G, \K)$-bibundle structure. The anchor maps are the unique maps which make the below diagram commute:
  \[
  \begin{tikzcd}
     &  \arrow[dl] P \times_N Q \arrow[dd] \arrow[dr] & \\
  M &               & O \\
    & \arrow[ul] P \otimes_\H \arrow[ur] Q
  \end{tikzcd}
  \]
  These maps exist due to Lemma~\ref{lemma:equivariantmaps}. Similarly, the left action map is the unique morphism such that the diagram below commutes.
  \[
  \begin{tikzcd}[column sep = huge]
  \G \times_{M} P \times_N Q \arrow[rr, "(\m_L ( \pr_1 \times \pr_2)) \times \pr_3"] \arrow[d, "{\Id \times (\pi \circ (\pr_1 \times \pr_2))}"]
  & & P \times_N Q \arrow[d, "\pi"]\\
  \G \times_{M} P \otimes_\H Q \arrow[rr] & & P \otimes_\H Q
  \end{tikzcd}
  \]
  The right action is defined similarly.

  \begin{lemma}
    These actions make $P \otimes_\H Q$ a left principal $(\G, \K)$-bibundle.
  \end{lemma}
  \begin{proof}
  We must show the total action is an isomorphism. Note that that the total action makes the following diagram commute:
  \[
  \begin{tikzcd}
    \G \times_M (P \times_N Q) \arrow[d] \arrow[r] & (P \times_N Q) \times_O (P \times_N Q) \arrow[d] \\
    \G \times_M (P \otimes_\H Q) \arrow[r] & (P \otimes_\H Q) \times_O (P \otimes_\H Q)
  \end{tikzcd}
  \]
  Where the vertical arrows are the projections under the action of $\H$. By Corollary~\ref{corollary:equivariantmaps}, the fact that the top arrow is an isomorphism implies that the bottom arrow is as well.

  Similarly, Corollary~\ref{corollary:equivariantmaps} implies that the left and right actions must commute, since they are obtained by reducing commuting left and right actions modulo $\H$.
  \end{proof}

  \begin{remark}
    We will use tensor product notation for constructing maps involving $P \otimes Q$. For instance, we may sometimes write $f \otimes g$ to denote the post composition of $f \times g$ with the projection $P \times_N Q \to P \otimes Q$. Similarly, if we write:
    \[ p \otimes q \mapsto f(p,q) \]
    We mean to denote the unique morphism obtained from the $\H$ invariant map $f$.
  \end{remark}

  With this product in hand, we can define the bicategory of generalized morphisms.
  \begin{definition}
    Let $\morpoid$ (generalized morphisms) denote the following bicategory:
    \begin{itemize}
      \item The objects of $\morpoid$ are $\C$-groupoids.
      \item The 1-morphisms are generalized morphisms of $\C$-groupoids.
      \item The 2-morphisms isomorphisms of generalized morphisms.
    \end{itemize}
    If $P$ is a $(\G, \H)$-bibundle, then we say the source of $P$ is $\H$ and the target of $P$ is $\G$.
    Composition of generalized morphisms is by the tensor product defined above.
    Vertical composition of 2-morphisms is by composition in $\C$.
    Suppose $P_1$ and $Q_1$ and $(\G, \H)$-bibundles and $P_2$ and $Q_2$ are $\H$, $\K$ bibundles.
    Given 2-morphisms $\alpha \colon P_1 \to P_2$ and $\beta \colon Q_1 \to Q_2$ then vertical composition is the unique 2-morphism $\alpha \otimes_\H \beta$ such that the below diagram commutes.
    \[
    \begin{tikzcd}
      P_1 \times_N Q_1 \arrow[d, "(\alpha \circ \pr_1) \times (\beta \circ \pr_2)" swap] \arrow[r] & P_1 \otimes_\H Q_1 \arrow[d, "\alpha \otimes_\H \beta"] \\
      P_2 \times_N Q_2 \arrow[r] & P_2 \otimes_\H Q_2
    \end{tikzcd}
    \]
    Again, existence and uniqueness of this morphism is given by Corollary~\ref{corollary:equivariantmaps}.

    Given a $\C$-groupoid $\G$, the unit 1-morphism is $\G$ (viewed as a $\G$-bibundle with the obvious actions).
    The left unit $\mathbb{U}_L$ is the unique 2-morphism $\G \otimes_\G P \to P$ obtained by reducing the left multiplication modulo $\G$.
    The right unit is defined similarly.

    Suppose $P_1$, $P_2$ and $P_3$ are $(\G_1, \G_2)$, $(\G_2, \G_3)$, and $(\G_3, \G_4)$-bimodules, respectively. Where $\G_i \grpd M_i$ are $\C$-groupoids.
    Associativity is the unique 2-morphism obtained by reducing the canonical isomorphism:
    \[ (P_1 \times_{M_2} P_2) \times_{M_3} P_3 \to P_1 \times_{M_2} (P_2 \times_{M_3} P_3) \, .  \]
    modulo the actions of $\G_2$ and $\G_3$.
  \end{definition}
  We have omitted the proof that this data satisfies the coherence conditions for a bicategory since we will show later that $\morpoid$ is equivalent to a (strict) bicategory and must therefore be coherent.

  \begin{lemma}\label{lemma:invertiblegmor}
    A generalized morphism is (weakly) invertible if and only if it is biprincipal.
  \end{lemma}
  \begin{proof}
    Suppose $P$ is a biprincipal $(\G,\H)$-bundle for $\C$-groupoids $\G \grpd M$ and $\H \grpd N$. Let $P^{op}$ be the $(\H,\G)$-bibundle obtained by taking the opposite actions on $P$.

    Then we claim that $P \otimes_\H P^{op} \isom \G$ as $\G$-bundles. By the definition of $P \otimes_\H P^{op}$ we have that:
    \[ P \otimes_\H P^{op} = P \times_N P /\H \]
    We have division map $\til \m_L \colon P \times_N P \to \G$ and it is clear that $\til \m_L$ is $\H$-invariant.
    By Lemma~\ref{lemma:equivariantmaps} we obtain a morphism $\Phi \colon P \otimes_\H P^{op} \to \G$. Since $\til \m_L$ is $\G$-equivariant we have that $\Phi$ is equivariant and so $\Phi$ is a $\G$-bundle morphism. Since $\Phi$ covers the identity it is an isomorphism (see Lemma~\ref{lemma:principalbundlemorphisms}).
    The argument is symmetric for $P^{op} \times_\G P$.

    On the other hand, suppose $P$ has a weak inverse $Q$. Then since $P \otimes Q \isom \H$ we can conclude that $\t \circ \pr_1 \colon P \times_N Q \to M$ is a submersion. By (GS5) from the definition of a good site, we conclude that $\t$ is a submersion.

    To conclude the proof, we need to show that the right action on $P$ is principal.
    First note that $\s \colon Q \to M$ is a submersion (by a similar argument as before).
    Therefore, we can choose a covering $\{ U_i \to M \}$ of $M$ with sections $\sigma_i \colon U_i \to Q$ of $\s \colon Q \to M$.
    Such sections let us construct a family of maps $U_i \times_M P \to Q \otimes P$ by using the rule:
    \[ (x, p) \mapsto \sigma_i(x) \otimes p \]
    Similarly, we have a family of morphisms:
    \[ U_i \times_M P \times_{\t,\t} P \to (Q \otimes P) \times_{\t,\t} (Q \otimes P) \quad (x,p,q) \mapsto ((\sigma_i(x) \otimes p), (\sigma_i(x) \otimes q))  \]
    By using the division map for the right action on $Q \otimes P$ we get local division maps $\delta_i \colon U_i \times_M \times P \times_{\t,\t} P \to \H$ for the right action on $P$. Since $\C$ is a good site, we can glue the $\delta_i$ together into a global division morphism $\delta \colon P \times_{\t,\t} P \to \H$. That this division map gives inverts the total action follows from the definition of the action on $Q \otimes_\H P$.
  \end{proof}
  \subsection{The classifying stack}

  Suppose $\G$ is a $\C$-groupoid. Then consider the category $\B \G$ such that
  \begin{itemize}
    \item an object in $\B \G$ is a principal $\G$ bundle $P \to N$
    \item a morphism in $\B \G$ is an equivariant map $\phi \colon P \to Q$.
  \end{itemize}
  Observe that there is a natural forgetful functor $\B \G \to \mathbf{Sub}$. This follows from:
  \begin{lemma}\label{lemma:bgtosub}
    Suppose $\G$ is a $\C$-groupoid. Let $P \to N_1$ and $Q \to N_2$ be principal left $\G$-bundles and $\phi \colon P \to Q$ be a $\G$-bundle morphism covering $f \colon N_1 \to N_2$. Then
    \[
    \begin{tikzcd}
      P  \arrow[d] \arrow[r, "\phi"] & Q \arrow[d]\\
      N_1 \arrow[r, "f"] & N_2
    \end{tikzcd}
    \]
    is a pullback square in $\C$.
  \end{lemma}
  \begin{proof}
  We must show that $\til \phi \colon P \to f^* Q = Q \times_{N_2} N_1$ is an isomorphism. In other words, it suffices to prove the lemma in the case where $f = \Id$. Therefore, without loss of generality assume $N_1 = N_2 = N$ and $f = \Id$.

  Let $S = \{ s_i \colon U_i \to N_1 \}$ be a covering sieve such that the the pullback bundles $ \{ s_i^* P \to U_i \}$ and $\{ s_i^*Q \to U_i \}$ admit sections. Let $\phi_i \colon s_i^* P \to s_i^* Q \}$ be defined as below:
  \[
  \phi_i(p,x) := (\phi(p),x)
  \]
  By Lemma~\ref{lemma:principalbundlemorphisms} each $\phi_i$ is invertible. Hence the collection $\phi_i\inv$ defines the inverse of $\phi$ locally. Since $\C$ is a good site, the $\phi_i\inv$ can be glued to obtain a global inverse of $\phi$.
  \end{proof}

The functor $\B \G \to \mathbf{Sub}$ is a faithful inclusion. In general, it is not full or essentially surjective.
Now consider the forgetful functor $\pi:= \B \G \to \C$ which passes to the base of the $\G$-bundle.
  \begin{theorem}
    The functor $\pi \colon \B \G \to \C$ makes $\B \G$ a category fibered in groupoids.
  \end{theorem}
  \begin{proof}
    We must show (CFG1) and (CFG1) from Definition~\ref{defn:cfg}. First of all, (CFG1) follows from the construction of the pullback bundle (Example~\ref{ex:pullbackbundle}).

    To show (CFG2), suppose we are given morphisms $f\colon M_1 \to M_2$ and $g\colon M_2 \to M_3$ in $\C$ together with $a\colon P \to R$ and $b\colon Q \to R$ such that $\pi(a) = g \circ f$ and $\pi(b) = g$. We must show that there exists a unique morphism $c$ such that $\pi(c) = f$ and $b \circ c = a$.
    \[
    \begin{tikzcd}
      P \arrow[d] \arrow[rr, bend left, "a"] \arrow[r, dashed, "\exists! c"] & Q \arrow[d] \arrow[r, "b"] & R\arrow[d] \\
      M_1 \arrow[r, "f"]& M_2 \arrow[r, "g"] & M_3
    \end{tikzcd}
    \]
    Since $\B \G \to \mathbf{Sub}$ is a faithful inclusion, we know that there exists a morphism $c$ which makes the diagram commute. We only need to show that $c$ is $\G$-equivariant. By Lemma~\ref{lemma:bgtosub} we know that we can replace $Q$ with $g^*R$ and $P$ with $f^*g^*R$. Since $f^* g^* R \to g^* R$ is clearly $\G$-equivariant, the theorem follows.
  \end{proof}
  \begin{theorem}\label{thm:bgstack}
    $\B \G$ is a stack.
  \end{theorem}
  \begin{proof}
    We need to show that $\B \G$ satisfies (S1) and (S2) from Definition~\ref{defn:stack}.
    \begin{enumerate}[(S1)]
      \item Suppose we are given the following data:
        \begin{itemize}
          \item a covering sieve ${\{ s_i \colon U_i \to N \}}_{i \in I}$,
          \item objects $P$ and $Q$ in $\B\G_N$,
          \item morphisms $\phi_i \colon P|_{U_i} \to Q|_{U_i}$ covering the identity,
        \end{itemize}
        such that
      \[ \phi_i|_{U_{ij}} = \phi_j|_{{U_{ij}}} \, . \]
      We must show there exists a unique morphism $\phi: P \to Q$ such that $\phi|_{P|_{U_i}} = \phi_i$.
      Since $\B\G \to \mathbf{Sub}$ is an inclusion and $\mbox{Sub}$ is a stack, we already know that there is a morphism $\phi\colon P \to Q$. We must show that $\phi$ is $\G$ equivariant. In other words, we need the following diagram to commute:
      \[
      \begin{tikzcd}
        \G \times_M P \arrow[d, "\m_L"] \arrow[rr, "\pr_1 \times (\phi \circ \pr_2)"] & & \G \times_M Q \arrow[d, "\m_L"]  \\
        P \arrow[rr, "\phi"]  &  & Q
      \end{tikzcd}
      \]
      The maps $P|_{U_i} \to P$ and $Q|_{U_i} \to Q$ generate covering sieves of $P$ and $Q$, respectively. Since we know each $\phi_i \colon P|_{U_i} \to Q|_{U_i}$ is $\G$-equivariant. We can conclude that the above diagram commutes locally. Since $\C$ is a good site, it must commute globally.
      \item Now suppose we are given the following data:
          \begin{itemize}
            \item a covering sieve ${\{ s_i \colon U_a \to N \}}_{i \in I}$ of an object $N$ in $\C$,
            \item a collection of objects ${\{ P_i \in \X_{U_i} \}}_{i \in I}$,
            \item morphisms $\phi_{ij}: P_j|_{U_{ij}} \to P_i|_{U_{ij}}$
          \end{itemize}
            such that, for all $i,j,k \in I$ such that $s_{ijk}$ exists, $\phi_{ij}|_{U_{ijk}} \circ \phi_{jk}|_{U_{ijk}} = \phi_{ik}|_{U_{ijk}}$.

            We must show there exists a principal $\G$-bundle $P$ over $N$ together with morphisms ${\{ \phi_i: P|_{U_i} \to P_i \}}_{i \in I}$ such that
            \[ \phi_{ij} \circ \phi_{j}|_{U_{ij}} = \phi_i|_{U_{ij}}. \]

            Again, we know that there exists $P \to N$ in $\Sub_N$ which satisfies this data. We claim that there is an action of $\G$ on $P$ which makes each $\phi_i$ equivariant. In other words, we must construct a morphism $\m_L \colon \G \times_M P \to P$. To do this, observe that $\{ \G \times_M P_i \to \G \times_M P \}$ generates a covering sieve of $\G \times_M P$. Hence we can treat each left action of $\G$ on $P_i$ as a locally defined left action on $P$. Since these actions are assumed to agree on intersections and $\C$ is a good site, we conclude that they can be glued together into a unique map $\G \times_M P \to P$.
            The resulting morphism will constitute an action since diagrams which commute locally in a good site must commute globally.
    \end{enumerate}

  \end{proof}

\subsection{Geometric stacks}
  So far we have observed that a $\C$-groupoid gives rise to a stack. Let $\mathbf{Stacks}$ be the bicategory of stacks over $\C$. In this section we will extend the relationship $\G \mapsto \B \G$ to a pseudofunctor (see Definition~\ref{defn:pseudofunctor}) of bicategories $\B \colon \morpoid \to \mathbf{Stacks}$ and prove the main properties of $\B$. A pseudofunctor of bicategories is just a map of objects, 1-morphisms, and 2-morphisms, together with a compatibility map relating the compositions.

  We have already defined the action of $\B$ on objects.
  To define $\B$ on 1-morphisms we must construct a morphism of CFGs $\B P \colon \B \H \to \B \G$ for a left principal $(\G,\H)$-bibundle $P$.

  Suppose $Q$ is a principal left $\H$ bundle (i.e. an object of $\B \H$). Then we define $\B P(Q) := P \otimes_\H Q$. Given $\phi\colon Q_1 \to Q_2$ a $\H$-bundle morphism. Then $\B P( \phi) \colon P \otimes_\H Q_1 \to P \otimes_\H Q_2$ is the unique morphism making the following diagram commute:
  \[
  \begin{tikzcd}[column sep = huge]
    P \times_N Q_1 \arrow[rr, "\pr_1 \times (\phi \circ \pr_2)"] \arrow[d] &  &P \times_N Q_2 \arrow[d] \\
    P \otimes_\H Q_1 \arrow[rr, "\B P(\phi)"] & & P \otimes_\H Q_2
  \end{tikzcd}
  \]
  It is clear that $\B P (\phi \circ \psi) = \B P(\phi) \circ \B P(\psi)$ and so $\B P$ is a functor. This defines $\B$ on 1-morphisms.

  Now suppose $\Phi: P_1 \to P_2$ is a $(\G, \H)$-bibundle isomorphism. We must define a natural transformation $\B \Phi \colon \B \P_1 \to \B \P_2$.
  In other words, we need a function $\eta \colon (\B \H)_0 \to (\B \G)_1$. Given an object $Q$ in $\B \G$, let
  \[ \eta(Q) \colon P_1 \otimes_\H Q \to P_2 \otimes Q  \]
  be the unique morphism obtained by reducing $(\Phi \circ \pr_1) \times \pr_2 \colon P_1 \times_N Q \to P_2 \times_N Q$ modulo $\H$.
  Clearly this makes the below diagram commutes for any morphism $\phi$ in $\B\G$ and so $\eta$ defines a natural transformation.
  \[
  \begin{tikzcd}
    P_1 \otimes_\H Q_1 \arrow[r, "\eta(Q_1)"] \arrow[d, "\B P_1 (\phi)"] & P_2 \otimes_\H Q_1 \arrow[d, "\B P_2 (\phi)"] \\
    P_1 \otimes_\H Q_2 \arrow[r, "\eta(Q_2)"] & P_2 \otimes Q_2
  \end{tikzcd}
  \qquad
  \begin{tikzcd}
  p_1 \otimes q_1 \arrow[r, mapsto] \arrow[d, mapsto] &  \Phi(p_1) \otimes q_1 \arrow[d, mapsto] \\
  p_1 \otimes \phi(q_1) \arrow[r, mapsto] & \Phi(p_1) \otimes \phi(q_1)
\end{tikzcd}
  \]
  Suppose $P_1 \colon \H \to \G$ and $P_2 \colon \K \to \H$ are 1-morphisms in $\morpoid$. For $\B$ to be a pseudofunctor, we should exhibit a natural transformation $\B (P_1 \otimes P_2) \isom \B(P_1) \circ \B(P_2)$. In other words, for each object $Q$ in $\B\G$ we need an isomorphism $\mathbb{B}_1(P_1,P_2,P_3) \colon (P_1 \otimes P_2) \otimes Q \to P_1 \otimes (P_2 \otimes Q)$.
  This isomorphism is constructed canonically by the associator of the fiber product in $\C$. Lastly, for each $\C$-groupoid $\G$, we should exhibit a natural isomorphism $\mathbb{B}_2(\G)$ between the functor:
  \[ \B \G \to \B \G \qquad P \mapsto \G \otimes_\G P \]
  and the identity functor on $\B\G$. This natural transformation comes from the canonical isomorphism between $\G \otimes_\G P$ and $P$.

  Technically, to show that $\B$ is a well defined pseudofunctor, we need to check coherence of this data. Since this calculation is straightforward and not particularly enlightening, so we will omit the proof.

  \begin{proposition}
    Let $\G \grpd M$ be a $\C$-groupoid for $\C$ a good site. Then $\B \G$ is a geometric stack.
  \end{proposition}
  \begin{proof}
    Let $p \colon \bar M \to \B \G$ be the functor $f \mapsto f^* \G$.
    That is, it sends a morphism $f \colon N \to M$ to the associated trivial bundle $f^* \G$.
    We claim that $p$ is a presentation of $M$.

    Since every $\G$-bundle is locally trivial, it follows that $p$ is an epimorphism.
    Hence, we only need to show that $p$ is a representable morphism of CFGs.
    Suppose $q \colon \bar N \to \B \G$ is a morphism of CFGs. Without loss of generality, we can assume that $q(g) = g^* q(\Id_N)$.
    Then Lemma~\ref{lemma:bundleispullback} says that there is a pullback square:
    \[
    \begin{tikzcd}
      \bar Q \arrow[r, "\s"] \arrow[d, "\t"] & \bar N \arrow[d, "q"] \\
      \bar M \arrow[r, "p"]           & \B\G
    \end{tikzcd}
    \]
    Which immediately implies that $p$ is representable.
  \end{proof}

  Let $\GStacks \subset \mathbf{Stacks}$ be the full sub-bicategory of geometric stacks. Then the previous result says that we can restrict the codomain of $\B$ to the full sub-2-category of geometric stacks.
  We now state the main theorem of the chapter.
  \begin{theorem}[Fundamental theorem of geometric stacks]\label{thm:chapter1}
    Suppose $\C$ is a good site.
    The pseudofunctor $\B$ is an equivalence between $\morpoid$, the bicategory of $\C$-groupoids and bibundles, and $\GStacks$, the bicategory of geometric stacks.
  \end{theorem}

  Before we proceed to prove this theorem, we should mention some of the context for this result.
  In the case of $\C = \Man$, it is folklore. For instance, one can find a statement and (partial) proof of the result in Blohmann~\cite{BlohmannSLGs}. In the case where $\C$ is the site of topological spaces, we are not aware of a reference which states the above theorem in this form.
  However, since the study of higher structures in $\Top$ is very well developed, the author considers it likely that it is at least an easy corollary of an existing theorem.
  The main utility of the Theorem is not in applying it to a well known site.
  Rather, it is mainly useful for understanding geometric stacks for newly defined or uncommon sites.
  In fact, this is precisely what we will do in Chapter~\ref{chap:dman}.

  There are two (equivalent) definitions of equivalence of bicategories (see Leinster~\cite{Leinster} for a discussion of this). The one that we will use, (and the one that appears in Definition~\ref{defn:pseudofunctor}) is as follows: a pseudofunctor $F$ is an equivalence if and only if $F$ is an equivalence of categories at the level of Hom-categories and $F$ is surjective up to weak equivalences. Hence, we can split the proof of Theorem~\ref{thm:chapter1} into the next two propositions.
  \begin{proposition}
    The pseudofunctor $\B$ is biessentially surjective onto geometric stacks. That is, every geometric stack $\X$ is equivalent to $\B \G$ for some $\C$-groupoid $\G$.
  \end{proposition}
  \begin{proof}
    Suppose $p \colon \bar M \to \X$ is a presentation.
    Lemma~\ref{prop:CFGgroupoidstruct} says that $\bar M \til\times_{\X} \bar M$ is a strict groupoid internal to CFGs.
    Furthermore, since the CFG of arrows is representable, there exists a $\C$-groupoid $\G$ and a 2-pullback square:
    \begin{equation}\label{eqn:biessentiallysurj}
    \begin{tikzcd}
      \bar \G \arrow[r, "\s"] \arrow[d, "\t"] & \bar M \arrow[d] \\
      \bar M \arrow[r]                   & \X
    \end{tikzcd}
  \end{equation}

    We claim that $\B \G$ is equivalent to $\X$. By Lemma~\ref{lemma:stackification} it suffices to show that there exists a pre-stack $\Y$ together with a pair of local equivalences (see Definition~\ref{defn:localequivalence}) $\B \G \from \Y \to \X$.
    Let $\Y$ be defined to be the pre-stack of trivial $\G$ bundles. That is: the objects of $\Y$ are defined to be morphisms $f \colon N \to M$. The morphisms of $\Y$ are isomorphisms $\phi \colon f_1^* \G \to f_2^* \G$.

    The forgetful functor $\Y \to \B\G$ is clearly a local equivalence. We need to construct a local equivalence $E \colon \Y \to \X$.
    Recall that a local equivalence is a fully faithful and locally essentially surjective functor.
    Since $\Y$ has the same objects as $\bar M$, we can define $E \colon \Y \to \X$ to be equal to $p$ at the level of objects.
    Since we are working with CFGs, it suffices to define $E$ for morphisms covering the identity.
    Suppose $\phi \colon f_1^* \G \to f_2^* \G$ is a morphism in $\Y$ covering the identity for $f_i \colon N \to M$.

    By Lemma~\ref{lemma:gbundlecycle} there is a unique $\gamma \colon N \to \G$ such that
    \[ \phi(g,x) = (g \cdot \gamma(x), x) . \]
    Let $\eta \colon \bar M_0 \to \X_1$ be the natural transformation which witnesse the 2-commutativity of Diagram~\ref{eqn:biessentiallysurj}.
    Then we define $E(\phi) := \eta(\gamma)$ where $\gamma \colon N \to M$ is the morphism just described.

    To see that $E$ is a functor, observe that
    \[ E( \phi_1 \circ \phi_2) = \eta(\m(\gamma_2,\gamma_1)) = \eta(\gamma_1) \circ \eta(\gamma_2) = E(\phi_1) \circ E(\phi_2) \]
    The second equality comes from two facts:
    \begin{itemize}
      \item if $F \colon \bar \G \to \bar M \til\times_{\X} \bar M$ is the canonical map, then $\eta = \pr_{1.5} \circ F_0 \colon \G_0 \to \bar M \til\times_{\X} \bar M$;
      \item the groupoid operation on $\G$ is, by definition, the one obtained via $F \colon \bar \G \to \bar M \til\times_{\X} \bar M$.
    \end{itemize}

    Now that we have defined $E$, it is clear that $E$ is locally essentially surjective so we only need to show that it is fully faithful.
    It suffices to show this for morphisms covering the identity.
    Fix a pair of objects $f$ and $g$ in $\Y$.
    Suppose $a \colon f^* X_0 \to g^* X_0$ is a morphism in $\X$ covering the identity.
    Consider the object $(f, a, g)$ in $\bar M \til\times_{\B\G} \bar M$. Since $F$ is an equivalence, $(f, a, g)$ is in the orbit of $F(\gamma)$ for a unique $\gamma \in \bar \G$.
    Since this process inverts $E$ we conclude that $E|_{\Hom(f,g)}$ is a bijection.
  \end{proof}
  \newcommand{\Bun}{\mathbf{Bun}}
  \begin{proposition}
    Let $\G \grpd M$ and $\H \grpd N$ be $\C$-groupoids and let $\Bun(\G,\H)$ be the category whose objects are generalized morphisms $P \colon \G \to \H$ and morphisms are isomorphisms of generalized morphisms.
    On the other hand, let $\Hom(\B\G, \B\H)$ be the category whose objects are CFG morphisms $\B\G \to \B\G$ and morphisms are natural transformations.
    Then $\B \colon \Bun(\G,\H) \to \Hom(\B\G, \B\H)$ is an equivalence of categories.

    In other words, the psuedofunctor $\B$ is bi-fully faithful.
  \end{proposition}
  \begin{proof}
    We will first show essential surjectivity. Suppose $F \colon \B\G \to \B\H$ is a CFG morphism. Then consider $\bar N \til\times_{\B\H} \bar M$ which fits into the 2-commutative diagram:
    \[
    \begin{tikzcd}
      \bar N \til\times_{\B\H} \bar M \arrow[rr] \arrow[dd]  &    & \bar M \arrow[d] \\
                                                            &      & \B \G \arrow[d, "F"] \\
      \bar N \arrow[rr]                                       & & \B\H
    \end{tikzcd}
    \]
    By Lemma~\ref{lemma:CFGbundle}, $\bar N \til\times_{\B\H} \bar M$ has a left $\bar \H$ action and a right $\bar \G$ action. Its clear from the definitions of these actions that they commute. Since $\bar N \to \B\H$ is a presentation, there exists a $(\H,\G)$-bundle $P$ which fits into a 2-commutative square:
    \[
    \begin{tikzcd}
      \bar P \arrow[rr, "\s" ] \arrow[dd, "\t"]            &      & \bar M \arrow[d] \\
                                                           &      & \B \G \arrow[d, "F"] \\
      \bar N \arrow[rr]                                    &      & \B\H
    \end{tikzcd}
    \]
    The right action of $\G$ is principal since $\s$ is a submersion (see Lemma~\ref{lemma:CFGbundle} for why the action is principal).
    Note that $P \to \B\G$ is a presentation. Therefore, by replacing $M$ with $P$, we can assume without loss of generality that there exists a 2-commutative square:
    \[
    \begin{tikzcd}
      \bar M \arrow[r, "f"] \arrow[d] & \bar N \arrow[d] \\
      \B \G \arrow[r, "F"] & \B\H
    \end{tikzcd}
    \]
    In such a case, we get a morphism $\sigma \colon M \to P$ which splits $\s \colon P \to M$.
    Since $P$ admits a section, we can assume without loss of generality that $P = \G \times_{\s,f} N = f^* \G$.

    Now we claim that there is a natural transformation $\eta \colon \B P \to F$. Such a natural transformation must associate to each object $Q$ in $\B\G$ an isomorphism $P \otimes_\G Q \to F(Q)$.
    We will first define $\eta$ for trivial $\G$ bundles. Suppose $g^*\G \to N'$ is such a $\G$-bundle. Then $P \otimes_\G g^* \G$ can be identified with $g^*f^* \H$ (due to the fact that $P$ admits a section).

    The equivalence $\bar P \to \bar N \til\times_{\B\H} \bar N$ sends $\rho \colon N' \to P$ to the triple $(\t \circ \rho, \alpha(\rho), \s \circ \rho)$ where
    \[ \alpha(\rho) \colon (t \circ \rho)^*\H \to F((\s \circ \rho)^*\G) , . \]
    Then we can define $\eta$ on trivial $\G$-bundles to be $\eta(g^* \G) := \alpha(\sigma \circ g)$. Therefore:
    \[ \eta(g^* \G) \colon g^* f^* \H \to F(g^* \G) \, . \]

    To define $\eta$ for non-trivial $P$, we extending it (non canonically) by making a choice (for each $P$) of a local trivialization and define $\eta(P)$ to be the unique $\H$-bundle isomorphism obtained from gluing the identifications coming from $\eta$ applied to each trivialized piece. Since $\eta$ satisfies the condition for being a natural transformation locally, it also satisfies the condition globally after glueing. This shows $\B \colon \Bun(\G,\H) \to \Hom(\G,\H)$ is essentially surjective.

    We will now show that it is fully faithful. Fix a left principal $(\G, \H)$-bibundle $P$ for $\C$-groupoids $\G\grpd M$ and $\H \grpd N$. Now consider the function $\B \colon \Isom(P,P) \to \Isom(\B P, \B P)$. Suppose that $\B \phi = \Id_P$ for $\phi \colon P \to P$.
    Then for all $Q \in \B \G$, $\phi \otimes \Id \colon P \otimes Q \to P \otimes Q$ is the identity, hence $\phi \colon P \to P$ must be the identity. This shows that the function is injective.

    Now let $\eta \colon \B P \to \B P$ be a natural transformation, we want to show $\eta$ is in the image. We know that $\eta(Q) \colon P \otimes Q \to P \otimes Q$ and
    \[
    \begin{tikzcd}
      P \otimes Q_1 \arrow[r, "\eta(Q_1)"] \arrow[d, "\B P(\psi)"] & P \otimes Q_1 \arrow[d, "\B P(\psi)"] \\
      P \otimes Q_2 \arrow[r, "\eta(Q_1)"]                           & P \otimes Q_2
    \end{tikzcd}
    \]
    commutes for any morphism $\psi \colon Q_1 \to Q_2$ in $\B \G$.
    In particular, by taking $Q = \H$ and observing that there is a canonical isomorphism $P \otimes \H \isom P$ then we get a $\G$-bundle isomorphism $\phi \colon P \to P$. We claim that $\phi$ is a bibundle morphism. To show this we need to explain why $\phi$ is $\H$-equivariant. To see this, consider an arbitrary morphism $\gamma \colon N' \to \H$, by the natural isomorphism condition the following diagram commutes:
    \[
    \begin{tikzcd}
      P \otimes (\t \circ \gamma)^* \H \arrow[rr, "\eta(Q_1)"]
      \arrow[d, "{p \otimes (h,x) \mapsto p \otimes (h \cdot \gamma(x), x)}" swap]
      & & P \otimes (\t \circ \gamma)^* \H
      \arrow[d, "{p \otimes (h,x) \mapsto p \otimes (h \cdot \gamma(x), x)}"]  \\
      P \otimes (\s \circ \gamma)^* \H
      \arrow[rr, "\eta(Q_1)"]              &             & P \otimes (\s \circ \gamma)^* \H  \\
  \end{tikzcd}
    \]
    Using the definition of $\phi$, we can extend this diagram as below:
    \[
    \begin{tikzcd}
      (\t \circ \gamma)^* P \arrow[d, "{(p,x) \mapsto p \otimes (\gamma(x),x))}" swap] \arrow[rr, "{(p,x) \mapsto (\phi(p),x)}"]  & & (\t \circ \gamma)^* P  \arrow[d, "{(p,x) \mapsto p \otimes (\gamma(x),x))}"] \\
      P \otimes (\t \circ \gamma)^* \H \arrow[rr, "\eta(Q_1)"]
      \arrow[d, "{p \otimes (h,x) \mapsto p \otimes (h \cdot \gamma(x), x)}" swap]
      & & P \otimes (\t \circ \gamma)^* \H
      \arrow[d, "{p \otimes (h,x) \mapsto p \otimes (h \cdot \gamma(x), x)}"]  \\
      P \otimes (\s \circ \gamma)^* \H
      \arrow[d, "{p \otimes (h,x) \mapsto (p \cdot h, x)}" swap] \arrow[rr, "\eta(Q_1)"]              &             & P \otimes (\s \circ \gamma)^* \H \arrow[d, "{p \otimes (h,x) \mapsto (p \cdot h, x)}"] \\
        (\t \circ \gamma)^* P\arrow[rr, "{(p,x) \mapsto (\phi(p),x)}"] & & (\t \circ \gamma)^* P  \\
    \end{tikzcd}
    \]
    Which implies that $(\phi(p) \cdot \gamma(x),x) = (\phi(p) \cdot \gamma(x), x)$ and so $\phi$ must be $\H$-equivariant.

    To finish, we need to show that $\B (\phi) = \eta \colon \B P \to \B \P$. Which is equivalent to showing that
    \[ \eta(Q) = \phi \otimes \Id_Q \colon P \otimes Q \to P \otimes Q \]
    for all objects $Q \in \B\G$. It suffices to show this for trivial $\H$-bundles since we only need to check equality of maps locally.
    Suppose $Q = f^* \H$ then by the definition of $\phi$ we have that
    \[
    \begin{tikzcd}
      P \otimes f^* \H \arrow[rr, "\eta(f^* \H)"] \arrow[d] & & P \otimes f^* \H \arrow[d] \\
      f^* P \arrow[rr, "{(p,x) \mapsto (\phi(p),x)}"] &  & f^* P
    \end{tikzcd}
    \]
    commutes, where the vertical arrows are the canonical identifications. Since $\phi$ is equivariant, we also know that:
    \[
    \begin{tikzcd}
      P \otimes f^* \H \arrow[rr, "\eta(f^* \H)"] \arrow[d] & & P \otimes f^* \H \arrow[d] \\
      f^* P \arrow[rr, "{(p,x) \mapsto (\phi(p),x)}"] &  & f^* P
    \end{tikzcd}
    \]
    Since $\phi$ is $\H$-equivariant, we have that:
    \[
    \begin{tikzcd}
      P \otimes f^* \H \arrow[rr, "{\phi \otimes \Id_\H}"] \arrow[d] & & P \otimes f^* \H \arrow[d] \\
      f^* P \arrow[rr, "{(p,x) \mapsto (\phi(p),x)}"] &  & f^* P
    \end{tikzcd}
    \]
    commutes as well. Since the vertical arrows are the same isomorphisms in both of these diagrams, we can put them together to obtain the desired equality.
  \end{proof}
    These two propositions complete the proof of the main theorem. In the next section we will look at a few examples of sites and study the geometric meaning of these concepts in each setting.

    \section{Examples}\label{section:examples1}
    Here we will briefly explore a few examples of good sites and the geometric interpretation of our theorems.
    \subsection{Manifolds}
    Let $\C$ be the site of manifolds. Then a $\C$-groupoid is a Lie groupoid. A principal $\C$-groupoid bundle is the same as a principal bundle of the associated Lie groupoid. Principal bibundles of Lie groupoids are called Morita equivalences. The theorem says the study of geometric stacks over the site of smooth manifolds is equivalent to the study of principal groupoid bundles.

    It is clear that a submersion in the sense we have discussed thus far is the same as a surjective submersion of manifolds. However, if we wish, we can be more strict with the definition of submersion, and we will obtain a corresponding class of geometric stack represented by a Lie groupoid.

    \begin{example}
      Suppose we redefine submersion to mean \emph{locally trivial fibration}. Then a $\C$-groupoid becomes a Lie groupoid $\G \grpd M$ for which the source map makes $\G$ into a locally trivial fibration. By our theorem, the stacks represented by such groupoids are precisely those which admit a presentation $p \colon M \to \X$ which is a locally trivial fibration. In this case, we mean that $p$ is a representable epimorphism such that for any $\bar N \to X$, the associated map $\bar M \times_{\X} \bar N \to N$ is a locally trivial fibration of manifolds.

      A sub-example of this case are the stacks represented by \emph{source proper} Lie groupoids. That is, groupoids whose source maps are proper submersions.
    \end{example}
    Apart from alternative notions of submersion, we can also enrich the category of manifolds with additional data.
    \begin{example}
      Let $\C$ be the category whose objects are pairs $(M,X)$ where $X$ is a vector field on the smooth manifold $M$. A morphism in this category is a smooth map $f \colon (M,X) \to (N,Y)$ such that $X$ and $Y$ are $f$-related.

      We can give $\C$ a Grothendiek topology by setting a covering family to be a collection of open embeddings $(U_i, X|_{U_i}) \to (M, X)$ whose images cover $M$. This site comes with a forgetful functor to $\Man$. A submersion in this category is just a surjective submersion. A $\C$-groupoid is a Lie groupoid $\G$ equipped with a vector field $X$ which is \emph{multiplicative}.
      \[ \m_*(X,X) = X \]

      Recent work by Berwick-Evans and Lerman \cite{Berwick} has shown that such a vector field is a reasonable candidate for the definition of a vector field on a geometric stack.
      Hence, we obtain our first example of an interesting strategy. One can try to understand structures on a geometric stack by equipping the site itself with additional structure.
    \end{example}

    \subsection{Schemes}

    The site of schemes over a base satisfies our axioms for a \emph{good site} when equipped with the etale topology. Gluing along submersions is easy since schemes are already defined locally. If one is more strict and requires more topological properties such as seperatedness, then one might need to be careful, but most of the time these properties are preserved as long as you gluing inside of $\Sub$.
    In this context, it is also reasonable to replace site theoretic submersions with `smooth' morphisms of stacks.

    Our definition of a geometric stack is very close to the definition of an \emph{Artin} stack. For an Artin stack, one commonly requires that the diagonal embedding:
    \[ \X \to \X \times \X \]
    be a representable morphism of stacks. We have a good reason for not requiring this in the manifold setting since it would imply that:
    \[ (\t,\s) \colon \G \to M \times M \]
    is a local submersion, due to the following pullback square:
    \[
    \begin{tikzcd}
      \bar \G \arrow[r, "{(\t,\s)}"] \arrow[d]     & \bar M \times \bar M \arrow[d] \\
      \B\G     \arrow[r]               & \B\G \times \B\G
    \end{tikzcd}
    \]
    Since every map of schemes is cartesian, the condition of representability is much weaker in the algebraic geometry setting.
    A nice property of this Artin condition is that it is equivalent to requiring that for each $X$ and $Y$ in $\X_N$, there is a scheme $S$ such that maps into $S$ classify isomorphisms $X \to Y$ which cover the identity. For Lie groupoids, we can only expect this to be the case when $X$ and $Y$ are trivial $\G$-bundles.

    \subsection{Algebroids}

    A \emph{Lie algebroid} is a vector bundle $A \to TM$, equipped with a Lie bracket $[ \cdot , \cdot]_A $ on its space of sections and a \emph{anchor map} $\rho \colon A \to TM$ such that:
    \begin{equation}\label{eqn:algebroidleibnitz}
      [ \alpha , f\beta]_A = f [\alpha, \beta]_A + \rho(\alpha)(f) \beta \quad \forall \alpha, \beta \in \Gamma(A) \quad \forall f \in C^\infty(M)
    \end{equation}

    If $A$ and $B$ are Lie algebroids, a morphism of Lie algebroids is a vector bundle morphism $F \colon A \to B$ which commutes with the anchors and is compatible with the brackets. We will not state the bracket compatibility condition here, but we will comment that it implies that
    \[ F_*([\alpha,\beta]_A) = [F_*(\alpha),F_*(\beta)]_B \]
    whenever the pushforward of $\alpha$ and $\beta$ by $F$ are well defined.

    Let $\C$ be the category whose objects are algebroids and whose morphisms are algebroid morphisms. The cartesian morphisms in this category are precisely the ones whose underlying vector bundle map is a vector bundle epimorphism covering a (local) submersion.
    Equip $\C$ with the Grothendieck (pre)topology whose covering families are the inclusion $A|_U \into A$ for open embeddings $U \to M$.

    Under this definition of Grothendieck topology, the ``submersions'' are algebroid maps $F \colon A \to B$ which cover submersions and admit local algebroid splittings $B|_U \to A$. Unlike with ordinary vector bundles, finding splittings of algebroid morphisms is non-trivial and we cannot expect them to exist in general. Hence, a submersion in this category is a fairly strong condition.

    A $\C$-groupoid is a groupoid internal to algebroids $A \grpd B$ such that the source (and target) morphisms are submersions in the above sense. This is the same as a LA-groupoid~\cite{Mackenzie} except for the local splitting condition.
    \[
    \begin{tikzcd}
      A \darrow \arrow[r] & \G \darrow \\
      B \arrow[r]         & M
    \end{tikzcd}
    \]
    The forgetful functor from $\C$ to manifolds can be extended to a functor $\B A \to \B\G$ where $\G$ is the underlying Lie groupoid associated to $A \grpd B$. This is a candidate for a definition of a Lie algebroid on a geometric stack. One caveat to this is that $\B A$ is not a stack over $\Man$. This is in contrast to the non-singular setting, where the total space of a Lie algebroid is also a manifold.
    One can get around this issue by first applying the ``total space'' functor to $\B A$ and then stackifying.

    One should take care that a principal $A \grpd B$ bundle internal to algebroids is not the same as a principal bundle as a $\C$-groupoid. The important distinction is the local splitting conditions. The main distinction comes from the fact that we require submersions to admit local sections and, in general, one cannot expect a surjective morphism of algebroids $F \colon A \to B$ to admit local sections, even if the base of the morphism is the identity.

\chapter{Morita equivalence}\label{chap:morita}

In this chapter, we will will introduce and study various notions of Morita equivalence and the Picard group.
We will also review some strategies for computing Morita equivalences as well as take a look at a few examples.
Throughout this chapter, $\C$ is a good site (Definition~\ref{defn:goodsite}).

Recall that principal bibundles of $\C$-groupoids are precisely the isomorphisms in $\morpoid$.
Isomorphism in $\morpoid$ has a special name:
\begin{definition}
  Let $\G$ and $\H$ be $\C$-groupoids. A \emph{Morita equivalence} between $\G$ and $\H$ is a principal $(\G,\H)$-bibundle.
\end{definition}
This chapter will be organized as follows:
\begin{itemize}
  \item In Section~\ref{section:maplike} we explore the relationship between bibundles of $\C$-groupoids and $\C$-groupoid homomorphisms. In particular, we define a special class of $\C$-groupoid homomorphisms called \emph{weak equivalences}.
  \item In Section~\ref{section:moritaequiv} we prove the main results of this chapter. The main theorem essentially states that the equivalence relation generated by weak equivalences is the same as Morita equivalence.
  \item In Section~\ref{section:picard} we define the Picard group of a $\C$-groupoid $\G$, which turns out to be the group of automorphisms of $\B \G$.
  \item In Section~\ref{section:examples2} we look at a few examples of Morita equivalences and the Picard group for specific sites.
\end{itemize}

\section{Generalized morphisms and homomorphisms}\label{section:maplike}

In Chapter~\ref{chap:stacks} we saw that there is a functor $\B \colon \morpoid \to \mathbf{GStacks}$ from Groupoids with generalized morphisms to geometric stacks.
In this section we will note that there is a pseudofunctor (Definition~\ref{defn:pseudofunctor})
\[ \bP \colon \C\GRPD \to \morpoid \]
from the category of $\C$-groupoids (with homomorphisms as morphisms) to $\C$-groupoids with generalized morphisms.

\subsection{Homomorphisms of groupoids}

Recall that a $\C$-groupoid homomorphism from $\G \grpd M$ to $\H \grpd N$ covering $f \colon M \to N$, is a morphism $F \colon \G \to \H$ which is compatible with the groupoid structure maps.
\begin{definition}\label{defn:groupoidmaps}
  Let $F\colon \G \to \H$ be a $\C$-groupoid morphism covering $f \colon M \to N$. Then:
  \begin{itemize}
    \item $F$ is called \emph{essentially surjective} if $\t \circ \pr_1 \colon \H \times_N M \to N$ is a submersion.
    \item $F$ is called \emph{fully faithful} if $M \times_{f, \t} \H \times_{f, \s} M$ exists and:
    \[ (\t, F, \s ) \colon \G \to M \times_{f, \t} \H \times_{\s, f} M \]
    is an isomorphism.
    \item $F$ is called an \emph{weak equivalence} if it is both essentially surjective and fully faithful.
  \end{itemize}
  Observe that the essentially surjective condition depends only of the base map $f$. For this reason we say an arbitrary morphism $f \colon M \to N$ is essentially surjective if $1_M \to \H$ is essentially surjective.
\end{definition}
Unsurprisingly, one can make groupoids and groupoid homomorphisms into a bicategory.
To do this, we must give the analogue of a natural transformation internal to $\C$.
\begin{definition}\label{defn:cgroupoidnaturaltrasn}
  Let $F \colon \G \to \H$ and $G \colon \G \to \H$ be groupoid homomorphisms covering $f \colon M \to N$ and $g \colon M \to N$, respectively.
  A \emph{natural transformation} from $F$ to $G$ is a morphism $\eta \colon M \to \H$ such that:
  \begin{itemize}
  \item $ \s \circ \eta = f$ and $\t \circ \eta = g$;
  \item $\m(\eta \circ \t, F) = \m(G, \eta \circ s) \colon \G \to \H$
  \end{itemize}

  Given $F_i \colon \G \to \H$ covering $f_i \colon M \to N$ for $i=1,2,3$ and a pair of natural transformations $\eta_1 \colon F_1 \to F_2$ and $\eta_2 \colon F_2 \to F_3$ then the vertical composition $\eta_2 * \eta_1$ is defined to be:
  \[ \eta_2 * \eta_1  := \eta_2  \cdot \eta_1 \colon M \to \H \, . \]

  Now suppose we are given $\eta \colon F_1 \to F_2$ and $\zeta \colon G_1 \to G_2$ for $F_i \colon \G \to \H$ and $G_i \colon \H \to \K$. Then the horizontal product $\zeta \circ \eta$ is defined as below:
  \[
  \zeta \circ \eta (x) := (G_2 \circ \eta(x)) \cdot  (\zeta \circ f_1(x))
  \]
\end{definition}
These constructions give a natural bicategory structure on $\C$-groupoids which is analogous to the bicategory of standard groupoids.
The important caveat is that everything is performed internally to $\C$.
\newpage
\begin{definition}
Let $\C\GRPD$ be the bicategory defined as follows:
\begin{itemize}
  \item The objects of $\C\GRPD$ are $\C$-groupoids.
  \item The 1-morphisms of $\C\GRPD$ are $\C$-groupoid homomorphisms.
  \item The 2-morphisms of $\C\GRPD$ are natural transformations.
\end{itemize}
Composition of 1-morphisms is defined in the obvious manner. Horizontal and vertical composition of 2-morphisms is defined as above.
\end{definition}
It is not difficult to check that this satisfies the axioms of a bicategory. In fact, it is a \emph{strict} bicategory since the horizontal product is strictly associative.
We should offer a word of caution, however. Unlike the case of set-theoretic groupoids, a fully faithful and essentially surjective morphism of $\C$-groupoids does not necessarily admit a weak inverse.

\subsection{Pullback groupoids}

We would like to point out one construction which is particularly useful for computing Morita equivalences.
\begin{lemma}
  Let $\G \grpd M$ be a $\C$-groupoid and suppose $f \colon N \to M$ is essentially surjective. Then $N \times_M \G \times_M N$ exists and inherits a canonical groupoid structure such that projection to the middle component $N \times_M \G \times_M N \to \G$ is a groupoid homomorphism.
\end{lemma}
\begin{proof}
  Let $f^! \G$ be shorthand for $N \times_M \G \times_M N$. To give a groupoid structure on $f^! \G$ we should first define the source and target maps.

  Let the source $f^! \G \to N$ be projection to the last component and the target be projection to the first component. The source and target maps are submersions since they are the base changes of a submersion.

  Let $F \colon f^! \G \to \G$ be projection to the middle component. Then we define the product on $f^! \G$ as follows:
  \[ (z,g_2,y) \cdot (y,g_1, x):= (z, g_2 \cdot g_1 ,x). \]
  This clearly makes $F$ a $\C$-groupoid homomorphism.
  \end{proof}

\begin{definition}\label{defn:pullbackgroupoid}
  Let $G \grpd M$ be a $\C$-groupoid and $f \colon N \to M$ be essentially surjective. Then
  \[ f^! \G := N \times_M \G \times_M N \]
  is called the \emph{pullback groupoid} along $f$.
\end{definition}

The reason we do note use the notation $f^* \G$ to denote the pullback groupoid is due to possible confusion with the pullback of the trivial $\G$-bundle.
Recall that $f^* \G$ is defined to be $\G \times_M N$ while $f^! \G$ is $N \times_M \G \times_M N$.
Since the pullback of $\G$ as a bundle and as a groupoid do not coincide, it will spare us some possible confusion if we use distinct notation.

It is not necessary that $f \colon N \to M$ be essentially surjective for this construction to result in a $\C$-groupoid. It is clear however that it will not always work. Even if $f^! \G$ exists as an object in $\C$, it may not necessarily be the case that the source and target maps are submersions.
\begin{lemma}\label{lemma:pullbackequiv}
  Let $\G \grpd M$ be a $\C$-groupoid and suppose $f \colon N \to M$ is essentially surjective.
  The canonical homomorphism $F \colon f^! \G \to \G$ is a weak equivalence.
\end{lemma}
\begin{proof}
  The result is clear from the definition of fully faithful.
\end{proof}

\subsection{Map-like bibundles}

Given a $\C$-groupoid homomorphism $F \colon \G \to \H$ covering $f \colon M \to N$, there is a canonical construction of a generalized morphism.
Let
\[
\bP(F) := f^* \H  = \H \times_N M \, .
\]
Since this is just the trivial bundle associated to $f$, there is a canonical left principal $\H$-bundle structure over $M$.
To make $\bP(F)$ into a bibundle (and hence a generalized morphism) we must equip it with a right $\G$-action. Let:
\[
(h, \t(g)) \cdot  g := (h \cdot F(g), \s(g))
\]
\begin{definition}
  A 1-morphism $\G \to \H$ in $\morpoid$ is called \emph{map-like} if it is isomorphic to $\bP(F)$ for some homomorphism $F \colon \G \to \H$.
\end{definition}
The next lemma gives a geometric characterization of the map-like bibundles.
\begin{lemma}\label{lemma:maplike}
  A generalized morphism is map-like if and only if it admits a section (as a left $\G$-bundle).
\end{lemma}
\begin{proof}
  Let $F \colon \G \to \H$ be a $\C$-groupoid homomorphism and suppose $Q \isom \bP(F)$ is map-like. Since $\bP(F)$ clearly admits a section (it is trivial as a left $\H$-bundle) then $Q$ must admit a section.

  On the other hand, suppose $N \from Q \to M$ is a left principal $(\H, \G)$-bibundle and admits a section. We must show it is isomorphic to $\bP(F)$ for some $F \colon \G \to \H$. Let $\sigma \colon M \to Q$ be a section of $Q$. Now let $F$ be the unique morphism such that:
  \[
   \sigma(\t(g)) \cdot g = F(g) \cdot \sigma(\s(g))
  \]
  To see that this is a homomorphism, observe that:
  \begin{align*}
    F(g_1 \cdot g_2) \cdot \sigma(\s(g_1 \cdot g_2)) &= \sigma(\t(g_1 \cdot g_2)) \cdot g_1 \cdot g_2 \\
    &= \sigma(\t(g)) \cdot g_1 \cdot g_2 \\
    &= F(g_1) \sigma(\s(g_2)) \cdot g_2 \\
    &= F(g_1) \cdot F(g_2) \cdot  \sigma(\s(g_2)) \\
    &= F(g_1) \cdot F(g_2) \cdot \sigma(\s(g_1 \cdot g_2))
  \end{align*}
  Now we construct an isomorphism $\bP(F) = \H \times_N M \to Q$. Let
  \[ \phi(h,x) \mapsto h \cdot \sigma(x)  \]
  By Lemma~\ref{lemma:bundletrivialization} we know that $\phi$ is an isomorphism of principal $\H$-bundles. That $\phi$ is $\G$ equivariant follows from a simple computation:
  \[ \phi(h,x) \cdot g = h \cdot \sigma(x) \cdot g = h \cdot F(g) \cdot \sigma(\s(g)) = \phi(h F(g), \s(g)) \]
  \end{proof}
The assignment $F \mapsto \bP(F)$ is bifunctorial in the following sense.
\begin{lemma}
  Suppose $F \colon \G \to \H$ and $G \colon \H \to \K$ are $\C$-groupoid homomorphisms covering $f \colon M \to N$ and $g \colon N \to O$, respectively. Then there is a canonical isomorphism:
  \[ \mathbb{P}(G,F) \colon \bP(G) \otimes_\H \bP(F) \to \bP({G \circ F}) \]
\end{lemma}
\begin{proof}
  It suffices to construct an $\H$-invariant map $\Phi \colon \bP(G) \times_N \bP(F) \to \bP({F \circ G})$ which respects the left and right actions. Recall that
  \[ \bP(G) := g^* \K = \K \times_O N \quad \bP(F) := f^* \H = \H \times_O N \quad \bP({G \circ F}) := (g \circ f)^*\K = \K \times_O M \, . \]
  Hence we define
  \[
  \Phi((k,y),(h,x)) := ( k \cdot G(h), x) \, .
  \]
  It is a routine check that $\Phi$ is $\H$-invariant and respects the actions of $\K$ and $\G$. Therefore, it must descend to a bibundle isomorphism $\mathbb{P}(G,F)$.
\end{proof}

Isomorphisms of $\C$-groupoid homomorphisms can be related to bibundle isomorphisms by the following process.
Suppose $\eta \colon F \to G$ is a natural transformation of $\C$-groupoid homomorphisms from $\G$ to $\H$ covering $f \colon M \to N$ and $g \colon M \to N$, respectively.

Then let
\[ \bP(\eta)(h,x) := (h \cdot (\i \circ \eta)(x),x) \]
Note that this uses the (opposite) correspondence from Lemma~\ref{lemma:gbundlecycle}.
By that same lemma, this is $\H$ equivariant. That is is $\G$-equivariant follows from the natural transformation condition.
\begin{proposition}
  The mapping $\eta \mapsto \bP(\eta)$ gives a one-to-one correspondence between natural transformations $\eta \colon F \to G$ of $\C$-groupoid homomorphisms $\G \to \H$ and bibundle isomorphisms $\bP(F) \to \bP(G)$.

  Furthermore, this correspondence respects the vertical and horizontal products in the following senses:
  \begin{itemize}
    \item For all $\zeta \colon G_1 \to G_2$ and $\eta \colon F_1 \to F_2$ which are horizontally composable, the below diagram commutes:
    \begin{equation}\label{diagram:pcompatibility}
    \begin{tikzcd}
      \bP(G_1) \otimes \bP(F_1) \arrow[d, "{\mathbb{P}(G_1,F_1)}"]\arrow[rr, "{\bP(\zeta) \otimes \bP(\eta)}"] & & \bP(G_2) \otimes \bP(F_2)\arrow[d, "{\mathbb{P}(G_2,F_2)}"] \\
      \bP(G_1 \circ F_1) \arrow[rr, "{\bP(\zeta \circ \eta)}"] & & \bP(G_2 \circ F_2)
    \end{tikzcd}
  \end{equation}
    \item For all $\zeta$ and $\eta$ which are vertically composable $\bP(\zeta * \eta) = \bP(\zeta) \circ \bP(\eta)$.
  \end{itemize}
\end{proposition}
\begin{proof}
  The correspondence is one-to-one thanks to Lemma~\ref{lemma:gbundlecycle}.
  \begin{itemize}
    \item Suppose $\zeta$ and $\eta$ are horizontally composable such that $F_i \colon \G \to \H$ and $G_i \colon \H \to \K$. Then the bottom left path of Diagram~\ref{diagram:pcompatibility} is:
    \[
    \begin{tikzcd}
      (k , \t(h)) \otimes (h , x) \arrow[d, mapsto]  & \\
      (k \cdot G_1(h), x) \arrow[r] & (k \cdot G_1(h) \cdot (\i \circ \zeta)(f_1(x)) \cdot (\i \circ G_2 \circ \eta)(x), x)
      \end{tikzcd}
    \]
    Going the other direction, we get:
    \[
    \begin{tikzcd}
      (k , \t(h)) \otimes (h , x) \arrow[r, mapsto]
      & \left(k \cdot (\i \circ \zeta) (\t(h)), \t(h) \right) \otimes \left( h \cdot (\i \circ \eta)(x),x \right) \arrow[d, mapsto] \\
       & (k \cdot (\i \circ \zeta)(\t(h)) \cdot G_2(h) \cdot (\i \circ G_2 \circ \eta)(x),x )
      \end{tikzcd}
    \]
    But since $\zeta$ is a natural transformation, we have:
    \[
    (\i \circ \zeta)(\t(h)) \cdot G_2(h)  = G_1(h) \cdot (\i \circ \zeta)(\s(h)) = G_1(h)
    \]
    So the diagram commutes.
    \item Suppose $\zeta \colon F_2 \to F_3$ and $\eta \colon F_1 \to F_2$ are horizontally composable for $F_i \colon \G \to \H$ covering $f_i \colon M \to N$. Recall that $(\zeta * \eta)(x) := \zeta(x) \cdot \eta(x)$. Therefore:
    \[ \bP(\zeta * \eta)(h,x) := (h \cdot \i \circ (\zeta(x) \cdot \eta(x)), x) = (h \cdot (\i \circ \eta)(x) \cdot (\i \circ \zeta)(x), x) \]
    Which is the same as $\bP(\zeta) \circ \bP(\eta)$.
  \end{itemize}
\end{proof}
The following theorem summarizes our results so far:
\begin{theorem}
  $P \colon \C\GRPD \to \morpoid$ is a psuedofunctor from the bicategory of $\C$-groupoids to the bicategory of generalized morphisms.
  Furthermore, $P$ is a bijection on objects. The restriction of $P$ to the $\Hom$ categories
  \[ \Hom{(\G,\H)}_{\C\GRPD} \to \Hom{(\G,\H)}_{\morpoid} \]
  is a fully faithful embedding whose essential image consists of the the map-like generalized morphisms.
\end{theorem}
\begin{proof}
  We take $P$ to be the identity on objects. Except for coherence, the theorem follows from the propositions and lemmas we proved in this subsection. We will leave the proof of coherence for this functor to the bored reader.
\end{proof}

\section{Morita equivalence}\label{section:moritaequiv}
\subsection{Weak equivalences}
It is natural to ask whether the bifunctor defined in the previous section allows us to to study Morita equivalences purely in terms of $\C$-groupoid homomorphismse
To answer this question, we will begin with characterizing those morphisms which give rise to Morita equivalences.
\begin{lemma}
  Suppose $F \colon \G \to \H$ is a homomorphism of $\C$-groupoids. Then $\bP(F)$ is a Morita equivalence if and only if $F$ is a weak equivalence.
\end{lemma}
\begin{proof}
Suppose $F$ is a weak equivalence.
We only need to show that the right action of $\G$ on $\bP(F)$ is principal.
Since $F$ is a weak equivalence, we can assume without loss of generality that $\G \isom f^! \H$.
Recall that $\bP(F) := \H \times_N M$ and $f^!\H := M \times_N \H \times_N M$.
If $\C$ is made of sets and functions, the right action of $\bP(F)$ is of the form:
\[ (h, x) \cdot (x, h', y) = (h \cdot h', y) \, . \]
The total action is therefore:
\[ ((h,x), (x,h',y)) \mapsto ((h,x),(h \cdot h', y)) \]
This leads us to notice that
\[ ((h_1,x),(h_2,y)) \mapsto (x, h_1\inv h_2 , y) \]
is a division map for the right action.
Since this division map is unique, the action must be principal.
\end{proof}
Now that we have identified the $\C$-groupoid morphisms which give rise to Morita equivalences. It is reasonable to wonder if there are enough of these maps to generate the Morita equivalence relation. That is, given two Morita equivalent $\C$-groupoids, does there exist a chain of weak equivalences connecting them? The answer is yes and actually two weak equivalences suffice.
\begin{theorem}\label{thm:weakequivalences}
  The following are equivalent:
  \begin{itemize}
    \item$\G$ and $\H$ are Morita equivalent.
    \item There exists a $\C$-groupoid $\K$ and a pair of weak equivalences $\H \from \K \to \G$.
  \end{itemize}
\end{theorem}
\begin{proof}
  Going from the second to the first is clear.

  Let $P$ be a biprincipal $(\H,\G)$-bibundle for $\H \grpd N$ and $\G \grpd M$. By Lemma~\ref{lemma:pullbackequiv}, the mapping $(\s^P)^! \G \to \G$ is a weak equivalence. Let $\K := (\s^P)^! \G $. To conclude the proof, we must construct a weak equivalence $F \colon \K \to \H$. Let $F$ be the unique morphism satisfying:
  \[ p_1 \cdot g = F(p_1,g,p_2) \cdot p_2 \]
  The existence and uniqueness of $F$ follows from the existence and uniqueness of the division map for the left $\H$-action on $P$.
  A quick check verifies that this is indeed a homomorphism of $\C$-groupoids covering $\t^P \colon P \to N$.
  To prove that this is a weak equivalence it suffices to show that the induced map
  \[ \K = (\s^P)^! \G \to (\t^P)^! \H \]
  is an isomorphism.
  Since the right action is principal as well, by the same trick we can construct a homomorphism $F' \colon (\t^P)^! \H \to \K$ using the rule:
  \[ (p_1, h, p_2) \mapsto (p_1, g, p_2) \quad \Leftrightarrow \quad p_1 \cdot g = g \cdot p_2  \]
  Since $F'$ is clearly the inverse of $\K \to (\t^P)^! \H$ the result follows.
\end{proof}
The above proposition tells us that if we wish to show that something is Morita invariant, it suffices to show that it is invariant under weak equivalences.
\subsection{Examples of Morita equivalences}
\begin{example}[Manifolds]
    When $\C$ is the site of smooth manifolds. Then the $\C$-groupoids are precisely the Lie groupoids. Given a Lie groupoid $\G \grpd M$, the following are Morita invariants:
    \begin{itemize}
      \item The orbit space of $\G$, that is the topological space obtained by taking the quotient of $M$ by the equivalence relation on $M$ induced by the map $(\s,\t) \colon \G \to M \times M$.
      This is sometimes called the \emph{coarse moduli space} of $\G$.
      \item Suppose one is given an orbit $\O \subset M$ of $\G$ in $M$. In other words, a point in the coarse moduli space of $\G$. Then for any pair of points, $x,y \in \O$ the isotropy groups $\G_x := (\s,\t)\inv(x,x)$ and $\G_y$ are isomorphic.
      Hence, to each orbit we can associate an isomorphism class of an isotropy group, and these classes are preserved by Morita equivalences.
    \end{itemize}
    This is (not even close) to being an exhaustive list of the extensive number of Morita invariants which are known.
    However, these two Morita invariants fundamental in that they provide one with some justification for the intuition that the stack associated to a Lie groupoid should be thought of as a singular manifold equipped with additional symmetry groups.
\end{example}
\begin{example}[Dirac Structures]
  Suppose $\C = \DMan$ (see Chapter~\ref{chap:dman}). Then a $\C$-groupoid is a D-Lie groupoid. In the previous chapter, we say that a Morita equivalence of D-Lie groupoids always has an underlying Morita equivalence of Lie groupoids. Hence any Morita invariant of Lie groupoids is also a Morita invariant of D-Lie groupoids.

  A D-Lie groupoids often come with additional interesting invariants.
  \begin{itemize}
    \item Suppose $\G \grpd M$ is a symplectic groupoid. Then $M$ inherits a Poisson bivector. By choosing a density $\mu$ on $M$ one can construct what is called the \emph{modular vector field} $X_\mu$. It turns out that different choices of densities on $M$ give rise to modifications of $X_\mu$ by Hamiltonian vector fields. Hence, we obtain what is called the \emph{modular class}
    \[ [X_\mu] \in \X(M) / \X_{ham}(M) := H^1_\pi(M) \, .  \]
    Ginzburg and Lu showed~\cite{GinzburgPoissonCohomology} that symplectic Morita equivalences induced isomorphisms in $H^1_\pi(M)$.
    Later, it was shown by Crainic~\cite{crainicmod} that the modular class is preserved by such isomorphisms.
  \end{itemize}
\end{example}

\section{The Picard group}\label{section:picard}

Roughly, the Picard group is the group of Morita self equivalences.
\begin{definition}
  Suppose $\G \grpd M$ is a $\C$-groupoid. The Picard group of $\G$, denoted $\Pic(\G)$ is the set of principal $(\G,\G)$-bibundles, up to isomorphism.
\end{definition}
Under the correspondence developed in Chapter~\ref{chap:stacks}, we know that the Picard group is the automorphisms of the associated stack $\B\G$ modulo natural transformations. Of course, that point of view does not lend itself to calculations. In general, we should expect that the Picard group of a given groupoid is difficult, if not impossible, to compute.

\subsection{Outer automorphisms}
Let $P$ represent an element of the Picard group $\Pic(\G)$. Then we say that $[P] \in \Pic(\G)$ is \emph{map-like} if $P$ is map-like. Map-like elements of the Picard group do not form a subgroup in general since $[P]\inv$ need not be map-like.

\begin{definition}
  Consider the group homomorphism $\bP \colon \Aut(\G) \to \Pic(\G)$ which sends a $\C$-groupoid automorphism to the associated map-like element of $\Pic(\G)$. The kernel of this map is called the \emph{inner automorphisms} of $\G$, denoted $\InnAut(\G)$. The image of this map is canonically isomorphic to $\OutAut(\G) := \Aut(\G)/\InnAut(\G)$.
\end{definition}

We might wonder if one can characterize the image of $\bP$ geometrically. In order to do this, we need to introduce the notion of a bisection.

\begin{definition}
  Suppose $P$ is a principal $(\G,\H)$-bibundle for $\C$-groupoids $\G \grpd M$ and $\H \grpd N$. A \emph{bisection} of $P$ is a morphism $\sigma \colon N \to P$ such that $\t \circ \sigma \colon N \to M$ is an isomorphism.
\end{definition}

\begin{lemma}
  An element $[Q] \in \Pic(\G)$ comes from an automorphism $F \colon \G \to \G$ if and only if $Q$ admits a bisection.
\end{lemma}
\begin{proof}
  Suppose $Q$ admits a bisection. Then by Lemma~\ref{lemma:maplike} we can construct an $F \colon \G \to \G$ so that $Q \isom \bP(F)$. Since $P$ is principal we know that $F$ is a weak equivalence. Since $F$ is an isomorphism on objects it must be an isomorphism of $\C$-groupoids.

  The other direction is clear by taking the canonical section of $\bP(F)$ when $F$ is an isomorphism.
\end{proof}

We also have a characterization of inner automorphisms.

\begin{lemma}\label{lemma:innaut}
  A morphism $F \colon \G \to \G$ is an inner automorphism if and only if there exists a bisection $\sigma \colon M \to \G$ of $\G$ such that:
  \[ F(g) = (\sigma \circ \t)(g) \cdot g \cdot (\i \circ \sigma \circ \s(g))  \]
\end{lemma}
\begin{proof}
  By assumption there is an isomorphism $\Phi \bP(F) \to \G$ as $(\G,\G)$-bibundles. Let:
  \[
  \sigma(x) := \Phi(\u \circ f(x), x) \colon M \to \bP(F) \to \G
  \]
  Then we have that:
  \[
    \Phi(F(h),\s(h))  = \Phi(F(h)\cdot \u \circ f(\s(h)),\s(h)) = F(h) \cdot \sigma(\s(h))
  \]
  On the other hand:
  \[
  \Phi(F(h),x) = \Phi(F(h),x) = \Phi(\u \circ f(\t(h)),\t(h)) \cdot h =  \sigma(\t(h)) \cdot h
  \]
  Hence:
  \[ F(h) = \sigma(\t(h)) \cdot h \cdot (\i \circ \sigma)(\s(h)) \]
  On the other hand, suppose $F$ is given by such a formula. Then it is straightforward to check that:
  \[ \Phi \colon \P(F) \to \G \qquad \Phi(g,x) \mapsto g \cdot \sigma(x) \]
  yields the desired isomorphism.
\end{proof}
\section{Examples}\label{section:examples2}
\subsection{Manifolds}
Suppose $\C$ is the site of smooth manifolds.
\begin{definition}\label{defn:staticpicardgroup}
  Let $\G \grpd M$ be a Lie groupoid and let $M / \G$ denote the orbit space of $\G$ (as a purely topological object). The \emph{static Picard group} is the kernel of the homomorphism $\Pic(\G) \to \Aut(M / \G)$ denoted by $\Pic_Z (\G)$. In other words, it is the elements of the Picard group which act trivially on the orbit space.
\end{definition}
The static Picard group fits into an exact sequence:
\[  1 \to \Pic_Z (\G) \to \Pic(\G) \to \Aut(M / \G) \, . \]
Note that the last arrow is not surjective.
There is no guarantee that all topological symmetries of $M /\G$ can be witnessed by elements of the Picard group. We will treat a few instances for which the Picard group has been calculated.
\begin{example}[Trivial groupoids]
  Suppose $1_M \grpd M$ and $1_N \grpd N$ are trivial Lie groupoid. That is, every morphism is a unit. Then $1_M$ is Morita equivalent to $1_N$ if and only if they are diffeomorphic.
  A principal $(1_M,1_M)$-bibundle is the same as a diffeomorphism $f \colon M \to M$. Hence
  \[ \Pic(1_M) \isom \Diff(M) \, . \]
  Every element of the Picard group is map-like and the group of inner automorphisms of $1_M \grpd M$ is trivial. Lastly, the static Picard group is trivial as well since $\Pic(1_M) \to \Aut(M)$ is injective.
\end{example}
\begin{example}[Trivial line bundle]
  Let $ \G = \R_M \grpd M$ be the trivial line bundle. That is, the groupoid operation is just fiber-wise addition in $\R$. Suppose $P$ is a principal $(\R_M, \R_M)$-bibundle.
  The left action of $\G$ makes $P$ into an (affine) line bundle.
  Since every affine line bundle admits a global section, we know that $P$ comes from an automorphism of $\G$. Since $\G$ is abelian, the inner automorphisms of $\G$ are trivial and we conclude that
  \[ \Pic(\G) \isom \OutAut(\G) \isom \Diff(M). \]
  The static Picard group is trivial.
\end{example}
\begin{example}[Bundles of abelian groups]
  To the best of our knowledge, this case first appeared in a paper by Moerdijk~\cite{moerdijkgerbes}. It is a, more interesting, generalization of the previous example.

  Consider the case where $\G \grpd M$ is a Lie groupoid such that $\s = \t$ and is such that $\G_x := \t\inv(x)=\s\inv(x)$ is an abelian group for all $x \in M$. We can think of $\G$ as defining a sheaf whose local sections are bisections of $\G$. Any principal $(\G,\G)$-bundle $P$ defines a diffeomorphism $f_P \colon M \to M$ of the orbit space of $\G$.
  Furthermore, any $(\G,\G)$ bibundle is, in particular, a principal $\G$-bundle. Hence, there is an associated \v{C}ech cohomology class $c_P \in H^1(M,\G)$. This gives rise to a group isomorphism:
  \[ \Pic(\G) \to H^1(M,\G) \ltimes \Diff(M) \]
  where the action of $\Diff(M)$ on $H^1(M,\G)$ is by pullbacks.
\end{example}
\begin{example}[Lie groups]
  Suppose $\G \grpd \{ * \}$ is a Lie group. In other words, it is a Lie groupoid over a point. Then every $(\G,\G)$-bibundle admits a bisection and so $\Pic(\G) \isom \OutAut(\G)$. Note that, by Lemma~\ref{lemma:innaut}, the outer automorphisms of $\G$ as a groupoid are the same as the outer automorphisms of $\G$ as a Lie group.
\end{example}
\begin{example}[Transitive groupoids]
  Suppose $\G \grpd M$ is a Lie groupoid such that the orbit space $M /\ G$ is a point. For any $x \in M$, the inclusion of the isotropy group $\G_x \into M$ is a weak equivalence. Hence $\Pic(\G) \isom \Pic(\G_x) = \OutAut(\G_x)$.

  For example, if $\G = \Pi_1(M) \grpd M$ is the fundamental groupoid of $M$ and $M$ is connected, then we obtain $\Pic(\G) \isom \OutAut(\pi_1(M))$.
\end{example}
\subsection{Morita equivalence of Lie algebroids}

Recall that a Lie algebroid is a vector bundle $A \to M$ together with a Lie bracket
\[ [ \cdot , \cdot ] \colon \Gamma(A) \times \Gamma(A) \to \Gamma(A) \]
satisfying a Leibnitz type identity. When Pradines~\cite{Pradines} introduced the notion of Lie algebroid, he showed that any Lie groupoid can be differentiated to a Lie algebroid. If a Lie algebroid is isomorphic to one of this type, then we say it is integrable.
It is well known that not every Lie algebroid is integrable.

When a Lie algebroid is integrable, there is a unique (up to isomorphism) integration with simply connected source fibers called the \emph{canonical} integration.
For a reader interested in more details, we refer them to \cite{lecturesonint} and \cite{Cint}.
\begin{definition}
  Suppose $A \to M$ and $B \to N$ are integrable Lie algebroids, with canonical integrations $\G \grpd M$ and $\H \grpd N$, respectively. Then a Morita equivalence between $A$ and $B$ is a Morita equivalence of $\G$ and $\H$. Similarly, the Picard group of $A$ denoted $\Pic(A)$ is defined to be $\Pic(\G)$.
\end{definition}

\begin{example}[Tangent]
  The canonical integration of the tangent algebroid is the fundamental groupoid. Therefore, if $M$ and $N$ are connected, then $TM$ and $TN$ are Morita equivalent if and only if $\pi_1(M) \isom \pi_1(N)$. Furthermore, $\Pic(TM) = \OutAut(\pi_1(M))$.
\end{example}
\begin{example}
  If $A \to \{ * \}$ is a Lie algebroid over a point, then $A$ is the same as a Lie algebra. The canonical integration of $A$ is just the simply connected integration $G$ of the Lie algebra. Hence
  \[ \Pic(A) \isom  \OutAut(G) \isom \OutAut(\g). \]
  Two Lie algebras are Morita equivalent if and only if they are isomorphic.
\end{example}
\subsection{Symplectic Morita equivalence}
For $\C = \DMan$ (see Chapter~\ref{chap:dman}), there is a special class of $\C$-groupoids called \emph{symplectic groupoids}.
Morita equivalence of symplectic groupoids was originally developed by Xu~\cite{Morping} (without the use of Dirac structures). Picard groups of symplectic groupoids were first introduced and studied by Bursztyn and Weinstein\cite{BPic}.
We will recall the definition of these objects to make the following discussion more clear.
\begin{definition}
  A \emph{symplectic groupoid} is a Lie groupoid $\G$ together with a symplectic form $\Omega \in \Omega^2(\G)$ which is multiplicative:
  \[ \m^* \Omega = \pr_1^* \Omega + \pr_2^* \Omega \, . \]
  A \emph{Morita equivalence} of symplectic groupoids $\G$ and $\H$ is a principal $(\G,\H)$-bibundle $P$ together with a \emph{left} and \emph{right mulitplicative} 2-form $\omega$ on $P$.
  \[ \m_L^* \omega = \pr_1^* \Omega^\G + \pr_2^* \omega \qquad \m_R^* \omega = \pr_1^* \omega + \pr_2^* \Omega^\H \, . \]
  The \emph{Picard group} of a symplectic groupoid $\G$ is the set of isomorphism classes of Morita self equivalences.
\end{definition}

One important reason to be interested in symplectic groupoids is that they are the objects which integrate Poisson manifolds.
Given $(M, \pi)$ a Poisson manifold, then $T^*M$ inherits a Lie algebroid structure. When $T^*M$ is integrable, we say that $(M,\pi)$ is an integrable Poisson manifolds. The canonical integration of a Poisson manifold inherits a multiplicative symplectic form from the canonical symplectic form on $T^* M$.
\begin{definition}
  Suppose $(M, \pi_M)$ and $(N, \pi_N)$ are integrable Poisson manifolds. Then we say that $M$ and $N$ are Morita equivalent as Poisson manifolds if their canonical integrations are \emph{symplectically} Morita equivalent. Similarly, the Picard group of a Poisson manifold is defined to by the Picard group of its symplectic integration (as a $\DMan$-groupoid).
\end{definition}
\begin{example}
  Suppose $M$ and $N$ are any manifolds.
  Then $T^* M \grpd M$ and $T^* N \grpd N$ are symplectic groupoids via fiberwise addition.
  The Poisson structure that they integrate is given by the zero bivector.
  A principal $(T^* M, T^* N)$-bibundle $(P, \omega)$ always admits a bisection, $\sigma \colon N \to P$.
  The bisection gives rise to a diffeomorphism $f_P \colon N \to M$ and a 2-form $\beta := \sigma^* \omega$.
  Bursztyn and Weinstein showed that the resulting map:
  \[ \Pic(T^* M) \to H^2(M) \ltimes \Diff(M)  \]
  is a group isomorphism.
\end{example}
\begin{example}
  Suppose $(M, \omega\inv)$ is a connected Poisson manifold such that $\omega\inv \colon T^* M \to TM$ is the inverse of a symplectic form.
  It follows that $M$ is integrable and the canonical integration is the fundamental groupoid $\Pi_1(M)$.
  The symplectic structure on the fundamental groupoid is obtained by the pullbacks $\t^* \omega - \s^* \omega$.

  As we saw earlier, $\Pi_1(M)$ is Morita equivalent to the group $\pi_1(M,x)$ for any given base-point $x$.
  It turns out that this is Morita equivalence of symplectic groupoids.
  Hence, symplectic manifolds (interpreted as Poisson manifolds) are Morita equivalent if and only if their fundamental groups are isomorphic.

  This example tells us that we should not expect Morita equivalence to tell us very much about the \emph{leaf-wise} geometry of a Poisson manifold. Morita equivalences are fundamentally about the transverse geometry of the groupoid.
\end{example}
\begin{example}
  Suppose $\g$ is a Lie algebra. Then the dual $\g^*$ inherits a canonical Poisson structure.
  The canonical integration is $T^* G \isom G \times \g^*$ equipped with the (coadjoint) action groupoid structure, where $G$ is the simply connected integration of $\g$.

  For the compact and semi-simple case, Bursztyn and Fernandes~\cite{ruipic} proved the following theorem:
  \begin{theorem}
    Suppose $\g$ is a compact semi-simple Lie algebra. Then:
    \[ \Pic(\g^*) \isom \OutAut(\g) \]
  \end{theorem}
\end{example}

\chapter{The site of Dirac structures}\label{chap:dman}

In this chapter we will construct a site whose stacks include those associated to symplectic groupoids.
\footnote{Most of the following chapter has already appeared in an article by the author in Letters in Mathematical Physics~\cite{Villatoro2018}.}
The main concept used here is that of a Dirac structure which is simultaneously a generalization of a foliation and a symplectic manifold.
Intuitively, one should think of a Dirac structure as a smooth manifold equipped with a (poissibly singular) foliation by pre-symplectic manifolds.
More exposition on Dirac structures can be found in~\cite{BDiracintro}\cite{MMbook}.

The structure of this chapter is as follows:
\begin{itemize}
  \item Section~\ref{section:Diracstructures} gives a brief review of important concepts from the theory of Dirac structures. Topics covered include the definition of a Dirac structure, pullbacks and gauge transformations of Dirac structures.
  \item Section~\ref{section:Diracsite} will define the site of Dirac manifolds $\DMan$ and establish our notation. It concludes with a proof that the site of Dirac structures is good.
  \item Section~\ref{section:dmangroupoids} develops the theory of $\DMan$-groupoids called \emph{D-Lie groupoids}. While a D-Lie groupoid is more general than a symplectic groupoid, it turns out that this generalization is faithful.
  \item Section~\ref{section:principalgbundles} studies principal bundles and Morita equivalences of D-Lie groupoids. It proves that symplectic Morita equivalences and Morita equivalences of symplectic groupoids (as D-Lie groupoids) are the same.
  \item Section~\ref{section:stacksindman} turns to the study of stacks in the world of Dirac manifolds. It collects the work we have done so far and includes the proof of Theorem~\ref{thm:stackequiv} which is the main result of the chapter.
\end{itemize}

\section{Dirac structures}\label{section:Diracstructures}

\subsection{Generalized tangent bundle}
The Courant bracket was developed by Theodore Courant~\cite{Courant} with the introduction of Dirac structures.
We begin with defining the generalized tangent bundle, which is the vector bundle for which the Courant bracket is defined.
\begin{definition}\label{defn:gtangentbundle}
  Let $M$ be a smooth manifold. The \emph{generalized tangent bundle} is the direct sum of the tangent and cotangent bundles.
  \[ \T M : = TM \oplus T^* M \]
  Elements of $\T M$ are denoted $v \oplus \eta$ for some vector $v \in TM$ and covector $\eta \in T^*M$.
\end{definition}

The generalized tangent bundle comes with a bracket called the \emph{Courant bracket}.
\begin{equation}\label{eqn:Courantbracket}
{[V \oplus \eta , W \oplus \zeta ]}_{\T M} := [V,W] \oplus \L_V(\zeta) - \L_W(\eta) + \dif(\eta(W)) \, .
\end{equation}
For a 2-form $\omega$, the notation $\omega^\flat$ denotes the associated linear map $TM \to T^*M$ given by vector contraction.
Notably, the above bracket does not form a Lie bracket. While it satisfies the Jacobi and Leibnitz identities, it is not anti-symmetric.

If we are supplied with a 3-form $\phi$ on $M$ we can ``twist'' the bracket.
\begin{equation}\label{eqn:twistedCourantbracket}
{[V \oplus \eta , W \oplus \zeta ]}^\phi_{\T M} := [V,W] \oplus \L_V(\zeta) - \L_W(\eta) + \dif(\eta(W)) + \phi(X,Y,-)  \, .
\end{equation}
If we are considering the twisted case, then we call $\phi$ the \emph{background 3-form}. Twisted brackets are interesting because they make a natural appearance in Lie theory. See Example~\ref{example:twistedliegroups}.

The generalized tangent bundle comes with two pairings, one symmetric and one anti-symmetric.
\begin{equation}\label{eqn:gtangentsymmetric}
\langle v \oplus \eta, w \oplus \zeta \rangle_+ = \zeta(v) + \eta(w)
\end{equation}
\begin{equation}\label{eqn:gtangentasymmetric}
\langle v \oplus \eta, w \oplus \zeta \rangle_{-} = \frac{1}{2} (\zeta(v) - \eta(w))
\end{equation}

\subsection{Dirac geometry}

\begin{definition}\label{defn:diracstructure}
  Let $M$ be a smooth manifold. A \emph{Dirac structure} on $M$ is a subbundle $L \le \T M$ which is maximally isotropic with respect to $\langle \cdot , \cdot \rangle_+$ and whose sections are involutive (closed) under the Courant bracket.

  A $\phi$-\emph{twisted Dirac structure} is defined similarly, except we require that $L$ is involutive under the $\phi$-twisted Courant bracket.
\end{definition}

\begin{example}\label{example:2forms}
  Suppose $\omega$ is a 2-form on $M$ such that $\dif \omega = \phi$. Let $L := \{ v \oplus \omega^\flat(v) \}$ be the \emph{graph} of $\omega$. Then $L$ is a $\phi$-twisted Dirac structure.
\end{example}
\begin{example}
  Let $F$ be an integrable distribution on $M$ (i.e. a foliation on $M$). Now let $F^\circ \le T^* M$ denote the \emph{annihilator} of $F$. That is, it is the vector bundle of 1-forms $\eta$ such that $\eta(v) = 0$ for all $v \in F$. Then $L_\omega := F \oplus F^\circ$ is a Dirac structure.
\end{example}
\begin{example}\label{example:bivectors}
  Let $\pi \in \mathfrak{X}^2(M)$ on $M$ and denote by $\pi^\sharp$ the contraction map $T^* M \to TM$. Let $L_\pi$ be the graph of $\pi^\sharp$. Then $L_\pi$ is a Dirac structure if and only if $\pi$ is a \emph{Poisson} bivector. For our purposes, this will serve as our definition of a Poisson bivector.

  If $L$ is a $\phi$-twisted Dirac structure, then we say $\pi$ is a $\phi$-twisted Poisson bivector.
\end{example}
\begin{example}\label{example:twistedliegroups}
  Let $G$ be a compact Lie group. Let $\theta$ be the (right invariant) Mauer-Cartan 1-form on $G$. That is, for each $v \in T_g G$,
  \[ \theta(v) = \dif R_{g\inv} (v)   \]
  Since $G$ is compact, we can assume there is an bi-invariant metric $\rho( \cdot , \cdot)$ on $G$. Then
  \[ \phi(v_1,v_2,v_3) := \frac{1}{2} \rho \left(\theta(v_1),[\theta(v_2),\theta(v_3)] \right) \]
  defines a 3-form on $G$ called the \emph{Cartan 3-form}.

  Using the metric, we can identify $\T G$ with $TG \oplus TG$ and at each $g \in G$ consider the elements of $\T \G$ which have the form
  \[ (\dif L_g(v) - \dif R_g(v)) \oplus \frac{1}{2}  \left( \dif L_g(v)) - \dif R_g(v)) \right) \]
  It was observed by Severa and Weinstein~\cite{SeveraWeinstein} that this defines a $\phi$-twisted Dirac structure on $G$.
\end{example}

A Dirac structure is an example of a Lie algebroid. Any Dirac structure $L \le \T M$ comes with a projection to $TM$ which we will call the \emph{anchor map} and denote with $\rho$.
Since sections of $L$ are involutive, it comes with a Lie bracket $[ \cdot , \cdot]_L$. Anti-symmetry follows from the condition that $L$ is isotropic.
The bracket and anchor are compatible in the following sense:
\begin{equation}\label{eqn:diracleibnitz}
  [\alpha,f\beta]_L = f [\alpha, \beta]_L + \L_{\rho(\alpha)}(f)\beta \quad \forall \alpha,\beta \in \Gamma(L)
\end{equation}

At each $x \in M$, there is a pairing on the image of $\rho$ given by
\[ \omega(v,w) := \langle v \oplus \eta , w \oplus \zeta \rangle_- \quad \forall v \oplus \eta, w \oplus \zeta \in L_x \]
Well definedness of this pairing follows from the maximally isotropic condition on $L$.
Since $L$ is an algebroid, it follows that we can integrate the image of $\rho$ to a singular foliation of $M$ by regularly immersed submanifolds.
The pairing $\omega$ can be restricted to each orbit $\O$ to get a 2-form $\omega^\O \in \Omega^2(\O)$.
If $L$ is Dirac structure, then it follows that $\dif \omega^\O$ is closed. When $L$ is a $\phi$-twisted Dirac structure we instead get that $\dif \omega^\O + \phi = 0$.

\subsection{Morphisms}

Dirac structures come with two natural notions of morphism. To begin, we will explain the basic \emph{pushforward} and \emph{pullback} operations on Dirac structures.
\begin{definition}
  Suppose $f \colon M \to N$ is a smooth map and $L$ is a $\phi$-twisted Dirac structure on $M$. The \emph{pushforward} of $L$ along $f$ is defined to be
  \[
  f_* L := \{ \dif f(v) \oplus \eta : v \oplus f^* \eta \in L  \}
  \]
  Now suppose $L'$ is a Dirac structure on $N$. The \emph{pullback} of $L$ along $f$ is defined to be:
  \[
  f^* L' := \{ v \oplus f^* \eta : \dif f (v) \oplus \eta \in L' \}
  \]
  For each point $x \in M$, the pushforward and pullback operations always give maximally isotropic subspaces. However, the pullback operation will not always give a continuous subbundle of $\T M$. The pushforward operation has the same potential problems as the pushforward of vector fields in that it may give different answers over the same fiber.
  In the twisted case, the pullback of a $\phi$-twisted Dirac structure along $f$ is a $f^* \phi$ twisted Dirac structure (if it exists).

  If we are given Dirac structures $L_M$ on $M$ and $L_N$ on $N$, then we say that $f$ is \emph{forward Dirac} if $f_* L_M = N$. We say that $f$ is \emph{backwards Dirac} if $f^* L_N = L_M$.
\end{definition}
\begin{example}
  Let $f \colon M \to N$ be any smooth map and $L = TM \le \T M$ be the \emph{tangent Dirac structure}.
  Now consider the function $f(x) = x^2$. Then $f^* TM$ is not a valid Dirac structure since it is not a continuous subbundle of $\T N$.
\end{example}
\begin{example}
  Suppose $f \colon M \to N$ is a smooth map and $\omega$ is a 2-form on $M$. Then $f^* L_\omega = L_{f^* \omega}$.
\end{example}
\begin{example}
  Let $\pi_M$ and $\pi_N$ be Poisson bivectors on $M$ and $N$ respectively. Then $f_* L_{\pi_M} = L_{\pi_N}$ if and only if $\pi_M$ and $\pi_N$ are $f$-related.
\end{example}

Before we can define our site of Dirac manifolds, we will introduce one more operation on subbundles of $\T M$.
\begin{definition}
  Suppose $L_M$ is a ($\phi$-twisted) Dirac structure on $M$. Given a 2-form $\beta \in \Omega^2(M)$ then we define the \emph{gauge transformation} of $L_M$ by $\beta$ to be:
  \[
  L_M + \beta := \{ v \oplus (\eta + \omega^\flat(v)) : v \oplus \eta \in L_M \, . \}
  \]
  If $L_M$ is a $\phi$-twisted Dirac structure then $L_M + \beta$ is a $(\phi - \dif \beta)$-twisted Dirac structure.
\end{definition}
\begin{lemma}
  Suppose $f \colon M \to N$ is a smooth map and $L_N$ is a Dirac structure on $N$. Then for any 2-form $\beta \in \Omega^2(N)$ we have that:
  \[ f^* (L_N + \beta) = f^* L_M + f^* \beta \]
\end{lemma}
\begin{proof}
  Due to the maximality condition on Dirac structures, it suffices to show that one is contained in the other. An arbitrary element of $f^* (L_M + \beta)$ can be written in the form
  \[ v \oplus f^*(\eta + \beta^\flat(\dif f(v))) \]
  For some $v \oplus \eta \in L_M$. This expression equals
  \[ v \oplus f^*\eta + {(f^*\beta)}^\flat((v)) \, . \]
  But this is clearly an element of $f^* L_M + f^* \beta$.
\end{proof}

At each orbit $\O$ of the Dirac structure, the effect of the gauge transformation is that it adds $\beta|_{\O}$ to the to 2-form $\omega^\O$. Therefore, we should not be surprised by its compatibility with the pullback operation.

\section{DMan}\label{section:Diracsite}

\subsection{The category of Dirac manifolds}

We can now define the site of Dirac manifolds. Let $\DMan$ be the category defined as follows:
\begin{definition}\label{defn:DMan}
  Let $\DMan$ be the category defined as follows:
  \begin{itemize}
    \item The \emph{objects} of $\DMan$ are triples $(M, \phi, L_M)$ where $M$ is a smooth manifold and $L_M$ is a $\phi$-twisted Dirac structure.
    \item The \emph{morphisms} of $\DMan$ are pairs $(f, \beta) \colon (M, \phi_M, L_M) \to (N, \phi_N, L_N)$ where $f \colon M \to N$ is a smooth map and $\beta$ is a 2-form on $M$ such that:
    \[
    f^* \phi_N = \phi_M + \dif \beta \quad \mbox{ and } \quad
    f^* L_N = L_M + \beta \, .
    \]
  \end{itemize}
  Composition of pairs is given by the rule $(f, \beta_1) \circ (g, \beta_2) = (f \circ g, g^* \beta_1 + \beta_2)$.
\end{definition}

\begin{example}[Gauge Transformations]
Suppose $(M,L_M)$ is a Dirac manifold and $\beta$ is a closed 2-form on $M$, then $(\Id, \beta): (M,L_M) \to (M,L_M+\beta)$ is a morphism in $\DMan$.
We call such morphisms \emph{gauge transformations}.
\end{example}
\begin{example}[Smooth Maps]
Let $M$ and $N$ be any smooth manifolds.
Then $(M, TM)$ and $(N,TN)$ are Dirac manifolds.
For any smooth map $f: M \to N$ we have that $(f,0) : (M, TM) \to (N, TN)$ is a morphism in $\DMan$.
\end{example}
\begin{example}[Symplectic Leaves]
Suppose $(M, L_M)$ is a manifold and $L_M$ is the graph of a Poisson bivector.
Any orbit $\O$ of $M$ has an associated symplectic form $\omega^\O$ and the immersion $i:\O \to M$ satisfies $i^* L_M = L_{\omega^\O}$.
\end{example}
To simplify our notation we will sometimes denote a morphism $(f,\beta)$ in $\DMan$ by $f$ alone and the 2-form $\beta$ will be called the \emph{gauge part} of $f$.
Similarly we may sometimes denote a Dirac manifold $(M, \phi_M , L_M)$ by $M$ alone.
The notation $L_M$ and $\phi_M$ will always denote the $\phi_M$-twisted Dirac structure on $M$.
Lastly, if we say a morphism in $\DMan$ is a \emph{submersion} we mean that the underlying smooth map is a surjective submersion.
When we give $\DMan$ a topology, we will see that this agrees with our site theoretic notion of a submersion.

The category $\DMan$ comes with a natural functor $\Pr_1 : \DMan \to \Man$ by projection to the first factor of each triple.
This functor is split by a fully faithful functor $\mathbf{i}: \Man \to \DMan$ which takes any manifold $M$ to the Dirac manifold $(M, \phi_M = 0, TM)$ and any smooth map $f$ to $(f,0)$.

We can characterize commutative diagrams in $\DMan$ by considering the associated diagram in $\Man$ together with a \emph{gauge equation}.
For example, suppose we are given a triangle $T$ of morphisms in $\DMan$ as per (\ref{eqn:triangle}).
\begin{equation}\label{eqn:triangle}
T =
\begin{tikzcd}
 & M_2 \arrow[rd, "f_2"] &  \\
M_1 \arrow[ru, "f_1"] \arrow[rr,"f_3"] & & M_3
\end{tikzcd}
\end{equation}
The \emph{gauge part} of $T$ is the equation $\beta_1 + f_1^*\beta_2 = \beta_3$ (here $\beta_i$ is the gauge part of $f_i$).
More generally, any diagram $D$ in $\DMan$ comes with a set of gauge equations coming from each triangle in $D$.
It is not difficult to note that $D$ is a commuting diagram if and only if $\Pr_1(D)$ commutes in $\Man$ and each gauge equation holds.

Suppose we are given two morphisms $(f,\beta): M \to X$ and $(g,\alpha): N \to X$ in $\DMan$ such that the manifold $M \times_X N$ exists.
Then the \emph{fiber product} is defined to be $M \times_X N$ where
\begin{equation}\label{eqn:fiberproduct}
L_{M \times_X N} := {(f \circ \pr_1)}^* L_X - \pr_1^* \beta - \pr_2^* \alpha .
\end{equation}
Such a fiber product fits into a corresponding pullback square in $\DMan$:
\[
\begin{tikzcd}
M \times_X N \arrow[r, "\pr_2"] \arrow[d, "\pr_1"]  & N  \arrow[d,"g"] \\
M \arrow[r, "f"] & X
\end{tikzcd}
\]
We take the gauge parts of $\pr_1$ and $\pr_2$ to be $\pr_2^* \alpha$ and $\pr_1^* \beta$ respectively.
Observe that such a fiber product always exists if either $f$ or $g$ is cartesian in $\Man$.

Fiber products in $\DMan$ are true fiber products in that they still satisfy the same universal property.
Suppose we have the following diagram in $\DMan$:
\[
\begin{tikzcd}
Y  \arrow[dr, "k", dashrightarrow] \arrow[drr, "h_2", bend left] \arrow[ddr, "h_1" swap, bend right] & & \\
  & M \times_X N \arrow[r, "\pr_2"] \arrow[d, "\pr_1"]  & N  \arrow[d,"g"] \\
  & M \arrow[r, "f"] & X
\end{tikzcd}
\]
Let $\eta_1$ and $\eta_2$ be the gauge parts of $h_1$ and $h_2$, respectively.
Then the gauge equation arising from the outermost square is
\begin{equation}\label{eqn:universalprop}
h_1^* \beta + \eta_1 = h_2^* \alpha + \eta_2 \, .
\end{equation}
We already know that there is a unique smooth map $k:Y \to M \times_X N$ which makes this diagram commute.
We can define the gauge part of $k$, call it $\kappa$, one of two ways:
\[ \kappa + k^* \pr_1^* \beta = \eta_2 \quad \mbox{ or equivalently} \quad  \kappa + k^* \pr_2^* \beta = \eta_1 \, . \]
In the presence of (\ref{eqn:universalprop}), these definitions are equivalent.
They must hold in order for the diagram to commute since they represent the gauge equations of the top and left triangles created by inserting $k: Y \to M \times_X N$ into the diagram above.
Hence, $(k,\kappa)$ is the unique morphism which completes the diagram in $\DMan$.
\subsection{Topology}

The category $\DMan$ inherits a natural topology from the forgetful functor $\Pr_1 \colon \DMan \to \Man$.
\begin{definition}
  A collection of morphisms $C = \{ u_i \colon U_i \to M \}$ in $\DMan$ is a \emph{covering family} if $\Pr_1(C)$ is a covering family in $\Man$.
\end{definition}
A routine check shows that this satisfies the axioms of a pre-topology and so we obtain a Grothendieck topology on $\DMan$.
A site theoretic \emph{submersion} then is just a morphism in $\DMan$ which projects to a submersion in $\Man$.

\begin{proposition}\label{prop:Dmangood}
  $\DMan$ is a good site.
\end{proposition}
\begin{proof}
  $\DMan$ has an initial object. It is the empty manifold equipped with the unique Dirac structure on the empty manifold. Morphisms in $\DMan$ are clearly defined locally since any $(f,\beta)$ is uniquely determined by its restrictions to an open cover. We only need to show that $\Sub$ is a stack. In particular, we need to prove that given the following data:
  \begin{itemize}
    \item an open cover $\{ U_i \into M \}$ on some object $M$ in $\DMan$;
    \item submersions $p_i \colon P_i \to U_i$;
    \item and morphisms $\phi_{ij} \colon P_j|_{U_{ij}} \to P_i|_{U_{ij}}$;
  \end{itemize}
  such that $\phi_{ij}|_{U_{ijk}} \circ \phi_{jk}|_{U_{ijk}} = \phi_{ik}|_{U_{ijk}}$
  then there exists a submersion $P \to M$ in $\DMan$ and identifications $\phi_i \colon P|_{U_i} \to P_i$ such that:
  \[ \phi_{ij} \circ \phi_j |_{U_ij} = \phi_{j}|_{U_{ij}} \]
  Since $\Man$ is a good site we already know that a manifold and smooth maps $\phi_i$ exist. We only need to come up with a Dirac structure and a 3-form on $P$ and the gauge parts for each $\phi_i$.

  Without loss of generality, we can assume the gauge parts of the embeddings $\{ U_i \into M \}$ are zero and the gauge parts of the submersions $P_i \to U_i$ are zero. Let $\phi_P$ and $L_P$ be defined to be the pullback of $\phi_M$ and $L_M$ respectively. Then this structure is compatible with the identifications $\phi_i \colon P|_{U_i} \to P_i$ if we take the gauge parts to be zero.

  To finish, we need to show that $\phi_{ij} \circ \phi_j|_{U_{ij}} = \phi_i |_{U_{ij}}$. Let $\beta_{ij}$ be the gauge part of each $\phi_{ij}$. Since we took the gauge parts of each $\phi_i$ to be zero, the gauge equation associated to this expression is:
  \[ \phi_j^* (\beta_{ij}) = 0 \, . \]
  This is true since every $\beta_{ij}$ is zero. To see why note that $p_i|_{U_{ij}} \circ \phi_{ij} = p_j|_{U_{ij}}$ and the gauge part of each $p_i$ is zero. Then the gauge equation for this expression says $\beta_{ij} = 0$.
\end{proof}

  With this result, all of our work from Chapter~\ref{chap:stacks} and Chapter~\ref{chap:morita} can be applied to $\DMan$.

\section{DMan-groupoids}\label{section:dmangroupoids}

\subsection{D-Lie groupoids}
\begin{definition}
  A \emph{D-Lie groupoid} is a $\DMan$-groupoid. In other words, it is a groupoid internal to $\DMan$.
\end{definition}

This definition is elegant, but we should provide a geometric interpretation for such an object.

\begin{theorem}\label{thm:dliedata}
  D-Lie groupoids are in one-to-one correspondence with the following data:
  \begin{itemize}
    \item A Lie groupoid $\G \grpd M$,
    \item (a possibly twisted) Dirac structure on $M$,
    \item and a pair of two forms $\sigma$ and $\tau$ on $\G$.
  \end{itemize}
  such that for $\Omega := \tau - \sigma$:
  \begin{enumerate}[(DL1)]
    \item $\t^* \phi_M - \s^* \phi_M = \dif \Omega$,
    \item and $\t^* L_M = \s^* L_M + \Omega$
    \item $\m^* \Omega = \pr_1^* \Omega + \pr_2^* \Omega$ for $\m \colon \G \times_M \G \to \G$
  \end{enumerate}
\end{theorem}
\begin{proof}
  If we are given a D-Lie groupoid $\G \grpd M$. Then let $\sigma$ and $\tau$ be the gauge parts of the source and target map. We only need to show that $\Omega := \tau - \sigma$ satisfies the desired properties. (DL1) follows from combining $\s^* \phi_M = \phi_\G + \sigma$ and $\t^* \phi_M = \phi_\G + \tau$. (DL2) follows from the observation that $\t^* L_M = L_\G + \tau$ and $\s^* L_M = L_\G + \sigma$.

  To recover (DL3) we consider the gauge part of the compatibility of $\m$ with $\s$. That is:
  \[ \s \circ \m = \s \circ \pr_2 \colon \G \times_M \G \to \G \, . \]
  Let $\mu$ be the gauge part of $\m$. Then the gauge part of this equation is:
  \begin{equation}\label{eqn:gaugepartsandm}
  \m^* \sigma + \mu = \pr_2^* \sigma + \pr_1^* \sigma \, .
  \end{equation}
  If the reader is not sure how this equation is obtained, they should look at the construction of the fiber product in $\DMan$ from the previous section. The target map yields a similar equation:
  \begin{equation}\label{eqn:gaugeparttandm}
  \m^* \sigma + \mu = \pr_2^* \tau + \pr_1^* \tau \, .
  \end{equation}
  Then (DL3) is obtained by subtracting these equations.

  To obtain the opposite correspondence, we must define a Dirac structure on $\G$ and the gauge parts of the structure maps. Let $\phi_\G$ be defined to be $\s^* \phi_M - \sigma$ or equivalently $\t^* \phi_M - \tau$. Let $L_\G$ be defined to be $\s^* L_M - \sigma$ or equivalently $\t^*L_M - \tau$.

  We already have gauge parts for $\s$ and $\t$. To define the gauge part of $\m$, call it $\mu$, we can use either Equation~\ref{eqn:gaugepartsandm} or Equation~\ref{eqn:gaugeparttandm}. Property (DL3) guarantees that these are equivalent.

  For the unit map $\u$ we take consider the axiom $\s \circ \u = \Id_M$. The gauge part of this axiom is
  \begin{equation}\label{eqn:gaugepartsandu}
    \u^* \sigma + \upsilon = 0 \, m
  \end{equation}
  if $\upsilon$ is the gauge part of $\u$. Hence we can take Equation~\ref{eqn:gaugepartsandu} to be a definition for $\upsilon$.

  To define the gauge part of the inverse map $\i$ we take inspiration from the fact that $\s \circ \i = \t$.
  Let the gauge part of $\i$, called $\iota$ be defined to be the unique 2-form such that:
  \begin{equation}\label{eqn:gaugepartsandi}
    \i^* \sigma + \iota = \tau \, .
  \end{equation}

  If this choice of $\iota$, $\mu$ and $\upsilon$ define a D-Lie groupoid, then surely this inverts the correspondence we defined at the start.
  The only possible problem is that our choices may not result in a well defined D-Lie groupoid. That is, our assumptions (DL1), (DL2) and (DL3) may not be enough to imply the gauge parts of every groupoid axiom holds. It turns out this does not happen and we leave the proof of this fact to the appendix (see Lemma~\ref{lemma:dlieconstruction}).
\end{proof}

The 2-form $\Omega := \tau - \omega$ clearly plays an important role in the study of D-Lie groupoids. We will call this the \emph{characteristic form} of $\G$. When $\sigma = 0$ we say that $\G$ is \emph{target aligned}. We will see later that (up to isomorphism) every D-Lie groupoid is target aligned. The pair $(\tau,\sigma)$ is called the \emph{gauge pair} of $\G$.

Now let's look at a few examples of such data.
\begin{example}[Symplectic Groupoids]
Given $(\G, \Omega)$, a symplectic groupoid integrating a Poisson manifold $(M,L_\pi)$, then $t^* L_\pi = s^*L_\pi + \Omega$.
Therefore, a symplectic groupoid is the same as a target aligned D-Lie groupoid with non-degenerate characteristic form.
\end{example}
\begin{example}[Symplectic orbifolds]
Let $\G \rightrightarrows M$ be an {\'e}tale Lie groupoid and suppose $\omega$ is a symplectic form on $M$.
Suppose further that $\t^* \omega - \s^* \omega = 0$.
When $\G$ is proper, it can be thought of as the presentation of a (possibly non-effective) symplectic orbifold.
By thinking of $M$ as a Dirac manifold, then $\G$ can also be thought of as a D-Lie groupoid with characteristic form $0$.
\end{example}
\subsection{D-Lie groupoid morphisms}\label{subsection:groupoidmorphisms}
We will now take a closer look at homomorphisms of D-Lie groupoids.
A morphism of D-Lie groupoids is just a morphism of $\DMan$-groupoids as discussed in Chapter~\ref{chap:stacks}.
Throughout, $\G$ and $\H$ are D-Lie groupoids over $M$ and $N$ respectively.
Also, $\Omega^\G$ and $\Omega^\H$ will denote their respective characteristic forms.

There is a characterization of morphisms in terms of the characteristic forms.
\begin{lemma}\label{lemma:dliemorphism}
There is a one-to-one correspondence between morphisms of D-Lie groupoids, and the following data:
\begin{itemize}
  \item a {\bf Lie} groupoid homomorphism $F \colon \G \to \H$ covering $f \colon M \to N$;
  \item a 2-form $\beta$ which makes $(f,\beta)$ into a morphism of Dirac manifolds.
\end{itemize}
such that:
\begin{equation}\label{eqn:morphism}
F^* \Omega^\H = t^* \beta - s^* \beta + \Omega^\G \, .
\end{equation}
\end{lemma}
\begin{proof}
One direction of the correspondence is clear. We simply let $\beta$ be the gauge part of the base map of the given D-Lie groupoid morphism.
One only has to show that Equation~\ref{eqn:morphism} holds. This follows immediately by looking at the gauge parts of compatibility of $F$ with the source and target maps:
\begin{equation}\label{eqn:scompatible}
F^* \sigma^\H + \alpha = \s^* \beta + \sigma^\G .
\end{equation}
\begin{equation}\label{eqn:tcompatible}
F^* \tau^\H + \alpha = \t^* \beta + \tau^\G.
\end{equation}
In the above, $\alpha$ is the gauge part of $F \colon \G \to \H$.

To reverse this correspondence, we need to supply $\alpha$. This is easy since we can use either of the two equations above as definitions for $\alpha$. Our assumption means that these choices are equivalent. By construction, it is clear that this choice of $\alpha$ makes $(F,\alpha)$ compatible with the source and target maps. We only need to check that $(F, \alpha)$ is compatible with multiplication.

To see why, first note that the gauge part of $(F,F) \colon \G \times_M \G \to \H \times_N \H$ is $\pr_1^* \alpha + \pr_2^* \alpha - \pr_2^* \t^* \beta$.
That is:
\[ (F,F)^* \mu^\H + \pr_1^* \alpha + \pr_2^* \alpha = \m^* \alpha + \mu^\G \]
If we take Equation~\ref{eqn:scompatible} as the definition of $\alpha$ and apply Equation~\ref{eqn:gaugepartsandm}, we get:
\[
  \m^* \alpha + \mu^\G = \m^* \s^* \beta - \m^* F^* \sigma^\H + \pr_1^* \sigma^\G + \pr_2^* \sigma^\G
\]
On the other hand:
\begin{align*}
  (F,F)^* \mu + \pr_1^* \alpha + \pr_2^* \alpha - \pr_2^* \t^* \alpha &= (F,F)^* (\pr_1^* \sigma^\H + \pr_2^* \sigma^\H - \m^* \sigma^\H) + \pr_1^* \alpha + \pr_2^* \alpha - \pr_2^* \t^* \alpha \\
  &= \pr_1^* F^* \sigma^\H + \pr_2^* F^* \sigma^\H - \m^* F^* \sigma^\H + \pr_1^* \alpha + \pr_2^* \alpha - \pr_2^* \t^* \alpha \\
  &= \pr_1^* \s^* \beta + \pr_1^* \sigma^\G + \pr_2^* \s^* \beta + \pr_2^* \sigma^\G - \m^* F^* \sigma^\H - \pr_2^* \t^* \alpha \\
  &=   \m^* \s^* \beta - \m^* F^* \sigma^\H + \pr_1^* \sigma^\G + \pr_2^* \sigma^\G
\end{align*}
Which shows that $(F,\alpha)$ is a D-Lie groupoid homomorphism.
\end{proof}
Now we observe a useful corollary which clarifies why we mainly need to consider the characteristic form of a D-Lie groupoid.
\begin{lemma}\label{lemma:targetalign}
Every D-Lie groupoid is canonically isomorphic to a target aligned D-Lie groupoid.
\end{lemma}
\begin{proof}
Let $\G \rightrightarrows M$ be a D-Lie groupoid over $M$ with gauge pair $(\tau,\sigma)$.
Then the pair gauge pair $(\Omega,0)$ also determines a target aligned D-Lie groupoid. Furthermore, the identity map $\Id_\G$ of Lie groupoids together with $\beta = 0$ satisfies Lemma~\ref{lemma:dliemorphism} and so the gauge transformation:
\[ (\Id_\G, \sigma): ( \G, L_\G ) \to (\G , L_\G + \sigma), \, \]
is an isomorphism of D-Lie groupoids.
\end{proof}
\begin{example}[Symplectic Groupoids]
Suppose $\G \rightrightarrows M$ and $\H \rightrightarrows N$ are symplectic groupoids.
If we think of $\G$ and $\H$ as target aligned D-Lie groupoids then a morphism consists of a homomorphism of Lie groupoids $F: \G \to \H$ together with a closed 2-form $\beta \in \Omega^2(M)$ such that (\ref{eqn:morphism}) holds.
\end{example}
\subsection{D-Lie algebroids}
There is an infinitesimal version of D-Lie groupoids.
Burzstyn and Cabrera~\cite{BImf} showed that given a $\phi$-closed multiplicative form on a Lie groupoid, there is a corresponding \emph{infinitesimal multiplicative form} on the corresponding algebroid $A$.
An infinitesimal multiplicative form relative to a 3-form $\phi$ on $M$ is a bundle map
\[ \rho^*: A \to T^*M  \]
which satisfies for all $a_1,a_2 \in \Gamma(A)$:
\begin{enumerate}[(MF1)]
\item (Anti-symmetry)
\[  {\rho}^* (a_1)(\rho(a_2)) = -{\rho}^* (a_2) (\rho(a_1)) \]
\item (Compatibility with the bracket twisted by $\phi$)
\[ \rho^*([a_1,a_2]) = \L_{\rho(a_1)} \mu(a_2) - \L_{\rho(a_2)} \mu(a_1) + \phi(\rho(a_1),\rho(a_2), - ) \]
\end{enumerate}
Where $\rho \colon A \to TM$ is the anchor map of $A$. We should clarify that the $^*$ notation is intended to indicate that $\rho_A^*$ takes values in the cotangent bundle, not that it is the dual of the anchor map $\rho \colon A \to TM$.
 a compatibility condition with the bracket on $A$ and is \emph{anti-symmetric}
Note that we have used the $^*$ notation to emphasize that the map takes values in the cotangent bundle.
We do not mean that $\rho_A^*$ is the linear dual of the anchor.

\begin{example}
  Let $L$ be a $\phi$-twisted Dirac structure on $M$ and let $\rho^* \colon L \to T^* M$ be the projection to the cotangent bundle. Then $\rho^*$ is an infinitesimal multiplicative form relative to $\phi$. To see this one should observe that the anti-symmetry condition follows from the fact that $L$ is isotropic and the bracket condition follows from the closure of $L$ under the Courant bracket.
\end{example}

Now suppose $\Omega$ is a multiplicative form on a Lie groupoid.
From~\cite{BImf}, we can recover an IMF by defining $\rho_A^*$ in the following way:
\begin{equation}\label{eqn:IMFdef}
\rho_A^*(v) := \eta  \quad \Leftrightarrow \quad  t^* \eta = \Omega^\flat(v) \, .
\end{equation}
We now proceed to a simple lemma, which will motivate our definition of D-Lie algebroid.
\begin{lemma}\label{lemma:dliealg}
Suppose $\G \rightrightarrows M$ is a D-Lie groupoid with characteristic form $\Omega$.
Let $A$ be the corresponding algebroid and $\rho_A^*$ be the associated infinitesimal multiplicative form and $\rho_A$ be the anchor map.
Then $\rho_A(v) \oplus \rho_A^*(v) \in L_M$ for all $v \in A$ and $(\rho \oplus \rho_A^*): A \to L_M$ is a Lie algebroid homomorphism.
\end{lemma}
\begin{proof}
We have two things to show.
First, the claim that, for any $v \in A$, $\rho_A(v) \oplus \rho_A^*(v) \in L_M$.
Let $v \in A := \ker \dif \s|_M$ and suppose $\eta = \rho_A^*(v)$.
By the definition of the pullback, we know that $v \oplus 0 \in \s^* L_M$.
Consequently, $v \oplus \Omega^\flat(v) \in \t^* L_M$ by Theorem~\ref{thm:dliedata}.
Hence, by the definition of $\rho_A^*$, we have that $v \oplus \eta \in \t^* L_M$.
Therefore, we can conclude that $\rho_A(v) \oplus \rho_A^*(v) \in L_M$ (again by the definition of the pullback).

Now for the second part.
Recall the definition of the Courant bracket.
\begin{equation}\label{eqn:Courantbracket2}
{[V \oplus \eta , W \oplus \zeta ]}_{L_M} := [V,W] \oplus \L_V(\zeta) - {(\dif \eta)}^\flat(W) + \phi(V,W,-) \, .
\end{equation}
We need to show that
\[ {[\rho_A(V) \oplus \rho_A^*(V), \rho_A(W) \oplus \rho_A^*(W) ]}_{L_M} = \rho_A({[V,W]}_A) \oplus \rho_A^*({[V,W]}_A) \, . \]
This is clearly true for the $TM$ component, since $\rho: A \to TM$ is compatible with the standard Lie bracket.
For the $T^*M$ component, the result follows immediately from (MF2).
\end{proof}
\begin{definition}
A \emph{D-Lie algebroid} is a Lie algebroid $(A, {[ \cdot , \cdot ]}_A, \rho)$ over a Dirac manifold $(M,L)$ together with a Lie algebroid homomorphism $\til{\rho}: A \to L_M$
\end{definition}

We can recover an infinitesimal multiplicative form from this definition by taking $\rho_A^*$ to be the $T^*M$ component of $\til\rho$.
We say that a D-Lie algebroid is \emph{integrable} if there exists a D-Lie groupoid $\G$ whose corresponding infinitesimal multiplicative form is $\rho_A^*$.
When $\G$ is source simply connected and target aligned we say that $\G$ is the canonical integration.
\begin{example}[Poisson Manifolds]
Let $(\G,\Omega)$ be a symplectic groupoid over a Poisson manifold $(M,\pi)$.
Then let $\til\rho: A \to T^*M$ be the standard identification of the algebroid of $\G$ with the cotangent bundle of $M$.
In this way, we can think of $A$ as a D-Lie algebroid.
\end{example}
\begin{example}[Dirac Structures]
Let $L_M$ be any Dirac structure.
Then if we take $\til\rho: L_M \to L_M$ to be the identity morphism, we can think of $L_M$ as a D-Lie algebroid.
\end{example}
\begin{example}[Trivial algebroids]
Let $A$ be a rank zero vector bundle, thought of as a trivial algebroid.
Then for any Dirac structure $L_M$ on the base of $A$, we can take $\til\rho: A \to L_M$ to be the zero map.
\end{example}
The last example illustrates the interesting fact that integrability of $L_M$ is neither a necessary nor a sufficient condition for integrablity of the D-Lie algebroid.
\begin{definition}
Suppose $(A, \rho_A^*)$ and $(B, \rho_B^*)$ are D-Lie algebroids.
A \emph{morphism} of D-Lie algebroids is a Lie algebroid morphism $F: A \to B$ which cover a morphism of Dirac manifolds $(f,\beta): M \to N$ such that
\[ f^* \circ \rho_B^* \circ F = \rho_A^* + \beta^\flat \, . \]
\end{definition}
The left side of the equation above is the \emph{pullback} of the IM 2-form $\rho_B^*$ along $F$.
The right side is the infinitesimal form of the gauge transformation $\Omega + \t^* \beta - \s^* \beta$.
Hence, this is just the infinitesimal version of (\ref{eqn:morphism}).

\section{Principal bundles and Morita equivalence}\label{section:principalgbundles}
In this section, we will define $G$-bundles and $(\G,\H)$-bibundles in the setting of D-Lie groupoids. By our work in Chapter~\ref{chap:stacks}, these objects determine the behavior of geometric stacks over $\DMan$.
Furthermore, we will see that Morita equivalence of D-Lie groupoids faithfully generalizes Morita equivalence for symplectic groupoids.
\subsection{Principal \texorpdfstring{$\G$}{} bundles}
\begin{definition}
Let $\G \rightrightarrows M$ be a D-Lie groupoid.
A \emph{left $\G$-bundle} over $N$ is a $\G$-bundle internal to $\DMan$, i.e., an ordinary $\G$-bundle $\s^P:P\to N$, where $P$ and $N$  are Dirac manifolds and all the structure maps $\s^P:P \to N$, $\t^P:P \to M$ and $\m_L:\G \times_M P \to P$ are $\DMan$-morphisms.
We say that $P$ is principal if the induced morphism $\G \times_M P \to P \times_N P$ is an isomorphism.
\end{definition}

The reader should note that, by our construction of fiber products in $\DMan$, a $\G$-bundle for a D-Lie groupoid is principal if and only if the underlying action of a Lie groupoid is principal.

The morphisms $\s^P, \t^P$ and $\m_L^P$ come with gauge parts $\sigma^P, \tau^P$ and $\mu_L^P$.
The equation associated to $\s^P \circ \m_L(g,p) = \s^\P(p)$ is
\[ \mu_L + \m^* \sigma^P = \pr_1^* \sigma^\G + \pr_2^* \sigma^P \, . \]
Similarly, the gauge equation of $\t^P \circ \m_L(g,p) = \t(g)$ is
\[ \mu_L + \m^* \tau^P = \pr_1^* \tau^\G + \pr_2^* \tau^P \, . \]
Therefore, when defining a principal $\G$-bundle it suffices to specify the 2-forms $\sigma^P$ and $\tau^P$.
As with D-Lie groupoids, we say the characteristic 2-form of $P$ is $\Omega^P := \tau^P - \sigma^P$.
Let $\Omega^\G$ be the characteristic 2-form of $\G$.
Then
\[ \m_L^* \Omega^P = \pr_1^* \Omega^\G + \pr_2^* \Omega^P \, . \]
That is, the 2-form $\Omega^P$ is left multiplicative.
When $\sigma^P = 0$ we call $P$ \emph{target aligned}.
By a similar argument as in Lemma~\ref{lemma:targetalign}, every principal $\G$-bundle is canonically isomorphic to a target aligned principal $\G$-bundle.

\subsection{Bibundles}
We can now proceed to tackle bibundles and Morita equivalence in $\DMan$.
Throughout this section $\G$ and $\H$ are D-Lie groupoids over the Dirac manifolds $M$ and $N$ respectively.
\begin{definition}
Suppose $\G$ and $\H$ are D-Lie groupoids.
A \emph{$(\G, \H)$-bibundle} is defined to be a bibundle object internal to the category $\DMan$.
Hence, it is an object $P$ in $\DMan$ together with morphisms $\t^P, \s^P, \m_L,\m_R$ (again in $\DMan$) which satisfy the axioms of commuting left and right actions over $N$ and $M$, respectively.

A bibundle $P$ is said to be \emph{left principal bibundle} if the left action makes $P$ into a left principal $\G$-bundle over $N$.
We define \emph{right principal} similarly.
We call $P$ a \emph{principal bibundle} if $P$ is both left and right principal.
A principal $(\G,\H)$-bibundle is also called a \emph{Morita equivalence} of $\G$ and $\H$.
\end{definition}

Just like $\G$-bundles, a $(\G, \H)$-bibundle $P$ of D-Lie groupoids is determined by the data of the underlying bundle and the gauge part of the source and target maps $\sigma^P$ and $\tau^P$.
The \emph{characteristic form} $\Omega^P$ of $P$ is defined to be $\sigma^P - \tau^P$ as before.
Using the same techniques as before we can show that $\sigma^P$ and $\tau^P$ define a bibundle if and only if $\Omega^P$ is left and right multiplicative.
That is
\[ \m_L^* \Omega^P = \pr_1^* \Omega + \pr_2^* \Omega^P \quad \mbox{ and } \quad \m_R^* \Omega^P = \pr_1^* \Omega^P + \pr_2^* \Omega^\H.  \]
We say that $P$ is \emph{target aligned} if $\sigma^P = 0$.

An equivariant map of $(\G, \H)$-bibundles is a morphism $\Phi: P \to Q$ which commutes with the source and target maps and respects the multiplication.
In terms of the characteristic 2-form the condition on $F: Q \to P$ is just
\[ F^* \Omega^P = \Omega^Q \, . \]
This makes sense when compared to the case of left $\G$-bundles since we can think of any $(\G,\H)$-bibundle morphism as a left $\G$-bundle morphism covering the identity on $N$.
As with D-Lie groupoids, for any bibundle $P$ the gauge transformation $(\Id,\sigma^P): (P,L_P) \to (P,L_P + \sigma^P)$ is an isomorphism of $P$ with a target aligned bibundle.

The next few examples demonstrate how this notion of Morita equivalence of D-Lie groupoid relates to existing definitions of Morita equivalence.
\begin{example}[Morita equivalence of Lie groupoids]
Given a Morita equivalence $P$ of Lie groupoids $\G$ and $\H$, then thinking of $\G$ and $\H$ as D-Lie groupoids with the tangent Dirac structure allows us to view $(P, TP)$ as a Morita equivalence of D-Lie groupoids $(\G, T\G)$ and $(\H, T\H)$.
Furthermore, it is a simple exercise to check that any Morita equivalence of the D-Lie groupoids $(\G, T\G)$ and $(\H, T \H)$ is isomorphic to such a $(P, T P)$.
\end{example}
\begin{example}[Symplectic Morita equivalence]\label{ex:sympl:morita}
Given a symplectic Morita equivalence $(P, \Omega^P)$ of symplectic groupoids $(\G, \Omega^\G)$ and $(\H, \Omega^\H)$, we can think of $(P, \Omega^P)$ as a target aligned Morita equivalence of $\G$ and $\H$ viewed as D-Lie groupoids.
\end{example}

We can improve on the observation from the preceding example.
\begin{proposition}\label{prop:symform}
Suppose $\G$ and $\H$ are symplectic groupoids, i.e.\ target aligned D-Lie groupoids with symplectic characteristic forms.
There is a one-to-one correspondence between symplectic Morita equivalences and target aligned principal $(\G,\H)$-bibundles.
\end{proposition}

\begin{proof}

One direction is just Example~\ref{ex:sympl:morita}.
For the other direction, suppose $P$ is a target aligned principal $(\G,\H)$-bibundle.
We must show that $\Omega^P$ is symplectic at each $p \in P$. So fix $p$ and let $x=\s^P(p)  \in N$ and $y=\t(p)\in M$.
Suppose $e: U \to P$ is a local section of $\s^P$ around $x$ such that $e(x) = p$ and let $\beta := e^* \Omega^P$.
Next define $f:= \t^P \circ e$ and notice that
\begin{align*}
f^* L_M &= e^* {(\t^P)}^* L_M  \\
        &= e^* (L_P + \Omega) \\
        &= e^* L_P + \beta  \\
        &= e^* {(\s^P)}^* L_N + \beta = L_N + \beta \, .
\end{align*}
In other words, $(f,\beta): U \to M$ is a morphism in $\DMan$. Since $P$ is principal, $f$ is transverse to the orbits of $M$.

We will need these facts in a moment, but first we should use the section $e$ to `trivialize' our bibundle.

When $P$ is restricted to $U$, we can identify it with the trivial $\G$-bundle associated to this map.
That is,
\[ P|_U \isom \G \times_{\s,f} U, \text{ with $p$ corresponding to }(\u(f(x)),x) . \]
When $P|_U$ is written in this way, then we can use the left multiplicativity of $\Omega^P$ to see that
\[ \Omega^P = \pr_1^* \Omega^\G + \pr_2^* \beta.  \]
Hence, for any two vectors
\[ (v_i,w_i)\in T_p (\G \times_M N)=\{(v,w)\in T_{\u(f(x))}\G\times T_x N: \dif\s(v)=\dif f(w)\}, \]
we have that
\[ \Omega^P((v_1, w_1), (v_2, w_2)) = \Omega^\G(v_1, v_2) + \beta(w_1,w_2) \, . \]
Now suppose that $(v_1, w_1)$ is in the kernel of $\Omega^P$.
We will show that it must be zero by pairing it with a few careful choices of $(v_2,w_2)$.
First let us see what happens when $(v_2,w_2) = (v_2,0)$ for arbitrary $v_2 \in \ker \dif \s$. Then
\[ 0 = \Omega^\P((v_1,w_1),(v_2,0)) = \Omega^\G(v_1,v_2) \, . \]
Therefore, we can conclude that $v_1$ is $\Omega^\G$ orthogonal to $\ker \dif \s$. Since $\G$ is a symplectic groupoid, this implies that $v_1 \in \ker \dif \t$.

Now suppose $(v_2,w_2) = (\dif \u \dif f(w_2),w_2)$ for arbitrary $w_2 \in T \O_x$ tangent to orbit of $x$.
We can conclude that
\begin{equation}\label{eqn:symform1}
 0 = \Omega^\G(v_1, \dif \u \dif f(w_2)) + \beta(w_1,w_2) \, .
\end{equation}
Let $\G_{\O_y}=\s^{-1}(\O_y)=\t^{-1}(\O_y)$ be the restriction of $\G$ to the orbit $\O_y$ and let $\omega^{\O_y}$ be the leafwise symplectic form on the orbit.
For any symplectic groupoid, it turns out that
\[ \Omega^\G|_{\G_{\O_y}} = \t^* \omega^{\O_y} - \s^* \omega^{\O_y}. \]
Since both $v_1$ and $\dif \u \dif f(w_2)$ are tangent to $\G_{\O_y}$, we can conclude that
\begin{align*}
\Omega^\G(v_1, \dif \u \dif f(w_2)) &= \omega^{\O_y}(\dif \t (v_1), \dif f(w_2)) - \omega^{\O_y}(\dif \s (v_1), \dif f(w_2)) \\
                                   &= -\omega^{\O_y}(\dif \s (v_1), \dif f(w_2)) \\
                                   &= -\omega^{\O_y}(\dif f (w_1), \dif f(w_2)) \\
                                   &= -f^*\omega^{\O_y}(w_1,w_2) \\
                                   &= -\omega^{\O_x}(w_1,w_2) - \beta(w_1,w_2).
\end{align*}
In the second line we have use the fact that $v_1 \in \ker \dif \t$.
In the third line we used the fact that $(v_1,w_1)$ must be tangent to $\G \times_M U$.
The last line follows from the fact that $(f,\beta)$ is a morphism of Dirac manifolds.

Combining this with (\ref{eqn:symform1}) we get that
\begin{equation}\label{eqn:symform2}
 \omega^{\O_x}(w_1,w_2) = 0 \, .
\end{equation}
Recall that $w_2$ was an arbitrary vector tangent to $\O_x$.
Since $\omega^{\O_x}$ is symplectic, we can conclude that $w_1=0$.

So far we have shown that $(v_1,w_1) = (v_1,0)$ and that $v_1 \in \ker \dif \t$. It follows that $v_1 \in \ker \dif \s$.
We still need to show that $v_1= 0$.
To do this we will show that $\Omega^\G(v_1,v) = 0$ for arbitrary $v \in T_{\u(y)} \G$.
Since $\Omega^\G$ is symplectic, this will show that $v_1=0$.

First write $v$ in the form $v_A + v_\u$ for $v_A \in \ker \dif \s$ and $v_\u \in \mbox{Im}(\dif \u)$.
Let us see what happens when we pair it with $v_1$:
\[ \Omega^\G(v_1,v) = \Omega^\G(v_1,v_A) + \Omega^\G(v_1, v_\u)  = 0 + \Omega^\G(v_1, v_\u) \, . \]

Since $f$ is transverse to the foliation on $M$, we can write $v_\u = \dif \u \dif f(w) + \dif \u (v_\O)$, where $w \in T_x N$ and $v_\O \in T_y \O_y$.
Hence,
\begin{equation}\label{eqn:symform3}
\Omega^\G(v_1, v_\u) = \Omega^\G(v_1, \dif u \dif f(w)) + \Omega^\G(v_1, \dif \u(v_\O)) \, .
\end{equation}
Recall that we have assumed that $(v_1,0)$ is in the kernel of $\Omega^P$.
Therefore,
\begin{align*}
0 &= \Omega^P((v_1,0),(\dif{\u{}} \dif f (w),w)) \\
&= \Omega^\G(v_1, \dif \u \dif f(w)) + \beta(0,w) \\
 &= \Omega^\G(v_1, \dif \u \dif f(w)) \, .
\end{align*}
Therefore, we can conclude that the first summand on the right side of (\ref{eqn:symform3}) vanishes.
For the second summand, observe that both $v_1$ and $\dif\u(v_\O)$ are tangent to $\G_{\O_y}$ and so
\[ \Omega^\G(v_1,\dif\u(v_\O)) = \omega^{\O_x}(\dif \t (v_1), v_\O) -\omega^{\O_x}(\dif \s (v_1), v_\O) = 0 \, . \]
Hence, we conclude that $\Omega^\G(v_1,v)=0$.
Since $v$ was arbitrary and $\Omega^\G$ is symplectic, we conclude that $v_1 =0$.
So $\Omega^P$ is non-degenerate at $p$ and therefore symplectic.
 \end{proof}
%
%
\subsection{Weak equivalences}\label{subsection:weakequivalences}
Let $\G$ and $\H$ be D-Lie groupoids over $M$ and $N$.
Suppose $F: \H \to \G$ is a morphism of D-Lie groupoids covering $f: N \to M$.
Then we can construct the left principal $(\G, \H)$-bibundle:
\[ P_F := \G \times_{\s,f} N \, , \]
with the obvious commuting actions of $\G$ (on the left) and of $\H$ (on the right). We equip $P_F$ with the characteristic form:
\[ \Omega^{P_F} = \pr_1^* \Omega^\G + \pr_2^* \beta \, , \]
where $\beta$ is the gauge part of $f: N \to M$.
This is the same as the standard construction for Lie groupoids with the addition of the characteristic form.
The reader can easily check that these actions satisfy the axioms of a $(\G, \H)$-bibundle.

\begin{definition}
We say that a morphism of D-Lie groupoids $F: \H \to \G$ is a \emph{weak equivalence} if $P_F$ is a principal $(\G,\H)$-bibundle.
\end{definition}

In other words, a weak equivalence of D-Lie groupoids is a D-Lie groupoid morphism which gives rise to a Morita equivalence.
Later, it will be shown that these equivalences further correspond to an isomorphism of the underlying stacks.
The name weak equivalence is chosen because while they are not necessarily invertible as D-Lie groupoid morphisms, they become (weakly) invertible when passing to the 2-category of stacks.

Recall that $P_F$ is principal if and only if it is a principal bibundle of Lie groupoids.
Therefore, $F$ is a weak equivalence if and only if the underlying generalized map of Lie groupoids is a weak equivalence.
This immediately gives rise to a notion of symplectic weak equivalences.
\begin{example}[Symplectic Weak Equivalences]
Suppose $\G$ and $\H$ are symplectic groupoids (i.e., $\G$ and $\H$ are target aligned D-Lie groupoids and their characteristic 2-forms $\Omega^\G$ and $\Omega^\H$ are symplectic).
Then a weak equivalence $F: \H \to \G$ consists of a homomorphism of Lie groupoids, together with a closed 2-form $\beta$ on $N$ such that the following hold.
\begin{enumerate}[(a)]
\item $F: \H \to \G$ is fully faithful and essentially surjective.
\item $f: N \to M$ is transverse to $\pi_M$ (the Poisson structure on $M$).
\item $F^* \Omega^\G = \Omega^\H + \t^* \beta - \s^* \beta$.
\end{enumerate}
Condition (a) and (b) ensure that $F$ is a weak equivalence of the underlying Lie groupoids as per the usual definition.
That is, $P_F$ is principal as a Lie groupoid bibundle.
The last condition is the geometric condition for $F$ to consitute a morphism of D-Lie groupoid as per our discussion of D-Lie groupoid morphisms.
\end{example}
Composition of homomorphisms corresponds to the tensor product operation at the level of bimodules.
Our work in Chapter~\ref{chap:stacks} already constructed the tensor product site theoretically, but for completeness we will repeat the construction in our D-Lie setting.

Given a left principal $(\G_1, \G_2)$-bibundle $P$ and a left principal $(\G_2, \G_1)$-bibundle $Q$.
Assume that $\G_1$, $\G_2$, $\G_3$, $P$ and $Q$ are all target aligned.
Thinking of the $\G_i$ as Lie groupoids then
\[ P \otimes Q := P \times_{M_2} Q / \G_2 \, , \]
where the action of $\G_2$ on $(p,q)$ is defined to be $g_2 \cdot (p,q) = (p \cdot g_2\inv , g \cdot q)$.
In order to equip $P \otimes Q$ into a target aligned left principal $(\G_1,\G_3)$-bibundle, we only need to equip it with a multiplicative 2-form $\Omega^{P \otimes Q}$.
Multiplicativity of $\Omega^P$ and $\Omega^Q$ with respect to the action of $\G_2$ ensures that
\[ \widetilde{\Omega} := \pr_1^* \Omega^P + \pr_2^* \Omega^Q \, , \]
is basic with respect to the action of $\G_2$ on $P \times_{M_2} Q$.
Hence, $\widetilde{\Omega}$ descends to a 2-form on $P \otimes Q$.
Left and right multiplicativity of $\Omega^{P \otimes Q}$ can easily be checked.

\section{Stacks in DMan}\label{section:stacksindman}

Stacks over $\DMan$ are mainly interesting because they provide a rigorous foundation for taking about the stack associated to a (pre)-symplectic groupoid. We will make this claim precise in this section. First we begin with a somewhat obvious result.

\begin{proposition}
Suppose $\G$ is a D-Lie groupoid.
Let $\B\G$ denote the category whose objects are left principal $\G$-bundles and morphisms are equivariant maps $F: P \to Q$.
Let the functor $\pi: \B \G \to \DMan$ send $F: P \to Q$ to $f:M \to N$.
Then $\B\G$ is a stack.
\end{proposition}
\begin{proof}
This is just a corollary of Theorem~\ref{thm:bgstack} since Proposition~\ref{prop:Dmangood} showed that $\DMan$ is a good site.
\end{proof}

Our next theorem provides the basic correspondence for studying Morita equivalence of D-Lie groupoids.

\begin{theorem}\label{thm:stackequiv}
Suppose $\G \rightrightarrows M$ and $\H \rightrightarrows N$ are D-Lie groupoids.
Then the following are equivalent:
\begin{enumerate}[(i)]
\item $\B \G$ and $\B \H$ are isomorphic.
\item There exists a principal $(\G, \H)$-bibundle.
\item There exists a D-Lie groupoid $\K$ and weak equivalences $F_1: \K \to \G$ and $F_2:\K \to \H$.
\end{enumerate}
\end{theorem}
\begin{proof}
This result can be thought of as an analogue of Theorem 2.26 in~\cite{PXstacks}.
The above theorem is a corollary of Theorem~\ref{thm:chapter1} and Theorem~\ref{thm:weakequivalences} applied to the case where $\C = \DMan$.
\end{proof}
\begin{theorem}\label{thm:main1}
Let $\G$ and $\H$ be symplectic groupoids.
The following are equivalent:
\begin{enumerate}[(1)]
\item $\G$ and $\H$ are symplectically Morita equivalent.
\item $\B\G$ is isomorphic to $\B\H$.
\item There exists a principal $(\G,\H)$-bibundle of D-Lie groupoids.
\item There exists a pre-symplectic groupoid $\G'$ and a pair of weak equivalences of D-Lie groupoids $\G \from \G' \to \H$.
\end{enumerate}
\end{theorem}

\begin{proof}
This is an immediate corollary of Proposition~\ref{prop:symform} and Theorem~\ref{thm:stackequiv}.
\end{proof}

\begin{example}[The Stack of a Poisson Manifold]
We saw earlier that there is a one-to-one correspondence between symplectic groupoids $(\G, \Omega)$ integrating a Poisson manifold $(M, \pi)$ and target aligned D-Lie groupoid $\G$ with a symplectic characteristic form.

From this point of view, if $\G$ is a proper symplectic groupoid then $\B\G$ is the analogous in $\DMan$ of the notion of a \emph{separated stack}.
The space of objects of such proper symplectic groupoids are the \emph{Poisson manifolds of compact type}, studied by Crainic, Fernandes and Martinez-Torrez in~\cite{PMCT1,PMCT2}.
\end{example}
\begin{example}[A non-presentable stack]
One weakness of the category $\DMan$ is that it lacks a terminal object.
This is remedied by passing to stacks over $\DMan$.
In fact, the category $\DMan$ equipped with the identity projection is a terminal object in the 2-category of stacks.
This is, perhaps, the simplest example of a stack over $\DMan$ which does not admit a presentation.
\end{example}

\chapter{b-Symplectic structures}\label{chap:bsymp}
In this chapter, we solve the classification problem, up to Morita equivalence, for a remarkable class of Poisson manifolds.
\footnote{Most of the content of this chapter has appeared previously in a pre-print by the author~\cite{VillatoroBS}}
Along the way, we compute the Picard group of Poisson manifolds of this class.
The chapter is organized as follows:
\begin{itemize}
\item In Section \ref{section:bsymplecticintro} and Section \ref{section:picgroups} we will establish our notation and give a brief overview of b-symplectic structures and their symplectic groupoids.
\item In Section \ref{section:trategyofproof} we will outline the general strategy of our proof and reduce the problem to 2-dimensions.
\item Section \ref{section:affineplane} and Section \ref{section:affinecylinder} are concerned with classifying groupoids over the `affine plane' and the `affine cylinder' respectively via discrete data.
\item Finally, in Section \ref{section:bsymplecticmanifolds} we will complete our proofs of the two main theorem, concerning the classification and the Picard groups of our b-symplectic manifolds.
\item Section \ref{section:examples4} goes over a few explicit applications. In particular, we will compare these results with the classification of b-symplectic compact surfaces due to Bursztyn and Radko~\cite{BDirac} and the computation of the Picard group for b-symplectic surfaces obtained by Radko and Shlyakhtenko~\cite{Radko2}.

\end{itemize}
\section{Stable b-symplectic structures}\label{section:bsymplecticintro}
In this section we recall some basic facts on b-symplectic manifolds that we will need later and, at the same time, we establish our notation conventions and terminology for b-symplectic manifolds.
Since b-symplectic manifolds are examples of Poisson manifolds they are also examples of Dirac manifolds. Hence, they can be considered to be objects in $\DMan$.

A \emph{b-symplectic structure} (also known as a \emph{log symplectic structure}) on an $2n$ dimensional smooth manifold $M$ is a Poisson structure $\pi$ on $M$ such that the section $\wedge^n \pi$ of $\wedge^{2n} TM$ intersects the zero section transversely. The \emph{singular locus} $Z$ of $\pi$ is the zero set of $\wedge^n \pi$ and is a codimension one embedded submanifold (an hypersurface) in $M$.

There are two alternative languages used in the study of b-symplectic structures. The \emph{b-geometry} point of view treats b-symplectic structures as non-degenerate closed 2-forms on the \emph{b-tangent bundle}, while the \emph{Poisson} point of view treats b-symplectic structures as a special class of Poisson bivectors on $M$. We will mostly use the language of Poisson geometry.

We will always assume that $M$ is orientable. A choice of volume form $\mu$ determines a \emph{modular vector field} $X_\mu$: it is the unique Poisson vector field such that:
\[ L_{X_h}\mu=X_\mu(h)\mu, \]
for any hamiltonian vector field $X_h$. If $\mu'=e^f \mu$ is another choice of volume form, then the corresponding modular vector fields are related by $X_{\mu'}=X_{\mu}+X_f$.

In dimension 2, the following two examples of b-symplectic manifolds will play an important role. Understanding these examples and their integrations will be critical to our main result.
\begin{example}[Affine plane]\label{example:aff}
The plane $\R^2$ equipped with the Poisson structure $x \partial{y} \wedge \partial{x}$, is a b-symplectic manifold which we call the \emph{affine plane} and denote it by $\aff$. If $\mu=\diff x\wedge\diff y$ is the standard volume form in $\R^2$, the associated modular vector field is $X_\mu=\partial_y$. As a Poisson manifold, $\aff$ is the linear Poisson structure associated to the dual of the 2-dimensional affine Lie algebra.

For any real number, $\rho \neq 0$, we can modify the Poisson structure on $\aff$ by setting:
\[ \pi^\rho = \frac{1}{\rho}\pi.\]
The modular vector field of $\pi^\rho$ relative to the standard volume is $\rho \partial_y$.
\end{example}
\begin{example}[Affine cylinder]\label{example:caff}
Consider the action of $\Z$ on $\aff^\rho$ given by $n \cdot (x,y) = (x,y+n)$. This is an action by Poisson diffeomorphisms and the Poisson structure of the quotient $\R^2/\Z=\R \times \S^1$ takes the form:
\[ \left( \frac{x}{\rho}\right) \frac{\partial}{\partial \theta} \wedge \frac{\partial}{\partial x} \, . \]
Here $\partial_\theta$ denotes the projection of $\partial_y$ to $\R^2/\Z$. We call this manifold the \emph{affine cylinder of modular period $\rho$} and denote it by $\caff^\rho$. When $\rho=1$ we may denote $\caff^\rho$ by just $\caff$.

The modular vector field of $\pi^\rho$ relative to the standard volume $\mu=\diff x\wedge\diff \theta$ is $X_\mu=\rho \partial_\theta$. The modular period $\rho$ turns out to be a complete Morita invariant of $\caff^\rho$: the manifolds $\caff^{\rho_1}$ and $\caff^{\rho_2}$ are Morita equivalent if and only if $\rho_1 = \rho_2$.
\end{example}
For a b-symplectic manifold $(M,\pi)$ we will denote the corresponding singular 2-form as $\pi\inv$. In a neighborhood of a point on the singular locus, the Darboux-Weinstein splitting theorem gives the local normal form:
\[ \pi = x \dd{y} \wedge \dd{x} + \sum_{i}^{n-1} \dd{p_i} \wedge \dd{q_i}.  \]
Alternatively, the corresponding singular 2-form is given by:
\[ \pi\inv = \frac{1}{x} \diff x \wedge \diff y + \sum_{i}^{n-1} \diff q_i \wedge \diff p_i. \]
 In Darboux-Weinstein coordinates, the modular vector field associated to the canonical volume form $\mu$ in such coordinates is $X_\mu=\dd{y}$.
\begin{example}[Semi-local model]\label{example:mappingtoruspoiss}
Given a symplectic manifold $(L,\omega^L)$ and a symplectomorphism $f: L \to L$ recall that the \emph{symplectic mapping torus}:
\[ T_f := \frac{\R \times L}{(y+1,p) \sim (y,f(p))} \, , \]
yields a symplectic fibration $T_f\to \S^1$. The corresponding regular Poisson structure $\pi^f$  on $T_f$ can naturally be extended to a b-symplectic structure $\pi$ on $\R \times T_f$ with singular locus $Z=T_f$ by setting:
\[ \pi=\frac{x}{\rho}\dd{\theta}\wedge \dd{x} + \pi^f  ,\]
where $\rho \in \R$ is any non-zero real number.  We call it the \emph{canonical b-symplectic extension of $T_f$ with period $\rho$}.
\end{example}
The previous example furnishes the semi-local model around the connected components of $Z$ of the class of b-symplectic structures of interest to us:
\begin{definition}
A \emph{stable b-symplectic structure} is an oriented b-symplectic structure on $M$ for which each component $Z_i$ of the singular locus $Z$ admits a tubular neighborhood $U$ isomorphic to $\R \times T_f$, for some mapping torus $f: L \to L$ and some real number $\rho \neq 0$. The number $\rho$ is called the \emph{modular period} of $Z_i$.
\end{definition}
Notice that the leaf space of a stable b-symplectic manifold has a simple structure: it is a collection of circles connected by open points, one for each connected component of $M-Z$.

The results of Guillemin, Miranda, and Pires \cite{GMP1,GMP2} lead to the following useful criteria for stability (combine Theorem 50 from \cite{GMP2} and Theorem 59 from \cite{GMP1}):
\begin{theorem}[Guillemin, Miranda, Pires]\label{thm:GMPtorus}
Let $(M,\pi)$ be a orientable b-symplectic manifold. Then $(M,\pi)$ is stable if and only if $Z$ is compact and each of its connected components admits a closed leaf.
\end{theorem}
In particular, when $M$ is compact we need only verify the existence of closed leaves in each connected component of the singular locus.
\begin{example}[Closed b-symplectic surfaces]
In dimension 2, every b-symplectic structure on a compact oriented surface $M$ is stable. Such stable b-symplectic (or just b-symplectic) structures on were studied by Olga Radko in \cite{Radko2} who called them \emph{topologically stable surfaces}. Radko showed that, up to Poisson diffeomorphism, stable b-symplectic structures on a surface $M$ were classified by the following data:
\begin{enumerate}[(a)]
\item The topological arrangement of the singular curves.
\item The period of the modular vector field around each singular curve.
\item The `regularized' volume of $\pi\inv$ on $M$.
\end{enumerate}
Later, Bursztyn and Radko \cite{BDirac} proved that (a) and (b) alone classify a topologically stable surface, up to Morita equivalence.
\end{example}
We quote here one more fact from the work of Guillemen, Miranda, and Pires, that will be useful for us in the sequel (see \cite{GMP2}):
\begin{theorem}[Guillemin, Miranda, Pires]\label{thm:bcohomology}
The Poisson cohomology of an arbitrary b-symplectic structure is given by:
\[ H_\pi^n(M) \isom H_{\dR}^n(M) \oplus H_{\dR}^{n-1}(Z). \]
\end{theorem}

\section{Picard groups}\label{section:picgroups}
We defined the Picard group of $\C$-groupoid in Chapter~\ref{chap:morita}.
In this section $\C = \DMan$ (see Chapter~\ref{chap:dman}). Every (integrable) Dirac structure comes with an integrating groupoid.
In the case of b-symplectic manifolds, the integrating groupoid is a symplectic groupoid. In this section, we will review these definitions, with particular focus on the b-symplectic case.
One slight change from the previous discussion is that technically we permitting Dirac structures on \emph{non-Hausdorff} manifolds. However, we will always assume the space of \emph{objects} is Hausdorff. Formally, we are considering stacks over non-hausorff Dirac manifolds which are presented by Hausdorff manifolds.

\subsection{Symplectic groupoids over b-symplectic structures}

Recall that a \emph{symplectic groupoid} is a pair $(\G \grpd M,\Omega)$, where the symplectic form on $\G$ is required to be multiplicative:
\[ \m^* \Omega = \pr_1^* \Omega + \pr_2^* \Omega. \]
A symplectic groupoid induces a unique Poisson structure $\pi$ on its manifold of units $M$ for which the target map $\t: \G \to M$ is Poisson. In such a case we say that $\G$ \emph{integrates} $(M,\pi)$. If $\G$ has 1-connected $\s$-fibers, then we say that $\G$ is the unique (up to isomorphism) \emph{canonical integration} of $(M,\pi)$ and we write $\G=\Sigma(M)$.

The general integrability criteria of Crainic and Fernandes imply that a Lie algebroid with injective anchor map on an open dense set is integrable (see \cite{Cint}). It follows that every b-symplectic manifold $(M,\pi)$ admits an integration. Recalling our two examples of b-symplectic manifolds defined earlier, we give their canonical integrations.
\begin{example}[Integration of the affine plane]\label{example:sigmaaff}
The two dimensional \emph{affine group}, which we denote by $\Aff$, is $\R^2$ equipped with the product
\[ (c,d) \cdot (a,b) = (a + c, b + e^a d). \]
The Lie algebra of $\Aff$ is $\aff$, the two dimensional affine Lie algebra. Identifying $\aff$ and $\aff^*$ using the standard metric on $\R^2$, leads to the linear Poisson structure on $\aff$ from Example \ref{example:aff}. We can now use the fact that the integrations of linear Poisson structures on the dual of a Lie algebra $\mathfrak{g}$ are the action groupoids $T^*G\simeq G\times\mathfrak{g}^*\rightrightarrows \mathfrak{g}^*$ associated with the coadjoint action of a Lie group $G$ integrating $\mathfrak{g}$, equipped with the canonical symplectic form on the cotangent bundle.

We can write the co-adjoint action of $\Aff$ on $\aff\simeq \aff^*$ explicitly as:
\[ (a,b) \cdot (x,y) = (x e^a, y + xb). \]
The resulting action groupoid $\G \Aff=\Aff\times\aff\simeq \R^4$ has source and target maps defined by:
\[ \s(a,b,x,y) = (x,y),\qquad \t(a,b,x,y) = (a,b)\cdot(x,y). \]
and multiplication given by:
\[ (c,d,xe^a , y + xb) \cdot (a,b,x,y) = (a + c, b +e^a d, x, y). \]
The multiplicative symplectic form on the groupoid $\G\Aff$ is:
\begin{align*}
\Omega &= t^* (\diff \log x \wedge \diff y) - s^*( \diff \log x \wedge \diff y) \\
&= \diff x \wedge \diff b + \diff a \wedge \diff y + b \diff a \wedge \diff x + x \diff a \wedge \diff b.
\end{align*}
Since $\Aff$ is simply connected, $\G\Aff$ has simply connected source fibers, and we conclude that $\G\Aff \isom \Sigma(\aff)$, the canonical integration.

The Poisson manifolds $\aff^\rho$ and $\aff$ have isomorphic algebroids. Hence, $\Sigma(\aff^\rho)$ has the same underlying Lie groupoid as $\Sigma(\aff)=\G\Aff$ but with the new symplectic form $\rho \Omega$. We will denote this modified symplectic groupoid by $\G\Aff^\rho$.
\end{example}
\begin{example}[Integration of the affine cylinder]\label{example:sigmacaff}
The affine cylinder $\caff$ was constructed in Example \ref{example:caff} as a quotient of the affine plane $\aff$ by a free and proper action of $\Z$ by Poisson diffeomorphisms. The $\Z$-action and the $\Aff$-action on $\aff$ commute. We obtain a lifted $\Z$-action on $\G\Aff$ by symplectic groupoid isomorphisms and its quotient is the action groupoid:
\[ \Caff := \Aff \times \caff = \{ (a,b,x,\theta) \in \R^2 \times (\R \times \S^1) \}. \]
The quotient symplectic structure is:
\[ \Omega = \diff x \wedge \diff b + \diff a \wedge \diff \theta + b \diff a \wedge \diff x + x \diff a \wedge \diff b. \]
Again, the source fibers of $\Caff$ are simply connected so $\Caff \isom \Sigma(\caff)$.
\end{example}
\begin{remark}
Gualtieri and Li in \cite{Gualt} have found the source connected integrations of b-symplectic manifolds. Although we do not explicitly use their classification here, our definition of the discrete presentation bears resemblance to their work.
\end{remark}
For a Lie groupoid $\G \rightrightarrows M$ we will be using the following notations. The \emph{isotropy group} over $x \in M$ is the Lie group of arrows with source and target $x$:
\[ \G_x := \t\inv(x)\cap \s\inv(x).  \]
The \emph{restriction} of $\G$ to a subset $U \subset M$ is the subset of arrows in $\G$ whose source and target lie in $U$:
\[  \G|_U := \s\inv(U)\cap\t\inv(U). \]
The quotient space $M/\G$ is the topological quotient of $M$ by the orbits of $\G$ and we call it the \emph{orbit space} of $\G$. A \emph{bisection} of $\G$ is a smooth section $\sigma: M \to \G$ of the source map such that $\t \circ \sigma: M \to M$ is a diffeomorphism. A \emph{local bisection} around $g \in \G$ is a map $\sigma:U \to \G$ where $U \subset M$ is an open neighborhood of $\s(g)$ such that $\sigma(\s(g)) = g$ and $\t \circ \sigma: U \to \t\circ\sigma(U)$ is a diffeomorphism. A basic fact is:
\begin{lemma}
For a Lie groupoid $\G \rightrightarrows M$ and an arrow $g \in \G$ there always exists a local bisection $\sigma:U \to \G$ around $g$.
\end{lemma}
We leave the (easy) proof to the reader. Notice that such a bisection, in general, is not unique.
\subsection{The Picard group}
We defined the Picard and Bibundles in Chapter~\ref{chap:morita}. Lets take a look at a few examples which arise in the b-symplectic setting.
Recall that we use the following diagram to illustrate the notion of $(\G_2,\G_1)$-bibundle:
\[
\begin{tikzcd}
\G_2 \darrow & P \arrow[dl, "\t"] \arrow[loop left] \arrow[loop right, leftarrow] \arrow[dr, "\s" swap] & \G_1 \darrow \\
M_2 & & M_1
\end{tikzcd}
\]
From here on out, all bibundles are assumed to be \emph{principal} bibundles.
Recall that any $(\G_2,\G_1)$-bibundle $P$ induces a map of orbit spaces which we denote by the same letter $P$:
\[ P:M_1/\G_1 \to M_2/\G_2,\quad [x]\longmapsto [\t(p)] \text{ for any } p \in P\text{ with }\s(p) = x. \]
\begin{example}
\label{example:transversal}
Given a symplectic groupoid $(\G,\Omega)$ over a Poisson manifold $M$ and a complete Poisson transversal $T$ (i.e., a submanifold intersecting each symplectic leaf of $M$ transversely in a symplectic submanifold), then $(\G|_T,\Omega_{\G|_T})$ is a symplectic groupoid Morita equivalent to $(\G,\Omega)$: a $(\G,\G|_T)$-bibundle is given by:
\[
\begin{tikzcd}
\G \darrow & s^{-1}(T) \arrow[dl, "\t"] \arrow[loop left, looseness=2.5] \arrow[leftarrow, loop right, looseness=2.5] \arrow[dr, "\s" swap] & \G|_T \darrow \\
M & & T \, ,
\end{tikzcd}
\]
with the obvious left/right actions. By Proposition 8 in \cite{Crainic2}, both $s^{-1}(T)$ and $\G|_T$ are symplectic submanifolds. Using the multplicativity of $\Omega$, it follows easily that this is indeed a symplectic bibundle.
\end{example}
A bibundle $P$ is a generalized isomorphism $\G_1 \to \G_2$, where an element $p\in P$ is thought of as an ``arrow" from a point in $M_1$ to a point in $M_2$.
For any $(\G_2,\G_1)$-bibundle $P$ the restriction of $P$ to $U \subset M_1$ is the collection of all elements of $p$ whose source lies in $U$.
\[ P|_U := \s\inv(U) .\]
When $\G_1,\G_2 \rightrightarrows M$ have the same unit manifold, the isotropy of $P$ at $x \in M$ are the points in $P$ which have source and target equal to $x$.
\[ P_x := \s\inv(x) \cap \t\inv(x). \]
A bisection of ($\G_2$,$\G_1$)-bibundle $P$ is a section $\sigma:M_1\to P$ of $\s:P\to M_1$ such that $\t\circ \sigma:M_1\to M_2$ is a diffeomorphism.

We will now review the construction of \emph{map-like} bibundles. Suppose $F:\G_1 \to \G_2$ is an isomorphism of symplectic groupoids. The induced map $f: M_1 \to M_2$ on the space of units is necessarily a Poisson diffeomorphism and we say that $F$ \emph{covers} $f$. The map $F$ gives rise to a symplectic bibundle:
\[ P_(F):= \G_2 \times_{s,f} M_1, \]
with anchor maps:
\[ \t(g,x) = \t(g),\quad \s(g,x) = x, \]
and left/right multiplication defined by:
\[ g_2 \cdot (g,x) = (g_2g,x),\qquad (g,x) \cdot g_1 = (gF(g_1),x). \]
The symplectic structure on $P_F$ is the pullback $\pr_1^* \Omega_2$. The symplectic $(\G_2,\G_1)$-bibundle $P_F$is called the \emph{symplectic bibundle associated to} $F$.

In general, not every bibundle will arise from a symplectomorphism of symplectic groupoids. In fact, it is shown in \cite{ruipic} that:
\begin{proposition}\label{prop:lagbisection}
A symplectic $(\G_2,\G_1)$-bibundle $P$ is isomorphic to $P_F$ for some isomorphism of symplectic groupoids $F:\G_1 \to \G_2$ if and only if the bibundle admits a lagrangian bisection.
\end{proposition}
For a symplectic manifold $M$ the obstructions to finding a symplectomorphism $F:\Sigma(M) \to \Sigma(M)$ such that $P_F \isom P$ become a symplectic version of the Nielson realization problem (see \cite{ruipic}). For b-symplectic manifolds the obstructions are even more complicated, but when
$M$ is a closed b-symplectic surface these obstructions vanish \cite{Radko} and one can always find a bisection.
\begin{example}[Picard group of a closed b-symplectic surface]
Radko and Shylakthenko in \cite{Radko} computed the Picard group of any toplogical stable surface. They showed that for a toplogical stable surface $M$, every symplectic $(\Sigma(M),\Sigma(M))$-bibundle admits a lagrangian bisection, from which it follows that the Picard group of such a surface is isomorphic to the group of outer Poisson automorphisms.
\[ \Pic(M) \isom \mathbf{OutPoiss}(M) \]
Their result depends critically on the Dehn-Nielsen-Baer theorem which is false for dimensions greater than 2. They went on to describe this group with the help of labeled graphs. This is the primary inspiration for our definition of a \emph{discrete presentation} which we will use to both classify and compute invariants for stable b-symplectic manifolds in higher dimensions.
\end{example}
\subsection{The Picard Lie algebra}
The Lie algebra of the Picard group was defined and studied by Burzstyn and Fernandes in \cite{ruipic}. In general, the Picard group of a Poisson manifold can be infinite dimensional, however Corollary 1.3 from \cite{ruipic} says that the Picard Lie algebra, $\mathfrak{pic}(M)$, fits into a long exact sequence together with the Poisson and de Rham cohomologies of $M$:
\[
\begin{tikzcd}[column sep=small]
\dots \arrow[r]
& H^1_\dR(M) \arrow[r]
& H^1_\pi(M) \arrow[r]
& \mathfrak{pic}(M) \arrow[r]
& H^2_\dR(M) \arrow[r]
& H_\pi^2(M) \arrow[r]
& \dots
\end{tikzcd}
\]
Applying Theorem~\ref{thm:bcohomology}, the map $H^n_\dR(M) \to H^n_\pi(M)=H^n_\dR(M)\oplus H^{n-1}_\dR(Z)$ is just injection to the first coordinate. Hence the exact sequence becomes:
\[
\begin{tikzcd}[column sep=small]
& H^1_\dR(M) \arrow[r,hook]
& H^1_\dR(M) \oplus R^N \arrow[r]
& \mathfrak{pic}(M) \arrow[r]
& H^2_\dR(M) \arrow[r,hook]
& H_\dR^2(M) \oplus H^1_\dR(Z)
\end{tikzcd}
\]
where $N$ is the number of connected components of $Z$. We conclude that:
\begin{proposition}
If $M$ is a stable b-symplectic manifold whose singular locus $Z$ has $N$ connected components, then its Picard Lie algebra is the $N$-dimensional abelian Lie algebra:
\[ \mathfrak{pic}(M) \isom \R^N. \]
In particular, the Picard group of $M$ is finite dimensional.
\end{proposition}
\begin{proof}
The long exact sequence argument above gives the dimension of $\mathfrak{pic}(M)$. Moreover, by choosing appropriate volume forms in the local model, one can find compactly supported modular flows around each connected component of $Z$, which lead to $N$ commuting Poisson vector fields. The time-1 flows of these vector fields yield a $N$ dimensional family of bibundles. Hence, the Picard Lie algebra is abelian and the connected component of the identity of $\Pic(M)$ is a quotient of $\R^N$ by some (possibly trivial) discrete subgroup.
\end{proof}

\section{Strategy of the proof}\label{section:trategyofproof}
The main aim of this chapter is to prove Theorem~\ref{maintheorem1}, describing the Picard group $\Pic(M)$ of a stable b-symplectic manifold $M$. The main steps in the proof are:
\begin{enumerate}
\item[Step 1.] Reduce the computation of $\Pic(M)=\Pic(\Sigma(M))$ to the computation of $\Pic(\G)$, where $\G \rightrightarrows \C$ is a symplectic groupoid integrating a disjoint union $\C$ of affine cylinders (see Example \ref{example:caff}).
\item[Step 2.] Describe $\Pic(\G)$ in terms of discrete data associated with a graph whose vertices are the connected components of $M-Z$ and whose edges are the connected components of $Z$.
\end{enumerate}
Section \ref{subsection:reduction} takes care of Step 1, while Step 2 we will be taken care of in the remaining sections. Section \ref{subsection:gluing} proves a simple lemma about gluing bibundles. Finally Section \ref{subsection:pointedbibundles} will define \emph{pointed bibundles} which will be a bridge between discrete data and geometric data.
\subsection{Reduction to 2 dimensions}\label{subsection:reduction}
Step 1 will follow from restricting $\Sigma(M)$ to a complete 2-dimensional Poisson transversal $\C$: $\Sigma(M)$ and $\G=\Sigma(M)|_{\C}$ are Morita equivalent (see Example \ref{example:transversal}) so that $\Pic(\Sigma(M))=\Pic(\G)$. We will use the semi-local models around the singular hypersurface of a stable b-symplectic manifold  to construct the Poisson transversal.

Let the singular hypersurface $Z$ be decomposed into a disjoint union of connected components $Z = \sqcup_{i \in I} Z_i$. By the definition of stable b-symplectic structure, each $Z_i$ has trivial normal bundle and is a symplectic mapping torus $T_{f_i}$, for some symplectomorphism $f_i: L_i \to L_i$.  We have a Poisson diffeormorphism
\[ \phi_i:U_i\to \R \times Z_i,\]
defined on an open neighborhood $U_i$ of $Z_i$, where $\R \times Z_i$ is furnished with the Poisson structure:
\[ \pi=\frac{x}{\rho_i}\dd{\theta}\wedge \dd{x} + \pi^{f_i}, \]
where $\rho_i$ be the modular period of $Z_i$. We think of $L_i$ as a  leaf of $Z_i$, i.e, a fiber of the mapping torus $p: Z_i \to \S^1$. Also, we can assume that $f_i$ is the holonomy of the flat connection on $p: Z_i \to \S^1$ induced by the modular vector field on $Z_i$.
\begin{lemma}
\label{lemma:transversal}
For each component $Z_i$ there exists an embedded Poisson transversal:
\[ \iota_i: (-\epsilon,\epsilon) \times \S^1 \hookrightarrow M , \]
with induced Poisson structure the b-symplectic affine cylinder $ \caff^{\rho_i}$:
\[ \iota_i^* (\pi)=  \frac{x}{\rho_i}\dd{\theta}\wedge \dd{x}. \]
\end{lemma}
\begin{proof}
Since $L_i$ is connected there exists a section $\gamma_i: \S^1 \to Z_i$ of $p: Z_i \to \S^1$. Since the normal bundle to $Z_i$ is trivial we can extend $\gamma_i$ to a 2 dimensional embedded submanifold
\[ \iota_i: (-\epsilon,\epsilon) \times \S^1 \to \R \times Z_i. \]
The first coordinate corresponds to the normal directions to $Z_i$ so that $\iota_i|_{\{ 0 \} \times \S^1} = \gamma_i$. Thus the image of $\iota_i$ is transversal to each leaf. It is clear from Example \ref{example:mappingtoruspoiss} that the pullback by $\iota_i^* \pi\inv$ of the b-form will satisfy the (inverse) of the formula above.
\end{proof}
Fix an orientation for $M$. We say that an open symplectic leaf of $M$ is \emph{positive} (respectively, \emph{negative}) if the orientation provided by the symplectic form coincides (respectively, is the opposite) of the orientation of $M$. Notice that we can orient the transversal of the previous lemma such that $\iota_i(x,\theta)$ lies in a positive (respectively, negative) leaf if and only if $x$ is positive (respectively, negative). We will assume that we have done this from now on.
\begin{corollary}
Suppose $(M,\pi)$ is a stable b-symplectic manifold. There exists a complete Poisson transversal $\iota: \C \hookrightarrow M$, where $\C = \sqcup_{i \in I} C_i$ is a disjoint union of affine cylinders $C_i\simeq \caff^{\rho_i}$ with modular period $\rho_i$. In particular, $\Sigma(M)$ is Morita equivalent to the restriction:
\[ \G := \Sigma(M)|_{\iota(\C)}. \]
\end{corollary}
\begin{proof}
We take for $\C$ the disjoint union of the transversals constructed in Lemma \ref{lemma:transversal}. Since the embedding $\C \to M$ intersects every orbit of $M$, it is a complete Poisson transversal. Hence, $\G$ and $\Sigma(M)$ are Morita equivalent (see Example \ref{example:transversal}).
\end{proof}
The groupoids $\G$ integrating disjoint unions of affine cylinders arising from a complete Poisson transversal are not totally arbitrary. For example, they share with $\Sigma(M)$ one useful feature: the modular vector field can be lifted to a family of symplectic groupoid automorphisms of $\G$ as shown by the next lemma.
\begin{lemma}
Let $M$ be a stable b-symplectic manifold with a complete Poisson transversal $\iota: \C \to M$. There exists a choice of volume form $\mu$ on $M$ such that the modular vector field $X_\mu$ is tangent to $\iota(\C)$. In particular, if $\G=\Sigma(M)|_{\iota(\C)}$, then the 1-parameter family of symplectic groupoid automorphisms $\Phi^t:\Sigma(M)\to\Sigma(M)$ induced by $X_\mu$ preserves $\G$.
\end{lemma}
\begin{proof}
Let $\omega$ be any choice of volume form on $M$. We claim that we there exists a smooth function $g: M \to \R$ such that the modular vector field associated to the volume form $\mu := e^g \omega$ is tangent to the embedding  $\iota: \C \to M$.

We will perform this adjustment locally around each $Z_i$. Let $\gamma_i$ be the section of $M_{f_i} \to \S^1$ as in Lemma \ref{lemma:transversal}. We can lift $\gamma_i$ to a curve $\til \gamma_i: [0,1] \to [0,1] \times L_i$ which respects the equivalence relation $(0,f_i(p)) \sim (1,p)$. Without loss of generality assume that $\til \gamma_i$ is constant near the endpoints.

Let $g_t: L_i \to \R$ be a time dependent family of smooth functions such that $\pr_2 \circ \til \gamma_i$ is an integral curve of the (time dependent) Hamiltonian vector field $X_{g_t}$. Again, we can assume without loss of generality that the $g_t$ are zero near the endpoints of $[0,1]$. We can think of the family $g_t$ as a function $\til g: [0,1] \times L_i \to \R$. Since $g_t$ is trivial near the endpoints, it respects the equivalence relation $\sim$ and descends to a function $g: M_{f_i} \to \R$.

Let $X_g$ be the Hamiltonian vector field associated to $g$. If $X_\omega$ is the modular vector field relative to $\omega$ then $X_\omega + X_g$ is tangent to $\iota$ and hence
\[ X_\omega + X_g = X_{e^g \omega} = \X_\mu , \]
is tangent to the embedding $\iota_i$. By choosing a local bump function around $Z_i$ we can arrange that the support of $g$ to be contained in a small neighborhood of $Z_i$.
\end{proof}
\begin{remark}
Although the transversals we have defined are only `finite' cylinders of the form $(-\epsilon,\epsilon) \times \S^1$, any integration of a finite cylinder is Morita equivalent to an integration of $\R \times \S^1$.
\end{remark}
We summarize this discussion for future reference in the following proposition.
\begin{proposition}
Suppose $M$ is a stable b-symplectic manifold. Then $\Sigma(M)$ is Morita equivalent to an integration $\G$ of a disjoint union $\C$ of affine cylinders. Furthermore $\G$ satisfies the following properties:
\begin{enumerate}[(a)]
    \item the restriction of $\G$ to any single cylinder has connected orbits;
    \item the open orbits of $\G$ can be split into two categories, positive and negative. Positive (respectively, negative) orbits are disjoint unions of positive (respectively, negative) half cylinders;
     \item there exists a smooth family of symplectic groupoid automorphisms $\Phi^t: \G \to \G$ covering the flow of the modular vector field $X_\mu$, where $\mu$ is the standard volume form on $\C$.
\end{enumerate}
\end{proposition}
\begin{definition}\label{def:natural}
Suppose $\C$ is a disjoint union of affine cylinders. An integration $\G$ of $\C$ satisfying (a)-(c) above is said to be \emph{natural}.
\end{definition}
It will be our goal to find all natural integrations $\G$ of $\C$. In order to do this, we will first classify natural integrations of the affine plane in Section \ref{section:affineplane}, which will enable us to classify integrations of the affine cylinder in Section \ref{section:affinecylinder}. Next, we will need to extend the classification of natural integrations of the affine cylinder to a disjoint union of affine cylinders. The needed data will be a labeled graph called the \emph{discrete presentation} consisting of a labeled graph which encodes the topology of the orbit space together with isotropy and holonomy data.
\subsection{Gluing bibundles}\label{subsection:gluing}
Once one describes the integrations $\G$ of $\C$ in terms of a discrete presentation, we will see that it is possible to describe the bibundles of $\G$. In order to do this, ones needs a gluing lemma for bibundles.

Let $\G \rightrightarrows M$ be any symplectic groupoid. Suppose $\{ U_i \}_{i \in I}$ is a saturated open cover of $M$, so each $U_i$ is a collection of orbits of $\G$. We set $U_{ij}:=U_i \cap U_j$ and, in particular, $U_{ii}= U_i$.
\begin{lemma}\label{lemma:cocycle}
Let $f: I \times I \to I \times I$ be a bijection and assume we have a family of $(\G|_{U_{f(i)}},\G|_{U_i})$-bibundles:
\[
\begin{tikzcd}
\G|_{U_{f(i)}} \darrow & P_i \arrow[dl] \arrow[loop left] \arrow[loop right, leftarrow] \arrow[dr] & \G|_{U_i} \darrow \\
U_{f(i)} & & U_i \, .
\end{tikzcd}
\]
Suppose further that for each pair $(i,j)\in I \times I$ we have an isomorphism of bibundles $\phi_{ij}: P_j|_{U_i \cap U_j} \to P_i|_{U_i \cap U_j}$ satisfying the cocycle condition:
\[ \phi_{ij} \circ \phi_{jk} = \phi_{ik}. \]
Then there exists a $(\G,\G)$-bibundle $P$ together with an isomorphism of bibundles $\phi_i:P|_{U_i} \to P_i$ such that:
\[ \phi_{ij} \circ \phi_j = \phi_i. \]
\end{lemma}
\begin{proof}
Using the well-known description of principal $\G$-bundles in terms of Haeflieger cocycles, it is clear that one can glue them along a saturated cover, provided they are related on the intersections by isomorphisms satisfying the cocycle condition. Hence, starting with the right principal $\G|_{U_i}$-bundles $P_i$ we construct a right principal $\G$-bundle $P$ and isomorphism of principal $\G|{U_i}$-bundles $\phi_i:P|_{U_i} \to P_i$. These isomorphisms allow to define left $\G|{U_i}$-principal actions on $P|_{U_i}$ from the ones on the $P_i$, commuting with the right action, and which agree on intersections. Hence, we obtain a $\G$-left principal action that commutes with the right $\G$-action, making $P$ is a principal $(\G,\G)$-bibundle, and for which the $\phi_i:P|_{U_i} \to P_i$ become  isomorphism of bibundles.
\end{proof}
\subsection{Pointed bibundles}\label{subsection:pointedbibundles}
\begin{definition}\label{defn:pointedbibundle}
Suppose $\G \rightrightarrows (M, m_0)$ and $\H \rightrightarrows (N, n_0)$ are groupoids over pointed manifolds. We say $(P,p_0)$ is a \emph{pointed bibundle} if $P$ is a $(\G,\H)$-bibundle and the anchor maps are base-point preserving.
\end{definition}
To any such $(P,p_0)$ there is a canonical isomorphism $\psi_{p_0}: \G_{m_0} \to \H_{n_0}$ such that:
\[ p_0 g = \psi_{p_0}(g) p_0.  \]
An \emph{isomorphism} $\phi: (P,p_0) \to (Q,q_0)$ of pointed bibundles is an isomorphism of bibundles such that the source (target) of $p_0$ and $q_0$ are equal. We say $\phi$ is a \emph{strong isomorphism} if $\phi(p_0) = \phi(q_0)$. To any isomorphism of pointed bibundles, we can associate a unique element $h_\phi \in \H_{n_0}$ characterized by the property:
\[ h_\phi \phi(p_0) = q_0 .\]
We can check easily that $h_\phi$ satisfies:
\[ C_{h_\phi} \circ \psi_{p_0} = \psi_{q_0} . \]
This leads us to the following lemma:
\begin{lemma}\label{lemma:transbimid}
Suppose $\G \rightrightarrows (M,m_0)$ and $\H \rightrightarrows (N,n_0)$ are transitive groupoids over pointed manifolds. Let $(P,p_0)$ and $(Q,q_0)$ be pointed bibundles. The relation $\phi \mapsto h_\phi$ gives 1-1 correspondence between isomorphisms $\phi: (P,p_0) \to (Q,q_0)$ and elements $h \in H$ such that:
\begin{equation}\label{eqn:conjprop}
C_h \circ \psi_{p_0} = \psi_{q_0}.
\end{equation}
\end{lemma}
\begin{proof}
We begin by commenting that two isomorphisms $\phi_1,\phi_2: P \to Q$ are equal if and only if there exists $p \in P$ such that $\phi_1(p) =\phi_2(p)$. This immediately implies that $\phi \mapsto h_\phi$ is injective.

It only remains to show that given $h \in \H_{n_0}$ we can construct $\phi$ such that $h_\phi = h$. Any $p \in Q$ can be written in the form $p= h_1 p_0 g_1$ for $g_1 \in \G$ and $h_1 \in H$, we define $\phi(p) = h_1 (h\inv q_0) g_1$. Property (\ref{eqn:conjprop}) implies that this definition is invariant with respect to the decomposition of $p$. Clearly $h \phi(p_0) = q_0$ and therefore $h_\phi = h$.
\end{proof}
By interpreting isomorphisms of bibundles in this way, we get the following useful properties. For $\phi_2: P \to Q$ and $\phi_1: Q \to R$ then:
\begin{equation}\label{eqn:compositionisom}
h_{\phi_1 \circ \phi_2} = h_{\phi_1} h_{\phi_2}.
\end{equation}
On the other hand, given $\phi_1: P_1 \to Q_1$ and $\phi_2: P_2 \to Q_2$ such that $P_1 \otimes P_2$ is defined, then:
\begin{equation}\label{eqn:tensorproductisom}
h_{\phi_1 \otimes \phi_2} = h_{\phi_1} \psi_{p_1}(h_{\phi_2}).
\end{equation}

\section{Picard groups of the affine plane}\label{section:affineplane}
In this section, $\G$ will denote a natural integration of $\aff$ (see Definition \ref{def:natural}). We will also denote by $\G^+$ and $\G^-$ the restrictions of $\G$ to the positive plane $\aff^+$ and to the negative plane $\aff^-$. The symbol $\pm$ will be used to indicate that cases for both $+$ and $-$ are being treated simultaneously.

Our aim is to compute $\Pic(\G)$ and we will proceed as follows:
\begin{itemize}
\item in Section \ref{subsection:ses}, we show that any integration $\G$ of $\aff$ arises as a semi direct product  $\G\Aff \times_\aff \K$ where $\K$ is a discrete bundle of Lie groups;
\item in Section \ref{subsection:isodata}, we will show how to construct $\K$ from discrete data which we will refer to as the \emph{isotropy data} of $\G$;
\item in Section \ref{subsection:mapsofisodata}, we obtain a correspondence between maps of isotropy data and bibundles;
\item finally, in Section \ref{subsection:bimoveraff}, we compute the Picard group of $\G \rightrightarrows \aff$.
\end{itemize}
\subsection{The short exact sequence of $\G$}\label{subsection:ses}

Our goal is to find a `split fibration' of $\G$. Note that $\G\Aff$ is the only source connected integration of $\aff$, up to isomorphism. Hence, $\G^0\simeq \G\Aff$ and we denote by $i: \G\Aff \to \G$ the inclusion. The proof of the following lemma is inspired in Proposition 3 from \cite{Radko}.
\begin{lemma}\label{lemma:projection}
There is a Lie groupoid morphism $p:\G \to \G\Aff$ which is split by the canonical map $i: \G\Aff \to \G$.
\end{lemma}
\begin{proof}
Recall that $\G\Aff \isom \Aff \times \aff$. Let $g \in \G$ be an arrow with $\t(g) = (x_2,y_2)$ and $\s(g) = (x_1,y_1)$, where $x_1\not=0$. Then we can set:
\[ p(g):= \left( \log \left( \frac{x_2}{x_1}\right), \frac{y_2-y_1}{x_1},x_1,y_1 \right) \]
This map is a local symplectomorphism. Morover, its restriction to $\G^0\simeq \G\Aff$ is easily seen to be the identity, for an arrow $(a,b,x,y)\in  \G\Aff$ has source $(x,y)$ and target $(e^a x, y+xb)$, so that: $p(a,b,x,y)=(a,b,x,y)$ (see Example \ref{example:sigmaaff}). Hence, $p\circ i$ is the identity and it remains to show that $p$ extends to all of $\G$.

Suppose $g_0 \in \G_{(0,y_0)}$. We choose a local bisection $\sigma$ around $g_0$ and let  $g(t) = \sigma(t,y_0)$, so:
\[ \t(g(t)) = (x(t),y(t)) \qquad \s(g(t)) = (t,y_0). \]
Observe that for $t\ne 0$, one has $p(g(t)) = (\log(x(t)/t),(y-y_0)/t,t,y_0)$. Also:
\[ \lim_{t \to 0} \left( \frac{x(t)}{t} \right) = x'(0)>0, \quad \lim_{x \to 0}  \left( \frac{y(t) - y_0}{t} \right) = y'(0). \]
Since $\diff t \cdot g'(0) = (x'(0),y'(0))$, these limits exist. We define $p$ on all of $\G$ by letting:
\[ p(g_0) = (\log(x'(0)), y'(0),0,y_0). \]

We must check that this definition is independent of the choice of local bisection $\sigma$. If $\sigma_1$ and $\sigma_2$ are two local sections through $g_0$ we claim that
\[ \diff \t \cdot g'_1(0) = \diff \t \cdot g'_2(0),\]
so it follows that $p$ is well-defined. For this observe that $t\mapsto g_1(t) g_2(t)^{-1}$ is a smooth curve in $\G^0\simeq \G\Aff$ that passes through the point $(0,0,0,0)$ at $t=0$. Hence, if we write $g_1(t) g_2(t)^{-1}=(a(t),b(t),x(t),y(t))\in  \G\Aff$, we have:
\begin{align*}
\t(g_1(t))&=\t(g_1(t) g_2(t)^{-1})=\t(a(t),b(t),x(t),y(t))=(e^{a(t)} x(t), y(t)+x(t)b(t)),\\
\t(g_2(t))&=\s(g_1(t) g_2(t)^{-1})=\s(a(t),b(t),x(t),y(t))=(x(t), y(t)).
\end{align*}
Since we have $x(0)=y(0)=a(0)=b(0)=0$, it follows that:
\[ \diff \t \cdot g'_1(0) = (x'(0),y'(0))=\diff \t \cdot g'_2(0),\]
so the claim follows.

We leave the details of proving that $p$ is differentiable morphism of groupoids to the reader.
\end{proof}
The map $p$ fits into the following short exact sequence of Lie groupoids:
\begin{equation}\label{diagram:splitfibration}
\begin{tikzcd}
1 \arrow[r] & \mathcal{K} \darrow \arrow[r, hook] & \G \darrow \arrow[r, two heads, "p"] & \G\Aff \darrow \arrow[r] \arrow[l, dashed, bend right, "i", swap] \darrow & 1 \\
 & \aff  \arrow[r] & \aff  \arrow[r] & \aff \, . &
\end{tikzcd}
\end{equation}
Here $\mathcal{K}$, the kernel of $p$, is a bundle of discrete groups over $\aff$. The map $i$ yields a natural action of $\G\Aff$ on $\K$, and we have that $\G$ is the semi-direct product of $\G\Aff$ and $\K$:
\begin{lemma}
The map $F: \G\Aff \times_{\s,\t} \K \to \G$ such that $F(g,k) = i(g)k$ is an isomorphism.
\end{lemma}
\begin{proof}
The map $p$ is a fibration since it covers a submersion (the identity) and $\G \to \G\Aff \times_\aff  \aff  = \G\Aff$ is a submersion. The canonical map $i: \G \Aff \to \G$ is a section of $p$ so this exhibits a flat cleavage of $p$. Therefore, by \cite{Makenzi} (Thm 2.5.3) the map is an isomorphism.
\end{proof}
More explicitly, the multiplication in $\G\Aff \times_{\aff} \K$ is given by the formula:
\[ (g',k') \cdot(g,k) = (g' g , \theta_{g\inv} (k')k), \]
where $\theta_g(k)$ denotes the action of an element $g\in \G\Aff$ on $k\in \K$.
\begin{remark}
We note that we have not used all of our naturality assumptions about $\G$. In fact, the above lemmas are true for any integration of $\G$ with connected orbits.
\end{remark}
\subsection{The isotropy data of $\G$}\label{subsection:isodata}
Since $\G \isom \G\Aff \times_\aff \K$, the wide subgroupoid $\K$ determines $\G$ up to isomorphism. Our next task is to show that we can reduce $\K$ to a few pieces of discrete data. Since we can lift the modular vector field $\partial/\partial y$ to a flow of $\G$, we see that $\K|_{x=0}$ is a locally trivial bundle of discrete groups. Let,
\[  G^+ := \K_{(1,0)}, \quad G^-:= \K_{(-1,0)}, \quad H:= \K_{(0,0)}.\]
For each $h \in H$ there is a unique global section $\sigma_h : \aff \to \K$ such that $\sigma_h(0,0) = h$. Similarly, for $g \in G^\pm$ there is a unique section $\sigma_g: \aff^\pm \to \K$ such that $\sigma_g(\pm 1,0) = g$. So we can define maps
\begin{align*}
\aff \times H \to \K,&\quad (x,y,h)\mapsto \sigma_h(x,y),\\
\aff^\pm \times G^\pm \to \K,& \quad (x,y,g)\mapsto \sigma_g(x,y),
\end{align*}
which cover $\K$. Moreover, these give rise to group homomorphisms:
\[ \phi^\pm: H \to G^\pm,\quad \phi^\pm(h) := \sigma_h(\pm 1,0). \]
\begin{definition}\label{def:isodata}
We call \emph{isotropy data} $I$ a pair of homomorphisms $\phi^\pm: H \to G^\pm$, where $H,G^\pm$ are arbitrary discrete groups:
\[I := \left( \begin{tikzcd} G^- & H \arrow[l, "\phi^-" swap] \arrow[r, "\phi^+"] & G^+ \end{tikzcd} \right). \]
When $\phi^\pm$ arise from an integration $\G$ of $\aff$ as above, we call $I$ the \emph{isotropy data associated to $\G$}.
\end{definition}
Notice that isotropy data is only defined for $\G$ such that the discrete bundle $\K|_{x=0}$ is locally trivial. This condition is equivalent to requiring that the modular vector field lifts to a complete vector field on $\G$.

Given arbitrary isotropy data $I$, we denote by $\K(I)$ the following bundle of groups over $\aff$:
\[ \mathcal{K}(I) := \left(\left( \bigsqcup_{g \in G^\pm} \aff^\pm \times \{ g \} \right) \sqcup \left( \bigsqcup_{h \in H} \aff \times \{h \} \right)\right)/ \sim, \]
where $\sim$ is the equivalence relation generated by:
\[ (x,y,h_1) \sim (x,y,h_2)\quad \text{ if }\quad \phi^\pm(h_1)=\phi^\pm(h_2),\ (x,y)\in\aff^\pm.   \]
We equip $\K(I)$ with the quotient topology, so that it becomes a bundle of discrete groups over $\aff$. There is an obvious action of $\G\Aff$ on $\K(I)$ and we call the Lie groupoid:
\[ \G(I) := \G\Aff \times_\aff \K(I) \rightrightarrows \aff \]
the \emph{symplectic groupoid} associated to $I$. The symplectic structure is the pullback under the projection of the symplectic structure in $\G\Aff$. One checks easily that $\K(I)$ (and hence $\G(I)$) will be Hausdorff if and only if $\phi^\pm$ are injective.

The equivalence relation $\sim$ is precisely the equivalence relation given by the intersection of the images of $\sigma_h$ and $\sigma_g$ for $h \in H$ and $g \in G^\pm$. Therefore,
\begin{theorem}\label{thm:affclass}
Let $p: \G \to \G\Aff$ be the fibration from Lemma \ref{lemma:projection}, with kernel $\K$. If $\G$ is natural and $I$ is the isotropy data associated with $\G$ then $\K(I) \isom \K$ and $\G(I) \isom \G$.
\end{theorem}
Any point in $\G \isom \G(I)$ can be uniquely represented by a pair $(g,\alpha)$ where $g \in \G\Aff$ and
\[ \alpha \in
\left\{
\begin{array}{l}
G^\pm\quad \text{ if } g\in \G\Aff^\pm   \\
\\
H \quad \text{ if } g \in \G\Aff|_{x=0} \\
\end{array}
\right.
\]
In these ``coordinates'' the product is given by $(g,\alpha) \cdot (h,\beta) = (gh,\alpha \beta)$.

\subsection{Maps of isotropy data}\label{subsection:mapsofisodata}
In this section, $\G_1 \isom \G(I_1)$ and $\G_2 \isom \G(I_2)$ will be symplectic groupoids integrating $\aff$ with isotropy data $I_1$
and $I_2$ respectively where:

\[ I_i := \left( \begin{tikzcd} G_i^- & H_i \arrow[l, "\phi_i^-" swap] \arrow[r, "\phi_i^+"] & G_i^+ \end{tikzcd} \right) \mbox{ for }i=1,2. \]

Recall the discussion of pointed bibundles in Section \ref{subsection:pointedbibundles}. We think of $\G_1$ and $\G_2$ as groupoids over the pointed manifold $(\aff, (0,0))$. Therefore, if we say $(P,p_0)$ is a pointed $(\G_2,\G_2)$-bibundle we mean that $p_0 \in P_{(0,0)}$.

A bibundle $P$ is \emph{orientation preserving} if it relates the positive half-plane to the positive-half plane and \emph{orientation reversing} otherwise. A bisection $\sigma$ of a $Z$-static bibundle is called $Z$-static if $\s \circ \sigma (x,y) = (x,y)$ or $\s \circ \sigma(x,y) = (-x,y)$.
\begin{proposition}
Suppose $P$ is a $Z$-static $(\G_2,\G_1)$-bibundle. Then $P$ admits a $Z$-static bisection.
\end{proposition}
\begin{proof}
Suppose $P$ orientation preserving. We first make the following claim:
There is a local symplectomorphism $\Phi:P \to \G\Aff$ which makes the following diagram commute:
\[
\begin{tikzcd}
P \arrow[dr, "\t \times \s" swap] \arrow[rr, "\Phi"] & &\G\Aff \arrow[dl, "\t \times \s"] \\
&\aff \times \aff \, . &
\end{tikzcd}
\]
We can think of the claim as an analogue of lemma \ref{lemma:projection} for bibundles and the proof of this claim is similar. If $P$ is orientation preserving then for $p \in P$ let $\s(p) = (x_2,y_2)$ and $\t(p) = (x_1,y_1)$. If $P$ is orientation reversing then let $\s(p) = (-x_2,y_2)$ and $\t(p)=(x_1,y_1)$. Then we define $\Phi$ for $x_1\ne 0$ by:
\[ \Phi(p):= \left(\log \left(\frac{x_2}{x_1}\right),\frac{y_2-y_1}{x_1},x_1,y_1 \right),\]
so the claim follows.

Let $p_0 \in P_{(0,0)}$ be such that $\Phi(p_0) = (0,0,0,0) \in \G\Aff|_Z$. Notice that $Q:= \Phi\inv(\u(Z))$ is a principal $H_2$ bundle over $Z$. Therefore, $p_0$ extends to a unique section of $Q$. On the other hand, since $\Phi$ is a local diffeomorphism, the identity section of $\G\Aff$ gives a unique local extension of any $q \in Q$ to a static bisection. Therefore $p_0$ extends to a unique static bisection in some neighborhood of $Z$ which we can extend uniquely to all of $\aff$.

When $P$ is orientation reversing, we replace $\s$ by composing it with $(x,y) \mapsto (-x,y)$, and then the same argument applies.
\end{proof}
A \emph{pointed} bibundle $(P,p_0)$ is called $Z$-\emph{static} if $P$ is a $Z$-static bibundle and $p_0$ extends to a $Z$-static bisection. Let $(P,p_0)$ be such a orientation preserving pointed $(\G_2, \G_1)$-bibundle and $\sigma$ the associated bisection. Then the points $\sigma(\pm 1,0)$ and $\sigma(0,0) \in P$ determine group homomorphism $\psi^\pm: G^\pm_1 \to G^\pm_2$ and $\psi: H_1 \to H_2$ such that:
\begin{equation}\label{diagram:opisotropymap}
\begin{tikzcd}
G_1^- \arrow[d, "\psi^-"'] & H_1 \arrow[l, "\phi_1^-"']\arrow[r, "\phi_1^+"] \arrow[d, "\psi"] & G_1^+ \arrow[d, "\psi^+"] \\
G_2^- & H_2 \arrow[l, "\phi_2^-"'] \arrow[r, "\phi_2^+"] & G_2^+ \, ,
\end{tikzcd}
\end{equation}
commutes. If $P$ is orientation reversing then we get a similar diagram:
\begin{equation}\label{diagram:orisotropymap}
\begin{tikzcd}
G_1^- \arrow[d, "\psi^-"'] & H_1 \arrow[l, "\phi_1^-"']\arrow[r, "\phi_1^+"] \arrow[d, "\psi"] & G_1^+ \arrow[d, "\psi^+"] \\
G_2^+ & H_2 \arrow[l, "\phi_2^+"'] \arrow[r, "\phi_2^-"] & G_2^- \, .
\end{tikzcd}
\end{equation}
This motivates the following definition:
\begin{definition}\label{def:isodatamap}
An \emph{orientation preserving isomorphism} $\Psi:I_1 \to I_2$ is a triple of group isomorphisms $\Psi =(\psi,\psi^\pm)$ such that (\ref{diagram:opisotropymap}) commutes. An \emph{orientation reversing isomorphism} $\Psi: I_1 \to I_2$ is a triple of group isomorphisms such that (\ref{diagram:orisotropymap}) commutes.
\end{definition}
This gives a category where composition corresponds to composing the vertical arrows. For a map of isotropy data $\Psi: I_1 \to I_2$ and $\alpha \in H,G^-,G^+$ we may sometimes write $\Psi(\alpha)$ to mean $\psi(\alpha),\psi^-(\alpha),\psi^+(\alpha)$, respectively.
\begin{example}[Inner automorphisms]
Given $\alpha \in H$ the \emph{inner automorphism} associated to $\alpha$ is the orientation preserving isomorphism $\C_\alpha:I \to I$ where $\psi:= C_\alpha:H \to H$ is conjugation by $\alpha$ and $\psi^\pm := C_{\phi^\pm(\alpha)}:G^\pm \to G^\pm$ is conjugation by $\phi^\pm(\alpha)$.
\end{example}
\begin{example}[Pointed bibundles]
In the above discussion, we constructed the \emph{isomorphism of isotropy data} associated to $(P,p_0)$, a $Z$-static pointed $(\G_2,\G_1)$-bibundle.
\end{example}
The next lemma says that the second example is generic:
\begin{lemma}\label{lemma:bimvsisomap} There are 1-1 correspondences between:
\begin{enumerate}[(i)]
\item orientation preserving isomorphisms of isotropy data $\Psi: I_1 \to I_2$ and orientation preserving $Z$-static pointed $(\G_2,\G_2)$-bibundles.
\item orientation reversing isomorphisms of isotropy data $\Psi: I_1 \to I_2$ and orientation reversing $Z$-static pointed $(\G_2,\G_2)$-bibundles.
\end{enumerate}
For this lemma, we are considering pointed bibundles up to \emph{strong isomorphism}.
\end{lemma}
\begin{proof}
We prove the orientation preserving case. The orientation reversing case is similar, so can be left to the reader. Given a orientation preserving $Z$-static pointed $(\G_2,\G_1)$-bibundle we have already provided the construction of an isomorphism of isotropy data above. Suppose $\Psi$ is an isomorphism of isotropy data. Then we can define a symplectic groupoid isomorphism $F: \G_1 \to \G_2$. In our `coordinates' from before, $F$ takes the form:
\[ F(g,\alpha) = (g, \Psi(\alpha)). \]
The bibundle associated to this map $P_\Psi$ is symplectomorphic to $\G_2$ and comes with a canonical point $p_\Psi = \u(0,0)$. Clearly $(P_\Psi, p_\Psi)$ is an orientation preserving $Z$-static pointed bibundle. The definition of the actions on $P_\Psi$ make it clear that the isomorphism of isotropy data associated to $(P_\Psi,p_\Psi)$ is $\Psi$.
\end{proof}
From now on, we will use the notation $P_\Psi$ to denote the bibundle associated to a map of isotropy data. The map $\Psi \mapsto P_\Psi$ is functorial, i.e.
\[ P_{\Psi_1 \circ \Psi_2} \isom P_{\Psi_1} \otimes P_{\Psi_2} . \]
When we pass to ordinary isomorphism classes of bibundles, we can think of the map $\Psi \mapsto P_\Psi$ as corresponding to the forgetful map $(P,p_0) \mapsto P$.
\subsection{The Picard group of $\G(I)$}\label{subsection:bimoveraff}
We need one more lemma before calculating the Picard group.
\begin{lemma}\label{lemma:staticbibundleclass}
For any isomorphism $\Psi: I \to I$, the bibundle $P_\Psi$ is trivial if and only if $\Psi$ is an inner automorphism.
\end{lemma}
\begin{proof}
It is proved in \cite{BPic} that for any automorphism $F: \G \to \G$ then $P_F$ is a trivial bibundle if and only if $F$ is an inner automorphism associated to a static bisection. Now suppose $P_\Psi$ is trivial. Recall the $F$ from the construction of $P_\Psi$ in Lemma \ref{lemma:bimvsisomap}. Since $P_\Psi$ is trivial, we must have that $F$ is given by conjugation by a static bisection $\sigma$. Let $\alpha = \sigma(0,0) \in H$. Then $\psi: H \to H$ must be $C_\alpha$ and $\psi^\pm = C_{\phi^\pm(\alpha)}$ and  so $\Psi$ is an inner automorphism. On the other hand, if $\Psi = \C_\alpha$ let $\sigma$ be the unique static bisection extending $\alpha$. Then clearly $F$ is the inner automorphism induced by $\sigma$.
\end{proof}
In particular, the inner automorphisms are in the kernel of the map $\Psi \mapsto P_\Psi$ and form a normal subgroup. Therefore, it makes sense to define $\OutAut(I)$ to be the automorphisms of $I$ modulo inner automorphisms. We now prove the main theorem of this section:
\begin{theorem}
Let $\G = \G(I)$ be a natural symplectic groupoid integrating $\aff$. Then:
\[ \Pic(\G) \isom \OutAut(I) \times \R. \]
\end{theorem}
\begin{proof}
We start by observing that from Lemma \ref{lemma:staticbibundleclass} it follows that the subgroup $Z\Pic(\G) \subset \Pic(\G)$ of $Z$-static bibundles is isomorphic to $\OutAut(I)$. By naturality, we have a 1-parameter group of symplectic automorphisms of $\G$ integrating the modular vector field:
\[ \Modd^t: \G \to \G, \quad (a,b,x,y,\alpha)\longmapsto (a,b,x,y+t,\alpha).\]
We let $\mathcal{M}^t \in \Pic(M)$ be the element represented by the bibundle associated with $\Modd^t$, and this defines 1-parameter subgroup $\Mod \subset \Pic$ isomorphic to $\R$.

Every symplectic bibundle $P$ preserves the class of the modular vector field. It follows that an arbitrary $P \in \Pic(\G)$ acts on $\{x=0\}$ by translations. Therefore the subgroups $\Mod$ and $Z\Pic(\G)$ generate $\Pic(\G)$. Since $Z$-static bibundles commute with $\Mod^t$, elements in $Z \Pic(\G)$ commute with elements in $\Mod$. Hence:
\[  Z\Pic(\G) \times \Mod \to \Pic(\G), \quad (P,Q) \longmapsto PQ,\]
is a well defined surjective group homomorphism. We claim that the kernel is trivial: if $P \in Z\Pic(\G)$ and $Q \in \Mod$, then $PQ = 1$ if and only if $P = Q\inv$. But $Q\inv$ is $Z$-static if and only if $Q = 1$.
\end{proof}

\section{The Picard group of the affine cylinder}\label{section:affinecylinder}
Let $\caff$ be the affine cylinder. Points in $\caff$ are equivalence classes $[(x,y)]$ of pairs $(x,y) \in \R^2$, where:
\[  (x,y) \sim (x,y+n),\ \forall n \in \Z. \]
The projection $C:\aff \to \caff$, $(x,y) \mapsto [(x,y)]$, is a local Poisson diffeomorphism. Throughout this section $\G \rightrightarrows \caff$ is a natural integration of the affine cylinder. Our goal is to show that we can realize $\G$ as a quotient of some $\G(I)$. The projection $\G(I) \to \G$ will depend on the \emph{holonomy data} of $\G$. We will proceed as follows:
\begin{itemize}
    \item in Section \ref{subsection:holdata}, we extract discrete data from any natural integration of $\caff$ and give a definition of \emph{holonomy data}, denoted $(I,\Hol)$;
    \item in Section \ref{subsection:mapsofhol}, we classify bibundles over $\caff$ via maps of holonomy data;
    \item in Section \ref{subsection:caffpic}, we compute the Picard group a natural integration of $\caff$.
\end{itemize}
\subsection{Holonomy data}\label{subsection:holdata}
Similarly to the affine plane, natural integrations of $\G$ of the affine cylinder can be characterized with discrete data. We will begin by defining the isotropy data of $\G$.
\begin{lemma}
Suppose $\G$ is a natural integration of $\caff$. Then there exists isotropy data $I$ and a surjective groupoid homomorphism $\Proj: \G(I) \to \G$:
\[
\begin{tikzcd}
\G(I) \darrow \arrow[r, "\Proj"] & \G \darrow \\
\aff \arrow[r]                   & \caff \, ,
\end{tikzcd}
\]
such that $\Proj$ is an isomorphism at the level of isotropy groups and $\aff \to \caff$ is the universal cover.
\end{lemma}
\begin{proof}
First consider the pullback groupoid:
\[  C^* \G :=  \{ ((x_2,y_2), \, g , \, (x_1,y_1)) \in \aff \times \G \times \aff : \t(g) = C(x_2,y_1), \,  \s(g) = C(x_1,y_1) \}.\]
Let $\overline{\G}$ be the maximal open subgroupoid of $C^* \G$ with connected orbits. Let $\Proj:\overline{\G} \to \G$ be projection to the middle component. The homomorphism $\Proj$ is surjective and is an isomorphism when restricted to any isotropy group of $\overline{G}$. Furthermore, since $\G$ is natural, it is clear that $\overline{\G}$ is natural and by Theorem \ref{thm:affclass} we can identify $\overline{\G}$ with $\G(I)$ for some isotropy data $I$.
\end{proof}
If $\Proj: \G(I) \to \G$ is as in the above lemma, we say $I$ is the \emph{isotropy data} of $\G$ which we denote as before $(G^- \from H \to G^+)$.

Our goal is to compute the fibers of $\Proj$. To do this, we will define some convenient notation.
\begin{itemize}
\item The image $\Proj(\K(I)) = \K \subset \G$ is a discrete bundle of Lie groups. Clearly $H \isom \K_{(0,0)}$ so let $\hol: H \to H$ be the holonomy of $\K|_{\{ 0 \} \times \S^1}$. In other words:
\[ \Proj(g - n, \hol^n(\alpha)) = \Proj(g ,\alpha ) \, . \]
Where if $g = (a,b,x,y)$ then $g+n:= (a,b,x,y+n)$ for any integer $n$.
\item Let $\eta$ be the unique bisection $\eta: (\aff - Z) \to \G\Aff$ such that:
\[ \t \circ \eta(x,y) = (x,y+1) \, . \]
Observe that $\eta$ cannot be defined on the critical locus $Z$.
\item Since $\Proj(\eta(\pm 1,0)) \in \K_{(\pm 1, 0)} \isom G^\pm$, let $\gamma^\pm \in G^\pm$ be such that:
\[ \Proj(\eta(x,y)) = (\u(x,y),\gamma^\pm) \quad x \neq 0 \, . \]
\item Let $\zeta: (\aff - Z) \to \G(I)$ be the bisection $\zeta$ such that
\[ \zeta(x,y) = (\eta(x,y),(\gamma^\pm)\inv) \quad \mbox{for }x \in \aff^\pm. \]
\end{itemize}
\begin{proposition}\label{prop:equivrelation}
Suppose $\Proj: \G(I) \to \G$ is as in the above lemma. Let $\zeta^\pm$ and $\hol: H \to H$ be as we just defined them. Then for any $(g,\alpha),(g',\alpha') \in \G(I)$ we have that $\Proj(g,\alpha) = \Proj(g,\alpha)$ if and only if one of the following holds:
\begin{enumerate}[(i)]
\item $g,g' \in \G\Aff|_{\aff^\pm}$ and there exists integers $n$ and $m$ such that:
\[ \zeta^n( g , \alpha) \zeta^m = (g', \alpha') \, ;\]
\item $g_1,g_2 \in \G\Aff|_Z$ and there exists an integer $n$ such that:
\[ (g_1 - n, \hol^n(\alpha_1)) = (g_2, \alpha_2) \, . \]
\end{enumerate}
\end{proposition}
\begin{proof}
Suppose $g_1,g_2 \in \G\Aff|_{\aff^\pm}$ and there exists $n$ and $m$ as above. Observe that $\Proj(\zeta)$ is the identity section whenever $\zeta$ is defined and so $\Proj(g,\alpha) \Proj(\zeta^n (g, \alpha) \zeta^m) = \Proj(g',\alpha')$.

Now suppose $g_1,g_2 \in \G\Aff|_{\aff^\pm}$ such that $\Proj(g,\alpha) = \Proj(g',\alpha')$. Then the $y$ coordinates of the source and target of $g$ and $g'$ differ by integers. Therefore there exist $n$ and $m$ such that $\eta^{-n} g \eta^{-m} = g'$. Hence $\Proj(\zeta^n (g,\alpha) \zeta^m) = \Proj(g',\alpha')$. Since $\zeta^n \cdot (g,\alpha) \cdot \zeta^m$ has the same source and target as $(g',\alpha')$ and are mapped by $\Proj$ to the same point, they must be equal (recall $\Proj$ is a bijection at the level of isotropy groups).

Now suppose $g_1,g_2 \in \G\Aff|_Z$ and there exists an $n$ satisfying (ii). By the definition of the holonomy map $\Proj(g_1 - n, \hol^n(\alpha_1))= \Proj(g_1,\alpha_1)$ and so $\Proj(g_1,\alpha_1) = \Proj(g_2,\alpha_2)$. On the other hand, suppose $\Proj(g_1,\alpha_1) = \Proj(g_2,\alpha_2)$. Then there exists a unique $n$ such that $g_2 = g_1 - n$. By the definition of holonomy, we conclude that $\Proj(g_2, \hol^n(\alpha_1) ) = (g_2, \alpha_2)$. Since $\Proj$ is an isomorphism at the level of isotropy groups, we conclude that $\hol^n(\alpha_1) = \alpha_2$.
\end{proof}

Hence, the topology of $\G$ is uniquely determined by $I$, $\hol$ and $\gamma^\pm$. We call the tuple $\Hol = (I,\hol,\gamma^\pm)$ the \emph{holonomy data} of $\G$. The holonomy map $\hol$ and $\gamma^\pm$ also satisfy a compatibility condition. Let $\hol^\pm: G^\pm \to \G^\pm$ be conjugation by $\gamma^\pm$.
\begin{lemma}
The triple $\mathbf{H}=(\hol,\hol^\pm): I \to I$ is an automorphism of isotropy data.
\end{lemma}
\begin{proof}
We must show that:
\begin{equation}\label{diagram:hol}
\begin{tikzcd}
G^- \arrow[d, "\hol^-" swap] & H \arrow[l] \arrow[r] \arrow[d,"\hol"] & G^+ \arrow[d, "\hol^+"] \\
G^-                    & H \arrow[l] \arrow[r]                  & G^+ \, ,
\end{tikzcd}
\end{equation}
commutes.
We first claim that $\hol^\pm$ are the holonomy of $K$ around the circles $\{ \pm 1 \} \times \S^1$.
Let $\eta^t: \aff - Z \to \G\Aff$ be the unique 1-parameter family of bisections such that:
\[ \t(\eta^t(x,y)) = (x,y+t) \, . \]
Let $g(t): [0,1]$ be the path in $\G$ defined as follows:
\[ g(t) = \Proj( \eta^t(\pm 1,0) \cdot \Proj(\u(\pm 1,0), \alpha) \cdot (\eta^{-t}(\pm 1,t) ). \]
Then $g(t)$ is a path in $\K$ covering the circle $\{ \pm 1 \} \times \S^1 \subset \caff$. Furthermore:
\[ g(0) = \Proj(\u( \pm 1, 0) , \alpha ), \mbox{ and }  g(1) = \Proj(\u( \pm 1, (\gamma^\pm) \alpha (\gamma^\pm)\inv. \]
Hence the holonomy of $\K|_{\{ \pm 1 \} \times \S^1}$ is conjugation by $\gamma^\pm$ and the claim follows.

The holonomy of $\K$ can also be understood in terms of the projection $\K(I) \to \K$. In particular, (\ref{diagram:hol}) commutes as a consequence of the continuity of $\K(I) \to I$.
\end{proof}
This motivates our definition of holonomy data:
\begin{definition}\label{def:holofGdef}
Let $I$ be isotropy data $G^- \from H \to G^+$, $\hol: H \to H$ be an isomorphism, and $\gamma^\pm \in G^\pm$. Then $\Hol = (I,\hol,\gamma^\pm))$ is called \emph{holonomy data} if $\hol^\pm = \C_{\gamma^\pm}$ and $\mathbf{H}:=(\hol,\hol^\pm): I \to I$ is an automorphism of isotropy data.

A \emph{strong isomorphism} $(I_1,\hol_1,\gamma^\pm_1) \to (I_2,\hol_2,\gamma^\pm_2)$ is an isomorphism $(\psi,\psi^\pm): I \to I$ such that $\Psi$ commutes $\mathbf{H}$ and $\psi^\pm(\gamma_1^\pm) = \gamma_2^\pm$.
\end{definition}
We say strong isomorphism above since we will need a weaker notion of isomorphism in our treatment of bibundles over $\caff$. By Proposition \ref{prop:equivrelation} the fibers of $\G(I) \to \G$ can be computed entirely in terms of the holonomy data, hence:
\begin{theorem}
Suppose $\G_1$ and $\G_2$ are natural integrations of $\caff$. Then $\G_1$ and $\G_2$ are isomorphic if and only if their holonomy data is strongly isomorphic.
\end{theorem}
\subsection{Bibundles over $\caff$}\label{subsection:mapsofhol}
Suppose $\G_1$ and $\G_2$ are natural integrations with holonomy data $(I_1,\Hol_1)$ and $(I_2,\Hol_2)$ respectively. We will denote this data as follows:
\[
I_i := \begin{tikzcd}
\G^-_i & H_i \arrow[l,"\phi_i^-" swap] \arrow[r,"\phi_i^+"] & \G^+_i \, .
\end{tikzcd}
\]
\[ \Hol_i := (\hol_i,\gamma^\pm_i) \, . \]
Throughout, $\Proj_i : \G(I_i) \to \G_i$ denotes the projection from Section \ref{subsection:holdata}. As in the affine plane case, we say that a $(\G_2,\G_1)$-bibundle is $Z$-static if the induced map of orbit spaces fixes the critical line. By fixing a covering $\aff \to \caff$ we can think of $\caff$ together with the image of the origin in $\aff$ as a pointed manifold. Our aim is characterize pointed $(\G_1, \G_2)$-bibundles in terms of holonomy data.
\begin{lemma}\label{lemma:phi}
Suppose $P$ is a $Z$-static pointed $(\G_2,\G_1)$-bibundle. Then there exists a unique pointed $(\G(I_2),\G(I_1))$-bibundle $(P_\Psi,p_0)$ and a surjective submersion $\Phi: P_\Psi \to P$ satisfying:
\begin{equation}\label{eqn:caffbim}
\Phi( g \cdot p) = \Proj_2(g) \cdot \Phi(p) \, , \qquad \Phi( p \cdot g) = \Phi(p) \cdot \Proj_1(g) \, .
\end{equation}
\end{lemma}
\begin{proof}
Consider the pullback groupoids $C^* \G_1$ and $C^* \G_2$. Let $C^* P$ be the corresponding $(C^* \G_2 , C^* \G_1)$-bibundle:
\[ C^* P := \{ ((x_2,y_2), \, p , \, (x_1,y_1) ) \in \aff \times P \times \aff: C(x_2,y_2) = \t(p) \, , C(x_1,y_1) = \s(p) \} \, .\]
Let $\overline{P}$ be the complement of $\{ (0,y_2) , p , (0,y_1) ) : y_2 \neq y_1 \}$ in $C^* P$. Then $\overline{P}$ is an open submanifold of $C^* P$. Let $\overline{p}:= ((0,0),p_0, (0,0))$ so that $(\overline{P},\overline{p})$ is a pointed manifold. The left and right actions of $C^* \G_2$ and $C^* \G_2$ make $(\overline{P},\overline{p}))$ into a $Z$-static pointed $(\overline{\G_2}, \overline{\G_1})$-bibundle. Hence we can identify $(\overline{P},\overline{p})$ with an isomorphism $\Psi: I_1 \to I_2$ and the associated pointed bibundle $(P_\Psi,p_\Psi)$. The construction of $\overline{P}$ makes it clear that (\ref{eqn:caffbim}) holds.
\end{proof}
Let $\sigma$ be the canonical static bisection of $(P_\Psi, p_\Psi)$. Since the left action is principal, we know that there exists an $h \in H_2$ such that $h \cdot \Phi(\sigma(0,0))= \Phi(\sigma(0,1))$. Let $\delta$ be the unique static bisecion of $\G(I_2)$ extending $h$. The bisection $\delta$ satisfies:
\begin{equation}\label{eqn:hprop}
\Phi(\delta(x,y) \cdot \sigma(x,y)) = \Phi(\sigma(x,y+1)).
\end{equation}
We can think of $\delta$ as measuring the failure $\sigma$ to descend to a bisection of $P$. The topology of $P$ is determined by the pair $(\Psi,\delta(0,0)=h)$ since:
\[ \Phi((g,\alpha) \cdot \sigma(x,y)) = \Phi((g',\alpha') \cdot \sigma(x,y+n)) \, \Leftrightarrow \, \]
\begin{equation}\label{eqn:bimprojcriteria}
\Proj_2(g,\alpha) = \Proj_2(g',\alpha') \cdot \left ( \prod_{i=0}^{n-1} \Proj_2(\delta(x,y+i)) \right).
\end{equation}
The pairs $(\Psi,h)$ arising in this manner are not arbitrary. The next proposition lays out the appropriate compatibility condition.
\begin{proposition}\label{prop:propertiesofmaps}
Let $(P,p_0)$ be a $Z$-static pointed $(\G_2,\G_1)$-bibundle and suppose $\Psi: I_1 \to I_2$ and $h \in H$ are as defined above. Then the following holds:
\begin{enumerate}[(i)]
    \item $\Hol_2 \circ \Psi = \C_h \circ \Psi \circ \Hol_1$;
    \item $\phi^\pm(h) = (\gamma^\pm_2) \Psi(\gamma^\pm_1)\inv$.
\end{enumerate}
\end{proposition}
\begin{proof}
We will prove the case where $P$ is orientation preserving and leave the orientation reversing case to the reader. We first show that \emph{($i$)} holds. Notice that:
\begin{align*}
    \Phi(\,(\u(x,y+1),\Psi(\alpha)) \cdot \sigma(x,y+1) \,) &= \Phi(\sigma(x,y+1)) \cdot \Proj_1(\u(x,y+1),\alpha); \\
    &= \Phi(\, \delta(x,y) \cdot \sigma(x,y) \,) \cdot \Proj_1(\u(0,0),\Hol_1(\alpha)); \\
    &= \Phi( \,\delta(x,y) \cdot ( \u(x,y),\Psi \circ \Hol(\alpha) ) \cdot \sigma(x,y) \,) .
\end{align*}
On the other hand:
\begin{align*}
    \Phi( \, (\u(x,y+1),\Psi(\alpha)) \cdot \sigma(x,y+1) \, ) &= \Proj_2(\u(x,y+1),\Psi(\alpha)) \cdot \Phi(\sigma(x,y+1)); \\
    &= \Proj_2(\u(0,0),\Hol\circ \Psi(\alpha)) \cdot \delta(x,y) \Phi(\sigma(x,y)); \\
    &= \Phi( \, (\u(x,y), \Hol \circ \Psi(\alpha)) \cdot \delta(x,y) \cdot \sigma(x,y) \,).
\end{align*}
Combining these two equalities at the values $(x,y) = (\pm 1,0),(0,0)$ yields:
\[ \Hol \circ \Psi(\alpha) = \C_h \circ \Psi \circ \Hol(\alpha). \]
We now show \emph{($ii$)} holds. Recall the bisection $\eta: \aff -Z \to \G\Aff$ from Section \ref{subsection:holdata}. Let $\zeta_i(x,y) = (\eta(\pm 1,0),\gamma^\pm_i) \in \G(I_i)$ for $i=1,2$. Then $\Proj_i(\zeta_i)$ is the identity at every point and therefore:
\begin{align*}
\Phi(\sigma(\pm 1,1)) &= \Proj_2(\zeta_2)\inv \cdot \Phi(\sigma( \pm 1,1)) \cdot \Proj_1(\zeta_1); \\
&= \Phi_2(\zeta_2\inv \cdot \Psi(\zeta_1)) \cdot \Phi(\sigma(\pm 1,0)).
\end{align*}
Therefore $\phi^\pm(h) = (\gamma_2^\pm)\Psi(\gamma_1^\pm)\inv$.
\end{proof}
This motivates our definition of maps of holonomy:
\begin{definition}\label{def:mapofholonomy}
Suppose $h \in H$ and $\Psi: I_1 \to I_2$  is an isomorphism of isotropy data. We say that $(\Psi, h ): (I_1, \Hol_1) \to (I_2, \Hol_2)$ is an \emph{isomorphism of holonomy data} if ($i$) and ($ii$) from Proposition \ref{prop:propertiesofmaps} hold.
\end{definition}
This produces a category in which the objects are holonomy data and isomorphisms are composed via the following rule:
\[ (\Psi_1,h_1) \circ (\Psi_2,h_2) := (\Psi_1 \circ \Psi_2, h_1 \Psi_1(h_2)). \]
\begin{example}[Inner automorphisms]\label{example:holinnerauto}
Suppose $\Psi = \C_\alpha: I \to I$ is an inner automorphism and $h$ satisfies:
\begin{equation}\label{eqn:innerprop}
\hol(\alpha) = h \alpha.
\end{equation}
Then the isomorphism $(\Psi, h ): (I,\Hol) \to (I,\Hol)$ is called an \emph{inner automorphism}.
\end{example}
\begin{example}[Pointed bibundles]\label{example:holpointedbim}
Given any $Z$-static pointed $(\G_2, \G_1)$-bibundle $(P,p_0)$ then Proposition \ref{prop:propertiesofmaps} says that the $\Psi$ and $h$ we constructed is an isomorphism of isotropy data. In such a case we call $(\Psi, h)$ the holonomy isomorphism of $(P,p_0)$.
\end{example}
The reader is encouraged to check that composition of bibundles corresponds to composition of isomorphisms of holonomy.
Our main end with this definition is to show that Example \ref{example:holpointedbim} is the generic case. The next proposition says that the holonomy isomorphism of a $Z$-static pointed bibundle classifies $(P,p_0)$ strongly.

If we combine (\ref{eqn:bimprojcriteria}) with Proposition \ref{prop:propertiesofmaps} and Proposition \ref{prop:equivrelation} then:
\begin{proposition}\label{prop:bibundlequotientversions}
Suppose $\Phi: P_\Psi \to P$ is as in Lemma \ref{lemma:phi}. Let $\hol_h: H \to H$ be the bijection $\hol_h(\alpha) = \hol(\alpha) h$. Then $\Phi((g,\alpha) \cdot \sigma(x,y)) = \Phi((g',\alpha') \cdot \sigma(x',y'))$ if and only if one of the following holds:
\begin{enumerate}[(i)]
\item $(x,y),(x',y') \in Z$ and there exists an integer $n$ such that:
\[ (g - n , \, \hol_h^n(\alpha)) = (g', \, \alpha') \, ; \]
\item $(x,y),(x',y') \in \aff^\pm$ and there exists integers $n$ and $m$ such that:
\[ \zeta_2^{-n} (g, \alpha ) \cdot \Psi(\zeta_1^{-m})  = (g',\alpha') .\]
\end{enumerate}
\end{proposition}
Since we can characterize $P$ in terms of $\Psi$ and $h$, two $Z$-static pointed bibundles are strongly isomorphic if and only if their associated holonomy isomorphisms are equal. We conclude this section with our main result regarding bibundles over $\caff$.
\begin{theorem}\label{thm:holcorrespondence}
There is a 1-1 correspondence between the following:
\begin{enumerate}[(i)]
\item orientation preserving isomorphisms of holonomy data and orientation preserving $Z$-static pointed $(\G_2,\G_2)$-bibundles;
\item orientation reversing isomorphisms of holonomy data and orientation reversing $Z$-static pointed $(\G_2,\G_2)$-bibundles.
\end{enumerate}
Again, we consider pointed bibundles up to strong isomorphism.
\end{theorem}
\begin{proof}
We have already shown how to obtain holonomy data from a bibundle. Hence, we need to show that given a holonomy isomorphism $(\Psi,h)$, we can construct a $Z$-static pointed bibundle. Let $\sim$ be the following equivalence relation on $P_\Psi$:
\[ (g,\alpha) \cdot \sigma(x,y) \sim (g',\alpha') \cdot \sigma(x,y+n) \Leftrightarrow (\ref{eqn:bimprojcriteria}) \mbox{ holds}.   \]
Let $\Phi: \G \to P_\Psi/\sim$ be the natural projection. We need to check that the equivalence relation respects the left and right actions. That is:
\begin{equation}\label{eqn:leftaction}
\Phi((g_2,\alpha_2) \cdot p ) = \Phi((g_2',\alpha_2') \cdot p') \, ,
\end{equation}
whenever $\Proj_2(g_2, \alpha_2) = \Proj_2(g'_2, \alpha_2') \mbox{ and } \Phi(p) = \Phi(p')$; and
\begin{equation}\label{eqn:rightaction}
\Phi(p \cdot (g_1, \alpha_1) ) = \Phi(p' \cdot (g_1',\alpha_1') )
\end{equation}
whenever $\Phi(p) = \Phi(p') \mbox{ and } \Proj_1(g_1,\alpha_1) = \Proj_1(g_1',\alpha_1'),$ for appropriately composable pairs.

We first check, (\ref{eqn:leftaction}) holds. Suppose $p = (g,\alpha) \cdot \sigma(x,y)$ and $p' = (g',\alpha') \cdot \sigma(x,y+n)$. Since $p \sim p'$ we have:
\[ \Proj_2(g,\alpha) = \Proj_2(g',\alpha') \cdot \left( \prod_{i=0}^{n-1} \delta(x,y+i) \right) .  \]
Since $\Proj_2(g_2, \alpha_2) = \Proj(g_2' , \alpha_2')$:
\[ \Proj_2(g_2 g, \alpha_2 \alpha) = \Proj_2(g_2' g', \alpha_2' \alpha') \left( \prod_{i=0}^{n-1} \delta(x,y+i) \right) .\]
So (\ref{eqn:leftaction}) holds.

To check the right action, we use the interpretation of $\sim$ from Proposition \ref{prop:bibundlequotientversions} and separate into cases where $g,g_1 \in \G\Aff|_{\aff^\pm}$ and $g \in \G\Aff|_Z$. Suppose $n, m ,k$ and $l$ are integers such that:
\[ \zeta_2^{-n} \cdot (g, \alpha)  \cdot \Psi(\zeta^{-m}) =  (g',\alpha') \,; \]
\[ \zeta_1^{-k} \cdot (g_1, \alpha_1) \cdot \zeta_1^{-l} = (g_1',\alpha_1' ) \, . \]
Composability tells us that $m$ and $k$ must be equal. Hence:
\[ \zeta_2^{-n} \cdot (g g_1, \alpha \Psi(\alpha) ) \cdot  \Psi(\zeta_1^{-l}) = (g' g_1', \alpha \Psi(\alpha_1') ) \,. \]
So the right action is well defined. Now suppose $g,g_1 \in \G\Aff|_Z$. So there exists $n$ and $m$:
\[ (g -n , \hol_h^n(\alpha) ) = (g', \alpha' ) \, ; \]
\[ (g_1 - m , \hol^m(\alpha) ) = (g_1', \alpha_1' ) . \]
Again, we observe that $n = m$ by composability. Since, it suffices to prove the result for $n=1$ we will just show this case. Hence:
\begin{align*}
(g', \alpha ' ) \cdot (g_1' \Psi(\alpha')) &=  (g - 1 , \hol(\alpha) h ) \cdot (g'-1, \Psi(\hol(\alpha'))) \\
&= ((g g_1) -1 , \hol(\alpha) h (\Psi \circ) \hol(\alpha_1) h\inv h ) \\
&= ( (g g_1) -1 , \hol( \alpha \alpha_1) h ) \\
&= ( (g g_1) - 1, \hol_h(\alpha \alpha_1) ).
\end{align*}
We leave it to the reader to check that $P_\Psi/ \sim$ inherits a (non-hausdorff) smooth structure from $P_\Psi$.
\end{proof}
\subsection{The Picard group $\G(I,\Hol)$}\label{subsection:caffpic}
Before computing the Picard group of $\G = \G(I,\Hol)$ we need this next lemma:
\begin{lemma}\label{lemma:innerautocylinder}
Suppose $(\Psi,h): (I,\Hol) \to (I,\Hol)$ is an isomorphism. Then $P(\Psi,h)$ is isomorphic to the trivial bibundle if and only if $(\Psi,h)$ is an inner automorphism.
\end{lemma}
\begin{proof}
Suppose $(P,p_0)$ is a $Z$-static pointed bibundle such that $P$ is isomorphic to the identity bibundle. Suppose $(\Psi,h)$ is the holonomy isomorphism of $(P,p_0)$. Then $\overline{P} \isom P_\Psi$ is the identity bibundle and therefore $\Psi = C_\alpha: I \to I$ is an inner automorphism. Let $F:P_\Psi \to \G(I)$ be the trivialization of $P_\Psi$ such that $F(\sigma(x,y)) = (\u(x,y),\alpha)$. Then:
\[
\begin{tikzcd}
\G(I) \arrow[rr, "F"]\arrow[dr, "\Proj"] & & P_\Psi \arrow[dl, "\Phi"] \\
 & \G \isom P \, , &
\end{tikzcd}
\]
commutes and therefore $\Phi(\tau(x,y) \cdot \sigma(x,y)) = \Proj(\u(x,y), h \alpha)$. On the other hand:
\[ \Phi(\tau(x,y) \cdot \sigma(x,y)) = \Phi(\sigma(x,y+1)) = \Proj(\u(x,y+1),\alpha) = \Proj(\u(x,y),\Hol(\alpha)) . \]
Therefore: $\hol(\alpha) = h \alpha$.

Now suppose $(\Psi,h)$ is an inner automorphism and $\Psi = C_\alpha$. Consider $(\G(I),\alpha)$ as a $Z$-static pointed bibundle. Then $(\G,\Proj(\alpha))$ is a $Z$-static pointed $(\G,\G)$-bibundle with holonomy isomorphism $(\Psi,h)$ and therefore by Theorem \ref{thm:holcorrespondence} $P(\Psi,h) \isom \G$.
\end{proof}
Since the inner automorphisms occur as the kernel of a homomorphism, they are a normal subgroup and it makes sense to define $\OutAut(I,\Hol)$ to be the automorphisms modulo inner automorphisms.

The results of the preceding subsections work just as well for $\caff^\rho$ (see Example \ref{example:caff}), the only distinction is the scaling of the appropriate symplectic forms the constant term $\rho$. In this section we will fix some $\rho$ and natural integration $\G \rightrightarrows \caff^\rho$. We can now compute the Picard group of $\G \isom \G(I,\Hol)$. As we did for $\G(I)$ denote the subgroup of $\Pic(\G)$ generated by the modular flow by $\mathbf{Mod}$. In this section, $\rho$ will be the modular period of $\caff^\rho$ (see Example \ref{example:caff}).
\begin{lemma}\label{lemma:cafftwist}
The element $\mbox{\normalfont{Mod}}^t \in \mathbf{Mod}$ for $t=\rho$ is isomorphic to $P(\Hol,e)$.
\end{lemma}
\begin{proof}
It suffices to show that there exists $p_0 \in \mbox{Mod}^\rho$ such that $(\mbox{Mod}^\rho,p_0)$ is a $Z$-static pointed bibundle and $(\Hol,e)$ is its associated holonomy isomorphism. The $(\G(I),\G(I))$-bibundle $(P_\Hol,p_\Hol)$ comes with a projection $\Phi: P_\Hol \to \mbox{Mod}^\rho$. If we take $p_0 = \Phi(p_\Hol)$ then the holonomy data of $(\mbox{Mod}^\rho,p_0)$ is $(\Hol,e)$ so we are done.
\end{proof}
There is an action of $\R$ on $Z\Pic(\G) \isom \OutAut(I,\Hol)$ by conjugation
\[ t \cdot P = (\mbox{Mod}^t) \otimes P \otimes (\mbox{Mod}^{-t}). \]
Using this action, we can compute the Picard group.
\begin{theorem}\label{thm:picofghol}
For any $\G \isom \G(I,\Hol)$ an integration of $\caff$ with connected orbits there is a surjective homomorphism:
\[ \OutAut(I,\Hol) \rtimes \R  \to \Pic(\G). \]
The kernel of this map is subgroup generated by $(P(\Hol,e),-\rho)$.
\end{theorem}
\begin{proof}
Let $\Phi$ be the map given by $(P(\Psi,h),t) \mapsto P(\Psi,h) \otimes \mbox{Mod}^t$. Any bibundle in $\Pic(\G)$ can be written in the form $P(\Psi,h) \otimes [\mbox{Mod}^t]$ since for any $P$ we can find a $t$ such that $P \otimes M^t \in Z\Pic(\G)$ so the map is surjective.. Furthermore, the map is a homomorphism since for any $P_1,P_2 \in Z\Pic \isom \OutAut(I,\Hol)$ and $t_1,t_2 \in \R$:
\begin{align*}
\Phi(P_1,t_1) \otimes \Phi(P_2,t_2) &= P_1 \otimes \mbox{Mod}^{t_1} \otimes P_2 \otimes \mbox{Mod}^{t_2} \\
&= P_1 \otimes \mbox{Mod}^{t_1} \otimes P_2  \otimes\mbox{Mod}^{-t_1} \otimes \mbox{Mod}^{t_1} \otimes \mbox{Mod}^{t_2} \\
&= P_1 \otimes (\mbox{Mod}^{t_1} \otimes P_2  \otimes\mbox{Mod}^{-t_1}) \otimes \mbox{Mod}^{t_1 + t_2} \\
&= \Phi((P_1,t_1) \cdot (P_2,t_2))
\end{align*}
Finally, if $\Phi(P,t) = e$ then $P \isom [\mbox{Mod}^{-t}]$ meaning that $t = n \rho$ for some $n \in \Z$ and $P \isom P(\Hol,e)^{-n}$.
\end{proof}

\section{Picard groups of stable b-symplectic manifolds}\label{section:bsymplecticmanifolds}
Throughout this section it will be useful to index sets and group elements with superscripts. To avoid confusion, when the symbols $i,j, k, l, n$ and $m$ occur in
superscripts, we intend to treat them as indices. I.e an object $X^n$ denotes the element with index $n$ and not `$X$ to the power $n$'. Let $\Cy = \sqcup_{i \in I} C^i$
($I$ finite or countable) be a disjoint union of affine cylinders such that each $C^i \isom \caff^{\rho^i}$. In this section $\G$ will denote a natural integration of $\Cy
$. Our main end is to classify $(\G_2,\G_1)$-bibundles and compute $\Pic(\G)$.

Recall that the numbers $\rho_i$ are called the \emph{modular periods} of $\Cy$. Let $\{V^n \}_{n \in N}$ ($N$ finite or countable) be an enumeration of the open orbits of $\Cy$. To each $C^i$ let $E^i$ denote the saturation of $C^i$ in the foliation induced by $\G$. We also define adjacency maps $+: I \to N$ and $-: I \to N$ such that $V^{+(i)}$ is the positive orbit adjacent to $C^i$ (similarly for $-$). Throughout, $\G^i$ denotes the restriction of $\G$ to $C^i$ and $\G^n$ denotes the restriction of $\G$ to $V^n$.
\subsection{Discrete data of $\G$}
Let $\G^i:= \G|_{C^i}$ be the restrictions of $\G$ to a given cylinder. Since each $\G^i$ is an integration of an affine cylinder we can assume it has the form $\G^i \isom \G(I^i,\Hol^i)$ where $\Hol^i= (\hol^i,(\gamma^i)^\pm)$ and:
\[
\begin{tikzcd}
G^{-(i)} \arrow[d,"(\hol^i)^-" swap] & H^{i} \arrow[d, "\hol^i"]\arrow[l,"(\phi^{i})^-" swap] \arrow[r,"(\phi^{i})^+"] & G^{+(i)}  \arrow[d, "(\hol^i)^+"] \\
G^{-(i)} & H^{i} \arrow[l,"(\phi^{i})^-" swap] \arrow[r,"(\phi^{i})^-"] & G^{+(i)} \, ,
\end{tikzcd}
\]
commutes. If $\G$ is constructed by restricting to transverse cylinders of a stable b-symplectic manifold (as in Section \ref{section:trategyofproof}) then we have following geometric interpretation:
\[ G^n \isom \pi_1(U_n), \qquad H^i \isom \pi_1(L_i), \qquad \hol^\pm = C_{(\gamma^i)^\pm}, \]
\[ \hol^i = (f_i)_*: \pi_1(L_i) \to \pi_1(L_i). \]
Where each $U_n$ is an open orbit, $f_i: L^i \to L^i$ is the holonomy of the mapping torus $Z_i$ and $(\gamma^i)^\pm_i$ are the homotopy classes of the section of $Z_i \to \S^1$ in adjacent orbits.
\begin{definition}
The \emph{orbit graph}, $\mbox{Gr}(\G)$, of $\G$ is defined as follows:
\begin{itemize}
    \item there is one vertex $v^n$ for each open orbit $V^n \subset \Cy$;
    \item there is one edge $e^{i}$ for each cylinder $C^i \subset \Cy$;
    \item each edge $e^i$ is adjacent to a positive and negative vertex $v^{+(i)}$ and $v^{-(i)}$ respectively.
\end{itemize}
\end{definition}
For a stable b-symplectic manifold, the \emph{orbit graph}, $\mbox{Gr}(M)$, of $M$ is defined to be the orbit graph of any Morita equivalent $\G \rightrightarrows \Cy$. We will need to attach more data to the orbit graph in order to construct a complete set of Morita invariants and enable computation of the Picard group of $\Sigma(M)$.
\begin{definition}\label{def:discretedata}
To each vertex, $v_n$, assign the following data:
\begin{itemize}
    \item the \emph{sign} of a vertex $v^n$ is the sign of $V^n$;
    \item the \emph{isotropy} of $v^n$ is the isotropy group $G^n$.
\end{itemize}
To each edge, $e^{i}$ we assign the following data;
\begin{itemize}
    \item the \emph{modular period} of an edge $e^{i}$ is $\rho_i$ (the modular period of $C^i$);
    \item the \emph{isotropy data and holonomy data} of $e^{i}$ is the pair $(I^i,\Hol^i)$ such that $\G^i \isom \G(I^i,\Hol^i)$.
\end{itemize}
We call the orbit graph together with the above data $\mathfrak{Gr}$ the \emph{discrete data} of $\G$ and a \emph{discrete presentation} of $M$ (If $\G$ was obtained by a restriction of $\Sigma(M)$).
\end{definition}
\begin{remark}
The discrete data actually characterizes $\G \rightrightarrows \mathcal{C}$ up to isomorphism. We omit a proof of this for brevity, as it is not needed for our calculation. Also note that a discrete presentation of $M$ is not constructed in a canonical manner as it depends on a choice of transversal $\mathcal{C} \to M$.
\end{remark}
\subsection{Isomorphisms of discrete data}
Suppose $\Cy_1 := \sqcup_{i \in I_1} C_1^i$ and $\Cy_2 := \sqcup_{i \in I_2} C_2^i$ and $\G_1$ and $\G_2$ are natural integrations of each. Enumerate the open orbits of $\G_1$ and $\G_2$ as $V_1^{n}$ for $n \in N_1$ and $V_2^{m}$ for $m \in N_2$. We now need to characterize $(\G_2,\G_1)$-bibundles in terms of discrete data. In order to do this, it will be convenient to fix a choice of base points $v_1^n \in V^n$ and $v_2^m \in V^m$ as well as a choice of groupoid elements $g_n^i$ and $g_m^j$ for each $\pm(i) = n$ or $\pm(j) = m$ such that the source of $g_n^i$ is $v^n$ and the target of $g_n^i$ is $[\pm 1,0] \in C^i$. The purpose of this is to fix an identification of the isotropy groups of each half cylinder in the same orbit and so we can treat each open orbit as a pointed manifold.

Suppose $P$ is a $(\G_2, \G_1)$-bibundle. Clearly $P$ induces an isomorphism of the orbit graphs $F:\mbox{Gr}_1 \to \mbox{Gr}_2$. We can think of $F$ as a map $F: I_1 \to I_2$ (also $N_1 \to N_2$). We say $P$ is \emph{orientation preserving} if $P$ preserves the signs of the vertices and orientation reversing otherwise. Since $P$ preserves the modular vector field $F$ must preserve the modular periods.

Now restrict $P$ to the affine cylinders $C_2^{F(i)}$ and $C_1^i$ to obtain a $(\G_2^{F(i)}, \G_1^{i})$-bibundle $P^i$. Suppose $\{p^i \}_{i \in I_1} \subset P$ is a set of points such that $(P^i,p^i)$ is a $Z$-static pointed bibundle for all $i \in I_1$. Then we call $(P, \{ p_i \})$ a $Z$-static \emph{marked bibundle}. Any such marked bibundle has an associated collection of holonomy isomorphisms $(\Psi_i,h_i): (I_1^i, \Hol_1^i) \to (I_2^i, \Hol_2^i)$. Notice that $(\Psi_i,h_i)$ is orientation preserving/reversing if and only if $F$ is orientation preserving/reversing.

For each pointed bibundle $(P^i,p^i)$ let $\Phi^i:P_{\Psi^i} \to P^i$ be the projection from Lemma \ref{lemma:phi}. Let $\sigma^i$ be the static bisection of $P_{\Psi^i}$ extending $p^i$ and let:
\[ p_\pm^i := (g_{F(n)}^{F(i)})\inv \cdot \sigma(\pm 1,0) \cdot g^i_n \in P \, . \]
If $P^{\pm(i)}$ is the restriction of $P$ to the open orbits $V^{\pm(i)}$ and $V^{F(\pm(i))}$. Then $(P^{\pm(i)},p_\pm^{i})$ are pointed $(\G^{F(\pm (i))},\G^{\pm(i)})$-
bibundles of transitive groupoids. If $\pm (i) = \pm (j)$, the bibundles $P^{\pm(i)}$ and $P^{\pm(j)}$ are equal. There is a unique $g^\pm_{ij} \in G^{F(\pm(i))}$ such
that $g^\pm_{ij} \cdot p_\pm^j = p_\pm^i$. Furthermore by Lemma \ref{lemma:transbimid} $g^\pm_{ij}$ must satisfy:
\begin{equation}\label{eqn:cocyclecondition1}
\psi_i^\pm = C_{g^\pm_{ij}} \circ \psi_j^\pm,
\end{equation}
and for $\pm(i)=\pm(j)=\pm(k)$:
\begin{equation}\label{eqn:cocyclecondition2}
g^\pm_{ij} g^\pm_{jk} = g^\pm_{ik}
\end{equation}
This motivates our definition of an isomorphism of discrete data. Let $\mathfrak{Gr}_1$ and $\mathfrak{Gr}_2$ be the discrete data of $\G_1$ and $\G_2$ respectively.
\begin{definition}
An \emph{isomorphism} $\mathcal{F}:\mathfrak{Gr}_1 \to \mathfrak{Gr}_2$ consists of the following:
\begin{itemize}
    \item an \emph{underlying graph isomorphism} $F: \mbox{Gr}_1 \to \mbox{Gr}_2$;
    \item \emph{holonomy isomorphisms} $(\Psi_i,h_i):(I_1^i,\Hol_1^i) \to (I_2^{F(i)},\Hol_2^{F(i)})$;
    \item \emph{cocycles} $g^\pm_{ij} \in G_2^{F(\pm(i))}$ for all $i,j \in I_1$ such that $\pm(i)=\pm(j)$.
\end{itemize}
The underlying graph isomorphism, holonomy isomorphisms, and cocycles must satisfy the following compatibility conditions:
\begin{enumerate}[(i)]
    \item the holonomy isomorphisms are orientation preserving/reversing if and only if $F$ is orientation preserving/reversing;
    \item if $i \in I_1$ then $\rho_1^{i}=\rho_2^{F(i)}$;
    \item for any $i,j \in I_1$ such that $\pm(i) = \pm(j)$ then (\ref{eqn:cocyclecondition1}) holds;
    \item for any $i,j,k \in I_1$ such that $\pm(i) = \pm(j) = \pm(k)$ then (\ref{eqn:cocyclecondition2}) holds.
\end{enumerate}
\end{definition}
\begin{example}[Inner automorphisms]
Suppose $\Gr$ is discrete data. Define $\mathcal{F}: \Gr \to \Gr$ to be an automorphism of $\Gr$ such that:
\begin{itemize}
\item $F$ is the identity;
\item $(\Psi_i,h_i) = (\C_{\alpha_i},h_i)$ are inner automorphisms;
\item and $g^\pm_{ij} = (\phi_2^i)^\pm(\alpha_i) (\phi_2^j)^\pm(\alpha_j)\inv$ for all $\pm(i) = \pm(j)$.
\end{itemize}
Such an $\mathcal{F}$ is called an \emph{inner automorphism} of $\Gr$.
\end{example}
\begin{example}[Bibundles]
Modulo the choices of base-points made at the beginning of this section, we can canonically construct an isomorphism of discrete data $\mathcal{F}$ from any $Z$-static marked bibundle $(P,\{p^i \})$. In this case we say that $\mathcal{F}$ is the \emph{discrete isomorphism} associated to $(P,\{p^i\})$.
\end{example}
A \emph{strong isomorphism} $\varphi: (P,\{p^i\}) \to (Q, \{q^i \})$ is an isomorphism of bibundles which preserves the markings.
\begin{proposition}\label{prop:bsympcoresp}
There is a 1-1 correspondence between orientation preserving (reversing) isomorphisms of discrete data $\Gr_1 \to \Gr_2$ and orientation preserving (reversing) $Z$-static marked bibundles modulo strong isomorphisms.
\end{proposition}
\begin{proof}
Suppose $(P,\{ p^i \})$ and $(Q, \{ q^i \})$ are $Z$-static marked $(\G_2,\G_1)$-bibundles with identical discrete isomorphisms $\mathcal{F}$. By Theorem \ref{thm:holcorrespondence} there is a strong identification $\varphi^i :(P^i,p^i) \to (Q,q^i)$ of the restrictions of $P$ and $Q$ to the cylinders $C_1^i$ and $C_2^{F(i)}$. The isomorphisms $\varphi^i$ induce strong isomorphisms:
\[ \varphi^{\pm(i)}: (P^{\pm(i)},p^i_\pm ) \to (Q^{\pm(i)},q_\pm^i ), \]
of pointed bibundles. Since $P$ and $Q$ have the same cocycle, these isomorphisms satisfy $\varphi^{\pm(i)} = \varphi^{\pm(j)}$ whenever $\pm(i) = \pm(j)$. Therefore the $\varphi^i$ extend to a unique strong isomorphism $\varphi: (P, \{p^i \}) \to (Q, \{q^i \})$.

On the other hand, suppose $\mathcal{F}$ is an isomorphism of discrete data. Let $(P^i, p^i)$ be $P(\Psi_i,h_i)$ with the standard point. As we saw in the beginning of this section, each $(P^i,p^i)$ induces a pointed bibundle over the adjacent orbits. We denote these pointed $(\G^{F(\pm(i))},\G^{\pm(i)})$-bibundles $(P_\pm^i,p_\pm^i)$. To glue these bibundles together, Lemma \ref{lemma:cocycle} says that it suffices to construct bibundle isomorphisms $\varphi^\pm_{ij}: P_\pm^j \to P_\pm^i$ for each $\pm(i) = \pm(j)$ such that:
\begin{equation}\label{eqn:cocycleforbim}
\varphi^\pm_{ij} \circ \varphi^\pm_{jk} = \varphi^\pm_{ik} \quad \forall \pm(i)=\pm(j)=\pm(k).
\end{equation}
We apply Lemma \ref{lemma:transbimid} and define $\varphi^\pm_{ij}$ to be the pointed bibundle isomorphisms associated to $g_{ij}^\pm$. Then equation (\ref{eqn:compositionisom}) implies that (\ref{eqn:cocycleforbim}) holds if and only if (\ref{eqn:cocyclecondition2}) holds and so the proposition follows.
\end{proof}
Under this 1-1 correspondence, given isomorphisms of discrete data $\mathcal{F}: \G_1 \to \G_2$ and $\mathcal{F}': \G_2 \to \G_3$ the composition corresponds to the following isomorphism:
\begin{itemize}
\item the underlying graph map is the composition $F' \circ F$;
\item the holonomy isomorphisms are the compositions $(\Psi'_{F(i)},h'_{F(i)} )\circ (\Psi_i,h_i)$;
\item and the cocycles are $\{(g')^\pm_{F(ij)} (\psi')_{F(i)}^\pm(g^\pm_{ij}) \}$ for $\pm (i) = \pm(j)$.
\end{itemize}
The last one is a consequence of (\ref{eqn:tensorproductisom}). Given any $(\G_1,\G_2)$-bibundle $P$ one can make a suitable choice of base-points on $\Cy_1$, $\Cy_2$ and $P$ which make $P$ a marked bibundle. Therefore, Theorem \ref{maintheorem1} follows:
\begin{theorem}\label{maintheorem1}
Suppose $M_1$ and $M_2$ are stable b-symplectic manifolds and $\Gr_1$ and $\Gr_2$ are discrete presentations of each, respectively. Then $M_1$ and $M_2$ are Morita equivalent if and only if there exists an isomorphism $\Gr_1 \to \Gr_2$.
\end{theorem}
\subsection{The Picard Group of a stable b-symplectic manifold}
In this section $\G$ is a natural integration of $\Cy$. We continue to use the notation we have established thus far and fix a choice of base points $v^n \in V^n$ and arrows $g_n^i$ as before. We need one last lemma about $Z$-static bibundles.
\begin{lemma}\label{lemma:lastinner}
An automorphism $\mathcal{F}: \Gr \to \Gr$ of the discrete presentation of $\G \rightrightarrows \Cy$ gives rise to the trivial bibundle if and only if $\mathcal{F}$ is an inner automorphism.
\end{lemma}
\begin{proof}
Suppose $(P, \{ p^i \})$ is a marked $(\G,\G)$-bibundle and let $\varphi:P \to \G$ be an isomorphism as bibundles. Let $\mathcal{F}$ be the isomorphism of discrete data associated to $(P, \{ p^i \} )$. By Lemma \ref{lemma:staticbibundleclass} the holonomy isomorphisms are all inner automorphisms $(\C_{\alpha_i}, h_i)$. As a bibundle, $\G$ has a canonical marking induced by the identity section. Hence if we restrict $\varphi$ to open orbits, denoted $\varphi^n$ we get an isomorphisms of marked bibundles:
\[ \varphi^n:(P^{\pm(i)},p_\pm^i ) \to (\G^{\pm(i)}, \u(v^{\pm(i)})). \]
By Lemma \ref{lemma:transbimid}, for each $i$ in $I$ the isomorphism $\varphi$ corresponds to $(\phi^i)^\pm(\alpha_i)$. Furthermore, since:
\[
\begin{tikzcd}
 & (\G^n, \u(v^{\pm(i)})) & \\
(P^{\pm(i)}, p^{\pm(i)} ) \arrow[rr, "\Id"] \arrow[ur, "\phi^n"] & & (P^{\pm(j)}, p^{\pm(j)}) \, , \arrow[lu, "\phi^n", swap]
\end{tikzcd}
\]
commutes, we conclude that $g_{ij}^\pm \cdot (\phi^j)^\pm(\alpha_j) = (\phi^i)^\pm(\alpha_i)$. Therefore $\mathcal{F}$ is an inner automorphism.

Now assume $\mathcal{F}$ is an inner automorphism. Since the holonomy isomorphisms $(\C_{\alpha_i},h_i )$ are all inner, there are unique identifications of $\varphi^i:\G^i \to P^i$ such that $\alpha_i \cdot \varphi(\u(0,0)) = p^i$ for all $i \in I$. Let $\{ \overline{p}^i \}$ be a new marking of $P$ obtained from this identification. With this new marking the cocycles of $(P, \{ \overline{p}^i \})$ take the form:
\[ (\phi^i)^\pm (\alpha_i)\inv g^\pm_{ij} (\phi^j)^\pm(\alpha_j) = (\phi^i)^{\pm}(\alpha_i)\inv \cdot (\phi^i)^{\pm}(\alpha_i) \cdot (\phi^j)^{\pm}(\alpha_j)\inv \cdot (\phi^j)^\pm(\alpha_j) = e     \]
Since $(P, \{ \overline{p}^i \})$ and $\G$ have the same discrete data, $P \isom \G$.
\end{proof}
The inner automorphisms of $\Gr$ form a normal subgroup of $\Aut(\Gr)$. Therefore it makes sense to define the outer automorphisms $\OutAut(\Gr)$ to be the inner automorphisms modulo outer automorphisms. We can summarize the preceding work thus far with the equation:
\[ \OutAut(\Gr) \isom Z \Pic(\G), \]
where $Z \Pic(\G)$ are the $(\G, \G)$-bibundles such that the restriction to each cylinder is a $Z$-static bibundle. To compute the full Picard group, we also need to identify the subgroup of bibundles which correspond to modular flows.

Let $X_i \in \mathcal{X}(\Cy)$ be the vector fields on $\Cy$ such that $X_i|_{C_i}$ is the modular vector field of $C_i$ and $X_i = 0$ outside of $C_i$. Since $\G$ is natural, the flow of each such vector field produces a family of groupoid isomorphisms $\Phi^t_i: \G \to \G$. Let
\[ Q_i(t) := P_{\Phi^i(t)} \]
be the symplectic bibundles induced by these flows. Since the isomorphisms $\Phi^t_i$ commute, the bibundles $Q_i(t)$ must commute with each other. In other words, we have a homomorphism:
\[ (\R^N,+) \to \Pic(M) \]
where $\R^N$ denotes the vector space of functions $N \to \R$ (recall that $N$ may be an infinite index). We call the image of this map $\Mod(M) \subset \Pic(M)$. Topologically it is a connected Lie group integrating the abelian Lie algebra $\R^N$.
\begin{lemma}
Let $\mathcal{F}_i: \Gr \to \Gr$ be the automorphism of $\Gr$ given by the following data:
\begin{itemize}
    \item The underlying graph automorphism is trivial;
    \item The holonomy isomorphisms $(\Psi_k,h_k)$ are trivial for $k \neq i$ and $(\Psi_i,h_i) = (\Hol_i,e)$;
    \item The cocycles $g^\pm_{kj} = e$ for $k,j \neq i$ and $g^\pm_{ij} = \gamma_i^\pm$ for $j \neq i$.
\end{itemize}
We will call $\mathcal{F}_i$ the \emph{twisting automorphism} about $C_i$. Then the bibundle associated to the automorphism $\mathcal{F}_i$ is $Q_i(\rho_i)$.
\end{lemma}
\begin{proof}
Observe that the twisting automorphism behaves in the following manner. Suppose $g \in \G$ is in an open orbit. Then for $g \in \G|_{\C - C_i}$ we have that $\Phi^{\rho_i}_i(g) = g$. If $g \in \G$ such that the target of $g$ lies in $C_i$ and the source of $g$ lies in $\C - \C_i$, then $\Phi^{\rho_i}_i (g) = \gamma_i^\pm g$. For $g \in \G^i$ we have that $\Phi^{\rho_i}_i$ is conjugation by $\gamma_i^\pm$.

The bibundle $Q_i(\rho_i)$ is diffeomorphic to $\G_2$ where the right action of $\G_1$ satisfies $g_2 \cdot g_1 = g_2 \Phi^{\rho_i}_i$. Since $Q_i(\rho_i)$ has this form, it inherits a natural marking from the identity section $\u: \C \to \G$. From this it is easy to deduce that there is natural strong isomorphism of $P|_{\C - \C_i}$ and $\G_{\C - \C_i}$ (the identity bibundle). We denote this marking by the tuple $\{ q^i \}$. Furthermore, given any $g \in \G$ whose target is the base point of $q^i$ and whose source is the base point of $q^j$ for $j \neq i$. We notice that:
\[ g \cdot q^j \cdot g\inv =  g g\inv (\gamma_i^\pm)\inv \cdot q^i = (\gamma_i^\pm)\inv \cdot q^i \, .  \]
Therefore the cocycle part of this marked bibundle must as above.
\end{proof}

We see that the additive group $(\R^n,+)$ acts on $Z\Pic(M) \isom \OutAut(\Gr)$ by conjugation:
\[ (t_1,...,t_N) = Q_1^{t_i} \cdot ... \cdot Q_N^{t_N} P  Q_N^{-t_N} \cdot ... \cdot  Q_1^{-t_1}.  \]
This action measures the failure of elements of $Z \Pic(\G)$ and $\Mod(\G)$ to commute. Since any element $P \in \Pic(\G)$ can be written as a product of elements in these groups, we have proved Theorem \ref{maintheorem2}.
\begin{theorem}\label{maintheorem2}
Suppose $M$ is a stable b-symplectic manifold and $\mathfrak{Gr}$ is a discrete presentation of $M$. There is a surjective group homomorphism:
\[ \OutAut(\mathfrak{Gr}) \ltimes R^N  \to \Pic(M).  \]
The kernel of this group homomorphism is generated by elements of the form $([\mathcal{F}_i]\inv,Q_i^{\rho_i})$.
\end{theorem}
This concludes the main portion of the chapter. The next section will examine a few examples.

\section{Applications and examples}\label{section:examples4}
\subsection{Compact surfaces}
We mentioned earlier Radko and other authors in studying b-symplectic structures on surfaces. In particular the Picard group was computed by Radko and Shylakhtenko, \cite{Radko}. Their work contained two important results. The first was that any $(\Sigma(M),\Sigma(M))$-bibundle admitted a lagrangian bisection (for $M$ a 2-dim compact, oriented, b-symplectic manifold). The consequence of this is that the Picard group of $M$ can be interpreted as the group of outer Poisson automorphisms of $M$ Secondly, they were able to provide combinatorial data which would assist in the explicit calculation of $M$.

In this section, we will see what Theorem \ref{maintheorem1} and Theorem \ref{maintheorem2} mean in dimension 2, and compare our results to existing work.

Suppose $M$ is a compact surface with a b-symplectic structure. Then $M$ is automatically stable. Our procedure of finding transverse cylinders to $M$ corresponds to restricting $\Sigma(M)$ to collars around the singular locus $Z$, which is just a finite union of disjoint circles.

When $M$ is a compact surface, we automatically get many simplifications of the data encoded by $\Gr(M)$.
\begin{enumerate}
\item Since the groups $H_i$ are all trivial, the isotropy data just consists of assigning to each $v_i$ the fundamental group of its associated open leaf. Given the orbit graph, we can instead simply state the genus of each open leaf.
\item The elements $\gamma^\pm_i$ are the homotopy classes of loops around the collars of each open leaf. Therefore, they are completely determined by the orbit graph, together with the fundamental group (or genus) of each open leaf.
\item The modular periods are specified as before.
\end{enumerate}
Therefore Theorem \ref{maintheorem1} reduces to the classification of compact b-symplectic surfaces from \cite{BDirac}.

Let us see what automorphisms of $\Gr(M)$ correspond to. First let us restrict to static automorphisms of $\Gr(M)$ (i.e. those which fix the underlying graph). Such automorphisms produce static bibundles which we denote $S\Pic(M)$.

Recall that to specify a morphism we must give automorphisms:
\[ (\Psi_i,h) = (I_i,\Hol_i) \to (I_i,\Hol_i)\, , \]
and cocycles $\{ g^\pm_{ij} \}$. The groups $H_i$ are all trivial and the various diagrams commute trivially, so an automorphism $(\Psi_i,h)$ is the same as a pair of maps $\psi_i^\pm$ such that $\psi_i^\pm(\gamma_i^\pm) = \gamma_i^\pm$.

The cocycles must satisfy $C_{g^\pm_{ij}} \circ \psi_j^\pm = \psi_i^\pm$. The morphism $\Gr(M) \to \Gr(M)$ will be trivial.

Let $V_n$ be any open orbit. Then the cocycles $g^\pm_{ij}$ for $i$ and $j$ adjacent to $V_n$ together with the maps $\phi^\pm_i: G_n \to \G_n$ and $\phi^\pm_j: G_n \to G_n$ are equivalent (modulo inner automorphisms) to specifying an element of the mapping class group of the associated surface. This recovers the combinatorial description of the Picard group from \cite{Radko}.
\subsection{Examples due to Cavalcanti}
In \cite{Caval}, Cavalcanti demonstrated several ways of constructing b-symplectic manifolds. In particular, he showed that given a symplectic manifold $M$ with a compact symplectic submanifold $L$ of codimension 2 with trivial normal bundle admits a log symplectic structure such that $Z \isom \sqcup_{i=1}^N \S^1 \times L$ (see thm 5.1 in \cite{Caval}).

We can immediately see that any such example is a \emph{stable} b-symplectic manifold. We will restate the construction here and then compute the discrete presentation and Picard group of a simple example.

Since $L$ has a trivial normal bundle, the symplectic neighborhood theorem says that there is a tubular neighborhood, $U$ of $L$ such that $U \isom D^2 \times L$ as a symplectic manifold. Then for any embedding of closed curves $\Gamma:\sqcup_{i=1}^N \S^1 \into D^2$ we can rescale the Poisson bivector by a function which vanishes on $\Gamma$ linearly and which is constant near the boundary of $D^2$.
We can easily extend this function to all of $M$ and rescale the bivector on $M$.
The result will be a b-symplectic structure which agrees with the symplectic structure on $M$ outside of $U$, but which $Z$ has $N$ connected components, all of them contained in $U$.

The orbit graph of $M$ will be determined by the topological arrangement of the closed curves in $D^2$. There will be one orbit for each connected component of $D^2 \setminus Z$. We will now consider the concrete example mentioned in \cite{Caval}.
\begin{example}
Consider $M = \C P^2 \# \bar{\C P^2}$ (i.e. the blowup of $\C P^2$ at a point). $M$ has the structure of a locally trivial $\S^2$ bundle over $\S^2$. Let $p$ be a point in $\S^2$. Then
\[ M|_{{\S^2 \setminus p}} \isom \S^2 \times D^2. \]
We can now choose curves in $D^2$ to define the singular locus. Suppose there is only one. Then the induced b-symplectic structure on $M$ has two open leaves and every symplectic leaf of $M$ is simply connected.

The orbit graph of $M$ consists of two vertices connected by a single edge. The isotropy data and holonomy are all trivial.
When constructing the b-symplectic structure above, we can arrange for whichever modular period $\rho$ we want.
We see immediately from Theorem~\ref{maintheorem1} that $M$ is Morita equivalent to the Radko sphere, $\S^2$ with modular period $\rho$.

In fact, for any topological arrangement of curves in $D^2$, we can see that the induced b-symplectic structure on $M$ will be Morita equivalent to the sphere $\S^2$ with the same arrangement of singular curves (thinking of $D^2$ as a punctured sphere).

The Picard group can therefore be computed by existing methods for surfaces. The case of a single singular curve provides an easy calculation. The holonomy, and isotropy data is all trivial and therefore the Picard group of $M$ is $\Z_2 \times \S^1$. The factor of $\Z_2$ corresponds to the orientation reversing graph swap while the $\S^1$ factor corresponds to the modular flows.
\end{example}
\subsection{Cosymplectic boundaries}
The following class of examples is mentioned in the work of Frejlich et al (see \cite{Frei}). Suppose $N$ is a connected symplectic manifold with a cosymplectic boundary $\partial N$. We say that the boundary of $N$ is \emph(stable) if $Z:= \partial M$ is a disjoint union of mapping tori. We can construct a b-symplectic manifold $M$ by gluing two copies of $N$ along their boundary.
The orbit graph of $M$ will have two vertices and one edge for each connected component of $Z$.

The isotropy and holonomy data will be determined by the topology of the embeddings $Z = \partial N \to N$. Note that $\Gr(M)$ comes with a cannonical orientation reversing automorphism and $Z\Pic(M) \isom \Z_2 \times Z\Pic^+(M)$ where $Z\Pic^+(M)$ are the orientation preserving automorphisms of $\Gr(M)$.

Lefschetz fibrations over the two dimensional disk provide a well studied class of examples. Let $M$ be a 4 dimensional manifold and $p: M \to \mathbb{D}^2$ be a Lefschetz fibration such that $p\inv(x) \isom S_g$ is a compact surface of genus $g$ for any regular value $p$. By a theorem of Gompf \cite{Gompf} we can equip $M$ with a symplectic structure such that the boundary of $M$ is a cosymplectic mapping torus $M_f$.

In a paper by Eliashberg~\cite{Elia}, it was proved that any symplectic mapping torus could be realized as the boundary of a Lefschetz fibration over $\mathbb{D}^2$. The applications of these results to b-symplectic geometry were pointed out in \cite{Frei}. The topology of Lefschetz fibrations has been a subject of keen interest, and is determined by the monodromy around singular fibers. By computing the monodromy of these fibrations, one can apply Theorem \ref{maintheorem2} to compute the Picard group of the assocaited b-manifolds. We will conclude our comments with simple example.
\begin{example}
Let $T^2$ be two dimensional torus and $f: T^2 \to T^2$ be a single (right handed) Dehn twist (i.e $f(\theta_1,\theta_2) = (\theta_1 + \theta_2,\theta_2)$. There is a unique Lefschetz fibration $p: M' \to \mathbb{D}^2$ with monodromy $f$ around the the singular point $0 \in \mathbb{D}^2$.

As per our previous comments we can glue two copies of $M'$ along the boundaries $\partial M$ to form a b-symplectic manifold $M$ with a connected singular locus of the form $Z = M_f$.

By studying the topology of $M$ we can conclude that the isotropy data takes the form:
\[
\begin{tikzcd}
\Z \arrow[d, "\mbox{Id}"] & \Z \oplus \Z \arrow[l, "\pr_2", swap] \arrow[r, "\pr_2"] \arrow[d, "\hol"] & \Z \arrow[d, "\mbox{Id}"] \\
\Z  & \Z \oplus \Z \arrow[l, "\pr_2", swap] \arrow[r, "\pr_2"] & \Z
\end{tikzcd}
\]
Where the map $\hol$ is given by the matrix $(\begin{smallmatrix} 1 & 1 \\ 0 & 1 \end{smallmatrix})$. Notice that the $+$ and $-$ maps are not injective and so $\Sigma(M)$ is not Hausdorff.

An orientation preserving automorphism $(\Psi,h)$ of this data turns out to be a choice of matrix $(\begin{smallmatrix} 1 & a \\ 0 & 1 \end{smallmatrix})$ and an element $h= 0 \oplus z \in \Z \oplus \Z$. Since every group here is abelian, the inner automorphisms are trivial. The orientation reversing automorphisms can always be written as the obvious swapping map and an orientation preserving automorphism. Therefore, $Z \Pic(M) \isom \Z \times \Z \times \Z_2$.

The twisting automorphism from the modular flow corresponds to $((\begin{smallmatrix} 1 & 1 \\ 0 & 1 \end{smallmatrix}), 0 \oplus 0)$. Since this element has infinite order in $Z \Pic$ we can conclude that $\Mod(M) \isom \R$. The action of $\R$ on $\Z \Pic$ is trivial so we have a surjective group homomorphism,
\[ \R \times (\Z \times \Z \times \Z_2) \to \Pic(M). \]
The kernel of this map are elements of the form $(n,(-n,m,k))$ and so
\[ \Pic(M) \isom \R \times \Z \times \Z_2. \]
\end{example}

\begin{appendices}
\chapter{2-Category theory}\label{chap:app2cats}
There are many possible choices to make and competing definitions when it comes to higher categorical structures.
We have included this appendix in the hope that it will clarify the precise definitions we are using.
Loosely speaking, the philosophy is that when one invokes the existence of higher morphisms it is necessary to make specific and coherent choices. For instance, the data of a 2-commutative square must include the 2-morphism witnessing commutativity. More generally, a 2-commutative diagram must include the relevant 2-morphisms and the supplied 2-morphisms must satisfy the `obvious' identities.
\section{Bicategories}
We refer the interested reader to Leinster~\cite{Leinster} for a good discussion of the topic of bicategories. Our definition is equivalent to his. However, since there are differences in notation, it is worthwhile to state a definition here.
\begin{definition}\label{defn:strict2category}
  A \emph{bicategory} $\C$ is the following data:
  \begin{itemize}
    \item a collection of objects $\C_0$;
    \item for each pair of objects, $x,y \in \C_0$, a (possibly empty) category $\C(x,y) = \C_2(x,y) \grpd \C_1(x,y)$;
    \item for any three objects $x,y,z \in \C_0$, a functor (called \emph{horizontal composition}):
    \[ \circ \colon \C(y,z) \times \C(x,y) \to \C(x,z) \quad g \circ h :=\circ(g, h) \]
    \[ (f,g) \mapsto f \circ g \]
    \item For any x $\in \C_0$ a functor:
    \[ \{ x \} \to \C(x,x)  \]
    \[ x \mapsto \Id_x \]
    Here $\{ x \}$ is thought of as a trivial category with one morphism.
  \end{itemize}
    The elements of $\C_2 := \sqcup_{x,y \in \C_0} \C_2(x,y)$ are called the \emph{2-morphisms} of $\C$. The elements of $\C_1 := \sqcup_{x,y \in \C_0} C_1(x,y)$ are called the \emph{1-morphisms} of $\C$. The composition operation on $\C(x,y)$ is called \emph{vertical composition} and will be denoted by $*$.
    We do not assume that either of $\C_2$ or $\C_1$ together with $\circ$ constitutes a category over $\C_0$.
    We will use the notation $\C^{(n)} = \C_2^{(n)} \grpd \C_1^{(n)}$ to denote the \emph{category} of $n$ horizontally composable morphisms.
    Since we do not require that $\circ$ satisfies the axioms of a category, we instead require that $\C$ comes with the following data:
    \begin{itemize}
    \item A natural isomorphism of functors
    \[ \mathbb{A} \colon A_L \to A_R \]
    where $A_L$ and $A_R$ are the parallel functors $\C^{(3)} \to \C_2$ defined below:
    \[ A_L(\alpha,\beta,\gamma) := (\alpha \circ \beta) \circ \gamma \qquad A_R (\alpha, \beta, \gamma) := \alpha \circ (\beta \circ \gamma) \]
    Such a natural transformation is determined by a function $\mathbb{A}$ which associates to each $(g,h,k) \in \C_1^{(3)}$, a 2-morphism:
    \[ \mathbb{A}(g,h,k) \colon (g \circ h) \circ k \to g \circ (h \circ k )  \]
    \item A pair of natural isomorphisms of functors
    \[ \mathbb{U}_L \colon U_L \to \Id_{C_2} \qquad \mathbb{U}_R \colon U_R \to \Id_{\C_2} \]
    where $U_L$ and $U_R$ are parallel functors (with respect to vertical composition) $\C \to \C$ such that for any $\alpha \in \C_2(x,y)$:
    \[ U_L(\alpha) = 1_y \circ \alpha \qquad U_R(\alpha) = \alpha \circ 1_x \]
    Such natural transformations are determined by functions $\mathbb{U}_L$ and $\mathbb{U}_R$ which associate to
    We also For every 1-morphism $f\colon x \to y$ in $\C_1$, two choices of 2-morphisms:
    \[ \mathbb{U}_L(f) \colon \Id_y \circ f \to f \qquad \mathbb{U}_R(f) \colon f \circ \Id_x \to f  .\]
  \end{itemize}
  This data must satisfy the following \emph{coherence} axioms which are encoded as (vertically) commutative diagrams.
    \begin{enumerate}[(BC1)]
    \item for all 1-morphisms $f,g,h,k$ which are horizontally composable, the following diagram commmutes:
    \[\begin{tikzcd}[column sep = tiny]
                                     &                               & ((f \circ g) \circ h) \circ k
                                                                        \arrow[dll, "{\mathbb{A}(f,g,h) \circ 1_k}" swap]
                                                                        \arrow[drr, "{\mathbb{A}(f \circ g, h,k)}"]& & \\
      ( f \circ (g \circ h)) \circ k
      \arrow[dr, "{\mathbb{A}(f,g \circ h, k)}"] &                               &                               & &(f \circ g) \circ (h \circ k) \arrow[dl, "{\mathbb{A}(f,g,h \circ k)}" swap]  \\
                                     & f \circ ((g \circ h) \circ k) \arrow[rr, "{\Id_f \circ \mathbb{A}(g,h,k)}"] &                               & f \circ (g \circ (h \circ k)) &
    \end{tikzcd}\]
    \item for all 1-morphisms $f$ and $g$ which are horizontally composable:
    \[
    \begin{tikzcd}
      (f \circ \Id_x) \circ g \arrow[dr, "\mathbb{U}_R(f) \circ 1_g" swap] \arrow[rr, "{\mathbb{A}(f,\Id_x,g)}"] &  & f \circ (\Id_x \circ g) \arrow[dl, "{1_f \circ \mathbb{U}_L(g)}"] \\
      & f \circ g &
    \end{tikzcd}
    \]
  \end{enumerate}
      A bicategory is called \emph{strict} if $\mathbb{A}$, $\mathbb{U}_L$, and $\mathbb{U}_R$ take values in vertical identities.
\end{definition}
\begin{example}[Category of Categories]
  Suppose $\C$ is a category of categories. That is, the objects of $\C$ are categories and the morphisms of $\C$ are functors.
  Then we can make $\C$ into a bicategory by taking our 2-morphisms to be natural transformations between functors.

  Recall that given functors $F_i\colon \mathcal{A} \to \mathcal{B}$, a natural transformation $\alpha\colon F_1 \to F_2$ is a map $\alpha \colon \mathcal{A}_0 \to \mathcal{B}_1$ such that (\ref{eqn:naturaltransformation}) commutes.
  \begin{equation}\label{eqn:naturaltransformation}
  \begin{tikzcd}
    F_1(x) \arrow[r, "F_1(f)"] \arrow[d, "\alpha(x)"] & F_1(y) \arrow[d, "\alpha(y)"] \\
    F_2(x) \arrow[r, "F_2(f)"]                        & F_2(y)
  \end{tikzcd}
  \end{equation}
  Given $F_i\colon \mathcal{X} \to \mathcal{Y}$ and 2-morphisms $\alpha_1\colon F_1 \to F_2$ and $\alpha_2\colon F_2 \to F_3$ we can take $(\alpha_2 * \alpha_1) (x) := \alpha_2(x) \alpha_1(x)$ for any $x \in \mathcal{A}_0$.

  For horizontal composition, suppose we are given functors $F_i\colon: \X \to \Y$, $G_i\colon \Y \to \mathcal{Z}$ and 2-morphisms $\beta\colon G_1 \to G_2$, $\alpha\colon F_1 \to F_2$. Then
  \[ (\beta \circ \alpha) (x):= (\beta(F_2(x)) \circ (G_1(\alpha(x)) \in \mathcal{Z}_1 \quad \forall x \in \X_1 \, . \]

  A sub-example to this case is the category of stacks over a given site.
  In this case the natural transformations can be shown to be natural isomorphisms. Hence, all 2-morphisms are invertible.
  \end{example}
  \begin{example}[Bibundles]
    Take the objects of $\C$ to be Lie groupoids, the 1-morphisms of $\C$ to be Lie groupoid bibundles and the 2-morphisms to be bibundle isomorphisms.
    Then the resulting object satisfies the axioms of a bicategory. The associativity map is obtained from the cannonical isomorphism:
    \[ P_1 \otimes (P_2 \otimes P_3) \to (P_1 \otimes P_2) \otimes P_3 \qquad (p_1 \otimes (p_2 \otimes p_3) \mapsto (p_1 \otimes p_2) \otimes p_3  \, . \]
  \end{example}
  \begin{definition}\label{defn:weakisomorphism}
    Let $\C$ be a bicategory. Let $f\colon x \to y$ and $g \colon y \to x$ be 1-morphisms. We say that $g$ is a \emph{weak inverse} of $f$ if there exist 2-isomorphisms $\alpha\colon f \circ g \to 1_y$ and $\beta\colon g \circ f \to 1_x$. If a 1-morphism admits a weak inverse then we call it a \emph{weak isomorphism}. Two objects in $\C$ are \emph{weakly isomorphic} if there exists a weak isomorphism between them.
  \end{definition}
  \begin{example}
    Suppose $C$ is the bicategory of categories. Then a functor is a weak isomorphism if and only if it is an equivalence of categories.
  \end{example}
  \begin{definition}\label{defn:21category}
    A \emph{(2,1)-category} is a bicategory such that all 2-morphisms are vertical isomorphisms. Every (2,1)-category $\C$ has a 1-category $\bar{\C}$ associated to it which we call the \emph{truncation of $\C$}. The objects of $\bar{\C}$ are the same as $\C$ while the morphisms of $\bar{\C}$ are the 2-isomorphism classes of the 1-morphisms of $\C$.
  \end{definition}
  \section{Pseudofunctors}
  \begin{definition}\label{defn:pseudofunctor}
    Given bicategories $\C$ and $\D$, a \emph{pseudofunctor} consists of a function $F_0 \colon \C_{0} \to \D_{0}$, and for each pair of objects $x$ and $y$ in $\C_0$ a functor $F \colon \Hom(x,y) \to \Hom(F_0(x),F_0(y))$. Furthermore, a pseudo functor $F$ must be equipped with the following data:
    \begin{itemize}
      \item a natural isomorphism $\mathbb{F} \colon F( - \circ -) \to F( -) \circ F(-)$.
    Such a natural transformation is witnessed by a function $\mathbb{F}$ which assigns to each composable pair $(f,g) \in \C^{(2)}_1$ a 2-isomorphism:
    \[ \mathbb{F}_1(f,g) \colon F(f \circ g) \to F(f) \circ F(g) \]
    \item for every object $x \in \C_0$, a 2-isomorphism $\mathbb{F}_2(x) \colon F(\Id_x) \to \Id_{F(x)}$.
  \end{itemize}
    This data must be \emph{coherent}, which is encoded via vertical diagrams of 2-morphisms. We will not write the coherence conditions here but we instead refer the reader to Leinster~\cite{Leinster}.

    A pseudofunctor $F \colon \C \to \D$ is called an \emph{equivalence of 2-categories} if it is an equivalence of categories for each $\Hom(x,y)$ and is \emph{biessentially surjective}. That is, every object in $\D$ is weakly equivalent to an object in the image of $F$.
  \end{definition}
  \begin{remark}
    Equivalences of categories can also be defined in terms of pseudonatural transformations of pseudo functors. The precise statement is that a pseudofunctor $F$ is an equivalence if and only if there exists a pseudofunctor $G$ such that $F \circ G$ and $G \circ F$ are pseudonaturally equivalent to the identities. See Leinster~\cite{Leinster} for more detail.
  \end{remark}
  \begin{example}
    Suppose $\C$ is a weak (2,1) category with invertible 2-morphisms and let $\bar\C$ be its truncation. We can think of $\bar\C$ as a strict bicategory where all 2-morphisms are the identity. Then there is a natural pseudofunctor $\Pi: \C \to \bar\C$ which is obtained by passing to equivalence classes in the truncation.
  \end{example}
  \section{Fiber products in bicategories}
  \begin{definition}\label{defn:2square}
    Let $\C$ be a bicategory. A \emph{2-commutative square} consists of a square of 1-morphisms:
    \[
    \begin{tikzcd}
      x \arrow[r, "f"] \arrow[d, "g"] & y \arrow[d, "h"] \\
      z \arrow[r, "k"] & w
    \end{tikzcd}
    \]
    together with a 2-isomorphism $\alpha \colon h \circ g \to k \circ f$.
  \end{definition}
  \begin{definition}\label{defn:2pullback}
    Suppose $\C$ is a bicategory. A 2-commutative square is a diagram:
    \[
    \begin{tikzcd}
      w \arrow[r, "\til g"] \arrow[d, "\til f"]& y \arrow[d, "g"] \\
      x \arrow[r, "f"] & z
    \end{tikzcd}
    \]
    together with a 2-isomorphism $\alpha \colon f \circ \til g \to g \circ \til f$ is \emph{a pullback square} if given any other 2 commutative square:
    \[
    \begin{tikzcd}
      w' \arrow[r, "\til f'"] \arrow[d, "\til g'"]& x \arrow[d, "g"] \\
      y \arrow[r, "f"] & z
    \end{tikzcd}
    \]
    with associated 2-isomorphism $\alpha' \colon f \circ \til g ' \to g \circ \til f '$, there exists a morphism $h\colon w' \to w$, and 2-isomorphisms $\beta_1\colon \til f' \to \til f \circ h$ and $\beta_2 \colon \til g ' \to \til g \circ h$ such that the following square of 2-morphisms commutes:
    \begin{equation}\label{eqn:2pullback}
      \begin{tikzcd}
      f \circ \til g' \arrow[r, "{\alpha'}"] \arrow[d, "{1_f \circ \beta_2}"] & g \circ \til f' \arrow[d, "{1_g \circ \beta_1}"] \\
      f \circ \til g \circ h \arrow[r, "{\alpha \circ 1_h}"] & g \circ \til f \circ h
      \end{tikzcd}
    \end{equation}
    Lastly, we require such an $h$, $\beta_1$, $\beta_2$ satisfying this property to be unique in the following sense: Given $h'$, $\beta_1'$, and $\beta_2'$ with the same property, then there exists a unique 2-isomorphism $\gamma: h \to h'$ such that the below diagrams commute:
    \[\begin{tikzcd}
    \til f ' \arrow[r, "\beta_1"] \arrow[dr, "{\beta_1'}"] & \til f \circ h \arrow[d, "1_{\til f} \circ \gamma"] \\
    & \til f \circ h'
    \end{tikzcd}
    \quad
    \begin{tikzcd}
    \til g ' \arrow[r, "\beta_2"] \arrow[dr, "\beta_2'"] & \til g \circ h \arrow[d, "1_{\til g} \circ \gamma"] \\
    & \til g \circ h'
    \end{tikzcd}\]
    The element $w$ in a 2-categorical pullback square is called the \emph{homotopy fiber product of x and y} and may sometimes be written $x \til\times_z y$.
  \end{definition}
  \begin{example}\label{example:HFPofcats}
    For the bicategory of categories (and the bicategory of CFGs), there is a canonical construction of the homotopy fiber product of $F\colon \X \to \mathcal{Z}$ and $G \colon \Y \to \mathcal{Z}$.
    Let us first define the category $ \X \til\times_{\mathcal{Z}} \Y$ as follows:
    \begin{itemize}
      \item The objects $\X \times_{\mathcal{Z}} \Y$ are triples $(X,a,Y)$ such that $X$ and $Y$ are objects of $\X$ and $\Y$ respectively, and $a:F(X) \to G(Y)$ is an isomorphism in $\mathcal{Z}$.
      \item The morphisms are diagrams:
      \[\begin{tikzcd}
        X_1 \arrow[d, "b_1"] & F(X_1) \arrow[d, "F(b_2)"] \arrow[r, "a_1"] & G(Y_1) \arrow[d, "G(b_1)"]  & Y_1 \arrow[d, "b_1"] \\
        X_2    & G(Y_2 \arrow[r, "a_2"])                                        & G(Y_2) & Y_1 \\
      \end{tikzcd}\]
      such that the inner square commutes. The source and target of such diagrams are given by passing to the top and bottom row, respectively. Composition of two such diagrams is obtained by composing the vertical morphisms.
    \end{itemize}
      This construction comes with natural projections to $\X$ and $\Y$ which fit into a pullback square:
      \[
      \begin{tikzcd}
        \X \til\times_\mathcal{Z} \Y \arrow[r, "\pr_2"] \arrow[d, "\pr_1"] & \Y \arrow[d, "G"] \\
        \X \arrow[r, "F"] & \mathcal{Z}
      \end{tikzcd}
      \]
      The cannonical 2-morphism associated to this pullback square is obtained from the map
      \[ \pr_{1.5}\colon (\X \til\times_\mathcal{Z} \Y)_0 \to \mathcal{Z}_1 \quad \mbox{ where } \quad \pr_{1.5}(X,a,Y) = a . \]
      We can see that this satisfies the definition of a pullback since given any other 2-commutative square:
      \[
      \begin{tikzcd}
        \mathcal{W} \arrow[r, "\til F"] \arrow[d, "\til G"] & \Y \arrow[d, "G"]\\
        \X \arrow[r, "F"] & \mathcal{Z}
      \end{tikzcd}
      \]
      with 2-morphism $\alpha\colon F \circ \til G \to \til F \circ G$, then we can define a functor $H: \W \to \X \times_\mathcal{Z} \Y$ by letting $H(a\colon w_1 \to w_2)$ be equal to:
      \[\begin{tikzcd}
        \til G(w_1) \arrow[d, "\til G(a)"] & F \circ \til G (w_1) \arrow[d, "F \circ \til G (a)"] \arrow[r, "\alpha(w_1)"] & G \circ \til F(w_1) \arrow[d, "G \circ \til F (a)"]  & \til F (w_1) \arrow[d, "\til F (a)"] \\
        \til G(w_2)    & F \circ \til G (w_2) \arrow[r, "\til G(a)"])                                        & G \circ \til F(w_2) & \til F(w_2) \\
      \end{tikzcd}\]
      All that remains to show that $\X \times_\mathcal{Z} \Y$ is a 2-pullback is to define $\beta_1$ and $\beta_2$ from the definition. By taking $\beta_1= 1_{\til G}$ and $\beta_2 = 1_{\til F}$ we see that $\X \times_{Z} \Y$ does, in fact, fit into a 2-pullback square.
  \end{example}
  \begin{definition}\label{defn:CFGgroupoid}
    A (strict) groupoid $\G$ internal to a bicategory $\C$ consists of two objects $\G_1$ and $\G_0$ together with 1-morphisms
    \[ \s\colon \G_1 \to \G_0 \, ,\quad \t\colon \G_1 \to \G_0 \, , \quad \u\colon \G_0 \to \G_1 \, , \]
    \[ \m\colon \G_1 \til\times_{\s, \t} \G_1 \to \G_1 \, , \quad \i\colon \G \to \G \]
    which satisfy the axioms of a groupoid. Note that in this setting $\G_1^{(n)} := \G_1 \til\times_{\s,\t} \dots \til\times_{\s,\t} \G_1$.

    When $\C$ is the bicategory of CFGs then we call such an object a \emph{CFG groupoid}.
    A (strict) morphism of CFG groupoids is a pair of functors, one for the CFG of arrows and one for the CFG of objects which commute with the groupoid structure.
    Such a pair of functors is called an \emph{isomorphism} if they are equivalences of categories.
  \end{definition}
  \begin{proposition}\label{prop:CFGgroupoidstruct}
    Let $F: \X \to \Y$ be a functor. Then $\X \til\times_\Y \X$ is canonically a strict groupoid internal to the bicategory of categories.
  \end{proposition}
  \begin{proof}
    We will construct this groupoid structure directly. The category of objects will be $\X$. The source and target will be defined to be $\pr_2$ and $\pr_1$ respectively. The unit morphism $\X \to \X \til\times_Y \X$ is the 2-categorical diagonal embedding:
    \[
    a: x_1 \to x_2 \quad \mapsto \quad \begin{tikzcd}
      x_1 \arrow[d, "a"] & F(x_1) \arrow[d, "F(a)"] \arrow[r, "1"] & F(x_1) \arrow[d, "F(a)"]  & x_1 \arrow[d, "a"] \\
      x_2    & F(y_2 \arrow[r, "1"])                                        & F(x_2) & x_2 \\
    \end{tikzcd}
    \]
    The multiplication functor will be defined by horizontal concatenation:
    \[
    \left(\begin{tikzcd}
      x_1 \arrow[d, "f"] & F(x_1) \arrow[d, "F(f)"] \arrow[r, "a_1"] & F(x_3) \arrow[d, "F(g)"]  & x_3 \arrow[d, "g"] \arrow[r, "b_3"]
      & x_3 \arrow[d, "g"] & F(x_3) \arrow[d, "F(g)"] \arrow[r, "a_3"] & F(x_5) \arrow[d, "F(k)"]  & x_5 \arrow[d, "k"] \\
      x_2    & F(x_2 \arrow[r, "a_2"])                                        & F(x_4) & x_4 \arrow[r, "b_4"]
      & x_4    & F(x_4 \arrow[r, "a_4"])                                        & F(x_6) & x_6 \\
    \end{tikzcd}\right)
    \mapsto
    \]
    \[
    \begin{tikzcd}[column sep=huge]
      x_1 \arrow[d, "f"] & F(x_1) \arrow[d, "F(f)"] \arrow[r, "a_3 \circ F(b_3) \circ a_1"] & F(x_5) \arrow[d, "F(g)"]  & x_5 \arrow[d, "k"] \\
      x_2    & F(x_2 \arrow[r, "a_4 \circ F(b_4) \circ a_2"])                                        & F(x_6) & x_6 \\
    \end{tikzcd}
    \]
    Lastly, the inverse functor is obtained by flipping horizontally:
    \[
    \begin{tikzcd}
      x_1 \arrow[d, "f"] & F(x_1) \arrow[d, "F(f)"] \arrow[r, "a_1"] & F(x_3) \arrow[d, "F(g)"]  & x_3 \arrow[d, "g"] \\
      x_2    & F(x_2 \arrow[r, "a_2"])                                        & F(x_4) & x_4 \\
    \end{tikzcd}
    \quad \mapsto \quad
    \begin{tikzcd}
      x_3 \arrow[d, "g"] & F(x_3) \arrow[d, "F(g)"] \arrow[r, "a_1\inv"] & F(x_1) \arrow[d, "F(f)"]  & x_1 \arrow[d, "f"] \\
      x_4    & F(x_4 \arrow[r, "a_2\inv"])                                        & F(x_2) & x_2 \\
    \end{tikzcd}
    \]
    Recall that the composition operation on the arrows of $\X \til\times_\Y \X$ is by vertical composition. Since these operations clearly commute with vertical composition we conclude that they are all indeed functors. Checking that these maps satisfy the axioms of a groupoid is straightforward.
  \end{proof}
  \begin{lemma}\label{lemma:CFGbundle}
    Let $F: \X \to \Y$ and $G: \Zcal \to \Y$ be full functors. Then there is a left $ \G:= \X \til\times_\Zcal \X$ action on $\P := \X \til\times_\Y \Zcal$ defined by a functor:
    \[ \m_L \colon \G \til\times_{\s, \pr_1} \P \to \P \]
     Furthermore, this action is principal in the sense that the \emph{total action}
    \[ \m_L \times \pr_2 \colon \G \til\times_{\s, \pr_1} \P \to \P \til\times_{\pr_2} \P \]
    is an equivalence of categories.
  \end{lemma}
  \begin{proof}
    $\m_L$ is defined by the rule:
    \[
    \left(\begin{tikzcd}
      x_1 \arrow[d, "f"] & F(x_1) \arrow[d, "F(f)"] \arrow[r, "a_1"] & F(x_3) \arrow[d, "F(g)"]  & x_3 \arrow[d, "g"] \arrow[r, "b_3"]
      & x_3 \arrow[d, "g"] & F(x_3) \arrow[d, "F(f)"] \arrow[r, "c_1"] & G(y_1) \arrow[d, "G(k)"]  & y_1 \arrow[d, "k"] \\
      x_2    & F(x_2) \arrow[r, "a_2"]                                        & F(x_4) & x_4 \arrow[r, "b_4"]
      & x_4    & F(x_4) \arrow[r, "c_2"]                                        & G(y_2) & y_2 \\
    \end{tikzcd}\right)
    \mapsto
    \]
    \[
    \begin{tikzcd}[column sep=huge]
      x_1 \arrow[d, "f"] & F(x_1) \arrow[d, "F(f)"] \arrow[r, "c_1 \circ F(b_3) \circ a_1"] & G(y_1) \arrow[d, "G(g)"]  & y_1 \arrow[d, "k"] \\
      x_2    & F(x_2 \arrow[r, "c_2 \circ F(b_4) \circ a_2"])                                        & G(y_2) & y_2 \\
    \end{tikzcd}
    \]
    That this satisfies the axioms of a left action is straightforward. The only thing remaining to check is that the mapping:
    \[
    \left(\begin{tikzcd}
      x_1 \arrow[d, "f"] & F(x_1) \arrow[d, "F(f)"] \arrow[r, "a_1"] & F(x_3) \arrow[d, "F(g_1)"]  & x_3 \arrow[d, "g_1"] \arrow[r, "b_1"]
      & x_5 \arrow[d, "g_2"] & F(x_5) \arrow[d, "F(g_2)"] \arrow[r, "c_1"] & G(y_1) \arrow[d, "G(k)"]  & y_1 \arrow[d, "k"] \\
      x_2    & F(x_2) \arrow[r, "a_2"]                                        & F(x_4) & x_4 \arrow[r, "b_2"]
      & x_6    & F(x_6) \arrow[r, "c_2"]                                        & G(y_2) & y_2 \\
    \end{tikzcd}\right)
    \mapsto
    \]
    \[
    \begin{tikzcd}[column sep=large]
      x_1 \arrow[d, "f"] & F(x_1) \arrow[d, "F(f)"] \arrow[r, "c_1 \circ F(b_1) \circ a_1"] & G(y_1) \arrow[d, "G(k)"]  & y_1 \arrow[d, "k"] \arrow[r, "\Id"]
      & y_1 \arrow[d, "k"] & G(y_1) \arrow[d, "G(k)"] & F(x_3) \arrow[d, "F(g_2)"] \arrow[swap, l, "c_1"] & x_3 \arrow[d, "g_2"]\\
      x_2    & F(x_2 \arrow[r, "c_2 \circ F(b_2) \circ a_2"])                                        & G(y_2) & y_2 \arrow[r, "\Id"]
      & y_2                & G(y_2)                   & F(x_4) \arrow[swap, l, "c_2"]                   & x_4 \\
    \end{tikzcd}
    \]
    is an equivalence of categories. We begin by showing that it is essentially surjective. Consider the following diagram, which is a morphism in $\P \til\times_\Zcal \P$:
    \[
    \begin{tikzcd}
      x' \arrow[d, "1"] & F(x') \arrow[d, "1"] \arrow[r, "q"]            & G(y') \arrow[d, "G(r)"]  & y' \arrow[d, "r"] \arrow[r, "r"]      & y'' \arrow[d, "1"] & G(y'') \arrow[d, "1"] & F(x'') \arrow[d, "1"] \arrow[swap, l, "s"] & x'' \arrow[d, "1"] \\
      x' & F(x') \arrow[r, "G(r) \circ q"] & G(y'') & y'' \arrow[r, "1"] & y'' & G(y'') & F(x'') \arrow[swap, l, "s"] & x''
    \end{tikzcd}
    \]
    Note that the top row of the diagram can be an arbitrary object in $\P \til\times_Y \P$. On the other hand, the bottom row is the image of the following object in $\G \til\times_\X \P$ under the total action:
    \[
    \begin{tikzcd}[column sep=large]
      x' & F(x') \arrow[r, "s\inv \circ G(r) \circ q"] & F(x'') & x'' \arrow[r, "1"] & x'' & F(x'') \arrow[r, "s"] & G(y'') & y''
    \end{tikzcd}
    \]
    Hence the total action is essentially surjective. That the total action is faithful is relatively clear from the definition. To show that it is full, suppose we have a pair of objects in $\G \til\times_\X \P$:
    \[
    \begin{tikzcd}
      x_1  & F(x_1) \arrow[r, "a_1"] & F(x_3)   & x_3  \arrow[r, "b_1"]
      & x_5  & F(x_5) \arrow[r, "c_1"] & G(y_1)  & y_1  \\
      x_2    & F(x_2) \arrow[r, "a_2"]                                        & F(x_4) & x_4 \arrow[r, "b_2"]
      & x_6    & F(x_6) \arrow[r, "c_2"]                                        & G(y_2) & y_2 \\
    \end{tikzcd}
    \]
    as well as an arbitrary morphism in $\P \til\times_\Zcal \P$ between their images:
    \begin{equation}\label{eqn:principalbundlelemma}
    \begin{tikzcd}[column sep=large]
      x_1 \arrow[d, "f"] & F(x_1) \arrow[d, "F(f)"] \arrow[r, "c_1 \circ F(b_1) \circ a_1"] & G(y_1) \arrow[d, "G(k)"]  & y_1 \arrow[d, "k"] \arrow[r, "\Id"]
      & y_1 \arrow[d, "k"] & G(y_1) \arrow[d, "G(k)"] & F(x_3) \arrow[d, "F(g_2)"] \arrow[swap, l, "c_1"] & x_3 \arrow[d, "g_2"]\\
      x_2    & F(x_2 \arrow[r, "c_2 \circ F(b_2) \circ a_2"])                                        & G(y_2) & y_2 \arrow[r, "\Id"]
      & y_2                & G(y_2)                   & F(x_4) \arrow[swap, l, "c_2"]                   & x_4 \\
    \end{tikzcd}
    \end{equation}
    Then let $g_1 := b_2\inv \circ g_2 \circ b_1$. Then we claim that:
    \begin{equation}\label{eqn:principalbundlelemma2}
    \begin{tikzcd}
      x_1 \arrow[d, "f"] & F(x_1) \arrow[d, "F(f)"] \arrow[r, "a_1"] & F(x_3) \arrow[d, "F(g_1)"]  & x_3 \arrow[d, "g_1"] \arrow[r, "b_1"]
      & x_5 \arrow[d, "g_2"] & F(x_5) \arrow[d, "F(g_2)"] \arrow[r, "c_1"] & G(y_1) \arrow[d, "G(k)"]  & y_1 \arrow[d, "k"] \\
      x_2    & F(x_2) \arrow[r, "a_2"]                                        & F(x_4) & x_4 \arrow[r, "b_2"]
      & x_6    & F(x_6) \arrow[r, "c_2"]                                        & G(y_2) & y_2 \\
    \end{tikzcd}
    \end{equation}
    is a morphism in $\G \til\times_\X \P$. To prove this we need to show that the three squares commute. The central square commutes by the definition of $g_1$. We have already assumed the right square commutes since it occurs in (\ref{eqn:principalbundlelemma}). Lastly, the left square commutes by combining the commutativity of the following diagrams:
    \[
    \begin{tikzcd}[column sep = large]
      F(x_3) \arrow[r, "F(b_1)"] \arrow[d, "F(g_1)"] & F(x_5) \arrow[d, "F(g_2)"] &
      F(x_5) \arrow[r, "c_1"] \arrow[d, "F(g_2)"] & G(y_1) \arrow[d, "G(k)  "] &
      F(x_1) \arrow[r, "c_1 \circ F(b_1) \circ a_1"] \arrow[d, "F(f)"] & G(y_1) \arrow[d, "G(k)"] \\
      F(x_4) \arrow[r, "F(b_2)"] & F(x_6) &
      F(x_6) \arrow[r, "c_2"] & G(y_2) &
      F(x_2) \arrow[r, "c_2 \circ F(b_2) \circ a_2"] & G(y_2)
    \end{tikzcd}
    \]
  \end{proof}
\chapter{Groupoids and stacks}

\section{Groupoids and actions}
In this section we include several lemmas about groupoids, groupoid actions and principal bundles. Throughout $\C$ is a category with an initial object and is endowed with Grothendieck topology making it into a \emph{good site}. That is, a site where $\mbox{Sub}$ is a stack and morphisms in $\C$ are locally defined.

Importantly, we do not assume that $\C$ has a terminal object and, relatedly, we cannot assume that morphisms in $\C$ are characterized by their behavior on points.
All proofs must hold in terms of morphisms and commuting diagrams. However, since the notation for constructing morphisms can become cumbersome, we may sometimes include the ``pointwise" versions of our definitions to help the reader in parsing them.

\begin{lemma}\label{lemma:equivariantmaps}
  Let $\G$ have a left action on $P$ and suppose that $P / \G = B$ exists (i.e. the action is almost principal. That is, there exists a submersion $\s: P \to B$ such that the total action $\G \times_M P \to P \times_B P$ is a submersion. Then for any $f: P \to X$ such that $f \circ \m_L = f \circ pr_2$, there exists a unique $f' \colon B \to X$ such that the below diagram commutes:
  \[
  \begin{tikzcd}
    \G \times_{\s,J} P \arrow[r, "\m_L"] \arrow[d, "\pr_2"] & P \arrow[d, "\pi"] \arrow[rdd, "f", bend left] & \\
    P \arrow[r, "\pi"] \arrow[rrd, "f", bend right] & B \arrow[dr, dashed, "{f'}"] & \\
    & & X
  \end{tikzcd}
  \]
\end{lemma}
\begin{proof}
We we will define $f'$ locally. Suppose $S = \{ s_i \} $ is a sieve on $B$ such that there exist sections $\sigma_i \colon U_i \to P$ along $s_i$. Then we define $f'_i \colon U_i \to X$ to be $f \circ \sigma_i$. To see that this defines a morphism $B \to X$, we need to check that the $f_i'$ agree on intersections. This follows from the following claim: Given $g_1, g_2 \colon N \to P$ such that $\s \circ g_1 = \s \circ g_2$, then $f \circ g_1 = f \circ g_2$.

Since we only need to check this property locally, we can assume without loss of generality that there exists a right inverse of the total action, and hence a division map $\delta \colon P \times_B P \to \G$. That is, a mapping $\delta$ such that:
\[ \m_L (\delta(\pr_1,\pr_2),\pr_2) = \pr_1 \colon P \times_N P \to P \, . \]
Then $g_1 = \m_L(\delta(g_1,g_2),g_2)$ and
\begin{align*}
  f \circ g_1 &= (f \circ g_1)(\m_L(\delta(g_1,g_2),g_2)) \\
              &= (f \circ \pr_2)(\m_L(\delta(g_1,g_2),g_2)) \\
              &- f \circ g_2 \, .
\end{align*}
Hence, there exists $f'$ such that $f \circ s_i = f_i$. Uniqueness of $f'$ follows from the fact that any other such morphism must agree with $f'$ locally.
\end{proof}
\begin{corollary}\label{corollary:pmodgunique}
  Suppose $\G$ has an almost principal action on $P$. Then $P/\G$ is unique up to a unique isomorphism.
\end{corollary}
\begin{proof}
  Lemma~\ref{lemma:equivariantmaps} shows that $P/\G$ must satisfy a universal property. In more technical terms, it is the coequalizer of $\m_L \colon \G \times_M P \to P$ and $\pr_2 \colon \G \times_M P \to P$. It is a standard result of category theory that coequalizers are unique up to a unique isomorphism.
\end{proof}
\begin{corollary}\label{corollary:equivariantmaps}
      Suppose $\G$ acts on $P$, $Q$ and $R$ in such a way that $P/ \G$, $Q / \G$ and $R/\G$ exist.
  \begin{itemize}
    \item Given an equivariant morphism $f: P \to Q$ there exists a unique morphism $f/\G \colon P/\G \to Q/ \G$ compatible with the projections.
    \item In such a case, $f/\G$ is an isomorphism if $f$ is an isomorphism.
    \item Lastly, given another equivariant morphism $g\colon Q \to R$ for a $\G$-action on $R$, then $(g \circ f)/\G = (g /\G \circ f/\G)$
  \end{itemize}
\end{corollary}
\begin{proof}
  \begin{itemize}
  \item Since $f$ is equivariant, we can apply Lemma~\ref{lemma:equivariantmaps} to $\s^Q \circ f \colon P to Q/\G$ to obtain a unique morphism $f/\G \colon P/\G \to Q/\G$. Since any morphism compatible with the projections must satisfy Lemma~\ref{lemma:equivariantmaps} this morphism is unique.
  \item This follows by checking that $f\inv /\G \colon Q/\G \to P/\G$ is the inverse of $f/\G$. This is clear by looking at the local construction of $f/\G$ from Lemma~\ref{lemma:equivariantmaps}.
  \item This holds by observing that the equation is true locally. Since morphisms in $\C$ are locally defined, such an equation must hold globally.
  \end{itemize}
\end{proof}
\begin{lemma}\label{lemma:diagonalaction}
  Let $\G \grpd M$ be a $\C$-groupoid.
  Suppose $P \to N$ is a principal left $\G$-bundle and there is a left action of $\G$ on $Q$ such that the target map $\t^Q \colon Q \to M$ is a submersion.
  Then the diagonal action of $\G$ on $Q \times_M P$ is principal.
\end{lemma}
\begin{proof}
  The diagonal action of $\G$ on $Q \times_M P$ is defined to be
  \[ \m^\Delta_L = (\m^Q_L(\pr_1, \pr_1 \circ \pr_2), \m^P_L(\pr_1, \pr_2 \circ \pr_2)) \colon \G \times_M (P \times_M Q) \to P \times_M Q \]
  We must construct an object $\til N$ and a submersion $\s^\Delta \colon P \times_M P \to \til N$ which makes $P \times_M P$ a principal bundle over $\til N$.

  We will construct $\til N$ by realizing it as an object in $\Sub_N$. Let $S \colon \{ s_i \colon U_i \to N$ be a covering sieve of $N$ and $\sigma_i \colon U_i \to P$ be sections along $S$.
  Now suppose that $\til N_i := Q \times_{\t, \t \circ \sigma_i} U_i$. Each $\til N_i$ comes with a submersion $\til N_i \to U_i$, so this is a family of objects of in $\Sub$. Let $\gamma_{ij} \colon U_{ij} \to \G$ be the cycle associated to the sections $\sigma_i$ of $P$. Then $\Phi_{ij} \colon \til N_i|_{U_ij} \to \til N_j|_{U_ij}$ is defined to be:
  \[ \Phi_{ij} := (\m_L(\gamma_{ij}, \pr_1), \pr_2) \colon N \times_{\t^N, \t \circ \sigma_{ij}} U_ij \to N \times_{\t^N, \t \circ \sigma_{ij}} U_ij   \]
  Defines a coherent gluing map for the $\til N_i$. Therefore, there exists an object $\til N$ and identifications $\Phi_i \colon \til N|_{U_i} = \til N_i$.

  We still need to define a projection $Q \times_M P \to \til N$. To find this, we will repeat our argument, except instead gluing an object over $N$, we glue an object over $\til N$.
  Notice that the set $\{ \til N_i \to \til N \}$ generates a covering sieve of $\til N$. Furthermore, for each $i$ the morphism $\pr_1 \times \s^P \colon Q \times_M P_i \to Q \times_M U_i$ is a submersion.
  Again, we have a canonical way of gluing the $P_i := P|_{U_i}$ together, which we can extend to a gluing of the $Q \times_M P_i$. Since $\Sub$ is a stack we obtain an submersion $Q \times_M P \to \til N$.

  This submersion makes $Q \times_M P$ into a $\G$-bundle over $\til N$ since each restriction $Q \times_M P_i = (Q \times_M P)|_{N_i}$ makes the necessary diagram commute. Since it suffices to show diagrams commute locally, we conclude that $Q \times_M P \to \til N$ is a $\G$-bundle. To finish, we need to explain why the diagonal action is principal. We will be somewhat brief on this point and just provide the reader with the inverse of the total action.
  Recall $\til \m^P_L$ is the division map for the left action on $P$. Then
  \[ (\til \m^P(\pr_1 \circ \pr_2,\pr_2 \circ \pr_2), (\pr_2 \circ \pr_1,\pr_2 \circ \pr_1)) \colon (Q \times_M P) \times_M (Q \times_M P) \to \G \times_M (Q \times_M P) \]
  is the inverse of the total action. The trick here is that the $\G$ component is uniquely determined by the terms two terms coming from $P$. That this is the inverse of the total action is a straightforward computation.
\end{proof}
\subsection{Principal Bundle Lemmas}
The next lemma is the first step to understanding principal $\G$-bundles in terms of cocycles.
\begin{lemma}\label{lemma:gbundlecycle}
  Suppose $\G \grpd M$ is a $\C$-groupoid. Let $f \colon N \to M$ and $g \colon N \to M$ be morphisms in $\C$ and $\G \times_{\s, f} N$ and $\G \times_{\s, g} N$ be the trivial bundles associated to them.

  There is a one-to-one correspondence between principal bundle morphisms $\phi \colon \G \times_{\s, f} N \to \G \times_{\s, g} N$ covering the identity and morphisms $\gamma \colon N \to \G$ such that $\t \circ \gamma = f$ and $\s \circ \gamma = g$.
\end{lemma}
\begin{proof}
  Suppose $\phi \colon P \to Q$ is an isomorphism covering the identity. Let $\sigma := (\u \circ f) \times \Id \colon N \to \G \times_{\s, f} N$ be the canonical section.
  Then we define
  \[ \gamma := \pr_1 \circ \phi \circ \sigma \colon N \to \G. \]
  Observe that $\s \circ \gamma = g$ and $\t \circ \gamma = f$.

  On the other hand, suppose we are given such a $\gamma$. Then let
  \[ \phi := (\m \circ (\pr_1 \times (\gamma \circ pr_2))) \times \pr_2 \colon \G \times_{\s, f} N \to \G \times_{\s, g} N \]
  If $\C$ has points this formula is equivalent to $\phi(g,x) = (g \cdot \gamma(x), x)$.
  This morphism is $\G$-equivariant thanks to the associativity of $\G$. The correspondence is one-to-one since
  \[ \gamma = \pr_1 \circ \phi \circ \sigma   \, .\]
\end{proof}
\begin{lemma}\label{lemma:principalbundlemorphisms}
  Let $\G \grpd M$ be a $\C$-groupoid and suppose $P$ and $Q$ are principal left $\G$-bundles. If $\phi \colon P \to Q$ is a principal $\G$-bundle morphism covering the identity, then $\phi$ is an isomorphism.
\end{lemma}
\begin{proof}
Let $S = \{ s_i \colon U_i \to N \}$ be a covering sieve such that there exist $\{\sigma_i^P \}$ and $\{ \sigma_i^Q \}$ be sections of $s_i^* P$ and $s_i^*Q$. Such a covering exists since $P \to N$ and $Q \to N$ are submersions. Let $\phi_i$ be the unique morphism which makes
\[
\begin{tikzcd}
  s_i^* P \arrow[d] \arrow[r, "{\phi_i}"] & s_i^* Q \arrow[d]\\
  P \arrow[r, "\phi"] & Q
\end{tikzcd}
\]
commute. Then by Lemma~\ref{lemma:gbundlecycle} each $\phi_i$ is related to some $\gamma_i \colon N \to \G$. By the construction in the lemma, it is clear that $\i \circ \gamma_i$ gives rise to $\phi_i\inv$. Hence each $\phi_i$ is an isomorphism.
\end{proof}

\begin{lemma}\label{lemma:bibundlemorphisms}
  Let $P$ and $Q$ be left principal $(\G, \H)$-bibundles. Then any bibundle morphism $\phi \colon P \to Q$ covering the identity is an isomorphism.
\end{lemma}
\begin{proof}
  It suffices to show that $\phi$ is locally an isomorphism. Hence we can assume without loss of generality that $P = f^* P$ and $Q = f^* Q$ are trivial bundles. Then the result follows immediately from our work in Lemma~\ref{lemma:gbundlecycle}.
\end{proof}

\section{Stack lemmas}

Our stack lemmas will be made more concise by adopting the following terminology.
\begin{definition}
  Let $\pi\colon \X \to \C$ be a CFG over a site. Given an object $M$ in $\C$, descent data over $M$ consists of:
  \begin{itemize}
    \item a covering family ${\{ i_a \colon U_a \to M \}}_{a \in A}$ of $M$,
    \item a collection of objects ${\{ P_a \in \X_{U_a} \}}_{a \in A}$,
    \item morphisms $\phi_{ab}: P_b|_{U_{ab}} \to P_a|_{U_{ab}}$
  \end{itemize}
  such that, for all $a,b,c \in A$, $\phi_{ab}|_{U_{abc}} \circ \phi_{bc}|_{U_{abc}} = \phi_{ac}|_{U_{abc}}$.
  Given a morphism $f\colon M \to N$ in $\C$, descent data over $f$ consists of:
  \begin{itemize}
    \item a covering family ${\{ i_a \colon U_a \to M \}}_{a \in A}$ of $M$ in $\C$,
    \item objects $P \in \X_M$ and $Q \in \X_N$,
    \item morphisms $\phi_a \colon P|_{U_a} \to f^*Q|_{U_a}$ covering the identity,
  \end{itemize}
  such that
  \[ F_a|_{P|_{U_{ab}}} = F_b|_{f^*Q|_{U_{ab}}} \, . \]
\end{definition}
The above definition is designed to correspond with the stack axioms $(S1)$ and $(S2)$. Given descent data over $f$ as above, we say that $\phi: P \to Q$ is a \emph{realization} of this data if $\phi|_{U_a} = \phi_a$ for every $a$.
Similarly, given descent data over an object $M$, we say that $P$, together with maps $\phi_a: P|_{U_a} \to P_a$ is a \emph{realization} of this data if $\phi_{ab} \circ \phi_b|_{U_{ab}} = \phi_a|_{U_ab}$.

Using this terminology the stack axioms can be restated as follows.
\begin{enumerate}[(S1)]
    \item descent data over a morphism $f$ always admits a unique realization.
    \item descent data over an object $M$ always admits a realization.
\end{enumerate}

\begin{lemma}\label{lemma:basechangelemma}
  Suppose
  \[
  \begin{tikzcd}
    \til \X \arrow[d, "\til F"] \arrow[r] & \X \arrow[d, "F"] \\
    \til \Y \arrow[r] & Y
  \end{tikzcd}
  \]
  is a 2-pullback square of CFGs. Then:
  \begin{itemize}
    \item $F$ is an epimorphism implies that $\til F$ is an epimorphism
    \item $F$ is representable implies that $\til F$ is representable
    \item $F$ is a submersion implies that $\til F$ is a submersion
  \end{itemize}
\end{lemma}
\begin{proof}
  It suffices to show the first two.
  \begin{itemize}
    \item Suppose $F$ is an epimorphism. Let the map $\til \Y \to \Y$ be denoted by $G$. Given an object $\til Y \in \til \Y|_M$, we know that there exists a covering $\{ U_i \}$ of $M$ and isomorphisms $a_i \colon G(\til Y)|_{U_i} \to F(X_i)$ for some objects $X_i \in \X$. Hence the triples:
    \[ \til X_i := (\til Y |_{U_i}, a_i, X_i) \]
    are objects in the fiber product $\til \X \isom \til \Y \times_\Y \X$. Since $\til F (X_i) = \til Y|_{U_i}$ the claim follows.
    \item Suppose $\bar N \to \til \Y$ is a CFG morphism for $\bar N$, representable. Then we have a rectangle:
    \[
    \begin{tikzcd}
      \bar N \til\times_{\til \Y} \til \X \arrow[r] \arrow[d] & \til \X \arrow[d, "\til F"] \arrow[r] & \X \arrow[d, "F"] \\
      \bar N \arrow[r] & \til \Y \arrow[r] & Y
    \end{tikzcd}
    \]
    Since both of the inner squares are pullback squares, it follows that the outer square is also a pullback square. Hence
    \[ N \til\times_{\til \Y} \til \X \isom N \til\times_\Y \X \, .\]
    Since $F$ is representable, it follows that $N \til\times_{\til \Y} \til \X$ is a representable CFG.
  \end{itemize}
\end{proof}
\begin{corollary}\label{cor:basechangeofsubmersions}
  Suppose
  \[
  \begin{tikzcd}
    \til M \arrow[d, "\til f"] \arrow[r] & M \arrow[d, "f"] \\
    \til N \arrow[r] & N
  \end{tikzcd}
  \]
  is a pullback square in a site $\C$. Then $f$ is a submersion implies that $\til f$ is a submersion.
\end{corollary}
\begin{proof}
  Note that the functor which sends an object of $\C$ to its associated stack is representable.
\end{proof}

\begin{lemma}\label{lemma:fiberproductofstacks}
  Let $F\colon \X \to \Zcal$ and $G\colon \Y \to \Zcal$ be CFG morphisms such that $\X$ and $\Y$ are stacks and $\Zcal$ be is a pre-stack. Then $\X \til\times_\Zcal \Y$ is a stack.
\end{lemma}
\begin{proof}
  Recall that an object of $\X \til\times_\Zcal \Y$ consists of a triple $(X,a, Y)$ for objects $X \in \X$ and $Y \in \Y$ and a morphism $a\colon F(X) \to G(Y)$ in $\Zcal$ which covers the identity. Recall that a morphism in $\X \til\times_\Zcal \Y$ from $(X_1, a_1, Y_1) \to (X_2, a_2, Y_2)$ consists of two morphisms $(b,b')$ such that $a_2 \circ F(b) = G(b') \circ a_1$.

  (S1) Let $f\colon M \to N$ be a morphism in $\C$ and let $\{ (b_i, b'_i ) \}$ be descent data over $f$ for $\X \til\times_\Zcal \Y$ (we will denote the associated covering family by $\{ U_i \}$).
  Then $\{ b_i \}$ and $\{ b_i ' \}$ constitute descent data in $\X$ and $\Y$ over $f$ respectively.
  Let $b$ and $b'$ be morphisms given by the first stack axiom. Then we need to show that $(b,b')$ is a morphism in $\X \til\times_\Zcal \Y$ which realizes the given data. We conclude that $a_2 \circ F(b) = G(b') \circ a_1$ since this holds when restricted to each $U_i$ in the covering. Furthermore $(b, b')$ is clearly a realization of this data, since the restriction of $(b,b')$ to $U_a$ is just $(b|_{U_a}, b'|_{U_a})$.

  (S2) Suppose $(X_i, a_i, Y_i)$ is descent $\X \til\times_\Zcal \Y$ data over an object $M \in \C$. Let $(\phi_ab, \phi'_ab)$ be the cocycle associated to this data. The definition of composition in $\X \til\times_Zcal \Y$ implies that the collections $\{ \phi_ab \}$ and $ \{ \phi'_{ab} \}$ satisfy the cocycle condition for descent $\X$ and $\Y$ data respectively. Therefore, there must exist objects $X$ and $Y$ realizing this data. Furthermore, the collection $\{ a_i \}$ constitutes descent $\Zcal$ data over the identity morphism.
  Since $\Zcal$ is a prestack we conclude that there exists a morphism $a$ which realizes this data.
  Then the triple $(X, a, Y)$ together with the associated morphisms $(\phi_a, \phi'_a)$ constitutes a realization of the original data.
\end{proof}

\begin{lemma}\label{lemma:representabledomain}
  Let $\X$ be a CFG and $M$ be an object in $\C$. Suppose $F_1 \colon \bar M \to \X$ and $F_2 \colon \bar M \to \X$ are CFG morphisms such that $F_1(\Id_M) \isom F_2(\Id_M)$. Then is a one-to-one correspondence between natural transformation from $F_1$ to $F_2$ and $\Isom(F_1(\Id_M),F_2(\Id_M))$.
\end{lemma}
\begin{proof}
  Suppose $F_1(\Id_M) = X$ and $F_2(\Id_M) = Y$. Let $a \colon X \to Y$ be an isomorphism.
  To a natural transformation, we need a function $\eta \colon \bar M_0 \to \bar \X_1$. Let $\eta(f)$ is defined to be the unique isomorphism $F_1(f) \to F_2(f)$ which makes the below diagram commute:
  \[
  \begin{tikzcd}
    F_1(f) \arrow[r] \arrow[d, "\eta(f)"] & X \arrow[d, "a"] \\
    F_2(f) \arrow[r]                      & Y
  \end{tikzcd}
  \]
  To get the reverse correspondence, given such a natural transformation $\eta$, we can let $a := \eta(\Id_M)$.
\end{proof}

\subsection{Groupoids vs Stacks}
\begin{lemma}\label{lemma:gbgpullback}
  Suppose $\G$ is a $\C$-groupoid. Let $p \colon \bar{M} \to \B \G$ be the morphism which send $f \colon N \to M$ to the trivial $\G$-bundle $f^* \G$.
  Then there is a 2-pullback square:
  \[
  \begin{tikzcd}
    \bar \G \arrow[r, "\s"] \arrow[d, "t"] & \bar M \arrow[d, "p"] \\
    \bar M \arrow[r, "p"]                  & \B\G
  \end{tikzcd}
  \]
\end{lemma}
\begin{proof}
  To make the above into a 2-commutative diagram, we must provide the natural transformation. Given an object $\gamma \colon N \to \G$ we need to exhibit an isomorphism $\phi_\gamma \colon (\t \circ \gamma)^* \G \to (\s \circ \gamma)^* \G$.
  Let $\phi_\gamma$ be the isomorphism associated to $\gamma$ via the correspondence provided by Lemma~\ref{lemma:gbundlecycle}. It is routine to check that the resulting function is a natural transformation.

  Since we have a 2-commutative diagram, we obtain a morphism $\G \to \bar M \til\times_{\B\G} \bar M$. On morphisms, this functor acts as below:
  \[
  \begin{tikzcd}
    N_1 \arrow[dr, "\gamma_1"] \arrow[dd, "h"] & \\
    & \G \\
    N_2 \arrow[ur, "\gamma_2"] &
  \end{tikzcd}
  \quad
  \mapsto
  \quad
  \begin{tikzcd}
    \t \circ \gamma_1 \arrow[d, "h"] & (\t \circ \gamma_1)^* \G \arrow[r, "\phi_{\gamma_1}"] \arrow[d] & (\s \circ \gamma_1)^* \G \arrow[d] & \s \circ \gamma_1 \arrow[d, "h"] \\
    \t \circ \gamma_2                & (\t \circ \gamma_2)^* \G \arrow[r, "\phi_{\gamma_2}"] & (\s \circ \gamma_2)^* \G & \s \circ \gamma_2
  \end{tikzcd}
  \]
The 1-1 correspondence from Lemma~\ref{lemma:gbundlecycle} tells us that this functor is actually a bijection. Therefore, the diagram is a pullback square.
\end{proof}
\begin{lemma}\label{lemma:gbgpullback2}
  The identification of $\bar \G$ with $\bar M \til\times_{\B\G} \bar M$ from Lemma~\ref{lemma:gbgpullback} is a CFG groupoid isomorphism.
\end{lemma}
\begin{proof}
  Suppose we are given $\gamma_1,\gamma_2 \colon N \to \G$ such that $\s \circ \gamma_1 = \t \circ \gamma_2$.
  As in Lemma~\ref{lemma:gbgpullback}, let $\phi_\gamma \colon (\t \circ \gamma)^* \G \to (\s \circ \gamma)^* \G$ denote the $\G$-bundle isomorphism provided in Lemma~\ref{lemma:gbundlecycle}.
  Then $\phi_{\m \circ (\gamma_1 \times \gamma_2)} = \phi_{\gamma_2} \circ \phi_{\gamma_1}$.

  Using this relationship and the definition of the groupoid structure on $\bar M \til\times_{\B\G} \bar M$ (see Proposition~\ref{prop:CFGgroupoidstruct}), it is routine to check that $\bar \G \to \bar M \til\times_{\B\G} \bar M$ is an isomorphism of CFG groupoids.
\end{proof}
\begin{lemma}\label{lemma:bundleispullback}
  Let $\G \grpd$ be a $\C$-groupoid and $Q$ be a $\G$ bundle over $N$. Let $q \colon \bar N \to \B\G$ be the CFG morphism which sends $g\colon N' \to N$ to $g^* Q$. Then there is a 2-pullback square:
  \[
  \begin{tikzcd}
    \bar Q \arrow[r, "\s"] \arrow[d, "\t"] & \bar N \arrow[d, "q"] \\
    \bar M \arrow[r]           & \B\G
  \end{tikzcd}
  \]
  Furthermore, the identification $\bar Q \to \bar M \til\times_{\B\G} \bar N$ is a $\bar \G$-bundle isomorphism.
\end{lemma}
\begin{proof}
We should clarify that the left action of $\bar \G$ on $\bar M \til\times_{\B\G} \bar N$ comes from combining Lemma~\ref{lemma:gbgpullback2} with Lemma~\ref{lemma:CFGbundle}.

First off, we observe that a morphism $\rho \colon N' \to Q$ gives rise to a canonical section $\sigma_\rho \colon N' \to (\s \circ \rho)^* Q$.
Furthermore, by Lemma~\ref{lemma:bundletrivialization}, a section provides us with a canonical trivialization $\phi_\rho \colon (\t \circ \rho)^*G \to (\s \circ \rho)^*Q$.
Hence, we can define a functor $\bar Q \to \bar M \til\times_{\B\G} \bar N$:
\[
\begin{tikzcd}
  N_1 \arrow[dr, "\rho_1"] \arrow[dd, "h"] & \\
  & Q \\
  N_2 \arrow[ur, "\rho_2"] &
\end{tikzcd}
\quad
\mapsto
\quad
\begin{tikzcd}
  \t \circ \rho_1 \arrow[d, "h"] & (\t \circ \rho_1)^* \G \arrow[r, "\phi_{\rho_1}"] \arrow[d] & (\s \circ \rho_1)^* Q \arrow[d] & \s \circ \rho_1 \arrow[d, "h"] \\
  \t \circ \rho_2                & (\t \circ \rho_2)^* \G \arrow[r, "\phi_{\rho_2}"]           & (\s \circ \rho_2)^* Q           & \s \circ \rho_2
\end{tikzcd}
\]
The functor is a bijection since trivializations $\phi \colon f^* \G \to g^*Q$ over the identity $N' \to N'$ are in 1-1 correspondence with morphisms $\rho \colon N' \to Q$ such that $\s \circ \rho = g$.
Therefore $\bar Q \to \bar M \til\times_{\B\G} \bar N$ is an isomorphism. Now suppose $\gamma \colon N' \to \G$ and $\rho \colon N' \to Q$. Then one can easily check that
\[ \phi_{\m_L \circ (\gamma \times \rho)} = \phi_{\rho} \circ \phi_{\gamma}  \, .  \]
From this formula, it follows from a routine calculation using the definition of the $\bar M \til\times_{\B\G} \bar N$ action that the identification $\bar Q \to \bar M \til\times_{\B\G} \bar N$ is (strictly) $\bar \G$-equivariant.
\end{proof}
\begin{proposition}\label{prop:groupoidaspullback}
Let $\G$ be a $\C$-groupoid and $\bar M \to \B\G$ be the canonical presentation of its stack. Then $\bar \G \isom \bar M \til\times_{\B \G} \bar M$ as CFG groupoids.
\end{proposition}
\begin{proof}
  By Lemma~\ref{lemma:bundleispullback}, we already know that $\G$ fits into a 2-pullback square:
  \[
  \begin{tikzcd}
    \bar \G \arrow[d, "\t"]  \arrow[r, "\s"] & \bar M \arrow[d] \\
    \bar M \arrow[r]                         & \B \G
  \end{tikzcd}
  \]
  We only need to show that the isomorphism $\bar \G \to \bar M \times_{\B\G} \bar M$ is a CFG groupoid homomorphism.
  The mapping $\bar \G \to \bar M \times_{\B\G} \bar M$ sends $\gamma \colon N \to \G$ to the triple $(\t \circ \gamma, \phi_\gamma , \s \circ \gamma)$ where $\phi_\gamma$ is the isomorphism covering the identity associated to $\gamma$ (see Lemma~\ref{lemma:gbundlecycle}).

  The groupoid operation on $\bar \G$ is defined as follows: suppose $\gamma_1 \colon N \to \G$ and $\gamma_2 \colon N \to \G$ are morphisms such that $\s \circ \gamma_1 = \t \circ \gamma_2$. Then $\gamma_1 \cdot \gamma_2 = \m(\gamma_1,\gamma_2)$.
  We also know that $\phi_{\gamma_2} \circ \phi_{\gamma_1} = \phi_{\m(\gamma_1,\gamma_2)}$. Then it is clear from the definition of the CFG groupoid structure on $\bar M \til\times_{\B\G} \bar M$ that the functor is a groupoid homomorphism at the level of objects. Since both categories have trivial isotropy, we conclude that the functor $\bar \G \to \bar M \times_{\B\G} \bar M$ is a CFG groupoid isomorphism.
\end{proof}
\subsection{Stackification}
\begin{definition}
  Let $\pi\colon \X \to \C$ be a CFG over a site $\C$. Then the \emph{canonical stackification} of $\X$, denoted $\mathcal{S}\X$ is constructed as follows:
  \begin{itemize}
    \item An object of $\Scal\X$ over $M$ consist of descent data over $M$.
    \item Suppose $\{P_i \}$ and $\{Q_j \}$ are descent data over $M$ and $N$ respectively. Let the coverings associated to this data be $S := \{ U_i \to M \}$ and $T := \{ V_j \to N \}$.
    A morphism $f \colon \{ P_i \} \to \{ Q_j \}$ covering $f \colon M \to N$ is represented by a collection of morphisms $\{\phi_i \colon P_i \to Q_j \}$ defined relative to a refinement $\{ W_i \to M \} \subset V \subset S \cap f^*T$, which satisfy the descent data over $f$ condition relative to the cocycles identifying the $P_i$ and $Q_j$ on each double intersection.
    Two representations of a morphism are declared equivalent if they are equal on some refinement of their respective coverings.
  \end{itemize}
  Composition of morphisms is performed by finding representatives which are defined relative to the same covering and then composing in the natural way.
\end{definition}
The CFG structure comes from projecting descent data over $f$ to $f$. The axiom (CFG1) is satisfied by just pulling back the descent data along $f$. On the other hand, (CFG2) follows by observing that it holds locally, and since morphisms (by construction) are characterized by their local behavior, the necessary morphism must exist. That the canonical stackification results in a stack is true tautologically since it replaces objects and morphisms with descent data.
\begin{definition}\label{defn:localequivalence}
Suppose $F \colon \X \to \Y$ is a morphism of CFGs. Then $F$ is called a \emph{local equivalence} if both a monomorphism (fully faithful) and an epimorphism (locally essentialy surjective).
\end{definition}
The following is the main property of local equivalences that we are interested in.
\begin{lemma}\label{lemma:localequivalenceofstacks}
  Suppose $F \colon \X \to \Y$ is a local equivalence for $\X$ a stack and $\Y$ a CFG. Then $F$ is an equivalence of CFGs and therefore $Y$ is a stack.
\end{lemma}
\begin{proof}
  We only need to show that $F$ is essentially surjective. Suppose $Y$ is an object in $\Y_M$. Since $F$ is locally essentially surjective, we know that there exists a covering $\{ U_i \to M \}$ of $M$ and objects $X_i \in \X_{U_i}$ such that $F(X_i) = Y_i := Y|_{U_i}$. The $Y_i$ come with natural identifications $\phi_{ij} \colon Y_j \to Y_i$. We have assumed that $F$ is fully faithful so there exist corresponding identifications $\psi_{ij} \colon X_j \to X_i$.
  These identifications must satisfy the cocycle condition and (since we have assumed that $\X$ is a stack), we can conclude that there exists $X \in \X_M$ and $\psi_i \colon X|_{U_i} \to X_i$ realizing this data.

  Since $\psi_i := F(\psi_i) \colon F(X)|_{U_i} \to Y_i$ realizes the $\{Y_i \}$ descent data, we can conclude that $F(X)$ is isomorphic to $Y$.
\end{proof}

\begin{lemma}\label{lemma:prestacklocalequiv}
  If $\X$ is a pre-stack. Then the canonical map $\X \to \mathcal{S}\X$ is a local equivalence.
\end{lemma}
\begin{proof}
  Since $\X$ is a pre-stack and so morphisms in $\X$ are assumed to be determined by descent data, it follows that $\X \to \mathcal{S}\X$ is a monomorphism. That $\X$ is an epimorphism is clear since any descent data $ \{ P_i \}$ over $M$ in $\X$ relative to a covering $\{ U_i \to M \}$ can be restricted to any single component $P_i$, which is certainly in the image of $\X$.
\end{proof}
Our last lemma on the topic of stackification is the observation that it is functorial, and it is well behaved for local equivalences.
\begin{lemma}\label{lemma:stackification}
  Suppose $F \colon \X \to \Y$ be a morphism of pre-stacks. There is a canonical morphism of CFGs
  \[ \mathcal{S}F \colon \mathcal{S} \X \to \mathcal{S} \Y \]
  which makes
  \[
  \begin{tikzcd}
    \X \arrow[d, hook] \arrow[r, "F"] & \Y \arrow[d, hook] \\
    \mathcal{S} \X \arrow[r, "\mathcal{S}F"] & \mathcal{S} \Y
  \end{tikzcd}
  \]
  commute. Lastly, $\mathcal{S}F$ is an equivalence of stacks if $F$ is a local equivalence.
\end{lemma}
\begin{proof}
  Since a morphism of CFGs sends descent data to descent data, the functor $\mathcal{S}F$ is defined in the obvious manner such that it pushes forward descent data along $F$. This clearly makes the necessary diagram commute.

  Suppose $F$ is a local equivalence. For the second part, we only need to show that $\mathcal{S}F$ is a local equivalence by Lemma~\ref{lemma:localequivalenceofstacks}. Since $\Y \to \mathcal{S} \Y$ is a local equivalence and the image of $\mathcal{S}F$ includes the image of $\Y$, we can conclude that $\mathcal{S} F$ is locally essentially surjective.
  To show that $\mathcal{S}F$ is fully faithful, it suffices to show that $\Hom(\{ X_i \}, \{ X'_i \}) \to \Hom(\{ F(X_i) \}, \{ F(X'_i) \})$ is a bijection for small objects of $\mathcal{S} \X$. But this certainly holds for objects in the image of $\X \into \mathcal{S} \X$.
\end{proof}
\chapter{Dirac structures}

\section{D-Lie groupoids}

\begin{lemma}\label{lemma:dlieconstruction}
  Suppose $\G \grpd M$ is a Lie groupoid together with the following data:
  \begin{itemize}
    \item a pair of 2-forms $\sigma$ and $\tau$ on $\G$;
    \item and a $\phi$-twisted Dirac structure $L_M$ on $M$.
  \end{itemize}
  Such that, for $\Omega := \tau - \sigma$,
  \begin{itemize}
    \item $\m^* \Omega = \pr_1^* \Omega + \pr_2^* \Omega$
    \item $\t^* L_M = \s^* L_M + \Omega$
  \end{itemize}
  Then we can equip $\G$ with a $\s^* \phi - \dif \sigma$-twisted Dirac structure $L_\G := \s^* L_M - \sigma = \t^* L_M - \tau$. Furthermore, if we take the 2-forms
  $\mu$, $\iota$, and $\upsilon$ defined by the equations below to be the gauge parts of $\m$, $\i$, and $\u$ respectively. Then these maps constitute a well defined D-Lie groupoid.
  \begin{equation}
  \m^* \sigma + \mu = \pr_2^* \sigma + \pr_1^* \sigma \,
  \end{equation}
  \begin{equation}
    \i^* \sigma + \iota = \tau \,
  \end{equation}
  \begin{equation}
    \u^* \sigma + \upsilon = 0 \, \,
  \end{equation}
\end{lemma}
\begin{proof}
  We need to check that the groupoid axioms are satisfied, this amounts to checking that some diagrams in $\DMan$ commute. Since $\G \grpd M$ is already a Lie groupoid, we only need to check that the gauge equation associated to each axiom holds. In the table below, we have enumerated the axioms of a groupoid and computed the corresponding equations of 2-forms.
  \begin{center}
  \rowcolors{1}{gray!25}{white}
  \begin{tabular}{ c|c|c|c}
  \hline\noalign{\smallskip}

              & Axiom & Domain & Gauge Part  \\
  \noalign{\smallskip}\hline\noalign{\smallskip}
  (G1) & $\s \circ \u = \Id_M$ & $M$ & $\u^* \sigma + \upsilon = 0$\\
  (G2) & $\s \circ \m = \s \circ \pr_2$ &$ \G^{(2)}$& $\m^* \sigma + \mu = \pr_1^* \sigma + \pr_2^* \sigma$\\
  (G3) & $\s \circ \i = \t$&$ \G$& $\i^* \sigma + \iota = \tau$\\
  (G4) & $\i \circ \u = \u$&$ M$ &$\u^* \iota + \upsilon = \upsilon$ \\
  (G5) & $\m \circ ((\u \circ \t) \times \Id_\G) = \Id_\G$&$ \G$ & ${((\u \circ \t) \times \Id)}^* \mu = {(\u \circ \t)}^* \sigma$  \\
  (G6) & $\m \circ (\Id_\G \times (\u \circ \s)) = \Id_\G$&$ \G$& ${(\Id \times (\u \circ \s) )}^* \mu = {(\u \circ \s)}^* \tau$\\
  (G7) & $\m \circ (\i \times \Id_\G) = \u \circ \s$&$ \G$&${(\i \times \Id)}^* \mu - \i^* \sigma = \s^* \upsilon + \sigma$ \\
  (G8) & $\m \circ (\Id_\G \times \i) = \u \circ \t$&$ \G$&${(\Id \times \i)}^* \mu - \i^* \tau = \t^* \upsilon + \tau$\\
  (G9) & ${\!\begin{aligned} & \m\circ (\m (\pr_1 \times \pr_2) \times \pr_3) = \\
                             & \m \circ (\pr_1 \times \m(\pr_2 \times \pr_3))
                             \end{aligned}}$ & $\G^{(3)}$ &  see (\ref{eqn:assocgauge}) below.\\
                             \noalign{\smallskip}\hline
  \end{tabular}
  \end{center}
  The equations from (G1-G3) follow immediately by definition.
  The equation for (G4) holds since
  \[ \u^*(\iota) = \u^*(\tau - \sigma). \]
  And the pullback of a multiplicative form along $\u$ is always zero.
  The first equality follows from (G3) while the second follows from the fact that $\tau-\sigma$ is multiplicative.

  Next we show (G5) by computing directly.
  \begin{align*}
  {((\u \circ \t) \times \Id)}^* \mu  &={((\u \circ \t) \times \Id)}^*( \pr_1^* \sigma + \pr_2^* \sigma - \m^* \sigma) \\
                                    &= {(\u \circ \t)}^* \sigma + \sigma - {(\m ((\u \circ \t) \times \Id))}^* \sigma \\
                                    &= {(\u \circ \t)}^* \sigma + \sigma - \sigma = {(\u \circ \t)}^* \sigma \, .
  \end{align*}
  It follows from the multiplicativity of $\tau - \sigma$ that
  \begin{equation}
  \m^* \tau + \mu = \pr_1^* \tau + \pr_2^* \tau \, .
  \end{equation}
  By using this expression for $\mu$ we can show (G6) by a calculation essentially identical to (G5).
  Next up, we show (G7):
  \begin{align*}
  {(\i \times \Id)}^* \mu - \i^* \sigma &= {(\i \times \Id)}^* (\pr_1^* \sigma + \pr_2^* \sigma - \m^* \sigma) - \i^* \sigma \\
  &= \i^* \sigma + \sigma - {(\u \circ \s)}^* \sigma - \i^* \sigma \\
  &= -\s^* \u^* \sigma + \sigma = \s^* \iota + \sigma
  \end{align*}
  Since (G8) is similar we can proceed to (G9).
  The gauge equation for (G9) is
  \begin{equation}\label{eqn:assocgauge}
  \begin{aligned}
  {(\pr_1 \times \pr_2)}^* \mu + {(\m \circ (\pr_1 \times \pr_2) \times \pr_3)}^* \mu &= \\
  {(\pr_2 \times \pr_3)}^* \mu + {(\pr_1 \times \m \circ (\pr_2 \times \pr_3))}^* \mu & \,.
  \end{aligned}
  \end{equation}
  If we apply the substitution $\mu = \pr_1^* \sigma + \pr_2^* \sigma - \m^* \sigma$ throughout, we get:
  \begin{align*}
  \pr_1^* \sigma + \pr_2^* \sigma-{(\pr_1 \times \pr_2)}^*\m^* \sigma + {(\pr_1 \times \pr_2)}^*\m^* \sigma + \pr_3^* \sigma - \mathbf{A}_L^* \sigma &= \\
  \pr_2^* \sigma + \pr_3^* \sigma-{(\pr_2 \times \pr_3)}^*\m^* \sigma + {(\pr_2 \times \pr_3)}^*\m^* \sigma + \pr_1^* \sigma - \mathbf{A}_R^* \sigma \, . &
  \end{align*}
  Here $\mathbf{A}_L, \mathbf{A}_R: \G^{(3)} \to \G$ are the left and right hand associativity maps.
  Since $\G$ is a Lie groupoid and assumed to be associative, it follows immediately that (9) holds.

\end{proof}
\end{appendices}

\newpage\addcontentsline{toc}{chapter}{Bibliography}
\bibliography{articledb}

\end{document}